\documentclass{amsart}

\usepackage{amscd,color,tikz,latexsym,amssymb}
\usetikzlibrary{calc}
\usetikzlibrary{arrows}
\usepackage{tikz-cd}
\usepackage{array}
\usepackage[hidelinks,bookmarksnumbered]{hyperref}
\newtheorem{proposition}{Proposition}
\newtheorem{lemma}{Lemma}
\newtheorem{theorem}{Theorem}
\newtheorem{corollary}{Corollary}

\theoremstyle{definition}

\newtheorem{definition}{Definition}

\newtheorem{remark}{Remark}

\newtheorem{example}{Example}
\newtheorem{convention}{Convention}

\newcommand{\BQDILOG}{\bar{\Phi}_{\la}}
\newcommand{\cla}{c_\la}
\newcommand{\COMPLEXS}{\mathbb{C}}
\newcommand{\bC}{\COMPLEXS}
\DeclareMathOperator{\lk}{lk}
\newcommand{\FQDILOG }{\phi}
\newcommand{\Imun}{\mathsf{i}}
\newcommand{\INTEGERS}{\mathbb{Z}}
\newcommand{\bZ}{\INTEGERS}
\newcommand{\bS}{\mathbb{S}}
\newcommand{\la}{\mathsf{b}}
\newcommand{\ra}{\rightarrow}

\newcommand{\MOM}{\mathsf p}
\newcommand{\POS}{\mathsf q}
\newcommand{\QDILOG}{\Phi_{\la}}
\newcommand{\REALS}{\mathbb{R}}
\newcommand{\bR}{\REALS}

\begin{document}
\raggedbottom
\allowdisplaybreaks[1]

\title[A new formulation of the Teichm\"uller TQFT]{A new formulation of the Teichm\"uller TQFT}
\author{J{\o}rgen Ellegaard Andersen}
\address{Center for Quantum Mathematics\\
        Danish Institute for Advanced Studies\\
        University of Southern Denmark\\
        DK-5230, Odense, Denmark}
\email{jea@sdu.dk}

\author{Rinat Kashaev}
\address{University of Geneva\\
2-4 rue du Li\`evre, Case postale 64\\
 1211 Gen\`eve 4, Switzerland}
\email{rinat.kashaev@unige.ch}

\thanks{Supported in part by the center of excellence grant ``Center for quantum geometry of Moduli Spaces'', DNRF95, from the Danish National Research Foundation, and by the Swiss National Science Foundation}

\begin{abstract}
Using the Weil--Gel'fand--Zak transform of Faddeev's quantum
dilogarithm, we construct a compact state integral for positively
shaped triangulations of oriented pseudo three-manifolds. The integral
is labelled by a circle-valued first cohomology class of the vertex
complement (a ``sector'') and takes boundary values in a line bundle
over compatible boundary cocycles. We prove absolute convergence,
smoothness, invariance under shape gauge transformations and under
compatible shaped $2$--$3$ moves, branching independence, and a
single-valued global meromorphic continuation in all shape and
boundary parameters for each fixed quantum parameter. Gluing is
expressed by a conditional Haar integral which retains the final
cohomology class; the ordinary trace of integrated cut amplitudes
computes an average of the sectors over a torus of classes determined
by the cut, and is a sewing formula for every cut of a closed oriented
three-manifold, but not for arbitrary pseudo-manifolds. For closed
admissible triangulations, the original Teichm\"uller partition
function is the normalized average of the sectors; this yields its
global meromorphic continuation. A translation law relates sectors
differing by twice an extending class, with phases given by linking
pairings; for closed three-manifolds the sectors are supported in a
distinguished coset of twice the torsion group and are determined by a
single amplitude, which a nonvanishing twisted quadratic Gauss sum
recovers from the original partition function.
\end{abstract}

\maketitle

\tableofcontents
\medskip

\section{Introduction}
Topological quantum field theory was developed in axiomatic form by
Atiyah and Segal~\cite{MR1001453,MR981378}, and through Witten's
topological gauge theory~\cite{MR953828} and three-dimensional
Chern--Simons theory~\cite{Witten1989Jones}. In dimension $2+1$,
combinatorial constructions from finite-dimensional representations
of the quantum group $U_q(sl(2))$ at roots of unity were given
in~\cite{MR1091619,MR1191386}, followed by a purely topological
construction in~\cite{MR1362791}. The relation of these theories to
the geometric quantization of moduli spaces and conformal field
theory was conjectured in~\cite{Witten1989Jones}, further developed
in~\cite{TUY,ADW,H} and in a series of subsequent papers. The
equivalence of the geometric and combinatorial constructions was
verified in~\cite{AU1,AU2,AU3,AU4}; applications of this equivalence
include~\cite{A1,A2,A3,A4,A5,A6}.

Quantum Teichm\"{u}ller theory~\cite{MR1607296, MR1737362}, constructed via quantization of the geometric component of the moduli space of $PSL(2,\bR)$ flat connections on punctured surfaces of negative Euler characteristic, produces unitary representations of surface mapping class groups on  infinite-dimensional Hilbert spaces, thus providing a component of the two-dimensional part of a potential TQFT associated with Chern--Simons theory with the  non-compact gauge group $PSL(2,\bR)$. The essential  ingredients behind quantum Teichm\"{u}ller theory are Penner's decorated Teichm\"uller space, the associated Ptolemy groupoid \cite{MR919235} (see also the 
book~\cite{Penner2012} with many other applications), and Faddeev's quantum dilogarithm \cite{MR1345554}. Quantum Teichm\"{u}ller theory indeed  has been extended to a TQFT in~\cite{MR3227503} (see also \cite{AK2,AK3}) along with  the physical constructions  in~\cite{MR1865275,MR2330673,MR2551896,MR2795276,Dimofte2011}.

In this paper, we reformulate the Teichm\"uller TQFT
of~\cite{MR3227503} as a compact state-integral construction in the
same combinatorial setting. The state integral is of the Turaev--Viro
type, with the state variables living on the edges of a
triangulation, as in the model suggested
in~\cite{KashaevLuoVartanov2012}. Here, however, the state variables
take their values on the circle, the tetrahedral weights are
quasi-periodic functions of the state variables, that is, sections of
line bundles over tori, and the integrand is periodic in each
integration variable. The state integral is therefore performed over
a compact torus, and its convergence is determined only by the local
integrability of the integrand; for strictly positive shapes it
converges absolutely. Similar integrations over compact domains were
used in lattice integrable systems~\cite{MR3019403}. No homological
vanishing condition is needed for the compact integral; such a
condition enters only the comparison with the original theory.

The compact integral is labelled by a cohomology class of the vertex
complement, which we call a \emph{sector}. The main results, stated
precisely in Section~\ref{sec:main-results} below, are the following.
\begin{enumerate}
\renewcommand{\labelenumi}{\textup{(\alph{enumi})}}
\item \emph{Analytic properties.} For positive shapes the sector
function is smooth, has a holomorphic germ in all parameters, and
admits a unique single-valued meromorphic continuation in the complex
shapes, the level and the boundary data
(Theorems~\ref{main} and~\ref{thm:global-continuation}). The proof
reduces each torus integration to Rains's contour theorem by finite
torus covers and a refinement of the double-product nomes
(Theorem~\ref{thm:toric-period}).
\item \emph{Invariance.} The sector function is invariant under based
shape gauge transformations, under compatible shaped $2$--$3$ moves
and their inverses with an explicit level relation, and, by means of
canonical boundary transports, under changes of branching
(Proposition~\ref{prop:gauge-rigorous}, Proposition~\ref{2-3/3-2M},
Theorem~\ref{thm:branching}, Corollary~\ref{cor:all-branched-moves}).
\item \emph{Gluing.} Cutting along faces and gluing are described by a
conditional Haar integral which retains the cohomology class of the
glued complex (Theorem~\ref{prop:gluing}). The ordinary
trace of already integrated cut amplitudes equals the average of the
sectors over a torus $\Theta_\Sigma$ of classes determined by the cut
(Proposition~\ref{prop:trace-average}). It is therefore a sewing
formula whenever $\Theta_\Sigma=0$, which is an explicit, testable
condition on the cut, and for every cut of a triangulation of a
closed oriented three-manifold. For pseudo-manifolds the ordinary
trace can fail, also for embedded cut surfaces
(Examples~\ref{ex:trace-failure} and~\ref{ex:fig8-fibre-cut}); the
torsion-conditioned averages of Theorem~\ref{mainT}, in which the
classes of the cut complex are integrated as well, always obey the
ordinary trace formula. Thus the construction yields a gluing law for
the sector functions, but it does not by itself assert an unrestricted
cobordism functor on the integrated amplitudes of pseudo-manifolds;
this is what the term TQFT refers to in this paper.
\item \emph{Comparison with the original theory.} For closed
admissible triangulations, the original partition function
of~\cite{MR3227503} is the normalized average of the sectors
(Theorem~\ref{RelationON}); consequently it has a global meromorphic
continuation in the shapes and the level (Theorem~\ref{OldInv}).
\item \emph{Structure of the sectors.} The translations of the states
which leave all arguments of the tetrahedral weights unchanged relate
the sectors of a closed triangulation to each other by phases given
by linking pairings (Theorem~\ref{mainTranslation}). For closed
oriented three-manifolds the sectors are determined by a single
amplitude and the torsion linking form, and this amplitude is
recovered from the original partition function by a nonvanishing
twisted Gauss sum (Corollary~\ref{cor:all-torsion-sectors}).
\end{enumerate}

\subsection{Setting and notation}\label{sec:setting}
Let us now describe the setting of our new state-integral model, leaving the precise details of its construction to Sections~\ref{TP} and~\ref{DEF}.

We will assume the reader is familiar with the basic notions regarding triangulated pseudo 3-manifolds from Section 1 of \cite{MR3227503}, however to fix notation, we briefly recall a couple of the most important ones here. Let $X$ be a leveled shaped triangulation of an oriented pseudo 3-manifold.  We recall that ``leveled'' means $X$ comes with a real number $\ell_X$, which we will actually allow to be complex as discussed below. Here ``triangulation'' means that $X$ is a $\Delta$-triangulation in the sense of Hatcher \cite{MR1867354}, i.e. a CW-complex, where all cells are standard simplices with linearly ordered vertices, and all attaching  maps respect that order. 
We denote by $\Delta_i(X)$ the set of $i$-dimensional cells of $X$. We obtain a smooth compact manifold with boundary, say $X'$, if we remove a small tubular open neighbourhood of $\Delta_0(X)$. The first cohomology of this manifold with values in the circle will play an important role below.

Above, ``shaped'' means that each 3-cell $T\in\Delta_3(X)$ is provided with a shape map
\begin{equation*}
\alpha_T\colon \Delta_1(T)\to (0,\pi)
\end{equation*}
assigning  to edges of $T$ dihedral angles just as for an ideal hyperbolic tetrahedron. That means that $\alpha_T$ is constrained by the conditions
\begin{equation}\label{Spi}
\sum_{i=0}^2\alpha_T(\partial_i f)=\pi,\quad \forall f\in\Delta_2(T),
\end{equation}
which imply that the values  on opposite edges are the same and their sum at each vertex is equal to $\pi$. 
It will be essential for us to also consider complexified angles, but when we do, we will always keep the condition (\ref{Spi}).
We will denote the set of shapes for $X$ by $S(X)$ and the set of arbitrary real-valued generalised shape structures by $\tilde{S}(X)$ and further the set of complex generalised shapes by $\tilde{S}^\bC(X)$. 

Suppose now that $\Gamma$ is a subcomplex of $\Delta_1(X)$, such that all (if any) of its univalent vertices are vertices of $X$ on the boundary $\partial X$. In this paper we denote the subset of ${\tilde S}^\bC(X)$, where all internal edges, not in $\Gamma$, are balanced, by ${\tilde S}_b^\bC(X, \Gamma)$.
The graph $\Gamma$ enters only through the balancing
condition and the gauge quotient, and through the original partition
function $Z_\hbar(X,\Gamma)$ of \cite{MR3227503}, which is an invariant
of the pair $(X,\Gamma)$. 
The state integral constructed in this paper is defined for all positive
shapes, without first imposing balancing conditions associated with a
graph $\Gamma$. We denote this unrestricted integral by $F_\hbar(X)$.
For a specified pair $(X,\Gamma)$, the relevant family is obtained by
restricting to the locus where every internal edge outside $\Gamma$ is
balanced and descending to the corresponding gauge quotient, wherever
these operations are defined. Thus the formula has no additional
dependence on $\Gamma$ at fixed shapes, level, sector, and boundary data;
the restricted parameter space, however, depends on $\Gamma$.

The leveled versions of the sets $S(X)$ and $\tilde S(X)$, as introduced in Definition 4 in \cite{MR3227503}, are denoted $LS(X)$ and $\widetilde{LS}(X)$. In the complex case we allow the level to be complex and denote the resulting sets $\widetilde{LS}^\bC(X)$ and $\widetilde{LS}_b^\bC(X, \Gamma)$ respectively.

In Definitions 5 and 6 of \cite{MR3227503}, the (based) gauge
equivalence relation is defined on the sets $LS(X)$ and
$\widetilde{LS}(X)$. On these real spaces the gauge parameter is a
real function on $\Delta_1(X)$ which vanishes on the edges of
$\partial X$; in particular, internal edges belonging to $\Gamma$
are allowed. On $\widetilde{LS}^\bC(X)$ and
$\widetilde{LS}_b^\bC(X,\Gamma)$ we use the complexified action:
the gauge parameter is complex-valued, still zero on boundary
edges, and the angle and level changes are given by the same
algebraic formulae. We denote the quotients under the resulting based
equivalence relations by $LS_r(X)$, $\widetilde{LS}_r(X)$,
$\widetilde{LS}_r^\bC(X)$ and
$\widetilde{LS}_{b,r}^\bC(X,\Gamma)$ respectively, and the
corresponding quotients of the shape spaces by $S_r(X)$,
$\tilde S_r(X)$ and $\tilde S_{b,r}^\bC(X,\Gamma)$.
The edge-weight map of \cite[Theorem~1]{MR3227503} has one coordinate
for each edge of $X$, including internal edges. The balanced locus fixes
the coordinates of internal edges outside $\Gamma$ to $2\pi$;
the coordinates on $\Gamma$ and on the boundary remain free.
From Theorem 1 in \cite{MR3227503}, we have that $\widetilde{LS}_r(X)$ is an affine ${\mathbb R}$-bundle over  $\tilde{S}_r(X)$, which restricts to the affine ${\mathbb R}$-bundle $LS_r(X)$ over $S_r(X)$. It follows that $\widetilde{LS}_{b,r}^\bC(X, \Gamma)$ is an affine ${\mathbb C}$-bundle over ${\tilde S}_{b,r}^\bC(X, \Gamma)$.

Put $\mathbb S=\mathbb R/\mathbb Z$ and
$\mathbb S_{\mathbb C}=\mathbb C/\mathbb Z$. Write
\[
C_X=H^1(X',\mathbb S),\qquad
\mathcal Z_\Sigma=Z^1(\tilde\Sigma;\mathbb S),\qquad
\mathcal Q_X=\{(\alpha,\zeta)\in C_X\times\mathcal Z_{\partial X}:
[\zeta]=i_\partial^*\alpha\}.
\]
Here $\tilde\Sigma$ is the midpoint truncation described in
Section~\ref{TP}, and $i_\partial$ denotes the inclusion of the
truncated boundary $\widetilde{\partial X}$ into the truncation
$\tilde X$ of $X$, which has the homotopy type of $X'$. If
$\partial X=\emptyset$, the group $\mathcal Z_{\partial X}$ has
one element and $\mathcal Q_X=C_X$. The face multipliers define a
line bundle $\mathcal L_\Sigma$ on $\mathcal Z_{\Sigma,\mathbb C}$,
as specified in Section~\ref{DEF}. The boundary value of the state
integral is a section of the pullback of $\mathcal L_{\partial X}$ to
$\mathcal Q_X$.

\subsection{Main results}\label{sec:main-results}
Many identities in this paper are first proved for positive
shapes and then extended to complex parameters by analytic
continuation. The following convention fixes the meaning of such
statements.

\begin{convention}\label{conv:meromorphic}
Let $M$ be a connected complex manifold and $M_{\mathbb R}\subset M$
a nonempty open subset of a totally real submanifold of maximal
dimension (in this paper, $M$ is a space of complexified shapes,
levels and boundary data, and $M_{\mathbb R}$ the subset of positive
real shapes, real levels and real boundary data). An identity between
quantities defined on $M_{\mathbb R}$ which admit meromorphic
continuations to $M$ is said to \emph{persist as a meromorphic
identity} if the two continuations coincide as meromorphic sections
on $M$; by the identity theorem this is automatic once both
continuations exist. Let $N\subset M$ be a connected complex
submanifold. The \emph{meromorphic restriction} of a meromorphic
section $F$ on $M$ to $N$ is defined if $N$ is not contained in the
polar locus of $F$; in that case $F|_N$ is the meromorphic section
on $N$ obtained by restricting the holomorphic numerator and
denominator of a local representation $F=f/g$ with $g|_N\not\equiv0$.
If $N$ contains a point of $M_{\mathbb R}$ at which $F$ is holomorphic,
then $N$ is not contained in the polar locus, and the restriction is
defined. Balanced shape subspaces, and the based gauge quotients of
Section~\ref{sec:setting}, are always understood in this sense.
\end{convention}

For shapes which are not positive, the continuation of a state
integral is in general not the integral of the continued integrand
over the original real cycle; all meromorphic identities below are
identities between continuations in the sense of
Convention~\ref{conv:meromorphic}.

\begin{theorem}\label{main}
Fix $\hbar>0$ and a finite branched triangulation $X$ as above.
For positive shapes and a real level there is a well-defined smooth
section
\begin{equation}\label{eq:si}
\mathcal F_\hbar(X)\in
C^\infty\bigl(LS(X)\times\mathcal Q_X,\mathcal L_{\partial X}\bigr)
\end{equation}
of the pullback of $\mathcal L_{\partial X}$, the \emph{sector
function} of $X$, with the following properties:
\begin{enumerate}
\renewcommand{\labelenumi}{\textup{(\roman{enumi})}}
\item \emph{Regularity and continuation.} It has a holomorphic germ in
the complexified shapes, level, cohomology class and compatible
boundary states near every point of the real domain, and a unique
single-valued meromorphic extension to
$\widetilde{LS}^{\mathbb C}(X)\times\mathcal Q_{X,\mathbb C}$
(Lemmas~\ref{lem:weight-regularity} and~\ref{lem:holomorphic-integral},
Theorem~\ref{thm:global-continuation}).
\item \emph{Multiplicativity and branching.} It is multiplicative
under disjoint union. It is independent of the branching of $X$ up to
the canonical boundary transports of Theorem~\ref{thm:branching},
which preserve each cohomology sector and the level; for closed $X$
the functions $\mathcal F_\hbar(X)(\alpha)$ are independent of the
branching.
\item \emph{Gauge invariance.} It is invariant under based shape
gauge transformations between positive shapes
(Proposition~\ref{prop:gauge-rigorous}).
\item \emph{Move invariance.} It is invariant under every compatible
shaped $2$--$3$ move and its inverse, provided both sides have
positive shapes and the levels $\ell_2$ and $\ell_3$ of the
triangulations with two and with three tetrahedra in the bipyramid
satisfy $\ell_3-\ell_2=P_e/3$, where $P_e$ is the linear combination
\eqref{eq:Pe} of the dihedral angles of the three tetrahedra, which
depends only on the shapes of the bipyramid
(Proposition~\ref{2-3/3-2M}, Corollary~\ref{cor:all-branched-moves}).
\end{enumerate}
The identities in \textup{(ii)}--\textup{(iv)} persist as
meromorphic identities in the sense of
Convention~\ref{conv:meromorphic}.
\end{theorem}

A shaped $2$--$3$ move is called \emph{compatible} if
the two sides have the same total dihedral angles at the nine boundary
edges of the bipyramid, so that the new edge is balanced, and if their
branchings agree on its boundary, that is, if both branchings are
induced by one total order of the five vertices. The theorem does not
assert that an arbitrary pair of triangulations can be joined by
these moves through positive shapes. The global continuation in
\textup{(i)} does not require a positive balanced shape: the initial
domain is the unrestricted positive shape space. On the gauge
quotient, the theorem applies wherever the balanced meromorphic
restriction and quotient are defined (Convention~\ref{conv:meromorphic});
a positive balanced point guarantees that the balanced locus is not
contained in the polar set.

The gluing properties of the sector function are formulated in
Section~\ref{sec:cutting} and are summarized as follows. Cutting $X$
along a union $\Sigma$ of faces produces a complex $X^c$ whose
boundary contains two copies of $\Sigma$. The sector function of $X$
is obtained from the integrand of $X^c$ by a conditional Haar
integral which retains the cohomology class of $X$
(Theorem~\ref{prop:gluing}). The ordinary trace of the already
integrated sector function of $X^c$ equals the average of
$\mathcal F_\hbar(X)$ over a coset $\alpha+\Theta_\Sigma$ of a torus
$\Theta_\Sigma\subset C_X$ determined by the cut
(Proposition~\ref{prop:trace-average}); it equals
$\mathcal F_\hbar(X)(\alpha)$ when $\Theta_\Sigma=0$, which is the
explicit kernel condition \eqref{eq:gluing-kernel-condition}, and for
every cut of a triangulation of a closed oriented three-manifold,
but not in general for pseudo-manifolds
(Examples~\ref{ex:trace-failure} and~\ref{ex:fig8-fibre-cut}).

Let $T_X=\operatorname{Tor}H_1(X',\mathbb Z)$ and
$\mathcal T_X=\operatorname{Hom}(T_X,\mathbb S)$. Restriction of
characters gives a canonical map $r_T:C_X\to\mathcal T_X$ with
connected-torus kernel. No splitting of integral homology is chosen.
Write $\mathcal Q^T_X=(r_T\times\operatorname{id})(\mathcal Q_X)
\subset\mathcal T_X\times\mathcal Z_{\partial X}$; for closed $X$
this is $\mathcal T_X$.

\begin{theorem}\label{mainT}
On positive shapes there is a smooth, conditionally averaged boundary
section
\begin{equation}\label{eq:si-torsion}
\mathcal F^T_\hbar(X)(\tau,\zeta)
=\int_{\substack{r_T(\alpha)=\tau\\i_\partial^*\alpha=[\zeta]}}
\mathcal F_\hbar(X)(\alpha,\zeta)\,d\alpha,
\qquad(\tau,\zeta)\in\mathcal Q^T_X,
\end{equation}
where the measure is the conditional Haar probability measure on the
nonempty compact fibre $\{\alpha:r_T(\alpha)=\tau,\
i_\partial^*\alpha=[\zeta]\}$, which may be disconnected when
$\partial X\neq\emptyset$; for closed $X$ it is the entire torus
$r_T^{-1}(\tau)$. It extends uniquely to a single-valued meromorphic
section over the complexification of $\mathcal Q^T_X$ and has the
regularity, branching, gauge and move invariance properties
\textup{(i)}--\textup{(iv)} of Theorem~\ref{main}. Its gluing law is
the torsion-conditioned version of
Theorem~\ref{prop:gluing}; in that version the conditioned fibre is
always the full fibre of the cut complex, so that for every cut
$\mathcal F^T_\hbar(X)$ is the ordinary trace of the sector amplitudes
of the cut complex over its compatible classes and interface states
(Proposition~\ref{prop:trace-average}\textup{(iii)}).
\end{theorem}

Following \cite{MR3227503}, we call $X$ \emph{admissible} if
$H_2(X',\mathbb Z)=0$. This is the condition under which the original
partition function of \cite{MR3227503} is defined. For a connected
closed three-manifold it means that $X$ has one vertex and that
$b_1(X)=0$ (see the proof of Corollary~\ref{cor:all-torsion-sectors}).

\begin{theorem}\label{RelationON}
Suppose $\partial X=\emptyset$, $H_2(X',\mathbb Z)=0$, and the shapes
are positive, so that the original partition function
$Z_\hbar(X,\Gamma)$ of \cite{MR3227503}, recalled in
Section~\ref{proofTI} with the normalization fixed in
Remark~\ref{rem:normalization}, is defined. Then
\begin{equation}\label{ON-rel}
Z_\hbar(X,\Gamma)=\frac1{|\mathcal T_X|}
\sum_{\tau\in\mathcal T_X}\mathcal F^T_\hbar(X)(\tau),
\end{equation}
where the right-hand side is evaluated at the same shapes and level.
The right-hand side gives the unique global meromorphic continuation
of the left-hand side in the complex shape and level variables.
\end{theorem}

The sectors of a closed triangulation are not independent of each
other. Let $j\colon X'\hookrightarrow X$ be the inclusion, and call a
class in $C_X$ \emph{extending} if it belongs to
$j^*H^1(X,\mathbb S)$. In Section~\ref{sec:sector-translation} we
define, by explicit cochain formulae, two bi-additive pairings
\[
\lk_X\colon H^1(X,\mathbb S)\times H^1(X,\mathbb S)\to\mathbb S,
\qquad
\lk'_X\colon H^1(X,\mathbb S)\times C_X\to\mathbb S,
\]
which vanish as soon as one of their arguments lies in an identity
component and which satisfy
$\lk'_X(\beta,j^*\beta')=\lk_X(\beta',\beta)$.
For a closed three-manifold, $\lk_X$ is the torsion linking pairing
of the Bockstein classes.

\begin{theorem}\label{mainTranslation}
Suppose $\partial X=\emptyset$. For all $\beta\in H^1(X,\mathbb S)$
and $\alpha\in C_X$, at the same shapes and level,
\begin{equation}\label{eq:intro-translation}
\mathcal F_\hbar(X)(\alpha+2j^*\beta)
=e^{-2\pi\Imun(\lk_X(\beta,\beta)+\lk'_X(\beta,\alpha))}\,
\mathcal F_\hbar(X)(\alpha).
\end{equation}
In particular
$\mathcal F_\hbar(X)(-\alpha)=\mathcal F_\hbar(X)(\alpha)$
for every extending class $\alpha$. If $X$ is a closed oriented
three-manifold, then every class extends and, identifying $C_X$ with
$H^1(X,\mathbb S)$ by $j^*$:
\begin{enumerate}
\renewcommand{\labelenumi}{\textup{(\roman{enumi})}}
\item $\mathcal F_\hbar(X)(\alpha)$ depends only on the
connected component of $\alpha$, that is, on $r_T(\alpha)$;
\item $\mathcal F_\hbar(X)(\alpha)$ vanishes identically unless
$\lk_X(\beta,\beta)+\lk_X(\alpha,\beta)=0$ for all $\beta$ with
$2\beta=0$;
\item if the torsion subgroup of $H_1(X,\mathbb Z)$ has odd order,
then every class is of the form $2\beta$ and
$\mathcal F_\hbar(X)(2\beta)
=e^{-2\pi\Imun\lk_X(\beta,\beta)}\mathcal F_\hbar(X)(0)$.
\end{enumerate}
\end{theorem}

This is proved in Section~\ref{sec:sector-translation}; see
Theorem~\ref{thm:sector-translation} and
Corollaries~\ref{cor:sector-reversal}--\ref{cor:closed-manifold-sectors}.
In the situation of (iii), and if moreover $H_2(X',\mathbb Z)=0$, formula
\eqref{ON-rel} becomes
\begin{equation}\label{eq:intro-gauss}
Z_\hbar(X,\Gamma)=\mathcal F_\hbar(X)(0)\,
\frac1{|H|}\sum_{\beta\in H}e^{-2\pi\Imun\lk_X(\beta,\beta)},
\qquad H=H^1(X,\mathbb S),
\end{equation}
that is, the trivial sector times a quadratic Gauss sum of the linking
pairing. Corollary~\ref{cor:all-torsion-sectors} extends this conclusion
to closed oriented three-manifolds with arbitrary torsion. Write
$A_X=\pi_0C_X$, which $r_T$ identifies with
$\mathcal T_X$, let $A_X[2]=\{k\in A_X:2k=0\}$ be its $2$-torsion
subgroup, and let $\lk_X$ also denote the induced linking pairing
on $A_X$. The solutions $\tau_*\in A_X$ of
$\lk_X(\tau_*,k)+\lk_X(k,k)=0$ for every $k\in A_X[2]$ form exactly one
coset $\tau_*+2A_X$, the distinguished coset. The sectors outside it
vanish, and those within it are all determined by one amplitude and
explicit linking phases. If $X$ is admissible, then
\begin{equation}\label{eq:intro-twisted-gauss}
Z_\hbar(X,\Gamma)
=\frac{\mathcal F^T_\hbar(X)(\tau_*)}{|A_X|\,|A_X[2]|}
\sum_{\theta\in A_X}
e^{-2\pi\Imun(\lk_X(\theta,\theta)+\lk_X(\tau_*,\theta))}.
\end{equation}
This twisted Gauss sum is nonzero, so the original partition function
determines all sectors, also when the torsion has even order. The
untwisted sum
$\sum_{\theta\in A_X}e^{-2\pi\Imun\lk_X(\theta,\theta)}$, which occurs
in \eqref{eq:intro-gauss} for odd torsion, vanishes by contrast
whenever the distinguished coset does not contain zero, as for
$\mathbb{R}P^3$.

For a general closed pseudo-manifold the reversal symmetry is
only proved for extending classes;
Example~\ref{ex:sector-reversal-obstruction} shows that it can fail at
singular vertices. The example of $\mathbb{R}P^3$ in
Section~\ref{sec:rp3-example} illustrates the selection rule (ii) of
Theorem~\ref{mainTranslation}, and the examples which follow it in
Section~\ref{sec:examples} illustrate (i), (iii) and the behaviour at
singular vertices.

Formula \eqref{ON-rel} resembles the conjectural decomposition in
\cite[Eq.~(2.6)]{Gco}, whose labels are torsion classes modulo the
involution $\alpha\mapsto-\alpha$ and which also involves the linking
pairing. The resemblance is only formal at this stage. The
labels of \cite{Gco} are elements of $\operatorname{Tor}H_1(X,\mathbb
Z)/\{\pm1\}$, and our sectors are labelled by the dual group
$\mathcal T_X=\operatorname{Hom}(\operatorname{Tor}H_1(X,\mathbb
Z),\mathbb S)$, on which the reversal symmetry of
Corollary~\ref{cor:sector-reversal} induces the same involution; the
nondegenerate linking form identifies the two groups. On the other
hand, by Corollary~\ref{cor:all-torsion-sectors} the sectors of a
closed three-manifold have at most one independent amplitude as
functions of the shapes, and they vanish outside the distinguished
coset, whereas the summands in \cite[Eq.~(2.6)]{Gco} are expected to
be independent functions. A relation, if any, with the invariants proposed in
\cite{Gco} is not established here.

For a one-vertex triangulation of an integral homology sphere $X$,
$H^1(X',\mathbb S)$ is trivial. Thus Theorem~\ref{RelationON}
specializes to the following identity.
\begin{theorem}\label{integral-homology-sphere-comparison}
For a graph $\Gamma$ in an integral homology sphere $X$ with a one-vertex triangulation for which the original partition function is defined, we have that
$$ Z_{\hbar}(X,\Gamma) =\mathcal F_\hbar(X)(0).$$
\end{theorem}

The comparison with the original theory has the following precise
analytic consequence.

\begin{theorem}\label{OldInv}
For $\partial X=\emptyset$ and $H_2(X',\mathbb Z)=0$, the original
positive-shape partition function has a unique global meromorphic
continuation in the unreduced complex shape and level variables, for
fixed physical $\hbar>0$. These continuations satisfy the move identities
of Theorem~\ref{main}\textup{(iv)} for all compatible branchings and
the stated level relation; they are also independent of branching.
Balanced restrictions are taken in the sense of
Convention~\ref{conv:meromorphic}. The continuation is given by
\eqref{ON-rel}.
\end{theorem}

\subsection{Relation to earlier work}\label{sec:earlier}
A short version of this paper appeared as a preprint in 2013
(arXiv:1305.4291). It introduced the compact state integral and proved
the pentagon identity of Proposition~\ref{PI} and the invariance under
shaped $2$--$3$ moves for closed triangulations; the equivalence with
the original formulation was deferred. The edge--face transformation
of Section~\ref{sec:edge-face-local} (Lemma~\ref{eft}) was stated,
with an indication of its proof, as Theorem~5 of our ICM
survey~\cite{AKICM}. Everything else in the present paper is new: the
formulation with boundary line bundles and cohomology sectors
(Section~\ref{DEF}), the regularity and continuation theory
(Sections~\ref{sec:analytic-germs}--\ref{sec:extension}), branching
independence, the gluing law and the translation law
(Section~\ref{proof}), the comparison with the original theory
(Section~\ref{ONrel}) and the examples (Section~\ref{sec:examples}).
The state integral of~\cite{KashaevLuoVartanov2012} is likewise built
from Faddeev's quantum dilogarithm on shaped triangulations, with real
state variables on the edges; it is a different model, defined
without the Weil--Gel'fand--Zak transform, and we do not know a direct
relation between the two beyond the common tetrahedral building block.
The decomposition of Theorem~\ref{RelationON} into sectors labelled by
torsion classes resembles the conjectural decomposition in
\cite[Eq.~(2.6)]{Gco}; see the discussion after
Theorem~\ref{mainTranslation}.

\subsection{Proof outline and notation}
The construction in Section~\ref{DEF}, the local identities in
Section~\ref{PId}, and the analytic results in
Sections~\ref{sec:analytic-germs}--\ref{sec:extension} establish the
integral and its continuation. Section~\ref{proof} proves the move,
branching and conditional gluing statements and computes the ordinary
trace of integrated cut amplitudes
(Proposition~\ref{prop:trace-average}). The proofs of
Theorems~\ref{main} and~\ref{mainT} are assembled, clause by clause,
in Section~\ref{sec:proofs-main}. Theorem~\ref{RelationON} is the
comparison proved in Section~\ref{ONrel};
Theorem~\ref{integral-homology-sphere-comparison} is its
integral-homology-sphere specialization, and Theorem~\ref{OldInv}
combines that comparison with the global continuation theorem.
Independently of the comparison, Section~\ref{sec:sector-translation}
proves the cochain translation law of Theorem~\ref{mainTranslation};
its combination with Theorem~\ref{RelationON} yields the Gauss-sum
reconstruction in Corollary~\ref{cor:all-torsion-sectors}.

Examples~\ref{ex:separating-cut} and~\ref{ex:nonseparating-cut}
illustrate the kernel condition for sewing, and
Examples~\ref{ex:trace-failure} and~\ref{ex:fig8-fibre-cut} show that
the ordinary trace can fail without it. The computations in
Sections~\ref{sec:rp3-example}--\ref{sec:wedge-example} illustrate
the dependence on the sector: vanishing, linking phases, constancy
along a torus, and nonconstant families at singular vertices, as
well as the distinction between graph exteriors and vertex complements.
For reference, the principal notation is collected in
Table~\ref{tab:notation}.

\begin{table}[t]
\centering
\small
\renewcommand{\arraystretch}{1.12}
\begin{tabular}{@{}>{\raggedright\arraybackslash}p{0.29\textwidth}
                  >{\raggedright\arraybackslash}p{0.66\textwidth}@{}}
\hline
Notation & Meaning \\
\hline
$\mathbb S,\mathbb S_{\mathbb C}$ &
The groups $\mathbb R/\mathbb Z$ and $\mathbb C/\mathbb Z$. \\
$X',\tilde X$ &
The vertex complement and its midpoint-truncation deformation retract. \\
$a,b,c$; $\la,\hbar$ &
Normalized tetrahedral angles with $a+b+c=1/2$; the fixed quantum
parameters, with $\hbar^{-1}=(\la+\la^{-1})^2$. \\
$S(X),LS(X)$ &
Positive shapes and shapes with a level; a tilde allows generalized
shapes, and the superscript $\mathbb C$ denotes complexification. \\
Subscripts $b,r$; $\Gamma$ &
Balancing outside $\Gamma$ and the based gauge quotient;
$\Gamma$ specifies the internal edges allowed to be unbalanced. \\
$C_X$ & $H^1(X',\mathbb S)$, the group of sector labels. \\
$T_X,\mathcal T_X$ &
$\operatorname{Tor}H_1(X',\mathbb Z)$ and its character group
$\operatorname{Hom}(T_X,\mathbb S)$. \\
$r_T,A_X$ &
Restriction of characters to $T_X$, and
$A_X=\pi_0C_X\cong\mathcal T_X$. \\
$\mathcal Z_X,q_X$ &
$Z^1(\tilde X;\mathbb S)$ and the surjective map
$z\mapsto([z],z|_{\widetilde{\partial X}})$ of
Lemma~\ref{lem:boundary-surjection}. \\
$\mathcal Q_X,\mathcal L_\Sigma$ &
Compatible pairs $(\alpha,\zeta)$ with $[\zeta]=i_\partial^*\alpha$,
and the boundary line bundle over cocycles on $\tilde\Sigma$. \\
$\mathcal F_\hbar,\mathcal F^T_\hbar,Z_\hbar$ &
The sector integral, its conditional torsion average, and the
original partition function. \\
$\Theta_\Sigma$ &
The torus of classes over which the ordinary trace along a cut
$\Sigma$ averages the sectors. \\
$\lk_X,\lk'_X$ &
The linking pairings on extending classes and between extending
classes and arbitrary sectors, respectively. \\
\hline
\end{tabular}
\medskip
\caption{Principal notation.}\label{tab:notation}
\end{table}

\section{Topological preliminaries}

\label{TP}

Adopting the notation from the introduction, we have that $X$ is a leveled shaped triangulation of an oriented pseudo $3$-manifold, which we assume is connected in the notation below. All constructions
extend componentwise; in the disconnected case one quotients by
$\mathbb S^{\pi_0(X)}$, rather than by a single constant circle. Let now 
$$ Y = \bigsqcup_{T\in \Delta_3(X)} T$$
and $p : Y \ra X$ be the natural map we obtain when gluing the tetrahedra of $Y$ together to form $X$.

We recall that $X'$ is the complement of an open tubular neighbourhood of $\Delta_0(X)$ in $X$.

For a $\Delta$-triangulation $D$ (in particular $D=X$ or $Y$), let
$\tilde D$ be obtained by truncating each tetrahedron to the octahedron
whose vertices are its six edge midpoints, as in Figure~\ref{fig:truncation}.
\begin{figure}[htbp]
\centering
\begin{tikzpicture}[scale=.8]
\foreach \shift/\stage in {0/0,4.6/1,9.2/2}{
\begin{scope}[xshift=\shift cm]
\coordinate (v0) at (0,2.5); \coordinate (v1) at (-1.8,-.6);
\coordinate (v2) at (1.8,-.6); \coordinate (v3) at (.6,.6);
\foreach \i/\j in {0/1,0/2,0/3,1/2,1/3,2/3}
  \coordinate (m\i\j) at ($(v\i)!.5!(v\j)$);
\ifnum\stage<2
  \draw[gray] (v0)--(v1)--(v2)--cycle;
  \draw[gray,dashed] (v0)--(v3)--(v1) (v3)--(v2);
\fi
\ifnum\stage>0
  \fill[blue!8] (m01)--(m02)--(m03)--cycle;
  \fill[blue!8] (m01)--(m12)--(m13)--cycle;
  \fill[blue!8] (m02)--(m12)--(m23)--cycle;
  \fill[blue!8] (m03)--(m13)--(m23)--cycle;
  \draw[blue,thick] (m01)--(m02)--(m12)--cycle;
  \draw[blue,thick] (m01)--(m03)--(m02);
  \draw[blue,thick] (m01)--(m13)--(m12)--(m23)--(m02);
  \draw[blue,thick,dashed] (m03)--(m13)--(m23)--cycle;
\fi
\end{scope}}
\node at (2.3,.8) {$\longrightarrow$};
\node at (6.9,.8) {$\longrightarrow$};
\node at (0,-1.1) {$D$}; \node at (9.2,-1.1) {$\tilde D$};
\end{tikzpicture}
\caption{A tetrahedron and its midpoint truncation, in perspective.
Dashed edges are hidden; the four shaded triangles are the new
vertex-truncation faces.}
\label{fig:truncation}
\end{figure}
In Figure~\ref{fig:truncation}, the thick blue edges bound the octahedron;
the shaded vertex-truncation faces form part of $\partial\tilde D$. By construction, the complex $\tilde D$ is a deformation retract of $D\setminus \Delta_0(D)$ so that, in particular, $D\setminus \Delta_0(D)$ and $\tilde D$ have isomorphic homology and fundamental groups, though not generally relative boundary homology groups. Moreover, we have a canonical identification 
\begin{equation}\label{eq:d0d1eq}
\Delta_0(\tilde D)=\Delta_1(D),
\end{equation}
where each vertex of $\tilde D$ is, by construction, the mid-point of the corresponding edge of $D$. We observe that for every edge $e \in \Delta_1(\tilde X)$, which is not on the boundary of $X$, there are two edges $e_+, e_- \in \Delta_1(\tilde Y)$, which project to $e$ in $X$. If $e$ is on the boundary of $X$ we use the notation $e_+, e_- \in \Delta_1(\tilde Y)$ for one and the same edge in $\tilde Y$ mapping to $e$ to simplify formulae. We denote the corresponding quotient map $\tilde p : \tilde Y \ra \tilde X$ again obtained when we glue the octahedra of $\tilde Y$ to obtain $\tilde X$.

We now define 
\begin{eqnarray}\label{OmD}
\Omega(X) & = &\left\{y \in {\mathbb S}^{\Delta_0({\tilde Y})} \mid  \delta_{\tilde Y}(y)(e_+) = \delta_{\tilde Y}(y)(e_-) \ \  \forall  e\in \Delta_1(\tilde X)
\right\}
\end{eqnarray}
and consider the map
$$ H'_X : \Omega(X) \ra  {\mathbb S}^{\Delta_1(\tilde X)}$$
given by
$$ H'_X(y)(e) = \delta_{\tilde Y}(y)(e_+) = \delta_{\tilde Y}(y)(e_-)$$
for all $e\in \Delta_1(\tilde X)$.
Since we see that $\delta_{\tilde X} (H'_X(y))=0$, $H'_X$ induces a map
$$ H_X : \Omega(X) \ra H^1(\tilde X, {\mathbb S}).$$
Let 
$$ {\mathbb S}(X) =  \bigl({\mathbb S}^{\Delta_3(X)}\times {\mathbb S}^{\Delta_0(\tilde X)}\bigr)/{\mathbb S}$$
where ${\mathbb S}$ is embedded anti-diagonally, as the pairs $(c,-c)$ of constant maps on $\Delta_3(X)$ and $\Delta_0(\tilde X)$.
\begin{proposition}\label{Omega}
There is a natural free action of ${\mathbb S}(X)$ on $\Omega(X)$ giving $H_X$ the structure of a principal ${\mathbb S}(X)$-bundle over $H^1(\tilde X, {\mathbb S})$.
\end{proposition}

\begin{proof}
Suppose $y_1,y_2\in \Omega(X)$ such that 
$$H_X(y_1) = H_X(y_2).$$
Then there exists $w \in {\mathbb S}^{\Delta_0(\tilde X)}$ such that $H'_X(y_1)- H'_X(y_2) = \delta_{\tilde X}(w)$. So then
$$
\delta_{\tilde Y}(y_1)(e) - \delta_{\tilde Y}(y_2)(e) = (\tilde p^*\delta_{\tilde X}w)(e),
$$
for all $e \in \Delta_1(\tilde Y)$ and therefore
$$ \delta_{\tilde Y}(y_1 -y_2 - \tilde p^* w) = 0.$$
But this equation means that $y_1 -y_2 - \tilde p^* w$ is constant on each connected component of $\tilde Y$, thus there exists $c\in {\mathbb S}^{\Delta_3(Y)}$ such that
$$y_1 -y_2 - \tilde p^* w = \tilde q^*c,$$
where 
$$\tilde q : \Delta_0(\tilde Y) \ra \Delta_3(Y)$$
is the map which associates to a vertex in $\Delta_0(\tilde Y)$ the tetrahedron in $\Delta_3(Y)$ to which it belongs.
So this way we see that ${\mathbb S}^{\Delta_3(Y)} \times {\mathbb S}^{\Delta_0(\tilde X)}$ acts transitively on the fibres of $H_X$. The stabilizer is given by the pairs of constant maps defined on $\Delta_3(Y)$ and $\Delta_0(\tilde X)$ respectively, whose values are each other's negatives. 

Suppose now $u\in {\mathbb S}^{\Delta_1(\tilde X)}$ represents a class $[u]\in H^1(\tilde X, {\mathbb S})$. We then consider $\tilde p^*u \in {\mathbb S}^{\Delta_1(\tilde Y)}.$ But we have of course that $\delta_{\tilde Y}(\tilde p^*u )=0$ and since $H^1(\tilde Y, {\mathbb S}) = 0$, there exists $y\in {\mathbb S}^{\Delta_0(\tilde Y)}$ such that
$$ \delta_{\tilde Y}(y) = \tilde p^*u.$$
Thus we have established the proposition.
\end{proof}

Suppose now that ${\underline X}$ is obtained from $X$ by a $2 - 3$ move (and therefore $X$ is obtained from ${\underline X}$ by a $3 - 2$ move) as indicated locally in Figure~\ref{23}. We denote the two tetrahedra which are involved in the move in $X$ by $T'_1$ and $T'_2$ and the three in ${\underline X}$ by $T'_3,T'_4$ and $T'_5$.

We shall now construct a map 
$$ M : \Omega({\underline X}) \ra \Omega(X)$$
which fits in the diagram
$$
\begin{tikzcd}
\Omega({\underline X}) \arrow[r,"M"] \arrow[d,swap,"H_{{\underline X}}"] & \Omega(X) \arrow[d,"H_X"] \\
H^1(\tilde {\underline X}, {\mathbb S}) \arrow[r,"\cong"] & H^1(\tilde X, {\mathbb S})
\end{tikzcd}
$$
and preserves the fibres. As proved below, $M$ is itself a principal
$\mathbb S^2$-bundle. This map will give invariance under both the
$2$--$3$ move and its inverse.
 
Write $B$ for the bipyramid on which the move is performed. The common
boundary of the two subdivisions identifies all cocycle values on its
edges, including edges of faces lying in $\partial X$. These values are
read from $H'_{\underline X}(\underline y)$, not from tetrahedra
outside $B$. The truncated bipyramid has trivial first cohomology,
so this cocycle is exact. Hence there is one value $x_e\in\mathbb S$
at each geometric edge midpoint of the three-tetrahedron bipyramid,
and constants $c_3,c_4,c_5$, such that
\[
\underline y|_{T'_j}=x|_{\Delta_1(T'_j)}+c_j,\qquad j=3,4,5.
\]
The common ambiguity is removed, for example, by fixing $x_{e_*}=0$
at one boundary edge. In particular, the eighteen differences on the
six boundary triangles satisfy the cocycle conditions on the
vertex-truncation triangles of the two-tetrahedron side: around each
of the five vertices these are sums of the conditions on the
three-tetrahedron side, or equivalently telescoping differences of
$x$. Thus they extend uniquely to an exact cocycle on each old
truncated tetrahedron. Its potentials are $x$ restricted to the old
edges, up to one constant per tetrahedron. This also explains why no
additional assumption is needed when some faces of $B$ lie in
$\partial X$.

To fix the constants, choose a common boundary edge of $T'_1$ and $T'_4$
and a common boundary edge of $T'_2$ and $T'_5$. If $q_i$ and $\underline q_i$
are their midpoint copies in these tetrahedra, impose
$y(q_i)=\underline y(\underline q_i)$, $i=1,2$. In the preceding
coordinates this says
\[
y|_{T'_1}=x|_{\Delta_1(T'_1)}+c_4,\qquad
y|_{T'_2}=x|_{\Delta_1(T'_2)}+c_5.
\]
Thus $M$ forgets the value $x_{ab}$ on the new internal edge and
the constant $c_3$, while retaining $c_4,c_5$ as the two old
tetrahedron constants. These formulas are unchanged by a common
translation of $x$ and the opposite translation of all $c_j$.
Outside $B$ the states are unchanged, and $M$ induces the displayed
identification in cohomology.

\begin{figure}[htbp]
\centering
\begin{tikzpicture}[scale=.93,line join=round,line cap=round,
  every node/.style={font=\small}]
\foreach \shift/\stage in {0/0,5.65/1}{
\begin{scope}[xshift=\shift cm]
\coordinate (a) at (.28,2.02);
\coordinate (b) at (.28,-1.68);
\coordinate (p) at (-1.45,-.15);
\coordinate (q) at (1.45,-.6);
\coordinate (r) at (1.25,.55);
\ifnum\stage=0
  \fill[blue!65,fill opacity=.16] (p)--(q)--(r)--cycle;
\else
  \fill[blue!70,fill opacity=.17] (a)--(b)--(p)--cycle;
  \fill[orange!85!black,fill opacity=.16] (a)--(b)--(q)--cycle;
  \fill[green!55!black,fill opacity=.17] (a)--(b)--(r)--cycle;
\fi
\draw[black!75] (a)--(p)--(b)--(q)--(r)--(a);
\draw[black!65,dashed] (a)--(q) (b)--(r) (p)--(r);
\draw[black!75] (p)--(q);
\ifnum\stage=0
  \draw[black!75,line width=.65pt] (a)--(q);
  \draw[blue!70!black] (p)--(q)--(r);
  \draw[blue!70!black,dashed] (p)--(r);
  \draw[gray!65,thin] (-.18,.80)--(-1.70,1.20);
  \draw[gray!65,thin] (-.18,-.94)--(-1.90,-1.05);
  \node[left] at (-1.70,1.20) {$T'_1$};
  \node[left] at (-1.90,-1.05) {$T'_2$};
  \node at (0,-2.35) {};
\else
  \draw[blue!70!black,line width=.65pt] (a)--(p)--(b);
  \draw[orange!80!black,line width=.65pt] (a)--(q)--(b);
  \draw[green!45!black,line width=.65pt,dashed] (a)--(r)--(b);
  \coordinate (crossing) at (intersection of a--b and p--q);
  \draw[violet!75!black,line width=1.3pt]
    (a)--($(crossing)+(0,.085)$)
    ($(crossing)-(0,.085)$)--(b);
  \draw[black!75] (p)--(q);
  \draw[gray!65,thin] (-.60,-.69)--(-1.90,-1.05);
  \coordinate (t4in) at (.88,.03);
  \coordinate (t4out) at (1.95,.12);
  \coordinate (t5in) at (-.31,.82);
  \coordinate (t5out) at (-1.70,1.20);
  \coordinate (t4boundary) at (intersection of t4in--t4out and r--q);
  \coordinate (t5boundary) at (intersection of t5in--t5out and a--p);
  \draw[gray!65,thin] (t4out)--(t4boundary) (t5out)--(t5boundary);
  \draw[gray!65,thin,dash pattern=on 1.5pt off 1.2pt]
    (t4boundary)--(t4in) (t5boundary)--(t5in);
  \node[left] at (-1.90,-1.05) {$T'_3$};
  \node[right] at (1.95,.12) {$T'_4$};
  \node[left] at (-1.70,1.20) {$T'_5$};
  \node at (0,-2.35) {};
\fi
\foreach \v in {a,b,p,q,r}{\fill (\v) circle (1.2pt);}
\node[above=3pt] at (a) {$a$};
\node[below=3pt] at (b) {$b$};
\node[left=3pt] at (p) {$p$};
\node[right=3pt] at (q) {$q$};
\node[above right=3pt] at (r) {$r$};
\end{scope}}
\node at (2.82,.08) {$\longleftrightarrow$};
\end{tikzpicture}
\caption{The $2$--$3$ move.}
\label{23}
\end{figure}

The map is equivariant for the homomorphism
$\mathbb S(\underline X)\to\mathbb S(X)$ induced by restriction of the
edge coordinates to the old edges, and by the map on tetrahedron constants
\[
(c_{T'_3},c_{T'_4},c_{T'_5})\longmapsto(c_{T'_4},c_{T'_5}).
\]
Indeed, the chosen midpoint copies project to the same old edge, so
their edge translations agree; their tetrahedron translations are
precisely $c_{T'_4}$ and $c_{T'_5}$. This homomorphism is surjective and
descends through the anti-diagonal constant subgroup. Its kernel is
\[
\mathbb S(X,\underline X)\cong\mathbb S^2,
\]
generated by the constant on $T'_3$ and the state at the midpoint $v_0$
of the new internal edge. These act trivially on the common boundary
differences and on the two chosen normalising values. Since the
cohomology map is an isomorphism and the structure-group map is
surjective, equivariance on each torsor proves that $M$ is surjective
and its fibres are exactly these kernel orbits.

\begin{proposition}
The map $M$ is a map of principal bundles and is itself a principal
$\mathbb S(X,\underline X)$-bundle over $\Omega(X)$.
\end{proposition}

We remark that there are natural analogs of $\Omega(X)$, $H_X$ and $M$ if we replace ${\mathbb S}$ by ${\mathbb S}_{\mathbb C}$, which we will denote $\Omega_{\mathbb C}(X)$, $H_{X,{\mathbb C}}$ and $M_{\mathbb C}$.

Suppose $\Sigma$ is a triangulated oriented surface and that
$\Xi$ is the disjoint union of its triangles. Replace each triangle
by its midpoint triangle, and denote the resulting CW-complexes by
$\tilde\Sigma$ and $\tilde\Xi$. The midpoint triangles meet only at
vertices. Consequently every edge of $\tilde\Sigma$ has a unique
preimage in $\tilde\Xi$, and a one-cochain on $\tilde\Sigma$ is a
cocycle if and only if its coboundary vanishes on each midpoint
triangle. For $\Sigma=\partial X$, the complex
$\tilde\Sigma=\widetilde{\partial X}$ is the part of the boundary of
$\tilde X$ which lies in $\partial X$.

We will in Section \ref{ONrel} need Proposition \ref{H2} below. In order to formulate the needed result, we let $\Delta_2^i(\tilde X)\subset \Delta_2(\tilde X)$ be the subset made up of all the interior two simplices of $\tilde X$. We introduce the following subset
$$ \mathfrak C = \{ a \in {\mathbb S}^{\Delta_1(\tilde X)}\mid \delta_{\tilde X}(a)(f) = 0, \  \forall f \in \Delta_2^i(\tilde X); \ \ a(e) = 0, \ \forall e\in \Delta_1(\tilde X) - \Delta^i_1(\tilde X)\}$$
Here $\Delta^i_1(\tilde X)$ denotes the edges belonging to the paired
face triangles, rather than edges disjoint from the topological boundary.
For each paired triangle, choose its $+$ and $-$ copies consistently for
all three of its edges. With these conventions $\mathfrak C$ has two independent
circle coordinates per paired triangle. Write $D^i$ for the set obtained
by choosing two edges in each such triangle; then $\mathfrak C\cong\mathbb S^{D^i}$. 
Let us now observe that we have a map
$$ D_X : {\mathbb S}^{\Delta_0(\tilde Y)} \ra \mathfrak C$$
given by
$$ D_X(y)(e) = \delta_{\tilde Y}(y)(e_+) - \delta_{\tilde Y}(y)(e_-).$$
Observe that $D_X$ is constructed exactly such that
$$ \Omega(X) = \ker(D_X).$$

\begin{proposition} \label{H2}
We have a natural isomorphism 
$$ H^2(X', {\mathbb S}) \cong \mathfrak C/D_X({\mathbb S}^{\Delta_0(\tilde Y)}).$$
\end{proposition}

\begin{proof}
We recall that $H^2(X', {\mathbb S}) = H^2(\tilde X, {\mathbb S})$, since $X'$ deformation retracts to $\tilde X$. Suppose we have $u\in {\mathbb S}^{\Delta_2(\tilde X)}$ such that $\delta_{\tilde X}(u) = 0$. Then $\delta_{\tilde Y}(\tilde p^*u) = 0$, but since $H^2(\tilde Y, {\mathbb S}) = 0$, there exists $\tilde a\in {\mathbb S}^{\Delta_1(\tilde Y)}$ such that $\tilde p^*u = \delta_{\tilde Y}(\tilde a).$ Now define $a\in \mathfrak C$ given by $a(e) = \tilde a(e_+)-\tilde a(e_-)$ for all $e\in \Delta^i_1(\tilde X)$. Thus we have a map from the two-cocycles in $\tilde X$ to $\mathfrak C$. Let us now check that this map descends to cohomology. Let $c\in {\mathbb S}^{\Delta_1(\tilde X)}$ and consider $u' = u + \delta_{\tilde X}(c)$. We see that
$$\tilde p^*u' = \tilde p^* u + \delta_{\tilde Y}(\tilde p^* c) = \delta_{\tilde Y}(\tilde a+ \tilde p^*c)$$
and so
$$ a (e) = \tilde a(e_+)-\tilde a(e_-) = \tilde a(e_+) + \tilde p^*c(e_+) -(\tilde a(e_-) +  \tilde p^*c(e_-)).$$
Changing $\tilde a$ by another primitive changes it by
$\delta_{\tilde Y}y$, since $H^1(\tilde Y,\mathbb S)=0$, and hence changes
$a$ by $D_X(y)$. Thus we have a well-defined map from
$H^2(\tilde X,\mathbb S)$ to $\mathfrak C/D_X(\mathbb S^{\Delta_0(\tilde Y)})$. Let us now assume that for some two-cocycle $u$ in $\tilde X$, we have that $a \in D_X({\mathbb S}^{\Delta_0(\tilde Y)})$. That is, let $y\in {\mathbb S}^{\Delta_0(\tilde Y)}$ be such that $a = D_X(y)$. This means that
$$ \tilde a(e_+)- \tilde a(e_-) = a(e) = \delta_{\tilde Y}(y)(e_+) -\delta_{\tilde Y}(y)(e_-).$$
So
$$ (\tilde a-\delta_{\tilde Y}(y))(e_+) = (\tilde a-\delta_{\tilde Y}(y))(e_-)$$
for all $e\in \Delta_1(\tilde X)$. Let $c\in {\mathbb S}^{\Delta_1(\tilde X)}$ be given by $c(e) = (\tilde a-\delta_{\tilde Y}(y))(e_+) $.
Then
$$ \tilde p^*\delta_{\tilde X} c = \delta_{\tilde Y}(\tilde a- \delta_{\tilde Y}(y)) =  \delta_{\tilde Y}(\tilde a) = \tilde p^*u.$$
Hence $u = \delta_{\tilde X}c$, i.e. $u$ is zero in cohomology and we conclude that our map is injective.
To establish that this map is also surjective, we consider an arbitrary $a \in \mathfrak C$. The $f\in \Delta^i_2(\tilde X)$ are precisely the two simplices of $\tilde X$ such that $\tilde p^{-1}(f)$ consists of two simplices of $\tilde Y$. For  each $f\in \Delta^i_2(\tilde X)$ pick one of them to be $f^+$ and denote the other $f^-$. We can assume that for each edge $e$  of each such $f$, $e_+$ is an edge of $f_+$. Now define $\tilde a\in {\mathbb S}^{\Delta_1(\tilde Y)}$ by
$$\tilde a(e_+) = a(e)$$
for all $e\in \Delta^i_1(\tilde X)$ and zero on the rest of $\Delta_1(\tilde Y)$. 
Now let $u' = \delta_{\tilde Y}(\tilde a)$. For all faces $f\in \Delta^i_2(\tilde X)$ we get that 
$$ u'(f_+) =u'(f_-) = 0.$$
But then there exists $u\in {\mathbb S}^{\Delta_2(\tilde X)}$ such that
$$ \tilde p^*u = u'$$
and
$$ \tilde p^*\delta_{\tilde X}(u) = \delta_{\tilde Y}(\tilde p^*u) =  \delta_{\tilde Y}(u') = 0.$$
Thus $[u] \in H^2({\tilde X}, {\mathbb S})$ maps to $a\in \mathfrak C$ under our map. Hence the map is also surjective and thus then an isomorphism, since we have also proved it is injective.
\end{proof}

As we will see in the next section, the following map will play a fundamental role in our construction. Namely consider 
$$ E : \Omega(X) \ra ({\mathbb S}^2)^{\Delta_3(X)},$$
which on each $t\in \Delta_3(X)$ is given by
$$E(y)(t) = (y_{02}(t) + y_{13}(t) - y_{03}(t) - y_{12}(t), y_{02}(t) + y_{13}(t) - y_{01}(t)-y_{23}(t)).$$

\begin{proposition}\label{fiberrel}
Set $K_X=\ker(E|_{\Omega(X)})$. Then $E:\Omega(X)\to E(\Omega(X))$
is a principal $K_X$-bundle. If
\[
A:\mathbb S^{\Delta_0(Y)}\longrightarrow\mathbb S^{\Delta_0(\tilde Y)},
\qquad A(v)_{ij}(t)=v_i(t)+v_j(t),
\]
then $K_X=\Omega(X)\cap\operatorname{im}A$, and there is an exact sequence
\[
0\longrightarrow(\mathbb Z/2\mathbb Z)^{\Delta_3(X)}
\longrightarrow A^{-1}(\Omega(X))\overset{A}{\longrightarrow}K_X
\longrightarrow0.
\]
Moreover, let $Z^1(X,\mathbb S)$ be the group of simplicial
one-cocycles of the $\Delta$-complex $X$, that is, of functions
$\lambda$ on $\Delta_1(X)$ such that
$\lambda(\partial_0f)-\lambda(\partial_1f)+\lambda(\partial_2f)=0$
for all $f\in\Delta_2(X)$. An element $v\in\mathbb S^{\Delta_0(Y)}$
satisfies $A(v)\in\Omega(X)$ if and only if the coboundaries
$(\delta v)_{ij}(t)=v_j(t)-v_i(t)$, $i<j$, on the tetrahedra are the
restrictions of a cocycle $\lambda_v\in Z^1(X,\mathbb S)$. The
assignment $A(v)\mapsto\lambda_v$ is a well-defined surjective
homomorphism
\begin{equation}\label{eq:KX-cocycles}
K_X\longrightarrow Z^1(X,\mathbb S),
\end{equation}
whose kernel is the group $\mathbb S^{\Delta_3(X)}$ of tetrahedron
constants, and
\begin{equation}\label{eq:KX-class}
H_X(A(v))=2\,j^*[\lambda_v],
\end{equation}
where $j\colon X'\hookrightarrow X$ is the inclusion and
$H^1(\tilde X,\mathbb S)$ is identified with $H^1(X',\mathbb S)$.
\end{proposition}
\begin{proof}
The fibres of the continuous homomorphism $E|_{\Omega(X)}$ are precisely
the cosets of its kernel. On one tetrahedron, $E(z)=0$ says that
$z_{01}+z_{23}=z_{02}+z_{13}=z_{03}+z_{12}$. Choose $v_0\in\mathbb S$
with $2v_0=z_{01}+z_{02}-z_{12}$ and put $v_i=z_{0i}-v_0$ for $i>0$.
The displayed equalities give $z_{ij}=v_i+v_j$ for every edge.
There are exactly two such choices, differing by $1/2$ at all four
vertices. Applying this independently to all tetrahedra proves the
exact sequence. The quotient of compact abelian Lie groups by a closed
subgroup is a locally trivial principal bundle, proving the first claim.

For the second part, consider a face with ordered vertices $i<j<k$ of
a tetrahedron, and let $q_{ij},q_{ik},q_{jk}$ be the midpoints of its
edges. The values of $\delta_{\tilde Y}A(v)$ on the three edges of its
midpoint triangle are
\begin{equation}\label{eq:KX-midpoint-values}
\begin{aligned}
A(v)_{ik}-A(v)_{ij}&=v_k-v_j &&\text{on }(q_{ij}q_{ik}),\\
A(v)_{jk}-A(v)_{ik}&=v_j-v_i &&\text{on }(q_{ik}q_{jk}),\\
A(v)_{jk}-A(v)_{ij}&=v_k-v_i &&\text{on }(q_{ij}q_{jk}).
\end{aligned}
\end{equation}
Hence $A(v)$ satisfies the matching conditions in \eqref{OmD} if and
only if the coboundaries $\delta v(t)$ agree on the three edges of every
glued face. Since the identifications of the edges of $Y$ in $X$ are
generated by the face gluings, this is the case if and only if the
$\delta v(t)$ descend to a function $\lambda_v$ on $\Delta_1(X)$; it
is a cocycle because each $\delta v(t)$ is one. Conversely, a cocycle
on $X$ restricts to a coboundary on every tetrahedron, and its
primitives $v(t)$ are unique up to a constant $c(t)$; one has
$A(v+c)=A(v)+2c$. As the elements of the kernel of $A$ do not change
$\delta v$, the map \eqref{eq:KX-cocycles} is well defined; it is onto,
and its kernel consists of the elements $2c$, that is, of all
tetrahedron constants.

Finally, let $r_{\mathrm{low}},r_{\mathrm{up}}:\tilde X\to X$ send the midpoint of an edge to
its lower, respectively upper, endpoint, extended linearly over a
subdivision of each octahedron without new vertices. Both maps are
homotopic to the inclusion by simplexwise straight-line homotopies.
By \eqref{eq:KX-midpoint-values}, the cocycle $H'_X(A(v))$ takes the
values $\lambda_{jk},\lambda_{ij},\lambda_{ik}$ of $\lambda=\lambda_v$
on the three midpoint edges, whereas $r_{\mathrm{low}}^*\lambda$ takes the
values $0,\lambda_{ij},\lambda_{ij}$ and $r_{\mathrm{up}}^*\lambda$ the values
$\lambda_{jk},0,\lambda_{jk}$. Since
$\lambda_{ij}+\lambda_{jk}=\lambda_{ik}$, we get
$H'_X(A(v))=r_{\mathrm{low}}^*\lambda_v+r_{\mathrm{up}}^*\lambda_v$, which proves
\eqref{eq:KX-class}.
\end{proof}

\begin{remark}\label{rem:KX}
The elements of $K_X$ are exactly the translations of $\Omega(X)$
which leave all the arguments $E(y)(t)$ of the tetrahedral weights
unchanged, and by \eqref{eq:KX-class} the translation by $A(v)$ changes
the cohomology class by twice the extending class $j^*[\lambda_v]$.
In Section~\ref{sec:sector-translation} this is used to
relate the cohomology sectors of a closed triangulation to each other.
For a single tetrahedron $K_X\cong\mathbb S^4$.
\end{remark}

\section{Definition of the new formulation}

\label{DEF}

In order to specify the geometric setup of our new formulation of the Teichm\"{u}ller TQFT, we consider the following multiplier construction of line bundles over a two torus. We have a natural translation action of $\bZ^2$ on $\bR^2$  and $\bC^2$ with respective quotients $\Pi = {\mathbb S}^2$ and $\Pi_\bC = {\mathbb S}_\bC^2$.
Consider the following multipliers 
\begin{equation*}
\varphi_\pm\colon \bZ^2\times\bC^2\to\bC^{\times},\quad\varphi_\pm((m,k),(x,y))=(-1)^{mk}e^{\pm\pi\Imun(kx-my)},
\end{equation*}
which induce two actions of ${\mathbb Z}^2$ on the trivial bundle $\bC^2\times\bC$ and we define 
$$L_\bC^\pm = (\bC^2\times\bC)/\bZ^2$$  
as complex line bundles over $\Pi_\bC$. 
We further let 
$$L^\pm = L^\pm_\bC\mid_\Pi.$$
Since the multipliers take values in $U(1)$ on $\bR^2$, there is induced a Hermitian structure in $L^\pm$ induced from the usual inner product on ${\mathbb C}$. We will simply use the notation $L:= L^+$.
In terms of the sign $\sigma(n)=(-1)^{n_1n_2}$ of $n\in\bZ^2$ and of
the standard symplectic form $\omega(a,b)=a_1b_2-a_2b_1$ on $\bC^2$,
the multipliers are
\begin{equation}\label{eq:multiplier-omega}
\varphi_\pm(n,z)=\sigma(n)\,e^{\pm\pi\Imun\,\omega(z,n)},
\qquad n\in\bZ^2,\ z\in\bC^2.
\end{equation}

Suppose $T$ is a tetrahedron with a choice of branching, i.e. a choice of labelling of its vertices by $\{0,1,2,3\}$ and an orientation given by a sign $\varepsilon_T$. 
We use the following notation
 \begin{equation*}
x_{ij}:= x(v_T(i)v_T(j)),\quad\alpha_i:=\alpha_T(v_T(0)v_T(i))/2\pi,
\end{equation*}
where 
\begin{equation*}
v_T\colon \{0,1,2,3\}\to \Delta_0(T)
\end{equation*}
is the vertex ordering map and $v_T(i)v_T(j)$ is the edge connecting the vertices $v_T(i)$ and $v_T(j)$ and $x: \Delta_1(T) \ra \bC$. 
Then we have the following map
$$ p_T : \widetilde{LS}^\bC(T) \times {\mathbb S}_\bC^{\Delta_1(T)} \ra \bC^2\times\Pi_\bC,$$
given by 
\begin{eqnarray*}
\lefteqn{p_T(\ell, \alpha_1, \alpha_2,\alpha_3, x_{01},x_{02},x_{03},x_{12},x_{13},x_{23})} & \\
 & = (\alpha_1,\alpha_3,x_{02} + x_{13} -x_{03}-x_{12}, x_{02}+x_{13} - x_{01} - x_{23}).
\end{eqnarray*}
Let $L_T = p_T^*p_2^*(L^{\varepsilon_T}_\bC)$. We write ${\mathcal M}(V,\mathcal E)$ for the space of meromorphic sections of a holomorphic line bundle $\mathcal E$ over a complex manifold $V$, and define a meromorphic section
$$B(T) = p_T^* (g^{\varepsilon_T}) \in {\mathcal M}(\widetilde{LS}^\bC(T)\times{\mathbb S}_\bC^{\Delta_1(T)}, L_T)$$
for a shaped tetrahedron $T$ where 
$$ g^\pm \in {\mathcal M}(\bC^2\times \Pi_\bC, p_2^*L_\bC^\pm)$$
is a meromorphic section (to be specified below) of the pull back bundle under the projection $p_2$ to the second factor and $g^-$ is determined by
$$  g^-|_{{LS}(T)\times \Pi} =  \overline{g^+}|_{{LS}(T)\times \Pi},$$
that is, $g^-_{a,c}(x,y)=\overline{g^+_{\bar a,\bar c}(\bar x,\bar y)}$.
In fact we will see that
$$ g^\pm\mid_{{LS}(T)\times\Pi} \in C^\infty({LS}(T)\times \Pi, L^\pm),$$
and so 
$$B(T)\mid_ {{LS}(T)\times{\mathbb S}^{\Delta_1(T)}} \in C^\infty({LS}(T)\times{\mathbb S}^{\Delta_1(T)}, L_T).$$

Let now $F$ be a triangle with ordered vertices
and with fixed orientation specified by a sign $\varepsilon_F\in\{\pm1\}$.
Consider the map
\begin{equation*}
p_F : {\mathbb S}_\bC^{\Delta_1(F)} \ra \Pi_\bC
\end{equation*}
defined by 
\begin{equation*}
p_F(x) = (x_1 - x_0,x_2-x_1),\quad x_{i}:= x(\partial_iF).
\end{equation*}
Denote by 
\begin{equation*}
L_F := p_F^*(L_\bC^{\varepsilon_F})
\end{equation*}
 the pull back bundle  over ${\mathbb S}_\bC^{\Delta_1(F)}$. 

Let $\Sigma$ be a triangulation of an oriented smooth surface and let $\Xi$ be the disjoint union of the triangles in $\Sigma$. By using the restriction maps
\begin{equation*}
r_{F} : {\mathbb S}_\bC^{\Delta_1(\Xi)} \ra {\mathbb S}_\bC^{\Delta_1(F)},\quad \forall F\in\Delta_2(\Xi),
\end{equation*}
induced from the inclusion maps $\Delta_1(F)\to\Delta_1(\Xi)$,
we define a line bundle over ${\mathbb S}_\bC^{\Delta_1(\Xi)}$ by the formula
\begin{equation*}
L_{\Sigma} = \bigotimes_{F\in \Delta_2(\Xi)} r_{F}^* (L_{F}).
\end{equation*}
Let $\bar \Sigma$ be $\Sigma$ with the reversed orientation and consider the induced triangulation structure on it with $\bar \Xi$ being the disjoint union of its triangles, then we have a natural identification
$$L_{\bar \Sigma} = L^*_\Sigma.$$
\begin{proposition}\label{prop:boundary-bundle}
On globally identified edge states $\mathbb S_\bC^{\Delta_1(X)}$, the
tensor product of the tetrahedral line bundles is canonically isomorphic
to the pullback of the boundary line bundle. More generally, on
$\Omega_\bC(X)$ it is the pullback, by restriction to the unglued
boundary triangles, of their tensor product line bundle.
\end{proposition}
\begin{proof}
For an edge function $x$ on a tetrahedron $T$ write $E(x)=(s,t)$ with
$s=x_{02}+x_{13}-x_{03}-x_{12}$ and
$t=x_{02}+x_{13}-x_{01}-x_{23}$, and let $p_{\partial_iT}(x)$ be the
value of $p_F$ on the restriction of $x$ to the face $F=\partial_iT$.
With the form $\omega$ of \eqref{eq:multiplier-omega}, direct expansion
gives, for arbitrary complex edge functions $x$ and $n$ on $T$,
\begin{equation}\label{eq:bilinear-faces}
\omega\bigl(E(x),E(n)\bigr)
=\sum_{i=0}^3(-1)^i\,\omega\bigl(p_{\partial_iT}(x),p_{\partial_iT}(n)\bigr),
\end{equation}
and, for an integer shift $n$ with $(m,k)=E(n)$ and
$(m_i,k_i)=p_{\partial_iT}(n)$,
\[
mk\equiv\sum_{i=0}^3(-1)^i m_i k_i\pmod2.
\]
Together with \eqref{eq:multiplier-omega}, these two identities give
the exact multiplier equality
\begin{equation}\label{eq:tetra-face-multiplier}
\varphi_{\varepsilon_T}\bigl(E(n),E(x)\bigr)
=\prod_{i=0}^3
\varphi_{\varepsilon_T(-1)^i}
\bigl(p_{\partial_iT}(n),p_{\partial_iT}(x)\bigr).
\end{equation}
Thus $L_T$ is the tensor product of the four face lines with their
induced orientations. On paired faces the maps $p_F$ agree for
compatible states and the orientations are opposite. Their two
factors are a line $L_f$ and its dual, and the contraction is the
canonical map $v\otimes\lambda\mapsto\lambda(v)$ from
$L_f\otimes L_f^*$ to $\mathbb C$. This leaves exactly the unglued
boundary factors, with no scalar normalization. Balancing of shape
angles is not needed for this bundle identity.
\end{proof}

\subsection{Compatible cocycles and normalized fibres}\label{sec:normalized-fibres}
All cochain groups in this subsection are cellular cochain groups of
the finite complexes of Section~\ref{TP}, with values in $\mathbb S$.
They are compact abelian Lie groups, possibly disconnected when
defined by integral equations. Put
\[
\mathcal Z_X=Z^1(\tilde X;\mathbb S).
\]
The map $H'_X:\Omega(X)\to\mathcal Z_X$ is surjective and has kernel
$\mathbb S^{\Delta_3(X)}$: pull back a cocycle to each contractible
octahedron and choose a primitive, unique up to a constant. Thus the
tetrahedron constants may be divided out before integration.

For an ordered face $F$, let $q_0,q_1,q_2$ be the midpoints of
$\partial_0F,\partial_1F,\partial_2F$. A cocycle $z$ determines
\[
\lambda_F(z)=\bigl(z(q_0q_1),z(q_1q_2)\bigr).
\]
The third value is determined by the cocycle equation on the face.
Define
\[
\mathcal L_\Sigma=\bigotimes_{F\in\Delta_2(\Sigma)}
\lambda_F^*L^{\varepsilon_F}.
\]
This is exactly the descent of the face bundle previously defined on
triangle potentials, since adding a constant to such a potential
does not change $\lambda_F$. Reversing orientation gives its dual.
Proposition~\ref{prop:boundary-bundle} therefore produces, for every
shape, a contracted weight
\[
b_X\in C^\infty(\mathcal Z_X,r_\partial^*\mathcal L_{\partial X}),
\qquad r_\partial z=z|_{\widetilde{\partial X}}.
\]

\begin{lemma}\label{lem:haar}
Let $\varpi:G\to H$ be a surjective homomorphism of compact abelian Lie
groups. Every fibre $\varpi^{-1}(h)$ carries a canonical
$\ker\varpi$-invariant probability measure $\nu_h$, namely the
translate of the probability Haar measure of $\ker\varpi$, and, for a
continuous scalar function $f$,
\[
\int_G f\,d\nu_G=\int_H\left(\int_{\varpi^{-1}(h)}f\,d\nu_h\right)d\nu_H(h).
\]
The assertion is valid on affine translates of these groups, for a
composition of surjective homomorphisms, and for smooth sections of a
bundle pulled back from $H$. Fibre integration of such sections is
smooth.
\end{lemma}
\begin{proof}
Translate probability Haar measure on $\ker \varpi$ to each fibre.
The result is independent of the translating point. Integrating it
against probability Haar measure on $H$ gives a translation-invariant
probability on $G$, which is Haar measure by uniqueness. This proves
the formula, and the version for a composition $G\to H\to H'$ follows
by applying it twice. Locally a surjective
Lie-group homomorphism has a smooth section: split its surjective
Lie-algebra map and use the exponential map. In this trivialization
the integral is over a fixed compact group. Differentiation under
the integral is justified by uniform bounds on derivatives over
compact sets. This also proves the bundle assertion in a local frame.
If the kernel is disconnected, its components receive equal mass.
\end{proof}
The letter $\varpi$ is used here and in Lemma~\ref{lem:holomorphic-integral}
for a general homomorphism; it is unrelated to the projection
$p\colon Y\to X$ of Section~\ref{TP}.

\begin{lemma}\label{lem:boundary-surjection}
The homomorphism
\[
q_X:\mathcal Z_X\longrightarrow\mathcal Q_X,
\qquad q_X(z)=([z],r_\partial z),
\]
is surjective. Its nonempty fibres are compact torsors and carry the
probability measures of Lemma~\ref{lem:haar}.
\end{lemma}
\begin{proof}
Given $(\alpha,\zeta)\in\mathcal Q_X$, choose a cocycle $z_0$
representing $\alpha$. The cochain $\zeta-r_\partial z_0$ is exact,
so it equals $\delta u$ for a boundary zero-cochain $u$. Extend $u$
to all vertices of $\tilde X$ by assigning arbitrary values to the
remaining vertices. Then $z=z_0+\delta\tilde u$ satisfies
$q_X(z)=(\alpha,\zeta)$. The fibre is a translate of
$\ker q_X$, a closed subgroup of a compact abelian Lie group.
\end{proof}

The same definitions with $\mathbb S_{\mathbb C}$ give the
complexifications. Integral matrices describing these maps have
constant rank on each component. Their real kernels provide the
initial compact cycles in the complex fibres. Continuation through
pinches, and its single-valuedness, are established separately in
Section~\ref{sec:extension}.

Let us now specify the main ingredient in our TQFT, namely the function
\begin{equation}\label{eq:gac}
 g_{a,c} = W(\tilde\psi'_{a,c}),
\end{equation}
where $W$ is the Weil--Gel'fand--Zak transform (see Appendix \ref{WGZ}) and where 
\begin{equation}\label{eq:tpsip}
\tilde\psi'_{a,c}(s):= e^{-\pi \Imun s^2}\tilde\psi_{a,c}(s),
\end{equation}
\begin{equation}\label{eq:tpsi}
\tilde\psi_{a,c}(s):= \int_\REALS\psi_{a,c}(t)e^{-2\pi \Imun st}dt,
\end{equation}
and
\begin{equation}\label{eq:psi}
\psi_{a,c}(t):= \BQDILOG(t-2\cla(a+c))e^{-4\pi \Imun\cla a(t-\cla(a+c))}
e^{-\pi\Imun\cla^2(4(a-c)+1)/6}
\end{equation}
where
\begin{equation*}
 \cla = \frac{\Imun}{2}(\la+\la^{-1}),
\end{equation*}
where $\BQDILOG(x):= 1/\QDILOG(x)$ and 
\begin{equation*}
\QDILOG(x):= \exp\left(
\int_{\bR+\Imun\epsilon}
\frac{e^{-2\Imun xw}\, dw}{4\sinh(w\la)
\sinh(w/\la) w}\right)
\end{equation*}
is Faddeev's quantum dilogarithm. The function $g^+_{a,c}$ discussed above is simply equal to the function $g_{a,c}$.
The function $\QDILOG(x)$ depends on $\la$ only through the combination $\hbar$ defined by the formula
\begin{equation*}
\hbar^{-1}:= \left(\la+\la^{-1}\right)^{2}=-4\cla^2,
\end{equation*}
which we require to be real positive. We fix the branch with
$\Re\la>0$ and $\Re\la^{-1}>0$, so that $\cla=\Imun C$ with $C>0$.
Thus either $\la>0$ is real or $|\la|=1$ with positive real part.
All parameter-continuation statements below keep this $\la$ fixed.

Along with \eqref{eq:psi}, we also define
\begin{multline}\label{eq:bpsi}
\bar\psi_{a,c}(t):= \QDILOG(t+2\cla(a+c))e^{-4\pi \Imun\cla a(t+\cla(a+c))}e^{\pi\Imun\cla^2(4(a-c)+1)/6}\\
=\QDILOG(t+\cla(1-2b))e^{-4\pi \Imun\cla a(t-\cla b)}
e^{\pi\Imun\cla^2(4(b-a)-1)/6},\quad b:= \frac12-a-c.\end{multline}
The function $\bar\psi_{a,c}$ is an auxiliary function; it is related
to the Fourier transform \eqref{eq:tpsi} as follows.
The Fourier integral \eqref{eq:tpsi} converges absolutely for
positive $a,b,c$, that is, for $a,c>0$ and $a+c<\frac12$.
For later use, we observe the following two equalities, which hold
in this domain:
\begin{equation}\label{eq:tpsip-psi}
\tilde\psi'_{a,c}(x)=e^{-\frac{\pi\Imun}{12}}\psi_{c,b}(x),
\end{equation}
and
\begin{equation}\label{eq:tpsi-bpsi}
\tilde\psi_{a,c}(x)=e^{\frac{\pi\Imun}{12}}\bar\psi_{b,c}(-x)\Leftrightarrow\bar\psi_{a,c}(x)=e^{-\frac{\pi\Imun}{12}}\tilde\psi_{b,c}(-x).
\end{equation}
Indeed, \eqref{eq:tpsi-bpsi} is the Fourier transformation
formula \eqref{fourier2} of Appendix~\ref{QDL}, applied
to \eqref{eq:psi} after the shift $t\mapsto t+2\cla(a+c)$ of the
integration variable; the shift is legitimate because for
$a+c<\tfrac12$ the integrand is holomorphic and exponentially
decaying in the strip between the two lines (Lemma~\ref{lem:qdl-asymptotics}).
Equality \eqref{eq:tpsip-psi} follows from \eqref{eq:tpsi-bpsi} and
the inversion relation \eqref{eq:inversion}: the quotient
$\bar\psi_{b,c}(-x)/\psi_{c,b}(x)$ equals
$\zeta_{inv}^{-1}e^{\pi\Imun\cla^2/3}e^{\pi\Imun x^2}
=e^{-\pi\Imun/6}e^{\pi\Imun x^2}$ by a direct expansion of the
exponents in \eqref{eq:psi} and \eqref{eq:bpsi}.
For $a+c>\frac12$ the pole $t=\cla(2(a+c)-1)$ of $\psi_{a,c}$, which
comes from the first zero $-\cla$ of $\QDILOG$, lies above the line of
integration of \eqref{eq:tpsi}. The two sides of these equalities then
differ by residue terms, also where the integral still converges.

\begin{definition}\label{Maindef}
For positive shapes define
\begin{equation}\label{eq:si-preliminary}
\mathcal F_\hbar(X)(\alpha,\zeta)
=e^{\Imun\pi\ell_X/(4\hbar)}
\int_{q_X^{-1}(\alpha,\zeta)}b_X(z)\,d\nu_{\alpha,\zeta}(z).
\end{equation}
The integral takes values in the single line
$(\mathcal L_{\partial X})_\zeta$. If the boundary is empty it is a
scalar. Integration on $\Omega(X)$ gives the same answer because
the redundant tetrahedron constants have probability mass one.
\end{definition}

The integrand and the domain of integration in
\eqref{eq:si-preliminary} do not involve the subcomplex $\Gamma$,
which only specifies which internal edges are allowed to be
unbalanced in the gauge quotient; this is why $\Gamma$ is omitted
from the notation $\mathcal F_\hbar(X)$. Smoothness,
convergence, and the holomorphic-germ assertion follow from
Lemmas~\ref{lem:weight-regularity} and~\ref{lem:holomorphic-integral}.

For later use we make the contraction of
Proposition~\ref{prop:boundary-bundle} explicit in terms of real lifts,
in the closed case. Suppose
$\partial X=\emptyset$, let $y\in\Omega(X)$, and choose real lifts
$y_T\in\mathbb R^{\Delta_1(T)}$ of its restrictions to the tetrahedra.
Every face $f$ of $X$ is glued from a face of a tetrahedron $T_1$ and
a face of a tetrahedron $T_2$, taken in either order; possibly
$T_1=T_2$. Put
\begin{equation}\label{eq:face-lifts}
z_{s,f}(y)=p_f(y_{T_s})\in\mathbb R^2\quad(s=1,2),\qquad
\Delta_f(y)=z_{2,f}(y)-z_{1,f}(y)\in\mathbb Z^2,
\end{equation}
where $p_f(y_{T_s})$ is the value of $p_F$ on the restriction of
$y_{T_s}$ to the corresponding face $F$ of $T_s$, and let
$\eta_f\in\{1,-1\}$ be the orientation sign of $f$ as a face
of $T_1$; its sign as a face of $T_2$ is $-\eta_f$. The representative
of $B(T)$ at the lift $y_T$ is $g^{\varepsilon_T}_{a_T,c_T}(E(y_T))$,
where $(a_T,c_T)=(\alpha_1,\alpha_3)(T)$, and by the
multiplier identity in the proof of
Proposition~\ref{prop:boundary-bundle} it is the product of
representatives of its four face factors at the lifts
$p_{\partial_iT}(y_T)$. Moving the representative of the second factor
on $f$ from the lift $z_{2,f}$ to the lift $z_{1,f}$ multiplies it by
$\varphi_{-\eta_f}(\Delta_f,z_{1,f})^{-1}
=\varphi_{\eta_f}(\Delta_f,z_{1,f})$, after which the two dual
factors are evaluated on each other. Writing $b_X(y)$ for
$b_X(H'_X(y))$, we obtain
\begin{equation}\label{eq:contraction-lifts}
b_X(y)=\prod_{T\in\Delta_3(X)}g^{\varepsilon_T}_{a_T,c_T}\bigl(E(y_T)\bigr)
\prod_{f\in\Delta_2(X)}\varphi_{\eta_f}\bigl(\Delta_f(y),z_{1,f}(y)\bigr).
\end{equation}
The right-hand side is independent of the lifts and of the order of
the two sides of each face.

On the positive balanced locus the following proposition permits
descent to the based gauge quotient.

\begin{proposition}\label{prop:gauge-rigorous}
Definition~\ref{Maindef} is invariant under based shape gauge
transformations, when both endpoint shapes are positive.
\end{proposition}
\begin{proof}
Write the real gauge parameter as $h:\Delta_1(X)\to\mathbb R$,
zero on the boundary edges. For each tetrahedron put
\[
h_A=h_{01}+h_{23},\quad h_B=h_{02}+h_{13},\quad
h_C=h_{03}+h_{12},\quad u=h_B-h_C,\quad v=h_B-h_A.
\]
Use the cyclic order $(A,B,C)$ on opposite edge pairs in a positive
tetrahedron and the reverse order in a negative tetrahedron. The gauge change
of the normalized angles $(a,b,c)$ is
\[
\Delta a=\varepsilon_Tu/2,\qquad
\Delta c=-\varepsilon_Tv/2,\qquad
\Delta b=-\Delta a-\Delta c.
\]
Indeed, in \cite[Definition~5]{MR3227503} the angle change is
$\pi\sum_b\varepsilon_{p(a),p(b)}h_b$. Dividing by $2\pi$ and
using $\varepsilon_{A,B}=\varepsilon_{B,C}=\varepsilon_{C,A}
=\varepsilon_T$ gives the displayed changes; reversing the
tetrahedron sign reverses all three. The level formula in that
definition sums over local edge occurrences and has no orientation
sign. Grouping opposite edges, it reads
\begin{equation}\label{eq:gauge-level-explicit}
\Delta\ell=\sum_T\bigl((1/3-2a)h_A+
(1/3-2b)h_B+(1/3-2c)h_C\bigr).
\end{equation}
Let $(s,t)=E(y)_T$. The formula \eqref{eq:g-phi-global} for $g$,
or direct substitution in \eqref{eq:psi}, gives
\begin{align}
&g^\varepsilon_{a+A,c+D}
(s-2\varepsilon\cla A,t+2\varepsilon\cla D)
\nonumber\\[-2mm]
&\quad=g^\varepsilon_{a,c}(s,t)
\exp\bigl(-2\pi\Imun\cla(Ds+At)
+\varepsilon\pi\Imun\cla^2
[(4c-2/3)A+(2/3-4a)D]\bigr).
\label{eq:gauge-local}
\end{align}
For $\varepsilon=-1$ this follows by holomorphic reflection from
the positive formula, since $\bar\cla=-\cla$.
With $A=\varepsilon u/2$, $D=-\varepsilon v/2$, the arguments on
the left are $E(y-\cla h)_T$. The state-dependent exponential on
the right is the lift to $L^{\varepsilon_T}$ of this translation.
Explicitly, translation by $(U,V)$ lifts by multiplication by
$\exp(\varepsilon\pi\Imun(Ut-Vs))$; substitution in the lattice
multiplier verifies this assertion.

The bilinear identity in Proposition~\ref{prop:boundary-bundle}
shows that these translation lifts are products of the face lifts.
For a glued face with lifts $z_{1,f}$ and $z_{1,f}+\Delta_f$,
let $w_f$ be their common complex translation. The product of the
two translation lifts is
$e^{-\eta_f\pi\Imun\omega(w_f,\Delta_f)}$, while the change of
$\varphi_{\eta_f}(\Delta_f,z_{1,f})$ in
\eqref{eq:contraction-lifts} is
$e^{\eta_f\pi\Imun\omega(w_f,\Delta_f)}$. These factors cancel.
On the external boundary the translation is zero. Thus the
cancellation holds intrinsically, also for unequal real face lifts.
The remaining exponent
in the product of \eqref{eq:gauge-local} is
$\pi\Imun\cla^2\Delta\ell$: indeed its contribution at $T$ is
\[
(2c-1/3)(h_B-h_C)+(2a-1/3)(h_B-h_A)
=(1/3-2a)h_A+(1/3-2b)h_B+(1/3-2c)h_C.
\]
Since $\cla^2=-1/(4\hbar)$, it cancels the change of the level
factor. Translation by the global zero-cochain $-\cla h$ changes
$z$ by an exact cocycle and leaves the boundary value fixed.
For a sufficiently small transformation the real fibre and its
translated cycle are homologous in a pole-free neighbourhood.
Lemma~\ref{lem:holomorphic-integral} then proves invariance.
Finally the straight segment between two positive gauge-equivalent
shapes remains positive. Cover this compact segment by finitely many
such neighbourhoods and compose the small transformations. This
proves the assertion for finite transformations as well.
\end{proof}

For the torsion version set $q_X^T=(r_T\circ[\cdot],r_\partial)$
and use its image as target. By Lemma~\ref{lem:haar}, integration of
$b_X$ on a fibre of $q_X^T$ is integration first on a fibre of $q_X$
and then over the compatible characters. This proves
\eqref{eq:si-torsion} and its independence of a homology splitting.
In particular, for $\partial X=\emptyset$,
\begin{equation}\label{eq:torsion-average}
\mathcal F^T_\hbar(X)(\tau)
=\int_{r_T^{-1}(\tau)}\mathcal F_\hbar(X)(\alpha)\,d\alpha.
\end{equation}
For nonempty boundary, averaging over all characters at a fixed
arbitrary boundary state would not be a well-defined operation;
the compatibility condition in \eqref{eq:si-torsion} is essential.

\section{The Pentagon identity and symmetries}\label{PId}
The first property which we need to establish about our quantity $ \mathcal F_\hbar(X)$ is that it is invariant under both $3-2$ and $2-3$ Pachner moves. The fundamental tool for this is the following Pentagon identity for the function $g_{a,c}$.

\begin{proposition}\label{PI}
For any $(\alpha,\beta,\gamma,\delta)\in\REALS^4$, the following integral identity is satisfied
\begin{multline}\label{eq:pent1}
\int_{[0,1]}g_{a_4,c_4}(\gamma-\sigma,\alpha+\delta-\sigma)g_{a_2,c_2}(\sigma,\beta+\delta)g_{a_0,c_0}(\alpha-\sigma,\beta+\gamma-\sigma)
d\sigma\\
=e^{-\frac{\pi\Imun}{12\hbar} P_e}g_{a_1,c_1}(\alpha,\beta)g_{a_3,c_3}(\gamma,\delta),
\end{multline}
where
\begin{equation}\label{eq:Pe}
P_e:= 2(c_0+a_2+c_4)-\frac12,
\end{equation}
and $a_i, c_i \in\REALS_{+}$, for $ i = 0,1,\ldots,4$ are such that
\[
b_i:= \frac12-a_i-c_i>0,\  i = 0,1,\ldots,4,
\]
and
\begin{equation}\label{pentagonconditions}
 a_1=a_0+a_2,\ a_3=a_2+a_4,\ c_1=c_0+a_4,\ c_3=a_0+c_4,\ c_2= c_1+c_3.
\end{equation}
The integrand on the left is $1$-periodic in $\sigma$, so that the
integral is an integral over the circle $\REALS/\INTEGERS$.
\end{proposition}
The periodicity assertion follows from the lattice
multipliers \eqref{eq:multiplier-omega}: under $\sigma\mapsto\sigma+1$
the three factors acquire the multipliers of $L$ for the translations
$n_1=(-1,-1)$, $n_2=(1,0)$ and $n_3=(-1,-1)$ of their arguments
$z_1,z_2,z_3$, whose product is
$\sigma(n_1)\sigma(n_2)\sigma(n_3)\,
e^{\pm\pi\Imun[\omega(z_1,n_1)+\omega(z_2,n_2)+\omega(z_3,n_3)]}$.
Here
\[
\omega(z_1,n_1)+\omega(z_2,n_2)+\omega(z_3,n_3)
=(\alpha+\delta-\gamma)-(\beta+\delta)+(\beta+\gamma-\alpha)=0,
\]
and the sign factors are $\sigma(n_1)=\sigma(n_3)=-1$ and
$\sigma(n_2)=1$. This periodicity is used implicitly in
Proposition~\ref{2-3/3-2M}, where the integration over $\sigma$ is
the integration over the state of the new edge.
\begin{proof} Multiplying \eqref{eq:pent1} by $e^{-\pi \Imun(\alpha\beta+\gamma\delta)}$, integrating over the unit square in the $1$-periodic variables $\beta$ and $\delta$, and using definition~\eqref{eq:gac}, we obtain the following identity. Repeating the calculation after arbitrary
integer shifts of $\alpha$ and $\gamma$ recovers every Fourier coefficient
in $\beta$ and $\delta$ by quasi-periodicity; thus the family of transformed
identities is equivalent to \eqref{eq:pent1}:
\begin{multline}\label{eq:pent2}
\int_{[0,1]^3}\sum_{k,l,m\in\INTEGERS}\tilde\psi'_{a_4,c_4}(\gamma-\sigma+k)\tilde\psi'_{a_2,c_2}(\sigma+l)\tilde\psi'_{a_0,c_0}(\alpha-\sigma+m)\\
\times e^{\pi \Imun ((\alpha+\delta-\sigma)(\gamma-\sigma+2k)+(\beta+\delta)(\sigma+2l)+(\beta+\gamma-\sigma)(\alpha-\sigma+2m)-\alpha\beta-\gamma\delta)}
d\sigma d\beta d\delta\\
=e^{-\frac{\pi\Imun}{12\hbar} P_e}\tilde\psi'_{a_1,c_1}(\alpha)\tilde\psi'_{a_3,c_3}(\gamma)
\end{multline}
where the triple sum is absolutely convergent. Exchanging the order of the integrations and summations, the double integral over $\beta$ and $\delta$ gives two Kronecker's deltas, $\delta_{m,-l}\delta_{k,-l}$,  which allows us to eliminate the summations over $k$ and $m$ with the result
\begin{multline}\label{eq:pent3}
\mathrm{l.h.s.\ of\ \eqref{eq:pent2}}=
\int_{[0,1]}\sum_{l\in\INTEGERS}\tilde\psi'_{a_4,c_4}(\gamma-\sigma-l)\tilde\psi'_{a_2,c_2}(\sigma+l)\tilde\psi'_{a_0,c_0}(\alpha-\sigma-l)\\
\times e^{\pi \Imun ((\alpha-\sigma)(\gamma-\sigma-2l)+(\gamma-\sigma)(\alpha-\sigma-2l))}
d\sigma.
\end{multline}
Now, the identity 
\[
e^{\pi \Imun((\alpha-\sigma)(\gamma-\sigma-2l)+(\gamma-\sigma)(\alpha-\sigma-2l))}
=e^{2\pi \Imun (\alpha-\sigma-l)(\gamma-\sigma-l)},\quad \forall l\in\INTEGERS,
\]
allows us to combine the integral and the sum in \eqref{eq:pent3} into an absolutely convergent integral over the real axis
\begin{equation}\label{eq:pent4}
\mathrm{r.h.s.\ of\ \eqref{eq:pent3}}=
\int_{\REALS}\tilde\psi'_{a_4,c_4}(\gamma-\sigma)\tilde\psi'_{a_2,c_2}(\sigma)\tilde\psi'_{a_0,c_0}(\alpha-\sigma)e^{2\pi \Imun(\alpha-\sigma)(\gamma-\sigma)}
d\sigma.
\end{equation}
By using definition~\eqref{eq:tpsip}, we eliminate four ``primes'' out of five and we obtain the following equivalent form of \eqref{eq:pent2}
\begin{multline}\label{eq:pent5}
\int_{\REALS}\tilde\psi_{a_4,c_4}(\gamma-\sigma)\tilde\psi'_{a_2,c_2}(\sigma)\tilde\psi_{a_0,c_0}(\alpha-\sigma)d\sigma\\
=e^{-\frac{\pi\Imun}{12\hbar} P_e-2\pi\Imun\alpha\gamma}\tilde\psi_{a_1,c_1}(\alpha)\tilde\psi_{a_3,c_3}(\gamma).
\end{multline}
We now apply the Fourier transformation over the variable $\gamma$: we multiply \eqref{eq:pent5} by $e^{2\pi\Imun\beta\gamma}$, where $\beta$ is a new real variable (the variable $\beta$ of \eqref{eq:pent1} has been integrated in \eqref{eq:pent2}), and integrate over $\gamma$ along the real axis. As all integrals are absolutely convergent, we can exchange the order of integrations on the left hand side, so that the integration over $\gamma$ can be carried out on both sides of the identity by Fourier inversion of \eqref{eq:tpsi}. The function $\psi_{a_4,c_4}$ has no zeros on the real axis, and the Fourier transformation is injective. In this way we obtain the equivalent identity
\begin{equation}\label{eq:pent6}
\int_{\REALS}\tilde\psi'_{a_2,c_2}(\sigma)\tilde\psi_{a_0,c_0}(\alpha-\sigma)e^{2\pi\Imun\beta\sigma}d\sigma
=e^{-\frac{\pi\Imun}{12\hbar} P_e}\frac{\tilde\psi_{a_1,c_1}(\alpha)\psi_{a_3,c_3}(\beta-\alpha)}{\psi_{a_4,c_4}(\beta)},
\end{equation}
which, by using the equalities~\eqref{eq:tpsip-psi} and  \eqref{eq:tpsi-bpsi}, can also be written in the form
\begin{equation}\label{eq:pent7}
\int_{\REALS}\psi_{c_2,b_2}(\sigma)\bar\psi_{b_0,c_0}(\sigma-\alpha)e^{2\pi\Imun\beta\sigma}d\sigma
=e^{\frac{\pi\Imun}{12}\left(1-\frac{P_e}{\hbar}\right)}\frac{\bar\psi_{b_1,c_1}(-\alpha)\psi_{a_3,c_3}(\beta-\alpha)}{\psi_{a_4,c_4}(\beta)}.
\end{equation}
Now, it is a straightforward verification that \eqref{eq:pent7} is a direct consequence of the integral Ramanujan identity. Indeed, for the left hand side of \eqref{eq:pent7}, we have that
\begin{multline}\label{eq:pent8}
\mathrm{l.h.s.\ of\ \eqref{eq:pent7}}\\=\int_{\REALS}\frac{\QDILOG(\sigma-\alpha+2\cla(b_0+c_0))}{\QDILOG(\sigma-2\cla(c_2+b_2))}e^{-4\pi \Imun\cla c_2(\sigma-\cla(c_2+b_2))}
e^{-\pi\Imun\cla^2(4(c_2-b_2)+1)/6}\\
\times e^{-4\pi \Imun\cla b_0(\sigma-\alpha+\cla(b_0+c_0))}e^{\pi\Imun\cla^2(4(b_0-c_0)+1)/6}e^{2\pi\Imun\beta\sigma}d\sigma\\
=\int_{\REALS}\frac{\QDILOG(\sigma-\alpha+\cla(1-2a_0))}{\QDILOG(\sigma-\cla(1-2a_2))}
e^{2\pi\Imun(\beta-2\cla (b_0+c_2))\sigma}d\sigma\\
\times e^{4\pi \Imun\cla b_0\alpha}e^{4\pi \Imun\cla^2 (a_0b_0-a_2c_2)}e^{2\pi\Imun\cla^2(a_0-b_0-a_2+c_2)/3}\\
=\Psi(\cla(1-2a_0)-\alpha,\cla(2a_2-1),\beta-2\cla (b_0+c_2))
\\
\times e^{4\pi \Imun\cla b_0\alpha}e^{4\pi \Imun\cla^2 (a_0b_0-a_2c_2)}e^{2\pi\Imun\cla^2(a_0-b_0-a_2+c_2)/3}\\
=\frac{\QDILOG(\cla(1-2(a_0+a_2))-\alpha)\QDILOG
(\cla(1-2 (b_0+c_2))+\beta)}{\QDILOG(\cla(2(c_0+b_2)-1)-\alpha+\beta)}e^{4\pi \Imun\cla (b_0\alpha-a_2\beta)}\\
  \times e^{4\pi \Imun\cla^2 (a_0b_0+a_2c_2+2a_2b_0)}e^{2\pi\Imun\cla^2(a_0-b_0-a_2+c_2)/3}e^{\pi\Imun(1-4\cla^2)/12}.
\end{multline}
Here the conditions \eqref{restrictions1} hold for the arguments
$u=\cla(1-2a_0)-\alpha$, $v=\cla(2a_2-1)$ and
$w=\beta-2\cla(b_0+c_2)$ of $\Psi$: with $\cla=\Imun C$ one finds
$\Im(v+\cla)=2Ca_2>0$, $\Im(\cla-u)=2Ca_0>0$ and
\[
\Im(v-u)=-2C(1-a_0-a_2)<\Im w=-2C(b_0+c_2)<0,
\]
the first of these inequalities being $c_0+b_2>0$.
For the right hand side of \eqref{eq:pent7}, we have that
\begin{multline}\label{eq:pent9}
\mathrm{r.h.s.\ of\ \eqref{eq:pent7}}\\
=\frac{\QDILOG(\cla(1-2a_1)-\alpha)\QDILOG(\cla(2b_4-1)+\beta)}{\QDILOG(\cla(2b_3-1)-\alpha+\beta)}
e^{4\pi \Imun\cla ((b_1+a_3)\alpha+(a_4-a_3)\beta)}\\
\times e^{4\pi \Imun\cla^2 (a_1b_1-a_3b_3+a_4b_4)}e^{2\pi\Imun\cla^2(a_1-b_1+a_3-b_3-a_4+b_4)/3}
e^{\pi\Imun(1-\frac{P_e}{\hbar}-2\cla^2)/12}.
\end{multline}
The final expressions in \eqref{eq:pent8} and \eqref{eq:pent9} coincide provided
\begin{equation}
a_1=a_0+a_2,\ b_0+c_2+b_4=1,\ b_3=c_0+b_2,\ b_0=b_1+a_3,\ a_3=a_2+a_4
\end{equation}
which is equivalent to \eqref{pentagonconditions}.
\end{proof}

The above relation will be used in Section \ref{proof} to establish the $3-2$ and $2-3$ Pachner moves. We now turn to what we need to establish invariance under the symmetries of the tetrahedron.

In order to prove the needed symmetries of  the tetrahedron partition function, we observe that 
\begin{equation}\label{eq:tpsip-GF}
\tilde\psi'_{a,c}(x)=G^{-1}F^{-1}(\psi_{a,c})(x),
\end{equation}
where $G$ and $F$ are defined in Appendix \ref{WGZ}. We then have that
\begin{equation}\label{eq:gac=w}
g_{a,c}=W(\tilde\psi'_{a,c})=WG^{-1}F^{-1}(\psi_{a,c}),\quad \bar g_{a,c}(x,y)=WGF(\bar\psi_{a,c})(x,-y).
\end{equation}
The coherent face transports of Section~\ref{sec:branching} turn the
following identities into branching independence on glued complexes.
\begin{lemma}\label{lemma2}
The following relations are satisfied
\begin{equation}\label{eq:fund1}
g_{a,c}(x,y)=\bar g_{a,b}(-x,y-x-1/2)e^{\pi\Imun x/2},
\end{equation}
and
\begin{equation}\label{eq:fund2}
 g_{a,c}(x,y)=e^{-\frac{\pi\Imun}{6}}\bar g_{b,c}(x-y-1/2,-y)e^{-\pi\Imun y/2}.
\end{equation}
\end{lemma}
\begin{proof}
From \eqref{eq:gac=w} and Lemma~\ref{lemma1} in Appendix \ref{WGZ}, we have that
\begin{multline}\label{eq:gac-psiac}
g_{a,c}(x,y)=WG^{-1}F^{-1}(\psi_{a,c})(x,y)=WF^{-1}(\psi_{a,c})(x,y-x-1/2)e^{\pi\Imun x/2}\\
=W(\psi_{a,c})(y-x-1/2,-x)e^{\pi\Imun x/2}.
\end{multline}
On the other hand, equality~\eqref{eq:tpsip-psi} implies that
\begin{equation}\label{eq:gac-psicb}
g_{a,c}=W(\tilde\psi'_{a,c})=e^{-\frac{\pi\Imun}{12}}W(\psi_{c,b})\Leftrightarrow W(\psi_{a,c})=e^{\frac{\pi\Imun}{12}}g_{b,a}.
\end{equation}
Combining \eqref{eq:gac-psiac} and \eqref{eq:gac-psicb}, we obtain that
\begin{multline}\label{eq:gac-gba}
g_{a,c}(x,y)=e^{\frac{\pi\Imun}{12}}g_{b,a}(y-x-1/2,-x)e^{\pi\Imun x/2}\\
\Leftrightarrow
g_{a,c}(x,y)=e^{-\frac{\pi\Imun}{12}}g_{c,b}(-y,x-y+1/2)e^{\pi\Imun y/2}
\end{multline}

From \eqref{eq:gac=w}, \eqref{eq:tpsi-bpsi}, \eqref{eq:gac-psicb}, and Lemma~\ref{lemma1} it follows that
\begin{multline}\label{eq:bgac-gab}
\bar g_{a,c}(x,y)=e^{-\frac{\pi\Imun}{12}}WF^2(\psi_{b,c})(x,-y+x+1/2)e^{-\pi\Imun x/2}\\=
e^{-\frac{\pi\Imun}{12}}W(\psi_{b,c})(-x,y-x-1/2)e^{-\pi\Imun x/2}=
g_{a,b}(-x,y-x-1/2)e^{-\pi\Imun x/2},
\end{multline}
which is equivalent to \eqref{eq:fund1}.
Combining \eqref{eq:bgac-gab} and \eqref{eq:gac-gba} we also have that
\begin{equation}
\bar g_{a,c}(x,y)=g_{a,b}(-x,y-x-1/2)e^{-\pi\Imun x/2}=e^{\frac{\pi\Imun}{6}}g_{b,c}(-y+x+1/2,-y)e^{-\pi\Imun y/2}
\end{equation}
To obtain \eqref{eq:fund2}, also use the first-coordinate
quasiperiodicity
$g_{b,c}(s+1,t)=e^{-\pi\Imun t}g_{b,c}(s,t)$ with $t=-y$.
Replacing $x-y+1/2$ by $x-y-1/2$ contributes
$e^{\pi\Imun y}$ before complex conjugation, and hence
$e^{-\pi\Imun y}$ afterwards.
\end{proof}

\section{Regularity and integration of analytic germs}
\label{sec:analytic-germs}

\begin{lemma}\label{lem:weight-regularity}
For fixed physical $\la$, the functions $g^\pm_{a,c}$ are smooth
sections on the real state torus when $a,c,b=1/2-a-c>0$. They extend
holomorphically to a neighbourhood of every compact subset of this
positive real parameter and state space. Their derivatives are
uniformly bounded on smaller compact subsets of that neighbourhood.
\end{lemma}
\begin{proof}
Put $\cla=\Imun C$, $C>0$. The asymptotics of the quantum dilogarithm
in Lemma~\ref{lem:qdl-asymptotics}, applied to
\eqref{eq:psi}, gives decay rates $4\pi Ca$ at $-\infty$ and
$4\pi Cc$ at $+\infty$. On a compact positive shape set these rates
have a positive common lower bound. The shifted arguments lie
strictly inside the pole-free strip of the reciprocal quantum
dilogarithm. A sufficiently small complex thickening preserves both
properties. Cauchy's estimates on a smaller thickening give, for
every finite derivative order, a bound by a polynomial in $|t|$
times $e^{-\kappa|t|}$, for some $\kappa>0$ independent of the
parameters in the compact set.

By \eqref{eq:tpsip-psi}, the same conclusion holds for
$\tilde\psi'_{a,c}=e^{-\pi\Imun/12}\psi_{c,b}$. In the WGZ series
the extra factor has absolute value $e^{-2\pi m\Im y}$, apart from
a bounded prefactor. Choose $|\Im y|<\kappa/(4\pi)$ and let $x$
range over a small complex neighbourhood of $[0,1]$. Every
differentiated series is then dominated by a summable sequence of
the form $C_N(1+|m|)^N e^{-\kappa|m|/2}$. Normal convergence proves
holomorphicity and permits termwise differentiation. Integer
translates are handled by the line-bundle multipliers. Finally
$g^-_{a,c}(x,y)=\overline{g^+_{\bar a,\bar c}(\bar x,\bar y)}$
is the holomorphic reflection, and has the same estimates.
\end{proof}

Fibre integrals over compact abelian groups will be continued
to complex parameters by means of the following general lemma. For a
compact abelian Lie group $G$ we write $G_{\bC}$ for its
complexification; if $G\cong(\REALS/\INTEGERS)^n\times F$ with $F$
finite, then $G_{\bC}\cong(\bC/\INTEGERS)^n\times F$, and a
homomorphism $\varpi\colon G\to H$ of compact abelian Lie groups
extends uniquely to a holomorphic homomorphism
$\varpi_{\bC}\colon G_{\bC}\to H_{\bC}$. The probability Haar measure
$\nu_K$ of a compact abelian Lie group $K$ is a density. Choose
an orientation in real Lie-algebra coordinates and regard this
density as a top-degree form. It extends uniquely to a
translation-invariant holomorphic top-degree form $\nu_K^{\bC}$ on
$K_{\bC}$, normalized so that each oriented component of $K$ has
integral $1/|\pi_0(K)|$. Reversing the coordinate orientation changes
both the form and the cycle orientation, leaving all integrals
unchanged. A \emph{cycle} in a coset
of $K_{\bC}$ means a compact real submanifold of the form
$g+K$ with $g\in G_{\bC}$, carrying the translate of $\nu_K^{\bC}$.

\begin{lemma}\label{lem:holomorphic-integral}
Let $\varpi\colon G\to H$ be a homomorphism of compact abelian Lie
groups, let $H'=\varpi(G)$ be its image and $K=\ker\varpi$ its kernel,
possibly disconnected. Let $\Lambda$ be a connected complex manifold
of parameters, let $h_0\in H'$ and $\lambda_0\in\Lambda$, and let
$f$ be a holomorphic function on an open set
$\mathcal U\subset G_{\bC}\times\Lambda$ containing
$\varpi^{-1}(h_0)\times\{\lambda_0\}$.
\begin{enumerate}
\renewcommand{\labelenumi}{\textup{(\roman{enumi})}}
\item There are an open neighbourhood $V$ of $(h_0,\lambda_0)$ in
$H'_{\bC}\times\Lambda$ and a holomorphic map
$s\colon V\to G_{\bC}$ with $\varpi_{\bC}(s(h,\lambda))=h-h_0$ and
$s(h_0,\lambda)=0$, such that for every $(h,\lambda)\in V$ the cycle
$\Gamma_{h,\lambda}=s(h,\lambda)+\varpi^{-1}(h_0)$ lies in
$\varpi_{\bC}^{-1}(h)\cap\mathcal U_\lambda$, where
$\mathcal U_\lambda=\{g:(g,\lambda)\in\mathcal U\}$. The function
\[
I(h,\lambda)=\int_{\Gamma_{h,\lambda}}f(\,\cdot\,,\lambda)\,
\nu^{\bC}_{K}
=\int_{\varpi^{-1}(h_0)}f\bigl(s(h,\lambda)+g,\lambda\bigr)\,d\nu_{K}(g)
\]
is holomorphic on $V$.
\item For $(h,\lambda)\in H'\times\Lambda$ with $(h,\lambda)\in V$ and
$\varpi^{-1}(h)\times\{\lambda\}\subset\mathcal U$, the value
$I(h,\lambda)$ is the probability Haar fibre integral
$\int_{\varpi^{-1}(h)}f(g,\lambda)\,d\nu_h(g)$ of Lemma~\ref{lem:haar}.
In particular $I$ does not depend on the choice of $s$, and the germ of
$I$ at $(h_0,\lambda_0)$ is the unique holomorphic germ whose
restriction to $H'\times\Lambda$ is the Haar fibre integral;
only the group variable is required to be real.
\item More generally, $I(h,\lambda)$ is unchanged if
$\Gamma_{h,\lambda}$ is replaced by any cycle $g'+\varpi^{-1}(h_0)$,
$g'\in\varpi_{\bC}^{-1}(h-h_0)$, which is homotopic to
$\Gamma_{h,\lambda}$ through cycles of this form inside
$\varpi_{\bC}^{-1}(h)\cap\mathcal U_\lambda$.
\item The same statements hold when $f$ is a holomorphic section of
the pullback to $\mathcal U$ of a holomorphic line bundle on
$H'_{\bC}\times\Lambda$; then $I$ is a holomorphic section of that
bundle near $(h_0,\lambda_0)$, independent of the local frame used to
compute it. They also hold on affine translates of $G$, $H$ and on
the complexified cosets of Lemma~\ref{lem:haar}.
\end{enumerate}
\end{lemma}
\begin{proof}
(i) Let $\mathfrak g$ and $\mathfrak h'$ be the Lie algebras of $G$
and $H'$. The differential $d\varpi\colon\mathfrak g\to\mathfrak h'$
is surjective; choose a linear right inverse
$\sigma\colon\mathfrak h'\to\mathfrak g$ and extend it complex
linearly. Near $0\in H'_{\bC}$ the exponential map
$\exp\colon\mathfrak h'_{\bC}\to H'_{\bC}$ has a holomorphic local
inverse $\log$, and we put
$s(h,\lambda)=\exp\bigl(\sigma(\log(h-h_0))\bigr)$ (we write the
group operations additively). Then
$\varpi_{\bC}(s(h,\lambda))=\exp(d\varpi\,\sigma\log(h-h_0))=h-h_0$,
so $s(h,\lambda)+\varpi^{-1}(h_0)\subset\varpi_{\bC}^{-1}(h)$. Since
$\varpi^{-1}(h_0)\times\{\lambda_0\}$ is compact and contained in the
open set $\mathcal U$, and $s(h_0,\lambda)=0$, the cycle
$\Gamma_{h,\lambda}\times\{\lambda\}$ lies in $\mathcal U$ for
$(h,\lambda)$ in a neighbourhood $V$ of $(h_0,\lambda_0)$. The second
expression for $I$ follows by translating the cycle back to
$\varpi^{-1}(h_0)$, since $\nu^{\bC}_{K}$ is translation invariant.
Its integrand is holomorphic in $(h,\lambda)$ for each
$g\in\varpi^{-1}(h_0)$ and jointly continuous, hence uniformly
bounded on the product of a smaller neighbourhood with the compact
fibre; by Morera's theorem and Fubini's theorem, or by differentiation
under the integral sign, $I$ is holomorphic.

(ii) The right inverse $\sigma$ was chosen over $\REALS$.
Consequently $s(h,\lambda)$ is real whenever $h\in H'$ is real,
and it is independent of $\lambda$. Its translated cycle is then
exactly $\varpi^{-1}(h)$, so $I(h,\lambda)$ is the Haar fibre
integral for every complex $\lambda$ in the neighbourhood.

For another local holomorphic splitting $s'$ with
$s'(h_0,\lambda)=0$, the difference $s'-s$ lies in the identity
component of $K_{\bC}$ and is small after shrinking $V$.
Use a local logarithm in $K_{\bC}$ to join the two translated cycles
inside $\mathcal U_\lambda$. The holomorphic top-degree form
$f\,\nu_K^{\bC}$ is closed on the complex fibre, so Stokes' theorem
shows that its integrals over these cycles agree, component by
component. This proves independence of the local splitting.
For uniqueness, fix any complex $\lambda$ near $\lambda_0$ and use
the real form $H'$ as a uniqueness set in $H'_{\bC}$. Varying
$\lambda$ proves uniqueness of the joint holomorphic germ; no real
structure on $\Lambda$ is required.

(iii) The same Stokes argument applies to the stated homotopy,
with each component weighted by $1/|\pi_0(K)|$.

(iv) In a local holomorphic frame of the base line bundle over a
neighbourhood of $(h_0,\lambda_0)$, the section $f$ is a holomorphic
function to which (i)--(iii) apply. A change of frame multiplies $f$
by a nowhere vanishing holomorphic function of $(h,\lambda)$ alone,
which is constant on each fibre and factors out of the fibre
integral; hence $I$ transforms as a section of the base bundle. For
affine translates one translates the groups and the cycles by a fixed
element.
\end{proof}

In this paper Lemma~\ref{lem:holomorphic-integral} is applied
to the homomorphisms $q_X$ and $q^T_X$ of
Section~\ref{sec:normalized-fibres}, to $H_X$ for closed $X$, and to
the maps $\rho$ of Section~\ref{sec:conditional-gluing} on the cosets
$\mathcal A_{\alpha,\zeta}$; the integrand $f$ is the contracted
weight $b_X$, its cut version $\operatorname{ev}_\Sigma b_{X^c}$, or a
translate of one of these, with the complex shapes and the level as
parameters $\lambda$. In each case Lemma~\ref{lem:weight-regularity}
provides the holomorphy of $f$ on a neighbourhood of every compact
real fibre over a positive shape.

These two lemmas and Lemma~\ref{lem:haar} prove the existence and
regularity parts of Theorems~\ref{main} and~\ref{mainT}; see
Section~\ref{sec:proofs-main}. A torsion
component is kept fixed in taking a complex germ. The shape quotient
in Proposition~\ref{prop:gauge-rigorous} means invariant germs on
the positive balanced locus.

\section{Continuation of an individual weight}
\label{Ancont}
The analytical properties of $g_{a,c}$ are determined by those of the
function $\phi_\la$ which is obtained by applying the series
\eqref{eq:wgz} of the WGZ transformation to $\QDILOG$,
\[
\phi_\la(x,y)=e^{\pi\Imun xy}\sum_{m\in\INTEGERS}\QDILOG(x+m)e^{2\pi\Imun my}.
\]
As $\QDILOG$ is not a Schwartz function, this is not an instance of
Appendix~\ref{WGZ}: the series is shown below to converge in the domain
$D_1$ of \eqref{eq:d1}, and we write $\phi_\la=W\QDILOG$ in this sense.
The two functions are related as follows.

\begin{lemma}\label{lem:g-phi}
For $a,c>0$ with $b=\tfrac12-a-c>0$ and real $x,y$,
\begin{equation}\label{eq:g-phi-global}
g_{a,c}\left(x,y\right)
=\phi_\la\left(y +(1-2c)\cla,\,-x-2a\cla\right)
e^{\pi\Imun\cla ((1-2c)x-2ay)}e^{\pi\Imun\cla^2(8a+4c-1)/6},
\end{equation}
where the point $\bigl(y+(1-2c)\cla,-x-2a\cla\bigr)$ lies in the
domain $D_1$ of \eqref{eq:d1}, on which the series defining
$\phi_\la$ converges.
\end{lemma}
\begin{proof}
By the second form of \eqref{eq:gac-gba},
$g_{a,c}(x,y)=e^{-\pi\Imun/12}e^{\pi\Imun y/2}g_{c,b}(X,Y)$ with
$X=-y$ and $Y=x-y+\tfrac12$, and by \eqref{eq:gac} and
\eqref{eq:tpsip-psi},
\[
g_{c,b}(X,Y)=e^{\pi\Imun XY}\sum_{m\in\INTEGERS}\tilde\psi'_{c,b}(X+m)
e^{2\pi\Imun mY}
=e^{-\pi\Imun/12}e^{\pi\Imun XY}\sum_{m\in\INTEGERS}\psi_{b,a}(X+m)
e^{2\pi\Imun mY}.
\]
Write $\kappa=\cla$. The inversion relation \eqref{eq:inversion}
gives $\QDILOG(w)^{-1}=\zeta_{inv}e^{-\pi\Imun w^2}\QDILOG(-w)$.
Substituting this into \eqref{eq:psi} yields
\[
\psi_{b,a}(t)=\zeta_{inv}\,e^{-\pi\Imun(t-2\kappa(a+b))^2}
\QDILOG\bigl(2\kappa(a+b)-t\bigr)
e^{-4\pi\Imun\kappa b(t-\kappa(a+b))}
e^{-\pi\Imun\kappa^2(4(b-a)+1)/6}.
\]
Insert $t=X+m=-y+m$, replace the summation index $m$ by $-m$, and
put $X'=y+(1-2c)\kappa=y+2\kappa(a+b)$ and $Y'=-x-2a\kappa$. Then
$2\kappa(a+b)-t=X'+m$, and the $m$-dependent factors combine to
\begin{multline*}
e^{-\pi\Imun(X'+m)^2}e^{4\pi\Imun\kappa bm}e^{-2\pi\Imun mY}
=e^{-\pi\Imun X'^2}e^{-2\pi\Imun X'm}(-1)^m
e^{4\pi\Imun\kappa bm}e^{-2\pi\Imun m(x-y)}(-1)^m\\
=e^{-\pi\Imun X'^2}e^{2\pi\Imun mY'},
\end{multline*}
because $-X'+2\kappa b-x+y=-x-\kappa(1-2c-2b)=-x-2a\kappa=Y'$.
Hence
\begin{multline*}
g_{a,c}(x,y)=e^{-\pi\Imun/6}\zeta_{inv}\,
e^{\pi\Imun(y/2+XY-X'^2-X'Y')}e^{4\pi\Imun\kappa by}\\
\times e^{4\pi\Imun\kappa^2b(a+b)}e^{-\pi\Imun\kappa^2(4(b-a)+1)/6}
\,\phi_\la(X',Y').
\end{multline*}
Expanding $XY=-xy+y^2-y/2$, $X'^2$ and $X'Y'$, the terms $y/2$,
$xy$ and $y^2$ cancel, the terms linear in $x$ and $y$ add up to
$\pi\Imun\bigl(\kappa(1-2c)x-2a\kappa y\bigr)$, and the constant
terms add up to $\pi\Imun\kappa^2(8a+4c-1)/6$; for the latter one
uses $e^{-\pi\Imun/6}\zeta_{inv}=e^{\pi\Imun\kappa^2/3}$,
$a+b+c=\tfrac12$ and the identity
$-(1-2c)^2+4b(a+b)+2a(1-2c)=(1-2c)(2a+2b+2c-1)=0$.
This is \eqref{eq:g-phi-global}.
Finally, with $\kappa=\Imun C$, $\Im X'=(1-2c)C>0$,
$-\Im Y'=2aC>0$ and $\Im X'+\Im Y'=2bC>0$, so $(X',Y')\in D_1$.
\end{proof}

The relation \eqref{eq:gac-gba} implies that the following equalities hold true
\begin{multline}\label{eq:propor}
\phi_\la(x,y)=\zeta_oe^{\pi\Imun\left(\cla(x+y-\frac12)-y/2\right)}\phi_\la(\cla+y,1/2-x-y)\\
=\zeta_o^{-1}e^{\pi\Imun\left(\cla y+(x-\cla)/2\right)}\phi_\la(\cla+1/2-x-y,x-\cla),\quad \zeta_o=\exp(\pi\Imun(1-4\cla^2)/12).
\end{multline}
Consider the functions
\begin{equation}
\xi(x):= \prod_{(m,n)\in\INTEGERS_{\ge0}^2}\left(1-e^{2\pi(\Imun x-\la m-\la^{-1}n)}\right)
\end{equation}
and 
\begin{equation}
\chi(x,y):= \xi(\cla-x)\xi(-y)\xi(x+y+1/2).
\end{equation}
We then have the following
\begin{proposition}[Single-weight continuation]
\label{prop:single-weight}
The function $\phi_\la$ has a unique meromorphic continuation to
$\mathbb C^2$, and $\phi_\la\chi$ is holomorphic there.
\end{proposition}
\begin{proof}
Write
\[
r=\Re\la>0,\qquad s=\Re\la^{-1}>0,\qquad
C=\Im\cla=(r+s)/2>0.
\]
The product defining $\xi$ converges normally on compact subsets,
because
\[
\sum_{m,n\geq0}
\left|e^{2\pi(\Imun z-\la m-\la^{-1}n)}\right|
=\frac{e^{-2\pi\Im z}}{(1-e^{-2\pi r})(1-e^{-2\pi s})}.
\]
It therefore defines an entire, $1$-periodic function. Its zero
divisor counts every pair $(m,n)$ separately, including pairs
giving the same zero.

First we record the pole orders of $\QDILOG$. If $\Im z_0\geq C$,
let $N$ be the least positive integer for which
$\Im z_0-Nr<C$. Then
\[
-C<C-r\leq\Im(z_0-N\Imun\la)<C.
\]
Iterating \eqref{eq:shift} gives, near $z_0$,
\[
\QDILOG(z)=
\QDILOG(z-N\Imun\la)
\prod_{k=0}^{N-1}
\left(1+
e^{2\pi\la(z-(k+1/2)\Imun\la)}\right)^{-1}.
\]
The numerator is holomorphic and nonzero in the defining strip.
A denominator factor vanishes precisely when
\[
z_0=\cla+k\Imun\la+n\Imun\la^{-1},\qquad n\in\INTEGERS.
\]
For $k<N$, minimality of $N$ gives
$\Im z_0-kr\geq C$, hence $n\geq0$.
Conversely, every representation of $z_0$ with $k,n\geq0$
has $k<N$. Every individual factor has a simple zero, so the
pole order is the number of such representations. This argument
also covers coincident poles, for example when $\la^2$ is rational.
The defining strip and \eqref{eq:inversion} exclude any other poles.

The series $\sum_{m\in\INTEGERS}\QDILOG(x+m)e^{2\pi\Imun my}$
converges absolutely away from the summand poles in the domain
\begin{equation}\label{eq:d1}
D_1=\{(x,y)\in\COMPLEXS^2\vert\ \Im x>-\Im y>0\}
\end{equation}
to a meromorphic function with poles contained in
\begin{equation}
P_1:=
(\cla+\INTEGERS+\Imun\la\INTEGERS_{\ge0}
+\Imun\la^{-1}\INTEGERS_{\ge0})\times\COMPLEXS.
\end{equation}
Here are quantitative estimates for this convergence. Choose
\[
0<\rho<2\pi\min(r,s).
\]
By Lemma~\ref{lem:qdl-asymptotics}, uniformly for $x$ in a compact
set, the two tails satisfy
\begin{align*}
\QDILOG(x+m)&=1+O(e^{\rho m}), &&m\to-\infty,\\
\QDILOG(x+m)&=\zeta_{inv}^{-1}e^{\pi\Imun(x+m)^2}
  \bigl(1+O(e^{-\rho m})\bigr), &&m\to+\infty.
\end{align*}

Consequently, on a compact pole-free subset of $D_1$ there are
constants $K,\delta>0$ such that the absolute value of the summand
is at most $K e^{-2\pi\delta|m|}$ for sufficiently large $|m|$,
where one can take
\[
\delta=\min\{-\Im y,\Im(x+y)\}>0
\]
pointwise and then its minimum on the compact set. In particular,
for sufficiently large $N$ the sum of the tails with $|m|>N$
is bounded by
\[
\frac{2K e^{-2\pi\delta(N+1)}}{1-e^{-2\pi\delta}}.
\]
This proves normal convergence away from $P_1$. On a neighbourhood
of any point only finitely many summands can have poles.
The factor $\xi(\cla-x)$ cancels each such pole with its full
multiplicity: a pole of $\QDILOG(x+m)$ at
$x+m=\cla+k\Imun\la+n\Imun\la^{-1}$ is a zero of the
$(k,n)$ factor of $\xi(\cla-x)$.
If poles from several integer translates coincide, the multiplicity
of $\xi(\cla-x)$ is still at least the pole order of every summand,
and hence of their sum. The tails are already holomorphic.
Including the entire WGZ prefactor $e^{\pi\Imun xy}$, we conclude
that $F:=\chi\phi_\la$ is holomorphic throughout $D_1$.

Let
\[
A(x,y)=(\cla+y,1/2-x-y).
\]
Then $A^3$ is the identity and the $1$-periodicity of $\xi$ gives
$\chi\circ A=\chi$. The cyclic identities \eqref{eq:propor},
whose exponential factors are entire and nowhere zero, therefore
extend $F$ holomorphically to
\[
\begin{split}
D_2=A^{-1}D_1
 &=\{(x,y):\Im x<C,\ \Im(x+y)>0\},\\
D_3=A^{-2}D_1
 &=\{(x,y):\Im x<C,\ \Im y<0\}.
\end{split}
\]
The pairwise overlaps all equal the nonempty connected tube
\[
\{(x,y):\Im x<C,\ \Im y<0,\ \Im(x+y)>0\}.
\]
This is precisely the tube parametrized by
\eqref{eq:g-phi-global} with real state variables and
$a,c>0$, $a+c<\tfrac12$. Indeed, since $\cla=\Imun C$, the map
\[
(x,y,a,c)\longmapsto(X',Y')
=\bigl(y+(1-2c)\cla,-x-2a\cla\bigr)
\]
is an affine real isomorphism from $\REALS^4$ to $\bC^2$, with inverse
\[
x=-\Re Y',\qquad y=\Re X',\qquad
a=-\frac{\Im Y'}{2C},\qquad
c=\frac{C-\Im X'}{2C}.
\]
The three inequalities defining the overlap are exactly
$c>0$, $a>0$ and $a+c<\tfrac12$, respectively. Thus the
parametrized set is the entire open overlap, and the already proved
charged-weight symmetries give \eqref{eq:propor} directly at every
point of it. The three holomorphic definitions therefore agree.

We now apply Bochner's tube theorem \cite{Bochner1938}: a
holomorphic function on a tube over a connected open set
$B\subset\mathbb R^n$ extends holomorphically to the tube over
the convex hull of $B$. The union $D_1\cup D_2\cup D_3$ is such
a connected tube. Its imaginary base contains, for every $R>0$,
the three points
\[
(2R,-R),\qquad(-R,2R),\qquad(-R,-R).
\]
Their convex hulls exhaust $\mathbb R^2$ as $R\to\infty$.
Thus Bochner's theorem extends $F$ to all of $\mathbb C^2$.
Consequently $\phi_\la=F/\chi$ has a global meromorphic
continuation whose complete pole divisor, including
multiplicities, is bounded by the zero divisor of $\chi$.
The continuation is unique by the identity theorem.
\end{proof}

\section{Global meromorphic continuation}
\label{sec:extension}
Fix the physical value of $\la$ throughout this section and put
\[
\cla=\Imun C,\quad C>0,\qquad
p_0=e^{-2\pi\la},\quad q_0=e^{-2\pi\la^{-1}}.
\]
Both nomes have modulus strictly less than one, including when
$\la$ is positive real. They are not the nomes $q=e^{\pi\Imun\la^2}$,
$\bar q=e^{-\pi\Imun\la^{-2}}$ of the product formula
\eqref{eq:ratio} for $\Phi_\la$, which have modulus one when
$\la$ is real. For $0<|p|,|q|<1$ write
\begin{equation}\label{eq:double-product}
(u;p,q)_\infty=\prod_{m,n\geq0}(1-u p^m q^n).
\end{equation}
The product converges normally on compact $u$-sets, since
$\sum_{m,n\geq0}|u p^m q^n|$ is locally bounded. All multiplicities
are retained if different pairs $(m,n)$ give the same zero.

\begin{lemma}[Pole-cancelling factors for a tetrahedron]
\label{lem:tetra-denominator}
For $\varepsilon\in\{1,-1\}$ and $b=1/2-a-c$, set
\begin{align}
D^\varepsilon_{a,c}(s,t)
={}&(e^{-4\pi Ca+2\pi\Imun\varepsilon s};p_0,q_0)_\infty
(e^{-4\pi Cc-2\pi\Imun\varepsilon t};p_0,q_0)_\infty
\nonumber\\
&\qquad\cdot
(-e^{-4\pi Cb+2\pi\Imun\varepsilon(t-s)};p_0,q_0)_\infty.
\label{eq:tetra-denominator}
\end{align}
On the universal state cover,
$D^\varepsilon_{a,c}g^\varepsilon_{a,c}$ is holomorphic in
$(a,c,s,t)\in\mathbb C^4$. For positive shapes and real states all
three arguments of the double products have modulus less than one.
\end{lemma}
\begin{proof}
In \eqref{eq:g-phi-global} the arguments of $\phi_\la$ at the state
$(s,t)$ are $X=t+(1-2c)\cla$ and $Y=-s-2a\cla$. The three arguments of
$\xi$ in $\chi(X,Y)$ are
\[
\cla-X=2c\cla-t,\qquad -Y=s+2a\cla,\qquad
X+Y+1/2=t-s+2b\cla+1/2.
\]
Since $\xi(u)=(e^{2\pi\Imun u};p_0,q_0)_\infty$, their product
is exactly $D^+_{a,c}$. The single-weight theorem
(Proposition~\ref{prop:single-weight}) says that $\phi_\la\chi$
is holomorphic on $\mathbb C^2$. The remaining exponential in
\eqref{eq:g-phi-global} is entire, proving the positive-sign claim.
For the negative sign use holomorphic reflection. For physical
$\la$, conjugation either fixes the nomes or interchanges them;
the double product is symmetric in its nomes. Reflection therefore
gives precisely \eqref{eq:tetra-denominator} with $\varepsilon=-1$.
The last assertion follows from $C,a,b,c>0$.
\end{proof}

We use the following one-variable result. Its control of colliding
poles is essential; it is stronger than continuation along a
pole-free path. Statements of this kind are due to Rains
\cite[Section~10]{Rains2010}; for the convenience of the reader we
include a direct proof.

\begin{lemma}[Rains's contour-integral theorem]
\label{lem:rains}
Fix $0<|p|,|q|<1$ and a connected complex manifold $\Lambda$.
If $H(z,\lambda)$ is holomorphic on
$\mathbb C^*\times\Lambda$, then the germ
\begin{equation}\label{eq:rains-period}
J(\lambda,A,B)=\int_{|z|=1}
\frac{H(z,\lambda)}
{\prod_i(A_i z;p,q)_\infty\prod_j(B_j/z;p,q)_\infty}
\frac{dz}{2\pi\Imun z},\qquad |A_i|,|B_j|<1,
\end{equation}
has a unique single-valued meromorphic continuation to
$\Lambda\times(\mathbb C^*)^{r+s}$, where $r$ and $s$ are the numbers
of the $A_i$ and of the $B_j$, and
\[
\left(\prod_{i,j}(A_iB_j;p,q)_\infty\right)J
\]
is holomorphic there. Holomorphic spectator
parameters $\lambda\in\Lambda$ are allowed. Empty products are
permitted.
\end{lemma}
\begin{proof}
Write $A=(A_1,\ldots,A_r)$ and $B=(B_1,\ldots,B_s)$, and let
$\mathfrak S=\{p^mq^n:m,n\geq0\}\subset\mathbb C$. For
$c\in\mathfrak S$ let $\mu(c)$ be the finite number of pairs $(m,n)$
with $p^mq^n=c$, so that
\[
(u;p,q)_\infty=\prod_{c\in\mathfrak S}(1-uc)^{\mu(c)}.
\]
As a function of $z\in\mathbb C^*$, the integrand of
\eqref{eq:rains-period} is holomorphic outside the \emph{outer
points} $(A_ic)^{-1}$, which accumulate only at infinity, and the
\emph{inner points} $B_jc$, which accumulate only at zero; here $c$
ranges over $\mathfrak S$. The initial contour encloses the inner
points and excludes the outer ones.

\emph{Step 1: continuation off the collision locus.} An outer point
coincides with an inner point, $(A_ic)^{-1}=B_jc'$, if and only if
$A_iB_jcc'=1$. As $\mathfrak S\mathfrak S=\mathfrak S$, such a
coincidence occurs exactly on the zero set $Z$ of
\[
G(A,B)=\prod_{i,j}(A_iB_j;p,q)_\infty,
\]
which is the locally finite union of the hypersurfaces
$Z_{i,j,c}=\{A_iB_jc=1\}$ in $(\mathbb C^*)^{r+s}$,
$c\in\mathfrak S$. For $(A,B)\notin Z$ there is a cycle $\gamma$ in
$\mathbb C^*$ which avoids all outer and inner points, has winding
number one around zero and around every inner point, and has winding
number zero around every outer point: take a circle $|z|=\varrho$
which avoids these points, add small positively oriented circles
around the finitely many inner points outside it, and subtract small
circles around the finitely many outer points inside it. By Cauchy's
theorem, the integral over such a cycle of a function which is
holomorphic in $\mathbb C^*$ outside the outer and inner points
depends only on these winding numbers. The exceptional set is
discrete in $\mathbb C^*$ and the cycles are compact, so this use of
Cauchy's theorem requires no growth condition at zero or infinity.
Define $J(\lambda,A,B)$ as
the integral of the integrand of \eqref{eq:rains-period} over
$\gamma$. For $|A_i|,|B_j|<1$ the unit circle is such a cycle, so
that this agrees with \eqref{eq:rains-period}. The outer and inner
points depend continuously on $(A,B)$, and all but finitely many of
them remain in a neighbourhood of infinity or of zero which does not
meet $\gamma$; hence a cycle which serves at one point serves in a
neighbourhood of it. Consequently $J$ is holomorphic on
$\Lambda\times\bigl((\mathbb C^*)^{r+s}\setminus Z\bigr)$. This set is
connected, and $J$ is the unique continuation of the germ
\eqref{eq:rains-period} to it.

\emph{Step 2: a point on exactly one hypersurface.} Let $(A^0,B^0)$
lie on $Z_{i,j,c}$ and on no other of these hypersurfaces. If an
outer and an inner point coincide there,
$(A_{i'}c_1)^{-1}=B_{j'}c_2$, then $(A^0,B^0)\in Z_{i',j',c_1c_2}$,
hence $i'=i$, $j'=j$ and $c_1c_2=c$. The coincidences are therefore
the points $(A_ic_1)^{-1}=B_jc_2$, one for each of the finitely many
factorizations $c=c_1c_2$ in $\mathfrak S$. A second outer point
at such a coincidence would put $(A^0,B^0)$ on a second hypersurface
$Z_{i',j,c'}$, and a second inner point would put it on a second
$Z_{i,j',c'}$, contrary to the hypothesis. Coincidences involving
only outer points or only inner points elsewhere are harmless.
Choose disjoint
closed discs $\Delta_{c_1}\subset\mathbb C^*$ around these points
which contain no other outer or inner point, and a cycle $\gamma_0$
outside the discs which avoids all outer and inner points and has
winding number one around zero and around the remaining inner points,
and winding number zero around the remaining outer points and around
the discs. For $(A,B)\notin Z$ close to $(A^0,B^0)$, the disc
$\Delta_{c_1}$ contains exactly the outer point $a=(A_ic_1)^{-1}$
and the inner point $b=B_jc_2$, which are distinct, and $\gamma_0$
together with a small circle around $b$ in every disc is a cycle as
in Step~1. Hence
\[
J=J_0+\sum_{c_1c_2=c}R_{c_1},
\]
where $J_0$, the integral over $\gamma_0$, is holomorphic in a
neighbourhood of $(A^0,B^0)$, and where, with $\mu=\mu(c_1)$ and
$\nu=\mu(c_2)$,
\[
R_{c_1}=\oint_{|z-b|=\epsilon}
\frac{h_{c_1}(z)\,dz}{(z-a)^{\mu}(z-b)^{\nu}}
=\frac{2\pi\Imun}{(\nu-1)!}\,
\partial_z^{\nu-1}\Bigl[\frac{h_{c_1}(z)}{(z-a)^{\mu}}\Bigr]_{z=b},
\qquad 0<\epsilon<|a-b|.
\]
Here $h_{c_1}$ is holomorphic in $z\in\Delta_{c_1}$ and in
$(\lambda,A,B)$ near $\Lambda\times\{(A^0,B^0)\}$. It follows that
$(b-a)^{\mu+\nu-1}R_{c_1}$ is holomorphic there, and
\[
b-a=(A_ic_1)^{-1}(A_iB_jc-1).
\]
Finally $\mu(c_1)+\mu(c_2)-1\leq\mu(c)$: the pairs $(m,n)$
representing $c_1$ and those representing $c_2$ add up to pairs
representing $c$, and the sum of two finite subsets of
$\mathbb Z^2$ with $\mu$ and $\nu$ elements has at least
$\mu+\nu-1$ elements, as one sees by ordering $\mathbb Z^2$
lexicographically. Therefore $(1-A_iB_jc)^{\mu(c)}J$ is holomorphic
near $(A^0,B^0)$. The remaining factors of $G$ do not vanish at this
point, so that $GJ$ is holomorphic there as well.

\emph{Step 3: codimension two.} Two different hypersurfaces
$Z_{i,j,c}$ and $Z_{i',j',c'}$ are disjoint if $(i,j)=(i',j')$, and
otherwise meet transversally, because the $A_i$ and $B_j$ are
independent variables. Hence the points of $Z$ which lie on more than
one of these hypersurfaces form an analytic subset of codimension
two. By Steps~1 and~2, $GJ$ is holomorphic outside this subset, and
by the second Riemann extension theorem (a holomorphic function on
the complement of an analytic subset of codimension at least two
extends holomorphically; see \cite[Chapter~7]{GrauertRemmert1984})
it is holomorphic on
$\Lambda\times(\mathbb C^*)^{r+s}$. If $r=0$ or $s=0$, then $Z$ is
empty, $G=1$, and Step~1 suffices. In all cases $J=(GJ)/G$ is
meromorphic on $\Lambda\times(\mathbb C^*)^{r+s}$, and by the identity
theorem it is the only meromorphic continuation of the germ
\eqref{eq:rains-period}.

The argument allows coincident nomes, through the multiplicities
$\mu(c)$. One may subsequently substitute holomorphic nonzero
functions for $A_i,B_j$, including equal functions for repeated
factors, provided the substituted $G$ is not identically zero on
any connected parameter component. Indeed, the substituted
holomorphic numerator $GJ$ divided by that nonzero holomorphic
denominator defines the meromorphic restriction. If the family is
contained in $G=0$, a restriction requires a separate proof of
removability; holomorphy of $GJ$ alone is insufficient. This is the
restriction convention of Convention~\ref{conv:meromorphic}.
\end{proof}

The one-variable continuation mechanism and its uniqueness are
those of \cite[Theorem~10.2]{Rains2010}. The additional reduction
needed here is to keep exponential-monomial denominators in this
class under successive integrations: finite covers make their roots
single-valued monomials, and nome refinement gives a common pair of
nomes. The next theorem makes this reduction and its finite global
pole majorant explicit; independence of the elimination choices
then follows from uniqueness.

\begin{theorem}[Toric period continuation]
\label{thm:toric-period}
Let $p=e^P$, $q=e^Q$, where $\Re P,\Re Q<0$. Let
$v_1,\ldots,v_N\in\mathbb Z^d$, let $L_j$ be complex affine-linear
functions on $\mathbb C^M$, and let $H$ be holomorphic on
$(\mathbb C^*)^d\times\mathbb C^M$. Suppose
\[
U=\{\lambda\in\mathbb C^M:|e^{L_j(\lambda)}|<1
\text{ for every }j\}
\]
is nonempty. This domain is open and convex, since its defining
conditions are $\Re L_j(\lambda)<0$. The holomorphic function on $U$ defined by
\begin{equation}\label{eq:toric-period}
I(\lambda)=\int_{|z_1|=\cdots=|z_d|=1}
\frac{H(z,\lambda)}
{\prod_{j=1}^N(e^{L_j(\lambda)}z^{v_j};p,q)_\infty}
\prod_{k=1}^d\frac{dz_k}{2\pi\Imun z_k}
\end{equation}
extends uniquely to a single-valued meromorphic function on
$\mathbb C^M$. More precisely, a finite elimination procedure
produces affine-linear $\widetilde L_r$ and an integer $D\geq1$
such that
\begin{equation}\label{eq:global-majorant}
\mathcal E(\lambda)=\prod_r
(e^{\widetilde L_r(\lambda)};e^{P/D},e^{Q/D})_\infty,
\qquad \mathcal E I\text{ is entire},
\end{equation}
and $\mathcal E$ has no zeros on $U$. Consequently the polar set of
$I$ is contained in the locally finite union of the affine
hyperplanes
\begin{equation}\label{eq:global-polar-hyperplanes}
\widetilde L_r(\lambda)
=-mP/D-nQ/D+2\pi\Imun k,
\quad m,n\geq0,\quad k\in\mathbb Z,
\end{equation}
the continuation has no monodromy, and it is independent of the
order in which the variables are eliminated and of the finite covers
used in the elimination.
\end{theorem}
\begin{proof}
We induct on $d$. For $d=0$, take $\mathcal E$ to be the denominator in
\eqref{eq:toric-period}. For the induction step write
$v_j=(m_j,v'_j)$ and let $L$ be the least common multiple of the
nonzero $|m_j|$, taking $L=1$ if all $m_j=0$. Replace the remaining
variables by $z'_k=w_k^L$. The normalized torus integral is unchanged:
the power map is a finite covering and pushes probability Haar
measure to probability Haar measure, so that no Jacobian or
covering-degree factor appears in \eqref{eq:toric-period}.

For $m_j\ne0$, put $k_j=|m_j|$,
$\sigma_j=\operatorname{sign}(m_j)$, and
\[
a_{j,r}(w,\lambda)=
\exp\bigl(L_j(\lambda)/k_j+2\pi\Imun r/k_j\bigr)
w^{L v'_j/k_j},\qquad 0\leq r<k_j.
\]
All exponents of $w$ are integers. Factoring each binomial in the
normally convergent product gives the exact identity
\begin{equation}\label{eq:root-factorization}
(e^{L_j}z_1^{m_j}w^{Lv'_j};e^P,e^Q)_\infty
=\prod_{r=0}^{k_j-1}
(a_{j,r}z_1^{\sigma_j};e^{P/k_j},e^{Q/k_j})_\infty.
\end{equation}
There are no parameter branch choices: the roots are the displayed
exponentials on $\mathbb C^M$.

Put $\widehat p=e^{P/L}$ and $\widehat q=e^{Q/L}$. Each nome pair
on the right of \eqref{eq:root-factorization} is
$(\widehat p^{h_j},\widehat q^{h_j})$ for the positive integer
$h_j=L/k_j$. A denominator can be replaced by one with the finer
nome pair $(\widehat p,\widehat q)$ because
\begin{equation}\label{eq:refine-nomes}
\frac{(u;\widehat p,\widehat q)_\infty}
{(u;\widehat p^h,\widehat q^h)_\infty}
=\prod_{(a,b)\notin h\mathbb Z_{\geq0}^2}
(1-u\widehat p^a\widehat q^b),\qquad a,b\geq0,
\end{equation}
is entire in $u$. Here the product omits pairs for which both indices
are divisible by $h$. Normal convergence proves that zeros of the
denominator in the quotient are removable with their multiplicities.
Multiply the numerator by these entire quotients, also for the
$m_j=0$ factors using $h=L$. The original integrand becomes
\[
\frac{\widehat H(z_1,w,\lambda)}
{D_0(w,\lambda)
 \prod_{a\in\mathcal A}(a z_1;\widehat p,\widehat q)_\infty
 \prod_{b\in\mathcal B}(b/z_1;\widehat p,\widehat q)_\infty},
\]
where $\widehat H$ is holomorphic globally, $\mathcal A$ and
$\mathcal B$ are the factors with $\sigma_j=1$ and $-1$, respectively,
and $D_0$ is the product of the refined factors with $m_j=0$.

Keep $D_0(w,\lambda)$ outside the $z_1$ integration. Apply
Lemma~\ref{lem:rains} with the holomorphic numerator $\widehat H$,
not the possibly meromorphic quotient $\widehat H/D_0$, treating
$(w,\lambda)$ as spectator parameters and the coefficients $a,b$
first as independent variables. Substitute their displayed
holomorphic monomials in the resulting holomorphic numerator and
collision product, and restore the spectator denominator $D_0$.
The first integral is therefore
\begin{equation}\label{eq:elimination-step}
\frac{H_1(w,\lambda)}
{D_0(w,\lambda)
 \prod_{a\in\mathcal A,b\in\mathcal B}
 (ab;\widehat p,\widehat q)_\infty},
\end{equation}
with $H_1$ holomorphic on $(\mathbb C^*)^{d-1}\times\mathbb C^M$.
Lemma~\ref{lem:rains} supplies precisely the global holomorphy of
$H_1$ required at the next step. The coefficients in
\eqref{eq:elimination-step} are again exponentials of affine-linear
functions of $\lambda$ times integral monomials in $w$. The required
nonvanishing of the substituted denominator follows on the initial
domain, as we now check.

For $\lambda\in U$ and $|w_k|=1$ one has
$|a_{j,r}|=|e^{L_j}|^{1/k_j}<1$. Hence every newly introduced
coefficient $ab$ has modulus less than one, as do the coefficients
in $D_0$. Thus the next integral has exactly the same hypotheses
in $d-1$ variables, with a nonempty initial domain containing $U$.
This also ensures that no denominator introduced by substitution
vanishes identically on parameter space. Fubini identifies this
iterated integral with the original one on $U$.

Repeat until no integration variables remain. There are finitely
many factors at each step and exactly $d$ elimination steps. Each
nome refinement divides $P,Q$ by a positive integer, so the final
pair has the form $(e^{P/D},e^{Q/D})$. The resulting numerator is
entire, and the denominator is \eqref{eq:global-majorant}. All its
arguments still have modulus less than one on $U$, proving that
$\mathcal E$ is nonzero there. The quotient therefore gives the desired
global continuation. Since $U$ is a nonempty open set and
$\mathbb C^M$ is connected, the identity theorem proves uniqueness,
including independence of the variable order and the finite covers.
\end{proof}

\begin{example}[Two-variable elimination with a finite cover]
\label{ex:toric-elimination}
Let $\lambda=(\lambda_1,\lambda_2,\lambda_3)\in\mathbb C^3$,
write $A=e^{\lambda_1}$, $B=e^{\lambda_2}$, $C_1=e^{\lambda_3}$,
and let $H$ be holomorphic on $(\mathbb C^*)^2\times\mathbb C^3$.
Consider
\[
I(\lambda)=\int_{|z_1|=|z_2|=1}
\frac{H(z_1,z_2,\lambda)}
{(Az_1^2z_2;p,q)_\infty
 (Bz_2/z_1;p,q)_\infty(C_1/z_2;p,q)_\infty}
\frac{dz_1}{2\pi\Imun z_1}\frac{dz_2}{2\pi\Imun z_2},
\]
initially for $\Re\lambda_j<0$. Here $p=e^P$, $q=e^Q$,
with the fixed logarithms $P,Q$ of the theorem.

For the first elimination, $L=\operatorname{lcm}(2,1)=2$.
Set $z_2=w^2$, $p_1=e^{P/2}$, $q_1=e^{Q/2}$, and
$\alpha=e^{\lambda_1/2}$. Probability Haar measure is preserved
by $w\mapsto w^2$, so there is no extra factor of two. The first
denominator factors as
\[
(Az_1^2w^2;p,q)_\infty
=(\alpha wz_1;p_1,q_1)_\infty
 (-\alpha wz_1;p_1,q_1)_\infty.
\]
Define the entire function
\[
R_2(u)=\frac{(u;p_1,q_1)_\infty}{(u;p,q)_\infty},
\]
where the quotient is interpreted by the product in
\eqref{eq:refine-nomes}. Refining the other two denominators
replaces the numerator by
\[
\widehat H(z_1,w,\lambda)
=H(z_1,w^2,\lambda)R_2(Bw^2/z_1)R_2(C_1/w^2).
\]
This is holomorphic globally on $(\mathbb C^*)^2\times\mathbb C^3$.
Lemma~\ref{lem:rains}, applied to $z_1$, therefore gives
\[
I(\lambda)=\int_{|w|=1}
\frac{H_1(w,\lambda)}
{(\alpha Bw^3;p_1,q_1)_\infty
 (-\alpha Bw^3;p_1,q_1)_\infty(C_1/w^2;p_1,q_1)_\infty}
\frac{dw}{2\pi\Imun w},
\]
where $H_1$ is holomorphic on $\mathbb C^*\times\mathbb C^3$.
The two new factors are precisely the collision products of
$\alpha w,-\alpha w$ with $Bw^2$.

For the second elimination the nonzero exponents are $3,3,-2$.
Factoring the two cubic terms produces the six coefficients
\[
a_k=\exp\left(\frac{\lambda_1}{6}
             +\frac{\lambda_2}{3}+\frac{\pi\Imun k}{3}\right),
\qquad k=0,\ldots,5,
\]
with nome pair $(e^{P/6},e^{Q/6})$. Factoring the quadratic term
produces $b_j=\exp(\lambda_3/2+\pi\Imun j)$, $j=0,1$,
with nome pair $(e^{P/4},e^{Q/4})$. Refine both pairs to
$p_*=e^{P/12}$, $q_*=e^{Q/12}$, multiplying the numerator
by the corresponding entire quotients. A second application of
Lemma~\ref{lem:rains} shows that
\[
\mathcal E(\lambda)=\prod_{k=0}^{5}\prod_{j=0}^{1}
\left(
e^{\lambda_1/6+\lambda_2/3+\lambda_3/2
   +\pi\Imun k/3+\pi\Imun j};p_*,q_*\right)_\infty
\]
makes $\mathcal E(\lambda)I(\lambda)$ entire. Repeated factors in this
product are retained. All its arguments have modulus less than
one when $\Re\lambda_j<0$, so $\mathcal E$ is nonzero on the initial
domain. Every root used above is an explicitly specified exponential
of an affine-linear parameter, so no parameter monodromy is
introduced. The example exhibits the finite cover, the closure of the
denominator class, and the global majorant with $D=12$.
\end{example}

\begin{remark}[Poles and monodromy]
The bound \eqref{eq:global-majorant} is generally not sharp. Its
zeros are the hyperplanes \eqref{eq:global-polar-hyperplanes}, and
only finitely many of them meet a compact parameter set:
real parts bound $m,n$ because $-\Re P,-\Re Q>0$, and imaginary
parts then bound $k$. Thus the proof controls finite-order poles
also where several pinches occur, and the representation of $I$ as
a global quotient of entire functions shows that the monodromy is
trivial. For comparison, the elementary contour integral
$\int_0^1(1-t\cos(2\pi x))^{-1}dx=(1-t^2)^{-1/2}$ does not have
the globally labelled exponential-monomial factors required here.
\end{remark}

\begin{theorem}[Global meromorphic state integrals]
\label{thm:global-continuation}
For fixed physical $\la$ and a fixed finite branched triangulation,
the positive-shape section in Definition~\ref{Maindef} extends
uniquely to a single-valued meromorphic section on
\[
\widetilde{LS}^{\mathbb C}(X)\times\mathcal Q_{X,\mathbb C}
\]
with values in the compatible boundary line bundle. The same is true
for the torsion-conditioned section on the complexification of the
image of $q_X^T$. No assumption on $H_2(X',\mathbb Z)$ is needed.
On the universal cover of each parameter component, a finite product
of the form \eqref{eq:global-majorant} makes the section holomorphic
in a lifted boundary frame.
\end{theorem}
\begin{proof}
We continue on an affine cover of each state component and then
descend in two stages. In Step~5 we quotient by the complex fibre
translations $N\mathbb C^d$; Step~6 imposes the remaining integral
deck identifications, with the boundary multipliers. The resulting
base is
$\widetilde{LS}^{\mathbb C}(X)\times\mathcal Q_{X,\mathbb C}$,
and the section takes values in the pullback of
$\mathcal L_{\partial X,\mathbb C}$. For the torsion-conditioned
integral, Steps~7--8 use the larger fibre, average its finitely many
connected components with equal masses, and perform the analogous
descent to
$\widetilde{LS}^{\mathbb C}(X)\times\mathcal Q^T_{X,\mathbb C}$,
where $\mathcal Q^T_{X,\mathbb C}$ is the complexified image of
$q_X^T$. We now specify these covers and actions.

\emph{Step 1: lifted states.} Write $\mathcal Z_{X,\bC}
=Z^1(\tilde X;\mathbb S_{\bC})$ and let
\[
\widehat{\mathcal Z}=\{y\in\bC^{\Delta_1(\tilde X)}:
\delta y\in\INTEGERS^{\Delta_2(\tilde X)}\}
\]
be the set of complex lifts of complex cocycles. The projection
$\widehat{\mathcal Z}\to\mathcal Z_{X,\bC}$ is surjective with kernel
$\INTEGERS^{\Delta_1(\tilde X)}$. The connected components of
$\widehat{\mathcal Z}$ are the complex affine spaces
$y_*+Z^1(\tilde X;\bC)$, and they are permuted by
$\INTEGERS^{\Delta_1(\tilde X)}$ with finitely many orbits, one for
each component of $\mathcal Z_{X,\bC}$ (these are indexed by
$\operatorname{Tor}H^2(\tilde X;\INTEGERS)$, by the Bockstein sequence
of $0\to\INTEGERS\to\bC\to\mathbb S_{\bC}\to0$). Fix a component
$\mathcal P$ of $\mathcal Z_{X,\bC}$, choose a component
$\mathcal A=y_*+Z^1(\tilde X;\bC)$ of $\widehat{\mathcal Z}$ above it
with $y_*$ real, and let
$\Lambda_{\mathcal A}=\{m\in\INTEGERS^{\Delta_1(\tilde X)}:\delta m=0\}
=Z^1(\tilde X;\INTEGERS)$ be the stabilizer of $\mathcal A$. Then
$\mathcal A\to\mathcal P$ is the universal covering of the complex
torus $\mathcal P$ with deck group $\Lambda_{\mathcal A}$, acting by
translation.

\emph{Step 2: the fibres of $q_X$ in lifted coordinates.} Let
$U_{\mathrm{rel}}\subset\mathbb S^{\Delta_0(\tilde X)}$ be the torus
of zero-cochains which are constant on each connected component of
the truncated boundary $\widetilde{\partial X}$, with coordinates
$u=(u_1,\ldots,u_d)$ given by the values at the interior vertices
and on the boundary components; for closed $X$, $d=|\Delta_0(\tilde
X)|$. The kernel of $q_X$ is $\delta(U_{\mathrm{rel}})$: a cocycle
$z$ has trivial class and trivial boundary restriction if and only if
$z=\delta u$ with $\delta u|_{\widetilde{\partial X}}=0$, that is,
with $u\in U_{\mathrm{rel}}$. The kernel of
$\delta\colon U_{\mathrm{rel}}\to\mathcal Z_X$ consists of the
locally constant cochains, which form a torus of probability mass
one in $U_{\mathrm{rel}}$; by Lemma~\ref{lem:haar} the probability
Haar fibre integral over $q_X^{-1}(\alpha,\zeta)=z_0+\delta(U_{\mathrm{rel}})$
equals the integral over $u\in U_{\mathrm{rel}}$ of the integrand at
$z_0+\delta u$. Lifting, a point of the fibre through a lift $y_0\in
\mathcal A$ is $y=y_0+N\tilde u$, where $\tilde u\in\REALS^d$ is a
lift of $u$ and $N$ is the integer matrix of $\delta$ in the
coordinates $u$. The same holds with complex $u$ and $y_0$.

\emph{Step 3: the integrand in lifted coordinates.} In the
one-skeleton of each tetrahedral octahedron $\tilde T$, fix a root
and a spanning tree. Let $P_T(y)$ be the potential obtained by signed
path sums of $y$ along this tree, with value zero at the root.
It is an integral-linear function of the lifted cochain $y$.
Since $y\bmod\INTEGERS$ is a cocycle on the contractible
octahedron, the cochain
\[
n_T(y)=\delta P_T(y)-y|_{\tilde T}
\]
is integral, and hence constant on each affine component
$\mathcal A$. We use $P_T(y)$ as the tetrahedral potential.

Let $\lambda=(a_T,c_T)_{T\in\Delta_3(X)}$ be the complex charges,
$\ell\in\bC$ the level, and $y=y_0+N\tilde u$ with $y_0\in\mathcal A$.
For a glued face $f$, define $z_{s,f}(y)=p_f(P_{T_s}(y))$ and
$\Delta_f(y)=z_{2,f}(y)-z_{1,f}(y)$ as in \eqref{eq:face-lifts}.
These discrepancies are integral and constant on $\mathcal A$.
For an unglued face $F\subset T$, put
\[
a_F(y)=p_F(P_T(y)),\qquad b_F(y)=\lambda_F(y),\qquad
m_F=a_F(y)-b_F(y)\in\INTEGERS^2.
\]
Here $\lambda_F$ is applied directly to the lifted cochain;
$m_F$ is again constant on $\mathcal A$. We use the
\emph{lifted boundary frame} determined by the direct boundary lifts
$b_F(y)$. Conversion from the potential lift $a_F(y)$ to this frame
multiplies the face representative by
\[
\varphi_{\varepsilon_F}(-m_F,a_F(y))
=\varphi_{\varepsilon_F}(m_F,b_F(y))^{-1}.
\]
Thus the scalar representative of the contracted weight is
\begin{align*}
\widehat b(\lambda;y)
={}&\prod_T g^{\varepsilon_T}_{a_T,c_T}\bigl(E(P_T(y))\bigr)
\prod_{f\ \mathrm{glued}}
\varphi_{\eta_f}\bigl(\Delta_f(y),z_{1,f}(y)\bigr)\\
&\quad\times\prod_{F\ \mathrm{unglued}}
\varphi_{\varepsilon_F}(m_F,b_F(y))^{-1}.
\end{align*}
For closed $X$ this is \eqref{eq:contraction-lifts} with the chosen
potentials. Every face factor, including each boundary conversion,
is a nowhere-zero exponential of an affine-linear function of
$(\tilde u,y_0)$, and all weight arguments are integral-linear in
$y$.
By Lemma~\ref{lem:tetra-denominator}, the function
\[
\widehat H(\lambda,\ell,y_0;\tilde u)=
e^{\Imun\pi\ell/(4\hbar)}\,
\widehat b(\lambda;y_0+N\tilde u)\prod_T
D^{\varepsilon_T}_{a_T,c_T}\bigl(E(P_T(y_0+N\tilde u))\bigr)
\]
is entire in all variables, and it is $\INTEGERS^d$-periodic in
$\tilde u$. Indeed, the factors $D^{\varepsilon_T}$ are periodic in
$(s,t)$, since their arguments are the exponentials
$e^{\pm2\pi\Imun s}$, $e^{\mp2\pi\Imun t}$ and the components of
$E(P_T(N\tilde u))$ are integral linear in $\tilde u$; and the
contracted weight $\widehat b(\lambda;y_0+N\tilde u)$ is periodic in
$\tilde u$ because, by the computation of Step~6 below applied to the
integral cocycle $m=Nm'$, $m'\in\INTEGERS^d$, its multiplier under
$y\mapsto y+Nm'$ is the boundary multiplier for the translation
$Nm'|_{\widetilde{\partial X}}=0$ of the boundary lift, which is
trivial. Hence $\widehat H$ descends to a
holomorphic function $H(z,\lambda,\ell,y_0)$ on
$(\bC^*)^d\times\bC^M$, $z_k=e^{2\pi\Imun u_k}$, where
$\bC^M=\bC^{2|\Delta_3(X)|}\times\bC\times\mathcal A$. The denominator
$\prod_TD^{\varepsilon_T}_{a_T,c_T}(E(P_T(y)))$ is a product of double
products \eqref{eq:tetra-denominator} whose arguments are
\[
e^{-4\pi Ca_T\pm2\pi\Imun s_T},\quad e^{-4\pi Cc_T\mp2\pi\Imun t_T},\quad
-e^{-4\pi Cb_T\pm2\pi\Imun(t_T-s_T)},\qquad (s_T,t_T)=E(P_T(y)),
\]
each of the form $e^{L_j(\lambda,y_0)}z^{v_j}$ with $L_j$
affine-linear (the sign $-1$ is $e^{\pi\Imun}$) and $v_j\in\INTEGERS^d$,
with the common nome pair $p_0=e^P$, $q_0=e^Q$, $P=-2\pi\la$,
$Q=-2\pi\la^{-1}$, whose real parts are negative. At a positive shape,
real level and real $y_0$ the moduli of these arguments are
$e^{-4\pi Ca_T}$, $e^{-4\pi Cc_T}$ and $e^{-4\pi Cb_T}$, all strictly
less than one, so the initial domain $U$ of
Theorem~\ref{thm:toric-period} contains the positive real locus and
is nonempty.

\emph{Step 4: continuation on the cover.} Define, for
$(\lambda,\ell,y_0)\in U$,
\begin{multline*}
\widehat{\mathcal F}(\lambda,\ell,y_0)
=\int_{(\REALS/\INTEGERS)^d}\frac{\widehat H(\lambda,\ell,y_0;\tilde u)}
{\prod_T D^{\varepsilon_T}_{a_T,c_T}(E(P_T(y_0+N\tilde u)))}\,du\\
=\int_{|z_k|=1}\frac{H(z,\lambda,\ell,y_0)}
{\prod_j(e^{L_j}z^{v_j};p_0,q_0)_\infty}\prod_k\frac{dz_k}{2\pi\Imun z_k}.
\end{multline*}
By Step~2 and Definition~\ref{Maindef}, for positive shapes, real
level and real $y_0$ this is the value of $\mathcal F_\hbar(X)$ at
$q_X(y_0)$ in the lifted boundary frame. By
Theorem~\ref{thm:toric-period}, $\widehat{\mathcal F}$ extends
uniquely to a single-valued meromorphic function on $\bC^M$, and a
finite product $\mathcal E$ of the form \eqref{eq:global-majorant} makes
$\mathcal E\widehat{\mathcal F}$ entire. Choose an affine complement
to $N\bC^d$ in $\mathcal A$ defined over $\REALS$ and restrict $\mathcal E$
to this complement. It is not identically zero there, since it is
nonzero at positive shapes and real points. By the invariance in
Step~5, this restriction is the asserted majorant on the universal
cover of the parameter component; no fibre-translation invariance
of $\mathcal E$ itself is required.

\emph{Step 5: invariance under the fibre translations.} For
$w\in U_{\mathrm{rel},\bC}$ with lift $\tilde w\in\bC^d$, the points
$y_0$ and $y_0+N\tilde w$ lie over the same point of
$\mathcal Q_{X,\bC}$, and $\widehat{\mathcal F}(\lambda,\ell,y_0+N\tilde w)
=\widehat{\mathcal F}(\lambda,\ell,y_0)$. Indeed, for real $w$ this
is the translation invariance of Haar measure on $(\REALS/\INTEGERS)^d$;
both sides are meromorphic in $(\lambda,\ell,y_0,\tilde w)$, and the
real locus is a uniqueness set. Hence $\widehat{\mathcal F}$ is a
meromorphic function on the quotient of $\bC^M$ by $N\bC^d$, which
is the universal cover of the component of
$\widetilde{LS}^{\bC}(X)\times\mathcal Q_{X,\bC}$ lying over
$\mathcal P$.

\emph{Step 6: descent.} The deck group $\Lambda_{\mathcal A}$ acts on
$\mathcal A$ by $y_0\mapsto y_0+m$. Since $m$ is an integral cocycle,
$P_T(m)$ is an integral potential and its face differences are the
direct differences $\lambda_F(m)$. Thus on a glued face both
potential lifts change by the same $n_f\in\mathbb Z^2$, while
$\Delta_f$ stays fixed; on a boundary face both $a_F$ and $b_F$
change by $n_F=\lambda_F(m)$, while $m_F$ stays fixed.
Equation~\eqref{eq:tetra-face-multiplier} splits each tetrahedral
multiplier into its four face multipliers. Their cancellation with
the conversion factors of Step~3 is explicitly
\begin{align*}
&\varphi_\eta(n,z)\varphi_{-\eta}(n,z+\Delta)
 \frac{\varphi_\eta(\Delta,z+n)}{\varphi_\eta(\Delta,z)}
 =e^{2\pi\Imun\eta\omega(n,\Delta)}=1,\\
&\varphi_\varepsilon(n,b+m_F)
 \frac{\varphi_\varepsilon(m_F,b)}
      {\varphi_\varepsilon(m_F,b+n)}
 =\varphi_\varepsilon(n,b),
\end{align*}
where $n,\Delta,m_F\in\mathbb Z^2$; the second identity uses
$e^{2\pi\Imun\varepsilon\omega(m_F,n)}=1$.
These are respectively the glued-face and unglued-face factors.
Consequently
\begin{equation}\label{eq:lifted-integrand-multiplier}
\widehat b(\lambda;y+m)
=\chi_m(y|_{\widetilde{\partial X}})\widehat b(\lambda;y),
\qquad
\chi_m=\prod_{F\ \mathrm{unglued}}
\varphi_{\varepsilon_F}\bigl(\lambda_F(m),b_F(y)\bigr).
\end{equation}
The calculation is valid on complex lifts wherever the weights are
finite, hence meromorphically. In particular $m=Nm'$ has zero
boundary restriction and gives multiplier one, as used in Step~3.
For positive shapes, real level and real $y_0$, $\chi_m$ is
the multiplier of
$\mathcal L_{\partial X}$ for the translation of the boundary lift
by $m|_{\widetilde{\partial X}}$; it depends on $y$ only through
$y|_{\widetilde{\partial X}}=y_0|_{\widetilde{\partial X}}$, is the
exponential of an affine-linear function, and equals one when
$\partial X=\emptyset$. Integrating over $u$ gives
$\widehat{\mathcal F}(\lambda,\ell,y_0+m)
=\chi_m(y_0|_{\widetilde{\partial X}})\,
\widehat{\mathcal F}(\lambda,\ell,y_0)$ on the real positive locus,
hence as an identity of meromorphic functions on $\bC^M$. The
functions $\chi_m$ are precisely the transition functions of the
pullback of $\mathcal L_{\partial X,\bC}$ between the lifted
boundary frames at $y_0$ and $y_0+m$. Therefore
$\widehat{\mathcal F}$ descends to a meromorphic section of the
compatible boundary line bundle over the component of
$\widetilde{LS}^{\bC}(X)\times\mathcal Q_{X,\bC}$ lying over
$\mathcal P$, and the choice of the lift $\mathcal A$ of $\mathcal P$
does not matter: different lifted components over $\mathcal P$
differ by an integral cochain and are related by the same boundary
multiplier calculation. Within one affine component, two real base
points can differ by any real cocycle. Changing between them is an
affine change of origin. The intrinsic Haar fibre integral and its
boundary value agree on the positive real locus, so the identity
theorem proves independence of this choice as well.

\emph{Step 7: continuation of the full torsion-conditioned fibre
integral.} Put
\[
q_X^T=(r_T\circ[\,\cdot\,],r_\partial)\colon
\mathcal Z_X\longrightarrow\mathcal Q_X^T,
\qquad K^T=\ker q_X^T.
\]
By Lemma~\ref{lem:haar} for the composition
$\mathcal Z_X\to\mathcal Q_X\to\mathcal Q_X^T$, on positive shapes
the section \eqref{eq:si-torsion} is the probability Haar integral
of $e^{\Imun\pi\ell/(4\hbar)}b_X$ over the full fibre of $q_X^T$.
We apply Theorem~\ref{thm:toric-period} to this integral itself.

Since $r_T$ takes values in a finite group, the Lie algebra of $K^T$
and its period lattice are
\[
V_T=\{v\in Z^1(\tilde X;\REALS):
             v|_{\widetilde{\partial X}}=0\},\qquad
\Lambda_T=V_T\cap\INTEGERS^{\Delta_1(\tilde X)}.
\]
The equations defining $V_T$ have integer coefficients, so
$\Lambda_T$ is a full lattice in $V_T$. Choose a basis of this
lattice and let $N_T$ be the integer matrix whose columns are its
elements. Thus $N_T$ identifies
$(\REALS/\INTEGERS)^{d_T}$ with the identity component
$(K^T)^0$, where $d_T=\dim V_T$. This parametrization pushes
probability Haar measure to probability Haar measure, also when
$d_T=0$.

Choose representatives $k_1,\ldots,k_m$ of the finite group
$K^T/(K^T)^0$. Each $k_j$ has a real edge lift $\kappa_j$ with
\[
\delta\kappa_j\in\INTEGERS^{\Delta_2(\tilde X)},\qquad
\kappa_j|_{\widetilde{\partial X}}=0.
\]
Indeed, its boundary restriction is zero in $\mathbb S$, so the
boundary values of any real lift are integers and may be subtracted
using an integer edge cochain on $\tilde X$. The resulting lift
need not be a real cocycle; only its coboundary is required to be
integral. In particular $\mathcal A+\kappa_j$ is a component of
$\widehat{\mathcal Z}$, and all the points
$y_0+\kappa_j+N_T\tilde u$ have the same boundary lift as $y_0$.
For real $y_0\in\mathcal A$, the full fibre through its image is
the disjoint union of these $m$ tori. In the common lifted boundary
frame its integral is therefore
\begin{align*}
\widehat{\mathcal F}^{\,T}(\lambda,\ell,y_0)
&=\frac1m\sum_{j=1}^m I_j(\lambda,\ell,y_0),\\
I_j(\lambda,\ell,y_0)&=e^{\Imun\pi\ell/(4\hbar)}
\int_{(\REALS/\INTEGERS)^{d_T}}
\widehat b(\lambda;y_0+\kappa_j+N_T\tilde u)\,du.
\end{align*}
The factor $1/m$ is exactly the mass of each connected component
under probability Haar measure. In particular the torus integrals
are taken before this finite average.

For each $j$, repeat Step~3 on $\mathcal A+\kappa_j$, with $N_T$
in place of $N$. Multiplication by
\[
\prod_T D^{\varepsilon_T}_{a_T,c_T}
\bigl(E(P_T(y_0+\kappa_j+N_T\tilde u))\bigr)
\]
makes the level times contracted weight entire. It is
$\INTEGERS^{d_T}$-periodic in $\tilde u$: every column of $N_T$ is
an integral cocycle with zero boundary restriction, so the
multiplier calculation of Step~6 gives multiplier one. The
denominator factors are also periodic. Hence the numerator descends
to a holomorphic function $H_j$ on
$(\bC^*)^{d_T}\times\bC^M$. The denominator arguments have the form
$e^{L_{j,r}(\lambda,y_0)}z^{v_{j,r}}$, with $L_{j,r}$ affine-linear
and $v_{j,r}\in\INTEGERS^{d_T}$, and the common nomes are $p_0,q_0$.
Since $\kappa_j$ is real, at positive shapes and real $y_0$ their
moduli are the same three numbers
$e^{-4\pi Ca_T}$, $e^{-4\pi Cc_T}$ and $e^{-4\pi Cb_T}$ as in
Step~3. The initial domains of all $I_j$ therefore contain a common
nonempty complex neighbourhood of every positive real parameter
point. Theorem~\ref{thm:toric-period} gives a single-valued
meromorphic continuation of each $I_j$ on $\bC^M$ and a majorant
$\mathcal E_j$ of the form \eqref{eq:global-majorant} for which $\mathcal E_jI_j$ is
entire. It follows that their finite average is meromorphic.
Refine the finitely many nome pairs to a common pair
$(e^{P/D},e^{Q/D})$ by \eqref{eq:refine-nomes} and multiply the
refined majorants. The resulting product $\mathcal E^T$ has the form
\eqref{eq:global-majorant}, is nonzero on the common initial domain,
and makes $\mathcal E^T\widehat{\mathcal F}^{\,T}$ entire.

\emph{Step 8: descent and uniqueness for the torsion-conditioned
section.} For real $w\in V_T$, translation of $y_0$ by $w$ is a
translation of the full Haar fibre and leaves its boundary frame
fixed. The identity theorem, applied also to the translation
parameter, extends this invariance to $w\in V_T\otimes\bC$.
Thus $\widehat{\mathcal F}^{\,T}$ descends first to
$\mathcal A/(V_T\otimes\bC)$, the universal cover of the component
of $\mathcal Q^T_{X,\bC}$ under consideration: the differential of
$q_X^T$ is boundary restriction, with kernel $V_T\otimes\bC$, and
the deck lattice is real. Restricting $\mathcal E^T$ to an affine complement
defined over $\REALS$ gives the required
majorant on this cover; it is not identically zero because the
complement contains real points over positive shapes.

Changing a fibre representative, a component representative $k_j$,
or its lift translates the Haar cycles, permutes the $m$ connected
components, or changes the lifted boundary frame. On the positive
real locus the resulting sums represent the same full-fibre integral
in $(\mathcal L_{\partial X})_\zeta$. If the boundary lift changes
by an integral cochain, their scalar values differ by exactly the
boundary multiplier computed in Step~6. These identities extend
meromorphically by the identity theorem. In particular they apply
to every deck transformation of the cover of $\mathcal Q^T_{X,\bC}$,
including those which permute the components of the fibre.
Consequently the local expressions descend and agree as sections
of the compatible boundary line bundle. This also shows that the
construction is independent of the chosen component $\mathcal A$
above a parameter component.

Finally, for both $q_X$ and $q_X^T$ the positive real locus is a
totally real submanifold of maximal dimension in each parameter
component. It determines the initial holomorphic germ and hence
the meromorphic continuation uniquely. In particular the result is
independent of all torus coordinates, lifts and elimination orders.
The entire level factor causes no additional poles. This proves the
theorem.
\end{proof}

\begin{corollary}\label{cor:global-identities}
The based gauge identity and the pentagon move, with its fixed
branching and level relation, extend as identities of meromorphic
sections in complex shape parameters. For closed $X$ with
$H_2(X',\mathbb Z)=0$, formula \eqref{ON-rel} gives the unique global
meromorphic continuation of the original positive partition function.
\end{corollary}
\begin{proof}
The unrestricted pentagon shape relations define a connected complex
affine family with a nonempty positive real open part. For example,
take $a_0=a_2=a_4=c_0=c_4=1/20$ and determine the other charges by
\eqref{pentagonconditions}; all remaining charges are positive.
The positive move identity and the identity theorem identify the
two meromorphic restrictions to this family. Likewise the complex
gauge identity follows from its real identity near zero gauge
parameter and positive shapes. These initial real sets are maximal
totally real in their complex parameter families, so they determine
the holomorphic germs uniquely. The old partition function is a
finite average of globally meromorphic torsion sections by
\eqref{ON-rel}; its equality on positive shapes proves the last
assertion.
\end{proof}

\begin{remark}[Restrictions and gluing]
The theorem concerns the full, unreduced complex shape space and
compatible state space. Restriction to a balanced subspace is a
meromorphic restriction in the sense of
Convention~\ref{conv:meromorphic} only when that subspace is not
contained in the polar locus, which by
\eqref{eq:global-polar-hyperplanes} is a locally finite union of
affine hyperplanes in the lifted parameters. A positive point on the
subspace guarantees this condition.
At a genuine pole, a finite value requires a separate regularization;
meromorphic continuation does not supply one. Gauge descent is
understood on loci where this restriction and quotient are defined.
The conditional gluing formulas have meromorphic continuations by
Theorem~\ref{thm:global-continuation} applied to their full compact
fibres, which are fibres of $q_X$ or, in the torsion version, of
$q_X^T$.
Section~\ref{sec:branching} proves the required branching
identifications, and Proposition~\ref{prop:branching-sewing} proves
their compatibility with conditional gluing. Branching independence does not affect the cohomology condition
\eqref{eq:gluing-kernel-condition} for an ordinary trace.
\end{remark}

\section{Move invariance and the exact gluing formula}
\label{tqft-rules}\label{proof}

\begin{proposition}\label{2-3/3-2M}
Let the five vertices of the bipyramid of a $2$--$3$ move be ordered
$0,1,2,3,4$, with apices $1$ and $3$, in such a way that the five
tetrahedra are positively oriented. Let $T_i$ be the tetrahedron which
omits the vertex $i$, with normalized angles $(a_i,b_i,c_i)$; thus
$T_1,T_3$ are the two tetrahedra called $T'_1,T'_2$ in
Section~\ref{TP}, and $T_0,T_2,T_4$ are the three tetrahedra called
$T'_3,T'_4,T'_5$ there (the primed names of Section~\ref{TP} refer
to Figure~\ref{23} and are not used elsewhere). Suppose that on the bipyramid both
triangulations carry the branching induced by this order, and that
the two sides have the same total dihedral angles at its nine boundary
edges, that is, that their normalized angles satisfy
\eqref{pentagonconditions}. Let $\ell_2$ and $\ell_3$ be the levels of
the triangulation with the two and of the triangulation with the three
tetrahedra. If the shapes are positive on both sides and
$\ell_3-\ell_2=P_e/3$, with $P_e$ as in \eqref{eq:Pe}, then the
section of Definition~\ref{Maindef} is unchanged by the move and by
its inverse. The equality holds in each cohomology sector and at each
compatible boundary cocycle, and also as an equality of analytic germs.
\end{proposition}
\begin{proof}
The map $M$ of Section~\ref{TP} preserves the global cohomology
class and the common boundary cocycle. Its kernel is $\mathbb S^2$:
one factor is the redundant tetrahedron constant and one is the
new internal edge state. Its restriction to a fibre at fixed
$(\alpha,\zeta)$ is surjective, by the equivariance and kernel
description there. Lemma~\ref{lem:haar} therefore identifies
integration on that fibre with integration first on $\ker M$ and
then on the corresponding fibre for the two-tetrahedron side.
The first redundant factor has mass one. Exactness of the cocycle
on the truncated bipyramid, as in the construction of $M$, permits
one common edge potential $x$ on it, up to tetrahedron constants.
Choose real lifts $x_{ij}$ of these geometric edge values and use
their restrictions as tetrahedral lifts. The paired face lifts then
agree on the one internal face of the old side and on the three
internal faces of the new side. All four internal face factors in
\eqref{eq:contraction-lifts} are consequently one, and the common
boundary representatives agree. Thus the scalar pentagon identity
\eqref{eq:pent1} applies without additional phases. Define its variables
$\alpha,\beta,\gamma,\delta$ of \eqref{eq:pent1}, which are not to be
confused with the class $\alpha$ above, by
$E(x|_{T_1})=(\alpha,\beta)$ and $E(x|_{T_3})=(\gamma,\delta)$. Then
\[
E(x|_{T_2})=(\sigma,\beta+\delta),\quad
E(x|_{T_4})=(\gamma-\sigma,\alpha+\delta-\sigma),\quad
E(x|_{T_0})=(\alpha-\sigma,\beta+\gamma-\sigma),
\]
with $\sigma=x_{03}+x_{14}-x_{04}-x_{13}$, the only one of these
numbers which contains the state $x_{13}$ of the new edge.
Since $\sigma$ is a translate of $-x_{13}$, integration over the
remaining circle in $\ker M$ is precisely probability Haar
integration $d\sigma$ over $[0,1]$. The full pentagon integrand is
periodic in $\sigma$, as noted after \eqref{eq:pent1}.
These are the arguments in \eqref{eq:pent1}, and the angle conditions of the
shaped move are \eqref{pentagonconditions}, as explained in the proof
of Corollary~\ref{cor:all-branched-moves}. Thus the local product of
three weights is replaced by the product of two weights times
$e^{-\pi\Imun P_e/(12\hbar)}$. This is cancelled by
$e^{\pi\Imun(\ell_3-\ell_2)/(4\hbar)}$. All other factors are
unchanged, and Fubini is legitimate for smooth functions on compact
fibres. Reading the equality backwards proves the inverse move.
Lemma~\ref{lem:holomorphic-integral} and the identity theorem prove
the germ assertion. Theorem~\ref{thm:global-continuation} and
Corollary~\ref{cor:global-identities} give the global meromorphic identity.
\end{proof}

\subsection{Coherent boundary identifications under rebranching}
\label{sec:branching}
A branching specifies an order on the vertices of each simplex,
compatible on identified faces. In this subsection the underlying
oriented triangulation, the geometric dihedral angles and the level
are fixed. Write $\beta$ for a branching. Edge states are attached to
the geometric, unoriented edges. Rebranching changes their coordinates
by a half-period translation.

For two branchings $\beta,\beta'$ define a circle-valued zero-cochain
on the midpoint complex by
\begin{equation}\label{eq:branching-half-shift}
h_{\beta\beta'}(q_e)=
\begin{cases}
0,&\text{if the orientations of $e$ agree},\\
1/2,&\text{if the orientation of $e$ is reversed}.
\end{cases}
\end{equation}
Here $q_e$ is the midpoint of the geometric edge $e$. Addition is
in $\mathbb S$, so
$h_{\beta\beta'}+h_{\beta'\beta''}=h_{\beta\beta''}$.
The same definition applies on a triangulated boundary surface.

We first give the required line-bundle maps on one oriented triangle.
Let its ordered vertices be $0,1,2$, its orientation sign be
$\eta\in\{1,-1\}$, and its edge potentials be $x_{01},x_{02},x_{12}$.
Set
\[
s=x_{02}-x_{12},\qquad t=x_{01}-x_{02},
\]
so that its line is $L^\eta$ at $(s,t)$. For the adjacent
transpositions $\tau_0=(01)$ and $\tau_1=(12)$, reverse the indicated
edge and add $1/2$ to its potential. The new face coordinates and
the factors mapping the new line to the old line are
\begin{align}
A_0(s,t)&=(-s,s+t+1/2),&
k_0^\eta(s,t)&=e^{\pi\Imun\eta(-1/12-s/2)},\label{eq:face-rebranch-0}\\
A_1(s,t)&=(s+t-1/2,-t),&
k_1^\eta(s,t)&=e^{\pi\Imun\eta(1/12-t/2)}.\label{eq:face-rebranch-1}
\end{align}
Both operations change the orientation sign to $-\eta$. 
Accordingly, for $i=0,1$ and $\eta\in\{1,-1\}$ we define the operator
\[
R_i^\eta\colon\Gamma(\Pi,L^{-\eta})\longrightarrow\Gamma(\Pi,L^{\eta}),
\qquad
(R_i^\eta f)(s,t)=k_i^\eta(s,t)\,f(A_i(s,t)),
\]
on lifted section representatives; here $\Gamma(\Pi,L^{\pm})$
denotes the smooth (or meromorphic) sections, identified with
quasi-periodic functions on $\bC^2$ with the multipliers
\eqref{eq:multiplier-omega}. The composition of two such operators
is defined because the orientation sign is reversed at each step: a
word $i_1i_2\cdots i_k$ in the letters $0,1$ gives the operator
$R_{i_1}^{\eta}R_{i_2}^{-\eta}\cdots R_{i_k}^{(-1)^{k-1}\eta}$ from
$\Gamma(\Pi,L^{(-1)^k\eta})$ to $\Gamma(\Pi,L^{\eta})$.

\begin{lemma}[Coherent face transports]\label{lem:face-transports}
The operators $R_i^\eta$ are holomorphic line-bundle isomorphisms,
that is, $k^\eta_i f\circ A_i$ is a section of $L^\eta$ whenever $f$
is a section of $L^{-\eta}$. They satisfy the exact $S_3$ relations
\[
R_0^{\eta}R_0^{-\eta}=R_1^{\eta}R_1^{-\eta}=1
\quad\text{on }\Gamma(\Pi,L^{\eta}),\qquad
R_0^{\eta}R_1^{-\eta}R_0^{\eta}=R_1^{\eta}R_0^{-\eta}R_1^{\eta}
\quad\text{on }\Gamma(\Pi,L^{-\eta}).
\]
Hence every permutation of the three vertices has a unique such
transport, obtained from any word representing it. On real states
it is unitary, and reversing the face orientation gives its dual
transport, $k_i^{-\eta}=(k_i^\eta)^{-1}$.
\end{lemma}
\begin{proof}
Let $M_i$ be the linear part of $A_i$ and let
$\varphi_\eta$ be the lattice multiplier of $L^\eta$. The condition
for descent is, for $n=(m,n_2)\in\mathbb Z^2$ and $z=(s,t)$,
\begin{equation}\label{eq:face-transport-descent}
k_i^\eta(z+n)\varphi_{-\eta}(M_i n,A_i z)
=\varphi_\eta(n,z)k_i^\eta(z).
\end{equation}
Direct substitution shows that, modulo even integers, the difference
of exponents in units of $\pi\Imun$ is respectively
\[
m^2-\eta m,\qquad n_2^2-\eta n_2.
\]
These are even integers, proving \eqref{eq:face-transport-descent}.

For the square relations the affine maps and accumulated factors are
\[
\begin{array}{c|c|c}
\text{word}&\text{final face coordinates}&\text{accumulated factor}\\ \hline
00&(s,t+1)&e^{-\pi\Imun\eta s}\\
11&(s-1,t)&e^{-\pi\Imun\eta t}\\
010&(-t,-s)&1\\
101&(-t,-s)&1.
\end{array}
\]
The lattice multipliers for $(0,1)$ and $(-1,0)$ are
$e^{\pi\Imun\eta s}$ and $e^{\pi\Imun\eta t}$, respectively.
They cancel the factors in the first two rows. The last two rows
give the braid relation, with no residual constant. These are a
presentation of $S_3$, proving independence of the word. The
exponentials have modulus one on real states, and
$k_i^{-\eta}=(k_i^\eta)^{-1}$, which proves the final assertions.
\end{proof}

Tensor these transports over the oriented faces of a surface
$\Sigma$. For the geometric cocycle, the base map is
\begin{equation}\label{eq:boundary-rebranch-base}
t_{\beta\beta'}:\mathcal Z_{\Sigma,\mathbb C}
\longrightarrow\mathcal Z_{\Sigma,\mathbb C},\qquad
\zeta\longmapsto\zeta+\delta h_{\beta\beta'}.
\end{equation}
Its expression in ordered face coordinates includes the passive
permutation of those coordinates in \eqref{eq:face-rebranch-0}--
\eqref{eq:face-rebranch-1}. Write the resulting bundle map as
\begin{equation}\label{eq:boundary-rebranch-line}
U_{\Sigma,\beta\beta'}:
t_{\beta\beta'}^*\mathcal L_{\Sigma,\beta'}
\longrightarrow\mathcal L_{\Sigma,\beta}.
\end{equation}
Because a globally reversed edge is shifted by the same half-period
in every incident face, these maps are compatible on shared edges.
Lemma~\ref{lem:face-transports} gives
\begin{equation}\label{eq:boundary-rebranch-cocycle}
U_{\Sigma,\beta\beta''}
=U_{\Sigma,\beta\beta'}\circ
t_{\beta\beta'}^*U_{\Sigma,\beta'\beta''},\qquad
U_{\Sigma,\beta\beta}=1.
\end{equation}
This assertion does not require a path of global single-edge
reversals through branchings: compare the two orders separately on each
face, and use the endpoint-dependent $S_3$ transport.

\begin{proposition}[Equivariance of a tetrahedral weight]
\label{prop:tetra-rebranch}
For two vertex orders on a geometric tetrahedron $T$, set
$y'=y+h_{\beta\beta'}$ on its six edge midpoints. With the same
geometric shapes and level,
\begin{equation}\label{eq:tetra-rebranch}
B_\beta(T,y)=
U_{\partial T,\beta\beta'}B_{\beta'}(T,y').
\end{equation}
This holds as an identity of meromorphic sections, and the transport
is independent of the permutation word.
\end{proposition}
\begin{proof}
Let $\varepsilon$ be the old tetrahedron sign, and put
$s=y_{02}+y_{13}-y_{03}-y_{12}$,
$t=y_{02}+y_{13}-y_{01}-y_{23}$. For an adjacent transposition in
the tetrahedral order, the new sign is $-\varepsilon$. The three
cases, including the half-shift of the reversed edge, are
\[
\begin{array}{c|c|c|c}
\text{transposition}&(a',c')&(s',t')&
\text{factor in the weight identity}\\ \hline
(01)&(a,b)&(-s,t-s-1/2)&e^{\pi\Imun\varepsilon s/2}\\
(12)&(b,c)&(s-t-1/2,-t)&
e^{-\pi\Imun\varepsilon(1/6+t/2)}\\
(23)&(a,b)&(-s,t-s-1/2)&e^{\pi\Imun\varepsilon s/2}.
\end{array}
\]
The weight identities are exactly \eqref{eq:fund1} and
\eqref{eq:fund2}, with holomorphic reflection for
$\varepsilon=-1$. To identify their bundle maps, the faces opposite
vertex $j$ have signs $\eta_j=\varepsilon(-1)^j$.
Only the two faces containing the reversed edge change their orders.
In an ordered face $(i,j,k)$, call $ij$ the low edge and $jk$
the high edge. For the low edge the exponent of its factor is
\[
\pi\Imun\eta\left(-\frac1{12}
 +\frac{y_{jk}-y_{ik}}2\right);
\]
for a high edge it is
\[
\pi\Imun\eta\left(\frac1{12}
 +\frac{y_{ik}-y_{ij}}2\right).
\]
Adding the two incident-face exponents gives, in the three rows,
$\pi\Imun\varepsilon s/2$,
$-\pi\Imun\varepsilon(1/6+t/2)$, and
$\pi\Imun\varepsilon s/2$. The unchanged faces contribute the
identity. This proves \eqref{eq:tetra-rebranch} for the generators
of $S_4$, and hence for every vertex permutation. On each boundary
face its composite is the unique transport of
Lemma~\ref{lem:face-transports}. Extra integer shifts arising in a
nonreduced word are exactly the lattice identifications checked
there. Thus the result is independent of the word, and the level is
unchanged.
\end{proof}

\begin{theorem}[Branching independence]\label{thm:branching}
Let $\beta,\beta'$ be any two compatible branchings of a fixed
finite oriented triangulation $X$. Regard its shapes as assignments
to geometric edges. On the compatible boundary target put
\begin{equation}\label{eq:rebranch-Q}
T_{\beta\beta'}(\alpha,\zeta)
=(\alpha,\zeta+\delta h_{\beta\beta'}|_{\partial X}).
\end{equation}
The meromorphic sections constructed in this paper satisfy
\begin{equation}\label{eq:branching-invariance}
\mathcal F_{\hbar,\beta}(X)
=U_{\partial X,\beta\beta'}\,
T_{\beta\beta'}^*\mathcal F_{\hbar,\beta'}(X).
\end{equation}
The same identity holds for the torsion-conditioned sections. The
level is unchanged, every cohomology class is preserved, and the
identifications obey the strict cocycle identity
\eqref{eq:boundary-rebranch-cocycle}. In particular, for closed $X$
the functions in each fixed cohomology or torsion sector are
independent of branching, without any boundary change of variables
or scalar anomaly.
\end{theorem}
\begin{proof}
Translation of tetrahedral potentials by the pullback of the global
zero-cochain $h_{\beta\beta'}$ preserves compatibility. On cocycles
it is the affine isomorphism
\[
z\longmapsto z+\delta h_{\beta\beta'}.
\]
Its added term is exact, so it preserves $[z]$ and its torsion
restriction. It sends the fibre with boundary value $\zeta$ to
the fibre with boundary value
$\zeta+\delta h_{\beta\beta'}|_{\partial X}$. On real fibres it
preserves probability Haar measure, including their components.

Apply Proposition~\ref{prop:tetra-rebranch} to every tetrahedron.
On an identified face the two transports are dual, by
Lemma~\ref{lem:face-transports}; their contraction is the identity.
This cancellation is an identity of bundle maps and therefore holds
also when the two face potentials have different real lifts. The
only remaining factor is $U_{\partial X,\beta\beta'}$, evaluated
at the fixed boundary cocycle. It can be taken outside the fibre
integral. The common level factor is unchanged. Change of variables
now proves \eqref{eq:branching-invariance} for positive shapes.
Theorem~\ref{thm:global-continuation} and the identity theorem extend
it meromorphically to all unreduced complex parameters.

The finite torsion condition is preserved by the same affine map,
so the identical argument proves the torsion assertion. The strict
composition law is the tensor product of the face composition laws.
No connectivity assumption on the graph of branchings is used.
Finally the boundary tensor product is empty for closed $X$,
so its transport is the identity.
\end{proof}

Equations~\eqref{eq:boundary-rebranch-base} and
\eqref{eq:boundary-rebranch-cocycle} give an intrinsic boundary target:
take the disjoint union of the cocycle spaces for all branchings and
identify $(\beta,\zeta)$ with
$(\beta',\zeta+\delta h_{\beta\beta'})$, identifying their line
fibres by $U_{\Sigma,\beta\beta'}$. Each branching is a coordinate
presentation of this same space and line bundle. Its map to
$H^1(\tilde\Sigma;\mathbb S_{\mathbb C})$ is well defined, since
the translation is exact. Theorem~\ref{thm:branching} thus gives a
section on the compatible intrinsic boundary target. This is the
precise boundary meaning of branching independence; equality at an
unchanged numerical boundary cocycle is not required. The transports
are independent of the shapes and commute with the gauge state
translations, so they also descend to the balanced gauge quotients
where those meromorphic restrictions are defined.

\begin{corollary}[All compatible shaped $2$--$3$ moves]
\label{cor:all-branched-moves}
The move identity holds for all branchings of the two sides of an
oriented shaped $2$--$3$ bipyramid which agree on its boundary, and
for the inverse move, with the same level normalization; these are
the compatible shaped moves of Theorem~\ref{main}. Consequently it
holds inside an
arbitrary glued triangulation in each cohomology and torsion sector.
It is a meromorphic identity on the full compatible complex shape
family, subject to the stated rules for meromorphic restriction.
\end{corollary}
\begin{proof}
Choose a positive standard order $0,1,2,3,4$ of the bipyramid
vertices as in Proposition~\ref{2-3/3-2M}: the apices are $1,3$, the
tetrahedron $T_i$ omits the vertex $i$, and the parity is chosen so
that the three tetrahedra $T_0,T_2,T_4$ and the two tetrahedra
$T_1,T_3$ have the positive signs used in \eqref{eq:pent1}. The
angle-matching equations on the boundary
edges $34,01,23,12,04$ are precisely
\eqref{pentagonconditions}. With $a_i+b_i+c_i=1/2$, they imply the
remaining boundary angle equations and the balance condition
$b_0+c_2+b_4=1$ at the new edge $13$. Thus this is exactly the
geometric shaped move family, not an extra restriction on it.

Proposition~\ref{2-3/3-2M} proves the identity for this standard
order. Apply Theorem~\ref{thm:branching} to each of the two balls.
Their boundary transports to any common prescribed boundary
branching are identical, since they depend only on that boundary.
They therefore preserve the move identity. There is no need to
extend the standard order to the complement of the ball. Boundary
contraction and conditional Haar integration then insert the
identity in an arbitrary $X$; equivalently use the same fibre map
$M$ as in Proposition~\ref{2-3/3-2M}.

For completeness the level factor is independent of the auxiliary
standard order. Its six possible positive orders are generated by
$(024)$ and $(02)(13)$. They send the ordered triple
$(c_0,a_2,c_4)$ to $(a_2,c_4,c_0)$ and $(a_2,c_0,c_4)$,
respectively, by the geometric edge-charge assignments.
Geometrically, $c_0,a_2,c_4$ are the angles at the edges
$14,01,12$ in $T_0,T_2,T_4$, respectively: these join apex $1$
to the preceding equatorial vertex in the cycle $(024)$.
This also makes their cyclic permutation visible.
Hence $P_e=2(c_0+a_2+c_4)-1/2$ and the relation
$\ell_3-\ell_2=P_e/3$ are unchanged. For the six standard orders in
which the five tetrahedra are negatively oriented the same expression
takes the value $-P_e$: for instance, the exchange of the labels $1$
and $3$ turns it into $2(a_0+b_2+a_4)-1/2$, and the two expressions add
up to zero by the balance condition $b_0+c_2+b_4=1$. This is why the
standard order is required to be positive. No rebranching phase is
added.
The identity first holds on the nonempty positive open part of the
affine move family and then globally by
Corollary~\ref{cor:global-identities}. The inverse identity follows
by reading it backwards.
\end{proof}

A branching of the side with three tetrahedra is induced by a total
order of the five vertices unless its equatorial triangle is cyclic,
because every other triangle of the bipyramid is a face of one of the
three tetrahedra. Of the $132$ branchings of that side, $12$ have a
cyclic equatorial triangle. (Indeed, a branching of that side
is an orientation of all ten edges of the five vertices whose
restriction to each of the three tetrahedra is acyclic; its only
possible $3$-cycle is the equatorial triangle. The $120$ total orders
give the acyclic ones. If the equatorial triangle is cyclic, each
apex must either dominate or be dominated by all three equatorial
vertices, since otherwise a $3$-cycle through the apex would lie in
one of the tetrahedra; the two apices together with the oriented
axis then give $2\cdot2\cdot2$ possibilities for each of the two
orientations of the triangle, of which the $2\cdot2$ with a $3$-cycle
through both apices are excluded, leaving $12$.) As this triangle is
a face of the side
with two tetrahedra, they agree on the boundary with no branching of
that side. The move is not available for them until the branching
has been replaced, by means of Theorem~\ref{thm:branching}, by one
with an acyclic equatorial triangle; such a branching exists if and
only if the triangulation with two tetrahedra admits a branching. In
the other direction there is no restriction: each of the $96$
branchings of the side with two tetrahedra extends to the side with
three.

\subsection{Translation and reversal of cohomology sectors}
\label{sec:sector-translation}
In this subsection $\partial X=\emptyset$. Let
$j:X'\hookrightarrow X$ be the inclusion. A class in
$C_X=H^1(X',\mathbb S)$ will be called extending if it belongs to
$\operatorname{im}j^*$. This distinction matters when vertices have
nonspherical links. The homomorphism
$j^*\colon H^1(X,\mathbb S)\to C_X$ is injective, because the vertex
links are connected, so that $\pi_1(X')\to\pi_1(X)$ is onto.

By Proposition~\ref{fiberrel}, translation of $\Omega(X)$ by an
element of $K_X$ does not change the arguments $E(y)(t)$ of the
tetrahedral weights, and it changes the class $H_X(y)$ by twice an
extending class. Its only effect on the integrand is therefore through
the face contractions in \eqref{eq:contraction-lifts}. We show that
this effect is a phase which is constant on each sector, and we compute
it. This relates the sectors $\alpha$ and $\alpha+2j^*\beta$ for all
$\alpha\in C_X$ and $\beta\in H^1(X,\mathbb S)$.

\subsubsection*{Cochain conventions}
We use simplicial cochains of the $\Delta$-complex $X$, and of a single
tetrahedron, with values in $\mathbb R$, $\mathbb Z$ or $\mathbb S$.
The edges of a simplex with ordered vertices are indexed by the pairs
$i<j$, and $\lambda_{ij}$ denotes the value of a one-cochain $\lambda$
on the corresponding edge. For one-cochains $\lambda,\lambda'$, a
two-cochain $c$ and a zero-cochain $\mu$, the formulas below use
ordinary multiplication with real or integer coefficients, or the
natural $\mathbb Z$-module pairing when one coefficient is in
$\mathbb S$. No multiplication of two circle-valued coefficients
is intended: whenever necessary we first choose real lifts and
reduce the resulting expression modulo $\mathbb Z$. With these
conventions, coboundaries and cup products are given by
\begin{gather*}
(\delta\mu)_{ij}=\mu_j-\mu_i,\quad\
(\delta\lambda)(i,j,k)=\lambda_{ij}+\lambda_{jk}-\lambda_{ik},\quad\
(\delta c)(T)=\sum_{i=0}^3(-1)^ic(\partial_iT),\\
(\lambda\cup\lambda')(i,j,k)=\lambda_{ij}\lambda'_{jk},\qquad
(c\cup\mu)(i,j,k)=c(i,j,k)\,\mu_k,\\
(c\cup\lambda)(T)=c(\partial_3T)\,\lambda_{23},\qquad
(\lambda\cup c)(T)=\lambda_{01}\,c(\partial_0T),
\end{gather*}
and they satisfy
$\delta(a\cup b)=\delta a\cup b+(-1)^{\deg a}a\cup\delta b$.
The face $\partial_iT$ of a tetrahedron $T$ carries the induced
orientation sign
\[
\eta_{T,\partial_iT}=\varepsilon_T(-1)^i .
\]
Since $X$ is closed and oriented, every face of $X$ is the image of
exactly two faces of tetrahedra, and these have opposite signs. Hence,
for every real two-cochain $c$ on $X$,
\begin{equation}\label{eq:stokes-X}
\sum_{T\in\Delta_3(X)}\varepsilon_T\,(\delta c)(T)=0 .
\end{equation}

\begin{definition}\label{def:linking}
Let $\beta,\beta'\in H^1(X,\mathbb S)$. Choose cocycles
$\lambda,\lambda'\in Z^1(X,\mathbb S)$ representing them and real
lifts $\tilde\lambda,\tilde\lambda'\colon\Delta_1(X)\to\mathbb R$.
The two-cochain $\delta\tilde\lambda$ is integral, and we define
$\lk_X(\beta,\beta')\in\mathbb S$ as the class modulo $\mathbb Z$ of
\begin{equation}\label{eq:lk}
\sum_{T\in\Delta_3(X)}\varepsilon_T\,
(\delta\tilde\lambda\cup\tilde\lambda')(T)
=\sum_{T\in\Delta_3(X)}\varepsilon_T\,
(\delta\tilde\lambda)(\partial_3T)\,\tilde\lambda'_{23}(T),
\end{equation}
where $\tilde\lambda'_{23}(T)$ is the value of $\tilde\lambda'$ on the
edge joining the last two vertices of $T$.
\end{definition}

\begin{lemma}\label{lem:linking}
Modulo $\mathbb Z$, the expression \eqref{eq:lk} depends only on
$\beta$ and $\beta'$. The pairing $\lk_X$ is bi-additive and
symmetric, and it vanishes if $\beta$ or $\beta'$ belongs to the
identity component of $H^1(X,\mathbb S)$.
\end{lemma}
\begin{proof}
Let $\tilde\lambda,\tilde\lambda'$ be arbitrary real one-cochains on
$X$. The Leibniz rule for $\tilde\lambda\cup\tilde\lambda'$ and
\eqref{eq:stokes-X} give
\begin{equation}\label{eq:lk-two-forms}
\sum_T\varepsilon_T\,(\delta\tilde\lambda\cup\tilde\lambda')(T)
=\sum_T\varepsilon_T\,(\tilde\lambda\cup\delta\tilde\lambda')(T).
\end{equation}
Now let them be lifts of cocycles as in Definition~\ref{def:linking}.
Adding an integral cochain $m'$ to $\tilde\lambda'$ changes the left-hand
side by the integer
$\sum_T\varepsilon_T(\delta\tilde\lambda\cup m')(T)$. Adding an
integral cochain $m$ to $\tilde\lambda$ changes it by
$\sum_T\varepsilon_T(\delta m\cup\tilde\lambda')(T)$, which by
\eqref{eq:lk-two-forms} equals the integer
$\sum_T\varepsilon_T(m\cup\delta\tilde\lambda')(T)$. If $\lambda'$ is
replaced by $\lambda'+\delta\mu'$, with $\mu'$ a zero-cochain having a
real lift $\tilde\mu'$, we may replace $\tilde\lambda'$ by
$\tilde\lambda'+\delta\tilde\mu'$. Since
$\delta\delta\tilde\lambda=0$, the sum then changes by
\[
\sum_T\varepsilon_T\,(\delta\tilde\lambda\cup\delta\tilde\mu')(T)
=\sum_T\varepsilon_T\,\delta(\delta\tilde\lambda\cup\tilde\mu')(T)=0,
\]
again by \eqref{eq:stokes-X}. Replacing $\tilde\lambda$ by
$\tilde\lambda+\delta\tilde\mu$ does not change $\delta\tilde\lambda$.
Hence $\lk_X$ is well defined, and it is bi-additive because
\eqref{eq:lk} is bilinear in the lifts.

The image of $H^1(X,\mathbb R)$ in $H^1(X,\mathbb S)$ is a compact
connected subgroup, and by the Bockstein sequence of
$0\to\mathbb Z\to\mathbb R\to\mathbb S\to0$ its cokernel is the finite
group $\operatorname{Tor}H^2(X,\mathbb Z)$. Being closed, connected and
of finite index, this image is the identity component of
$H^1(X,\mathbb S)$. A class in the identity component therefore has a
representative with a lift which is a real cocycle. If
$\delta\tilde\lambda=0$, the left-hand side of \eqref{eq:lk-two-forms}
vanishes; if $\delta\tilde\lambda'=0$, the right-hand side does.

For the symmetry, put $N=\delta\tilde\lambda$ and
$N'=\delta\tilde\lambda'$, and let $h$ be the real two-cochain on $X$
with $h(i,j,k)=\tilde\lambda_{ik}N'(i,j,k)$. Since $\delta N'=0$,
expanding the definitions on a tetrahedron $T$ gives
\[
(\tilde\lambda\cup N'-N'\cup\tilde\lambda)(T)
=-(\delta h)(T)+N(\partial_2T)N'(\partial_0T)
-N(\partial_1T)N'(\partial_3T);
\]
this is the coboundary formula for Steenrod's $\cup_1$-product
\cite{Steenrod1947} of $\tilde\lambda$ and $N'$. Sum over $T$ with
the signs $\varepsilon_T$. On the right the first term disappears by
\eqref{eq:stokes-X}, and the others are integers. On the left we
obtain, by \eqref{eq:lk-two-forms}, the difference of the expressions
\eqref{eq:lk} for $(\beta,\beta')$ and for $(\beta',\beta)$.
\end{proof}

\begin{remark}\label{rem:linking}
Formula \eqref{eq:lk} is the evaluation
$\langle B\beta\cup\beta',[X]\rangle$, where
$B\colon H^1(X,\mathbb S)\to H^2(X,\mathbb Z)$ is the Bockstein
homomorphism and $[X]=\sum_T\varepsilon_TT$ is the fundamental cycle.
By Lemma~\ref{lem:linking} the pairing descends to the finite group
$\pi_0(H^1(X,\mathbb S))\cong\operatorname{Tor}H^2(X,\mathbb Z)$.
Suppose that $X$ is a closed oriented three-manifold. Then
$\langle x\cup\beta',[X]\rangle=\beta'([X]\cap x)$ for
$x\in H^2(X,\mathbb Z)$, where $\beta'$ is regarded as a character of
$H_1(X,\mathbb Z)$, and $x\mapsto[X]\cap x$ maps
$\operatorname{Tor}H^2(X,\mathbb Z)$ isomorphically onto
$\operatorname{Tor}H_1(X,\mathbb Z)$ by Poincar\'e duality.
Consequently $\lk_X(\,\cdot\,,\beta')$ vanishes identically if and
only if $\beta'$ is trivial on $\operatorname{Tor}H_1(X,\mathbb Z)$,
that is, if and only if $\beta'$ belongs to the identity component.
Thus $\lk_X$ induces a nondegenerate symmetric pairing on
$\pi_0(H^1(X,\mathbb S))$; it is the classical linking pairing of
$\operatorname{Tor}H^2(X,\mathbb Z)\cong\operatorname{Tor}H_1(X,\mathbb Z)$,
with the sign convention fixed by \eqref{eq:lk} and by the
orientation of $X$ (reversing the orientation of $X$ replaces $\lk_X$
by $-\lk_X$; the examples of Section~\ref{sec:examples} are computed
with the orientation in which the first tetrahedron is positive).
In these conventions the generators used for $\mathbb RP^3$,
$L(3,1)$ and $L(4,1)$ have self-linking
$1/2$, $2/3$ and $3/4$, respectively, that is, $-1/p$ modulo
$\mathbb Z$ for $p=2,3,4$. This specifies the orientation convention
when comparing with the usual lens-space linking forms.
\end{remark}

We shall use the following local identity. Let $T$ be a tetrahedron,
let $v\in\mathbb R^{\Delta_0(T)}$ and $u=\delta v$, and for a face
$F=(i,j,k)$ of $T$ put $U_F=(u_{ij},u_{jk})$. With the map $A$ of
Proposition~\ref{fiberrel} we have
$p_F(A(v))=(v_i-v_j,v_j-v_k)=-U_F$, while $E(A(v))=0$. Hence
\eqref{eq:bilinear-faces}, applied to $A(v)$ and to an arbitrary
complex edge function $w$ on $T$, gives
\begin{equation}\label{eq:local-cocycle-identity}
\sum_{i=0}^3(-1)^i\,
\omega\bigl(U_{\partial_iT},p_{\partial_iT}(w)\bigr)=0 .
\end{equation}

\begin{definition}\label{def:mixed-linking}
Let $\beta\in H^1(X,\mathbb S)$ be represented by
$\lambda\in Z^1(X,\mathbb S)$ with a real lift $\tilde\lambda$. For a
face $f$ of $X$ with ordered vertices $(i,j,k)$ put
\[
\Lambda_f=(\tilde\lambda_{ij},\tilde\lambda_{jk})\in\mathbb R^2 .
\]
For $y\in\Omega(X)$, with real lifts $y_T$ and the notation of
\eqref{eq:face-lifts}, put
\begin{equation}\label{eq:lkprime}
\lk'_X(\beta,y)=\sum_{f\in\Delta_2(X)}\eta_f\,
\omega\bigl(\Lambda_f,\Delta_f(y)\bigr)\quad\in\mathbb S .
\end{equation}
\end{definition}

\begin{lemma}\label{lem:mixed-linking}
The right-hand side of \eqref{eq:lkprime} depends only on $\beta$ and
$y$. It is additive in each of them. It vanishes if $\beta$ belongs to
the identity component of $H^1(X,\mathbb S)$, or if $y$ belongs to the
identity component of $\Omega(X)$; in particular it depends on $y$
only through the connected component of $H_X(y)$ in $C_X$. The
resulting bi-additive pairing
\[
\lk'_X\colon H^1(X,\mathbb S)\times C_X\longrightarrow\mathbb S,
\qquad \lk'_X(\beta,H_X(y))=\lk'_X(\beta,y),
\]
satisfies
\begin{equation}\label{eq:lkprime-extending}
\lk'_X(\beta,j^*\beta')=\lk_X(\beta',\beta),
\qquad \beta,\beta'\in H^1(X,\mathbb S).
\end{equation}
\end{lemma}
\begin{proof}
The sum is unchanged if the two sides of a face are exchanged, since
$\eta_f$ and $\Delta_f(y)$ both change sign. An integral change of
$\tilde\lambda$ changes each $\Lambda_f$ by an integral vector and
hence the sum by an integer. Next, let the lifts $y_T$ be replaced by
$y_T+w_T$ with $w_T\in\mathbb Z^{\Delta_1(T)}$. On a tetrahedron $T$
the cocycle $\lambda$ is exact, so there are $v_T\in\mathbb R^{\Delta_0(T)}$
and an integral one-cochain $n_T$ on $T$ with
$\delta v_T=\tilde\lambda|_T+n_T$. With
$\nu_{T,F}=(n_{T,ij},n_{T,jk})$ for a face $F=(i,j,k)$ of $T$,
\eqref{eq:local-cocycle-identity} gives
\[
\sum_{i=0}^3(-1)^i\,
\omega\bigl(\Lambda_{\partial_iT},p_{\partial_iT}(w_T)\bigr)
=-\sum_{i=0}^3(-1)^i\,
\omega\bigl(\nu_{T,\partial_iT},p_{\partial_iT}(w_T)\bigr)\in\mathbb Z .
\]
The change of \eqref{eq:lkprime} is minus the sum of the left-hand
sides multiplied by $\varepsilon_T$, hence an integer.

Suppose now that $\tilde\lambda$ is a real cocycle. Then we can take
$n_T=0$ above, and \eqref{eq:local-cocycle-identity} with $w=y_T$ gives
\[
0=\sum_T\sum_{i=0}^3\eta_{T,\partial_iT}\,
\omega\bigl(\Lambda_{\partial_iT},p_{\partial_iT}(y_T)\bigr)
=-\sum_f\eta_f\,\omega\bigl(\Lambda_f,\Delta_f(y)\bigr).
\]
Thus \eqref{eq:lkprime} vanishes for such a lift. Applied to
$\tilde\lambda=\delta\tilde\mu$, this proves independence of the
representative $\lambda$ of $\beta$; applied to the classes in the
identity component, which have real cocycle lifts by the proof of
Lemma~\ref{lem:linking}, it proves the first vanishing assertion.
Additivity is clear from \eqref{eq:lkprime}.

The group $\Omega(X)$ is the closed subgroup of the torus
$\mathbb S^{\Delta_0(\tilde Y)}$ defined by the integral linear
equations in \eqref{OmD}. Its identity component is therefore the
image of the space of real solutions of these equations; that is, it
consists of the elements which have real lifts satisfying them exactly.
For such lifts $\Delta_f(y)=0$ for all $f$, so $\lk'_X(\beta,y)=0$.
The fibres of $H_X$ are connected by Proposition~\ref{Omega}, so the
connected components of $\Omega(X)$ are the preimages of those of
$C_X$. This proves that $\lk'_X(\beta,y)$ depends only on the
component of $H_X(y)$.

Finally let $\beta'$ be represented by $\lambda'$ with real lift
$\tilde\lambda'$, and on each tetrahedron choose $v'_T$ and $n'_T$ with
$\delta v'_T=\tilde\lambda'|_T+n'_T$, so that
$\delta n'_T=-\delta\tilde\lambda'|_T$. Put
$y'_{ij}(T)=v'_{T,i}$ for $i<j$. On the three edges
$(q_{ij}q_{ik})$, $(q_{ik}q_{jk})$, $(q_{ij}q_{jk})$ of the midpoint
triangle of a face $(i,j,k)$ the differences of $y'$ are
$0,\lambda'_{ij},\lambda'_{ij}$ modulo $\mathbb Z$. Hence
$y'\in\Omega(X)$ and $H'_X(y')=r_{\mathrm{low}}^*\lambda'$, in the notation of the
proof of Proposition~\ref{fiberrel}, so that $H_X(y')=j^*\beta'$.
Moreover $p_F(y'_T)=(v'_{T,i}-v'_{T,j},0)$, and therefore
$\Delta_f(y')=(n'_{T_1,ij}-n'_{T_2,ij},0)$. Consequently
\[
\sum_f\eta_f\,\omega\bigl(\Lambda_f,\Delta_f(y')\bigr)
=-\sum_T\sum_{i=0}^3\eta_{T,\partial_iT}\,
(n'_T\cup\tilde\lambda)(\partial_iT)
=-\sum_T\varepsilon_T\,\delta(n'_T\cup\tilde\lambda)(T).
\]
By the Leibniz rule,
$\delta(n'_T\cup\tilde\lambda)
=-\delta\tilde\lambda'\cup\tilde\lambda-n'_T\cup\delta\tilde\lambda$
on $T$, and the last term is an integer. This proves
\eqref{eq:lkprime-extending}.
\end{proof}

\begin{remark}\label{rem:mixed-linking}
The pairing $\lk'_X$ has the following intrinsic description. For an
edge $e$ of a face $f$ of $X$ put $d_{f,e}=y_{T_2}(e)-y_{T_1}(e)$,
the difference of the lifts on the two sides of $f$. If $f=(i,j,k)$,
then $\Delta_f(y)=(d_{f,ik}-d_{f,jk},d_{f,ij}-d_{f,ik})$; in
particular the three numbers $d_{f,e}$ are congruent modulo
$\mathbb Z$. Substitute this in \eqref{eq:lkprime} and use that
$\sum_i(-1)^i(\delta c_T)(\partial_iT)=0$ for the one-cochains
$c_T(e)=\tilde\lambda_e\,y_T(e)$ on the tetrahedra. The result is
\begin{equation}\label{eq:lkprime-dual}
\lk'_X(\beta,y)=\sum_{f\in\Delta_2(X)}\eta_f\,
(\delta\tilde\lambda)(f)\,\bigl(y_{T_1}(q_f)-y_{T_2}(q_f)\bigr)
\quad\in\mathbb S,
\end{equation}
where $q_f$ is the midpoint of an arbitrary edge of $f$; the summand
does not depend on this choice, because $(\delta\tilde\lambda)(f)$
is an integer and the differences $y_{T_1}(q)-y_{T_2}(q)$ at the three
midpoints of $f$ are congruent modulo $\mathbb Z$. Let $f^*$ be
the edge dual to $f$, oriented from $T_1$ to $T_2$. Since
$\delta\tilde\lambda$ is an integral cocycle representing the
Bockstein class $B\beta$, the chain
$\sum_f\eta_f(\delta\tilde\lambda)(f)\,f^*$ is a one-cycle in $X'$.
Its class $\beta^\vee\in H_1(X',\mathbb Z)$ depends only on $B\beta$,
because the coboundary of an integral one-cochain on $X$ corresponds
to the boundary of the dual two-chain; it is a torsion class, since
$B\beta$ is one. On the other hand, the value of the class $H_X(y)$ on
a dual one-cycle $\sum_fm_ff^*$ is
$\sum_fm_f(y_{T_1}(q_f)-y_{T_2}(q_f))$, as one sees by following the
cycle through the truncated tetrahedra. Hence \eqref{eq:lkprime-dual}
says that
\[
\lk'_X(\beta,\alpha)=\alpha(\beta^\vee),
\]
the value of the character $\alpha$ of $H_1(X',\mathbb Z)$ on this
class. In particular $\lk'_X$ does not depend on the branching:
the coefficient $\eta_f(\delta\tilde\lambda)(f)$ is the value of
$\delta\tilde\lambda$ on $f$ with the orientation induced from $T_1$,
where $\tilde\lambda$ is regarded as a cochain on oriented edges, and
this value does not depend on the order of the vertices of $f$; hence
the class $\beta^\vee$ does not depend on the branching. Moreover
$\lk'_X(\beta,\alpha)$ depends only on $r_T(\alpha)$.
\end{remark}

\begin{theorem}[Translation of sectors]\label{thm:sector-translation}
Suppose $\partial X=\emptyset$. Let $v\in\mathbb S^{\Delta_0(Y)}$ be
such that $A(v)\in\Omega(X)$, let $\lambda_v\in Z^1(X,\mathbb S)$ be
the cocycle of Proposition~\ref{fiberrel}, and put
$\beta=[\lambda_v]$. Then, for all $y\in\Omega(X)$ and at arbitrary
positive shapes,
\begin{equation}\label{eq:weight-translation}
b_X(y+A(v))=e^{-2\pi\Imun(\lk_X(\beta,\beta)+\lk'_X(\beta,y))}\,b_X(y).
\end{equation}
Consequently, for all $\beta\in H^1(X,\mathbb S)$ and
$\alpha\in C_X$, at the same shapes and level,
\begin{equation}\label{eq:sector-translation}
\mathcal F_\hbar(X)(\alpha+2j^*\beta)
=e^{-2\pi\Imun(\lk_X(\beta,\beta)+\lk'_X(\beta,\alpha))}\,
\mathcal F_\hbar(X)(\alpha).
\end{equation}
This identity holds on positive shapes and continues meromorphically
in the shape and level variables. It also holds on the complexified
sectors wherever the meromorphic restriction is defined, with the
phase determined by the connected component of $\alpha$.
\end{theorem}
\begin{proof}
Write $\lambda=\lambda_v$. Choose a real lift
$\tilde\lambda$ of $\lambda$ and real lifts
$v_T\in\mathbb R^{\Delta_0(T)}$ of the restrictions of $v$ to the
tetrahedra, and put $N=\delta\tilde\lambda$. By
Proposition~\ref{fiberrel}, $\delta v_T$ is a lift of $\lambda|_T$, so
that
\begin{equation}\label{eq:primitive-lifts}
\delta v_T=\tilde\lambda|_T+n_T,\qquad
n_T\in\mathbb Z^{\Delta_1(T)},\qquad\delta n_T=-N|_T .
\end{equation}
For a face $F=(i,j,k)$ of $T$ put $\nu_{T,F}=(n_{T,ij},n_{T,jk})$.
Then
\begin{equation}\label{eq:pf-of-A}
p_F(A(v_T))=-\Lambda_F-\nu_{T,F}.
\end{equation}

\emph{Step 1: the quotient of the contractions.}
Let $y\in\Omega(X)$ with real lifts $y_T$. Then $y_T+A(v_T)$ are real
lifts of $y+A(v)$, and $E(y_T+A(v_T))=E(y_T)$ as real vectors. Thus
the first product in \eqref{eq:contraction-lifts} is the same for $y$
and for $y+A(v)$. For a face $f$ of $X$ write $z_1=z_{1,f}(y)$,
$\Delta=\Delta_f(y)$ and
\[
a_s=p_f(A(v_{T_s}))\quad(s=1,2),\qquad
D=a_2-a_1=\nu_{T_1,f}-\nu_{T_2,f}\in\mathbb Z^2,
\]
so that $z_{1,f}(y+A(v))=z_1+a_1$ and $\Delta_f(y+A(v))=\Delta+D$. By
\eqref{eq:multiplier-omega},
\[
\varphi_{\eta_f}(\Delta+D,z_1+a_1)
=\sigma(\Delta+D)\sigma(\Delta)\,
e^{\pi\Imun\eta_f(\omega(z_1,D)+\omega(a_1,\Delta)+\omega(a_1,D))}\,
\varphi_{\eta_f}(\Delta,z_1),
\]
and $\sigma(\Delta+D)\sigma(\Delta)=\sigma(D)(-1)^{\omega(D,\Delta)}$.

\emph{Step 2: elimination of $z_1$.}
Apply \eqref{eq:bilinear-faces} on each tetrahedron to $y_T$ and
$A(v_T)$. Since $E(A(v_T))=0$, multiplying by $\varepsilon_T$ and
summing over the tetrahedra gives
\[
0=\sum_T\sum_{i=0}^3\eta_{T,\partial_iT}\,
\omega\bigl(p_{\partial_iT}(y_T),p_{\partial_iT}(A(v_T))\bigr)
=\sum_f\eta_f\bigl(\omega(z_1,a_1)-\omega(z_1+\Delta,a_1+D)\bigr),
\]
that is,
\[
\sum_f\eta_f\,\omega(z_1,D)
=\sum_f\eta_f\,\omega(a_1+D,\Delta).
\]
Since $\omega(D,\Delta)$ is an integer,
$e^{\pi\Imun\eta_f\omega(D,\Delta)}=(-1)^{\omega(D,\Delta)}$. Combining
this with Step~1 we obtain
\begin{equation}\label{eq:translation-two-factors}
b_X(y+A(v))=\Theta\,
\exp\Bigl(2\pi\Imun\sum_f\eta_f\,\omega(a_{1,f},\Delta_f(y))\Bigr)\,b_X(y),
\end{equation}
where the index $f$ has been restored and
\begin{equation}\label{eq:Theta}
\Theta=\prod_f\sigma(D_f)\,e^{\pi\Imun\eta_f\omega(a_{1,f},D_f)} .
\end{equation}

\emph{Step 3: the $y$-dependent factor.}
By \eqref{eq:pf-of-A} and the integrality of $\nu_{T_1,f}$ and
$\Delta_f(y)$, the exponential in \eqref{eq:translation-two-factors}
equals
$\exp\bigl(-2\pi\Imun\sum_f\eta_f\,\omega(\Lambda_f,\Delta_f(y))\bigr)
=e^{-2\pi\Imun\lk'_X(\beta,y)}$.

\emph{Step 4: the constant $\Theta$.}
Write $\nu_s=\nu_{T_s,f}$ in \eqref{eq:Theta}. By \eqref{eq:pf-of-A},
\[
\omega(a_1,D)=-\omega(\Lambda_f,\nu_1)+\omega(\Lambda_f,\nu_2)
+\omega(\nu_1,\nu_2),\qquad
\sigma(D)=\sigma(\nu_1)\sigma(\nu_2)(-1)^{\omega(\nu_1,\nu_2)} .
\]
Since $\omega(\nu_1,\nu_2)$ is an integer, the factor of $\Theta$ at
$f$ is therefore
\[
\sigma(\nu_1)\,e^{-\pi\Imun\eta_f\omega(\Lambda_f,\nu_1)}\cdot
\sigma(\nu_2)\,e^{\pi\Imun\eta_f\omega(\Lambda_f,\nu_2)},
\]
a product of two terms, one for each side of $f$. Hence
\[
\Theta=\prod_T\prod_{i=0}^3\sigma(\nu_{T,\partial_iT})\,
e^{-\pi\Imun\eta_{T,\partial_iT}\,
\omega(\Lambda_{\partial_iT},\nu_{T,\partial_iT})}.
\]
In terms of cup products on $T$ we have
$\omega(\Lambda_F,\nu_{T,F})
=(\tilde\lambda\cup n_T-n_T\cup\tilde\lambda)(F)$ and
$\sigma(\nu_{T,F})=(-1)^{(n_T\cup n_T)(F)}$, so that
\[
\Theta=\prod_T(-1)^{\delta(n_T\cup n_T)(T)}\,
e^{-\pi\Imun\varepsilon_T\,
\delta(\tilde\lambda\cup n_T-n_T\cup\tilde\lambda)(T)} .
\]
Now $u_T=\tilde\lambda|_T+n_T$ is a cocycle on $T$, whence
$\delta(u_T\cup u_T)=0$, that is,
\[
\delta(\tilde\lambda\cup n_T+n_T\cup\tilde\lambda)
=-\delta(\tilde\lambda\cup\tilde\lambda)-\delta(n_T\cup n_T)
\quad\text{on }T .
\]
Therefore
\[
-\delta(\tilde\lambda\cup n_T-n_T\cup\tilde\lambda)
=\delta(\tilde\lambda\cup\tilde\lambda)+\delta(n_T\cup n_T)
+2\,\delta(n_T\cup\tilde\lambda).
\]
The integer $\delta(n_T\cup n_T)(T)$ contributes
$(-1)^{\delta(n_T\cup n_T)(T)}$, which cancels the sign above. The
cochain $\tilde\lambda\cup\tilde\lambda$ is defined on all of $X$, so
its contribution disappears by \eqref{eq:stokes-X}. Finally, by
\eqref{eq:primitive-lifts},
$\delta(n_T\cup\tilde\lambda)=-N\cup\tilde\lambda-n_T\cup N$, and
$(n_T\cup N)(T)$ is an integer. Hence
\[
\Theta=\exp\Bigl(-2\pi\Imun\sum_T\varepsilon_T\,
(N\cup\tilde\lambda)(T)\Bigr)=e^{-2\pi\Imun\lk_X(\beta,\beta)} .
\]
Together with Step~3 this proves \eqref{eq:weight-translation}.

\emph{Step 5: integration.}
Given $\beta$, choose $v$ with $A(v)\in\Omega(X)$ and
$[\lambda_v]=\beta$; this is possible because \eqref{eq:KX-cocycles}
is surjective. By
Definition~\ref{Maindef}, the sector
$\mathcal F_\hbar(X)(\alpha)$ is the level factor times the
integral of $b_X$ over $H_X^{-1}(\alpha)$ with respect to the
invariant probability measure of this $\mathbb S(X)$-torsor. By
\eqref{eq:KX-class}, translation by $A(v)$ maps $H_X^{-1}(\alpha)$
onto $H_X^{-1}(\alpha+2j^*\beta)$. It commutes with the action of
$\mathbb S(X)$, and the invariant probability measure of a torsor is
unique; hence it carries one of these measures to the other. By
Lemma~\ref{lem:mixed-linking} the phase in
\eqref{eq:weight-translation} is constant on $H_X^{-1}(\alpha)$.
Integrating \eqref{eq:weight-translation} over this fibre therefore
gives \eqref{eq:sector-translation} for positive shapes.
Theorem~\ref{thm:global-continuation} and the identity theorem extend
it meromorphically. The same argument applies on each connected
component of the complexified group of sectors, using its real compact
form as the uniqueness set.
\end{proof}

\begin{corollary}[Reversal of extending sectors]
\label{cor:sector-reversal}
Suppose $\partial X=\emptyset$. For every extending class
$\alpha\in C_X$, at the same shapes and level,
\[
\mathcal F_\hbar(X)(-\alpha)
=\mathcal F_\hbar(X)(\alpha).
\]
\end{corollary}
\begin{proof}
Let $\alpha=j^*\beta'$ and apply \eqref{eq:sector-translation} with
$\beta=-\beta'$. Then $\alpha+2j^*\beta=-\alpha$, and by
\eqref{eq:lkprime-extending} and bi-additivity
\[
\lk_X(\beta,\beta)+\lk'_X(\beta,\alpha)
=\lk_X(\beta',\beta')+\lk_X(\beta',-\beta')=0 .
\qedhere
\]
\end{proof}

\begin{corollary}[Constancy along extending tori]
\label{cor:sector-constancy}
Suppose $\partial X=\emptyset$, and let $H^1(X,\mathbb S)^0$ be the
identity component of $H^1(X,\mathbb S)$. Then
\[
\mathcal F_\hbar(X)(\alpha+\gamma)
=\mathcal F_\hbar(X)(\alpha),\qquad
\alpha\in C_X,\quad\gamma\in j^*\bigl(H^1(X,\mathbb S)^0\bigr).
\]
\end{corollary}
\begin{proof}
The torus $H^1(X,\mathbb S)^0$ is divisible, so $\gamma=2j^*\beta$ with
$\beta$ in this torus. Both pairings in \eqref{eq:sector-translation}
then vanish, by Lemmas~\ref{lem:linking} and~\ref{lem:mixed-linking}.
\end{proof}

\begin{corollary}[Selection rule]\label{cor:selection}
Suppose $\partial X=\emptyset$. Let $\alpha\in C_X$, and let
$\beta\in H^1(X,\mathbb S)$ satisfy $2\beta=0$. If
\[
\lk_X(\beta,\beta)+\lk'_X(\beta,\alpha)\ne0\quad\text{in }\mathbb S,
\]
then $\mathcal F_\hbar(X)(\alpha)$ vanishes identically in the
shapes and the level.
\end{corollary}
\begin{proof}
In this case $\alpha+2j^*\beta=\alpha$, and
\eqref{eq:sector-translation} states that
$\mathcal F_\hbar(X)(\alpha)$ equals itself times a phase
different from one.
\end{proof}

\begin{corollary}[Closed three-manifolds]
\label{cor:closed-manifold-sectors}
Let $X$ be a triangulation of a closed oriented three-manifold. Then
$j^*$ is an isomorphism, by which we identify $C_X$ with
$H^1(X,\mathbb S)$, and for all $\alpha,\beta\in C_X$ the following
hold.
\begin{enumerate}
\renewcommand{\labelenumi}{\textup{(\roman{enumi})}}
\item $\mathcal F_\hbar(X)(\alpha+2\beta)
=e^{-2\pi\Imun(\lk_X(\beta,\beta)+\lk_X(\alpha,\beta))}\,
\mathcal F_\hbar(X)(\alpha)$.
\item $\mathcal F_\hbar(X)(\alpha)$ depends only on
$r_T(\alpha)$, and
$\mathcal F^T_\hbar(X)(r_T(\alpha))
=\mathcal F_\hbar(X)(\alpha)$.
\item $\mathcal F_\hbar(X)(-\alpha)
=\mathcal F_\hbar(X)(\alpha)$ and
$\mathcal F^T_\hbar(X)(-\tau)
=\mathcal F^T_\hbar(X)(\tau)$ for all $\tau\in\mathcal T_X$.
\item If $2\beta=0$ and
$\lk_X(\beta,\beta)+\lk_X(\alpha,\beta)\ne0$, then
$\mathcal F_\hbar(X)(\alpha)$ vanishes identically.
\item If $\operatorname{Tor}H_1(X,\mathbb Z)$ has odd order, then
every class in $C_X$ is of the form $2\beta$, and
\[
\mathcal F_\hbar(X)(2\beta)
=e^{-2\pi\Imun\lk_X(\beta,\beta)}\,\mathcal F_\hbar(X)(0).
\]
\end{enumerate}
If, in the situation of \textup{(v)}, moreover $H_2(X',\mathbb Z)=0$,
then $H=H^1(X,\mathbb S)$ is finite,
\begin{equation}\label{eq:gauss-sum}
Z_\hbar(X,\Gamma)=\mathcal F_\hbar(X)(0)\,
\frac1{|H|}\sum_{\beta\in H}e^{-2\pi\Imun\lk_X(\beta,\beta)},
\end{equation}
and, at parameters where the amplitudes are finite,
\[
|Z_\hbar(X,\Gamma)|=|H|^{-1/2}\,|\mathcal F_\hbar(X)(0)|.
\]
\end{corollary}
\begin{proof}
The vertex links of a closed three-manifold are spheres, and filling
them by balls induces an isomorphism on $H^1(\,\cdot\,,\mathbb S)$.
Thus $j^*$ is an isomorphism. Assertion (i) is
\eqref{eq:sector-translation} combined with
\eqref{eq:lkprime-extending}. The kernel
of $r_T$ consists of the characters of $H_1(X',\mathbb Z)$ which are
trivial on its torsion subgroup; it is a torus of finite index in
$C_X$, hence the identity component of $C_X$. Therefore the first part
of (ii) is Corollary~\ref{cor:sector-constancy}, and the second part
follows from \eqref{eq:torsion-average}, whose integrand is then
constant. Assertion (iii) is Corollary~\ref{cor:sector-reversal}
together with (ii), and (iv) is Corollary~\ref{cor:selection}. For
(v), $C_X\cong\operatorname{Hom}(H_1(X,\mathbb Z),\mathbb S)$ is the
product of a torus and a finite group of odd order, on both of which
multiplication by two is surjective; now apply (i) with $\alpha=0$.

Suppose finally that also $H_2(X',\mathbb Z)=0$. Since
$\partial X'$ consists of spheres, the map
$H_2(X',\mathbb Z)\to H_2(X',\partial X';\mathbb Z)\cong
H^1(X',\mathbb Z)$ is onto, so that $H^1(X',\mathbb Z)=0$. Therefore
$H_1(X',\mathbb Z)$ is finite, $C_X=\mathcal T_X$ is finite, and
$\mathcal F^T_\hbar=\mathcal F_\hbar$. As $|H|$ is odd,
$\beta\mapsto2\beta$ is a bijection of $H$, and
Theorem~\ref{RelationON} together with assertion (v) gives
\eqref{eq:gauss-sum}.
The pairing $\lk_X$ on the finite group $H$ is symmetric by
Lemma~\ref{lem:linking} and nondegenerate by
Remark~\ref{rem:linking}. Writing $G$ for the Gauss sum, we therefore
have
\[
|G|^2=\sum_{\beta,\gamma\in H}
e^{-2\pi\Imun(\lk_X(\gamma,\gamma)+2\lk_X(\beta,\gamma))}=|H|,
\]
because the sum over $\beta$ vanishes unless $2\gamma=0$, that is,
unless $\gamma=0$. This proves the assertion on the modulus.
\end{proof}

\begin{corollary}[Sector classification for arbitrary torsion]
\label{cor:all-torsion-sectors}
Let $X$ be a triangulation of a closed oriented three-manifold, and put
\[
A_X=\pi_0C_X=\pi_0H^1(X,\mathbb S),\qquad K_2=A_X[2]=\{k\in A_X:2k=0\}.
\]
The map $r_T$ induces an isomorphism of $A_X$ onto $\mathcal T_X$, by
which we identify these groups. By Lemma~\ref{lem:linking} the pairing
$\lk_X$ descends to $A_X$, and we write
$\lk_X\colon A_X\times A_X\to\mathbb S$ also for the induced pairing.
For $\tau\in A_X$ put $F_\tau=\mathcal F_\hbar(X)(\alpha)$,
where $\alpha$ is any class in the component $\tau$; by
Corollary~\ref{cor:closed-manifold-sectors}\textup{(ii)} this is well
defined, and $F_\tau=\mathcal F^T_\hbar(X)(\tau)$. The
solutions $\tau_*\in A_X$ of
\begin{equation}\label{eq:characteristic-sector-coset}
\lk_X(\tau_*,k)+\lk_X(k,k)=0\qquad(k\in K_2)
\end{equation}
form exactly one coset $\tau_*+2A_X$ of $2A_X$, which we call the
distinguished coset. For any choice of $\tau_*$ in it, all sectors are
given by
\begin{equation}\label{eq:all-torsion-sector-classification}
\begin{aligned}
F_\tau&=0 &&(\tau\notin\tau_*+2A_X),\\
F_{\tau_*+2\theta}
&=e^{-2\pi\Imun(\lk_X(\theta,\theta)+\lk_X(\tau_*,\theta))}F_{\tau_*}
&&(\theta\in A_X).
\end{aligned}
\end{equation}
The phase in the second formula is independent of the choice of
$\theta$ with prescribed $2\theta$. Thus there is at most one
independent sector amplitude; the statement does not assert that
$F_{\tau_*}$ is nonzero, either at a given parameter value or
identically.

Suppose moreover that $X$ is admissible, that is,
$H_2(X',\mathbb Z)=0$; for a connected closed three-manifold this
means that $X$ has only one vertex and that its first Betti number
vanishes. Then
\begin{equation}\label{eq:all-torsion-gauss-sum}
\begin{gathered}
Z_\hbar(X,\Gamma)
=\frac{F_{\tau_*}}{|A_X|\,|K_2|}\,G_{\tau_*},\\
G_{\tau_*}:=\sum_{\theta\in A_X}
e^{-2\pi\Imun(\lk_X(\theta,\theta)+\lk_X(\tau_*,\theta))},
\end{gathered}
\end{equation}
where
\begin{equation}\label{eq:all-torsion-gauss-modulus}
|G_{\tau_*}|^2=|A_X|\,|K_2|,\qquad
|Z_\hbar(X,\Gamma)|
=\frac{|F_{\tau_*}|}{\sqrt{|A_X|\,|K_2|}}.
\end{equation}
The sector and comparison formulas are meromorphic identities in the
shape and level variables; the modulus identity holds wherever the
amplitudes are finite. In particular the original partition function
of an admissible triangulation determines all sectors.
\end{corollary}
\begin{proof}
The homomorphism $r_T$ is onto because $\mathbb S$ is divisible, and
its kernel is the identity component of $C_X$ by the proof of
Corollary~\ref{cor:closed-manifold-sectors}; hence it induces an
isomorphism of $A_X$ onto $\mathcal T_X$.
By Lemma~\ref{lem:linking} and Remark~\ref{rem:linking}, $\lk_X$ induces
a nondegenerate symmetric pairing on the finite group $A_X$. For
$k,k'\in K_2$,
\[
\lk_X(k+k',k+k')-\lk_X(k,k)-\lk_X(k',k')=2\lk_X(k,k')=0.
\]
Consequently $k\mapsto-\lk_X(k,k)$ is a character of $K_2$.
Nondegeneracy identifies $A_X$ with its character group, and restriction
of characters from $A_X$ to $K_2$ is surjective. The kernel of the map
\[
A_X\longrightarrow\operatorname{Hom}(K_2,\mathbb S),\qquad
\tau\longmapsto\lk_X(\tau,\cdot)|_{K_2}
\]
is $K_2^{\perp}=2A_X$: the inclusion $2A_X\subseteq K_2^{\perp}$ is immediate,
and both groups have order $|A_X|/|K_2|$. Hence the solutions of
\eqref{eq:characteristic-sector-coset} form exactly one coset of
$2A_X$.

Every $k\in K_2$ has a lift of order dividing two in $H^1(X,\mathbb S)$.
Indeed, if $\beta$ is any lift, then $2\beta$ belongs to the identity
torus; its divisibility provides $t$ in that torus with $2t=2\beta$,
and $\beta-t$ is the required lift. The selection rule therefore
applies to every $k\in K_2$ and proves the first formula of
\eqref{eq:all-torsion-sector-classification}. The second is
Corollary~\ref{cor:closed-manifold-sectors}(i), descended to components.
Replacing $\theta$ by $\theta+k$, with $k\in K_2$, changes its exponent
modulo $\mathbb Z$ by
\[
\lk_X(k,k)+2\lk_X(\theta,k)+\lk_X(\tau_*,k)=0,
\]
using \eqref{eq:characteristic-sector-coset} and $2k=0$.

For a triangulation of a connected closed three-manifold with $V$
vertices, the Mayer--Vietoris sequence of $X'$ and the vertex
neighbourhoods gives
\[
H_2(X',\mathbb Z)\cong H_2(X,\mathbb Z)\oplus\mathbb Z^{V-1},
\]
and $H_2(X,\mathbb Z)\cong H^1(X,\mathbb Z)$ is free, of rank equal to
the first Betti number of $X$. This gives the stated reformulation of
admissibility. If $H_2(X',\mathbb Z)=0$, the argument
in the proof of Corollary~\ref{cor:closed-manifold-sectors} shows that
$H^1(X,\mathbb S)$ is finite, and hence equals $A_X$.
Theorem~\ref{RelationON} gives
\[
Z_\hbar(X,\Gamma)=\frac1{|A_X|}\sum_{\tau\in A_X}F_\tau .
\]
The map $\theta\mapsto\tau_*+2\theta$ is $|K_2|$-to-one onto
$\tau_*+2A_X$; substitution of
\eqref{eq:all-torsion-sector-classification} gives
\eqref{eq:all-torsion-gauss-sum}. Finally, expanding the absolute
square and changing one summation variable from $\theta$ to
$\theta+\theta'$ gives
\begin{align*}
|G_{\tau_*}|^2
&=\sum_{\theta'\in A_X}
  e^{-2\pi\Imun(\lk_X(\theta',\theta')+\lk_X(\tau_*,\theta'))}
  \sum_{\theta\in A_X}e^{-4\pi\Imun\lk_X(\theta,\theta')}\\
&=|A_X|\sum_{\theta'\in K_2}
  e^{-2\pi\Imun(\lk_X(\theta',\theta')+\lk_X(\tau_*,\theta'))}
=|A_X|\,|K_2|.
\end{align*}
Here character orthogonality and nondegeneracy make the inner sum
zero unless $2\theta'=0$, and \eqref{eq:characteristic-sector-coset}
makes every remaining phase equal to one. In particular
$G_{\tau_*}\ne0$, which also proves the final assertion. The sector
and comparison formulas are meromorphic identities because
\eqref{eq:sector-translation} and \eqref{ON-rel} are, by
Theorems~\ref{thm:sector-translation} and~\ref{RelationON}.
\end{proof}

When $|A_X|$ is odd, $K_2=0$ and one may choose $\tau_*=0$, recovering
\eqref{eq:gauss-sum}. In general the distinguished coset contains zero
if and only if $\lk_X(k,k)=0$ for all $k\in K_2$. If it does not,
choose $k\in K_2$ with $\lk_X(k,k)=\frac12$; the substitution
$\theta\mapsto\theta+k$ shows that the untwisted Gauss sum
$\sum_{\theta\in A_X}e^{-2\pi\Imun\lk_X(\theta,\theta)}$ is equal to
its negative, so that it vanishes, as does the trivial sector. With
even torsion both cases occur, as Example~\ref{ex:Z4-sectors} and the
calculation for $\mathbb RP^3$ in Section~\ref{sec:rp3-example}
illustrate.

A further instance in which the distinguished coset contains zero is
provided by the triangulations $X$ of the quaternionic space
$S^3/Q_8$, the quotient of the group $S^3$ of unit quaternions by
$Q_8=\{\pm1,\pm\mathbf i,\pm\mathbf j,\pm\mathbf k\}$. Here
$A_X=K_2\cong(\mathbb Z/2\mathbb Z)^2$. Conjugation by the unit
quaternion $\frac12(1+\mathbf i+\mathbf j+\mathbf k)$ induces an
orientation-preserving diffeomorphism of $S^3/Q_8$ which permutes the
three nonzero elements of $A_X$ cyclically, so that their self-linkings
have a common value $c$. The self-linking is additive on $K_2$, and
the third nonzero element is the sum of the other two; hence $c=2c$,
that is, $c=0$. Therefore the distinguished coset is $\{0\}$: the
trivial sector is the only one which can be nonzero, the opposite
extreme to $\mathbb RP^3$.

\begin{example}\label{ex:Z3-sectors}
Let $X$ be a one-vertex triangulation of a closed oriented
three-manifold with $H_1(X,\mathbb Z)\cong\mathbb Z/3\mathbb Z$, for
instance of a lens space $L(3,1)$. Then $H_2(X',\mathbb Z)=0$, and
$\lk_X(\beta,\beta)=\pm1/3$ for both nontrivial classes $\beta$,
with a sign which is determined by the orientation.
Both nontrivial sectors are therefore equal to
$e^{\mp2\pi\Imun/3}\mathcal F_\hbar(X)(0)$, and
\eqref{eq:gauss-sum} reads
\[
Z_\hbar(X,\Gamma)
=\frac{1+2e^{\mp2\pi\Imun/3}}3\,\mathcal F_\hbar(X)(0)
=\mp\frac{\Imun}{\sqrt3}\,\mathcal F_\hbar(X)(0).
\]
If instead $H_1(X,\mathbb Z)\cong\mathbb Z/2\mathbb Z$, then
$\lk_X(\beta,\beta)=1/2$ for the nontrivial class $\beta$, and
Corollary~\ref{cor:closed-manifold-sectors}(iv) with $\alpha=0$ shows
that the trivial sector vanishes identically.
This is made explicit for $\mathbb RP^3$ in
Section~\ref{sec:rp3-example}; the lens space $L(3,1)$ is treated in
Section~\ref{sec:l31-example}.
\end{example}

\begin{example}\label{ex:Z4-sectors}
Suppose $X$ is a one-vertex triangulation of a closed oriented
three-manifold with $H_1(X,\mathbb Z)\cong\mathbb Z/4\mathbb Z$.
Identify $A_X$ with $\mathbb Z/4\mathbb Z$. Then
$\lk_X(x,y)=\pm xy/4$, with a sign which is determined by the
orientation. Here $K_2=\{0,2\}$ and $\lk_X(2,2)=0$, so the
distinguished coset is $2A_X=\{0,2\}$, and
\[
\begin{gathered}
F_1=F_3=0,\qquad F_2=e^{\mp\pi\Imun/2}F_0=\mp\Imun F_0,
\\
Z_\hbar(X,\Gamma)=\frac{1\mp\Imun}4\,F_0.
\end{gathered}
\]
Indeed $H_2(X',\mathbb Z)=0$ for a one-vertex triangulation of a
rational homology sphere, so the comparison formula applies.
In particular $|Z_\hbar(X,\Gamma)|=|F_0|/\sqrt8$ wherever finite, in
agreement with \eqref{eq:all-torsion-gauss-modulus}. The lens space
$L(4,1)$ is treated explicitly in Section~\ref{sec:l41-example}.
\end{example}

For a general closed pseudo-manifold,
Theorem~\ref{thm:sector-translation} relates only those sectors which
differ by twice an extending class, and the extension condition in
Corollary~\ref{cor:sector-reversal} cannot be omitted. The following
example gives an exact obstruction on the unrestricted positive shape
space.

\begin{example}[Failure for arbitrary ideal-vertex sectors]
\label{ex:sector-reversal-obstruction}
Take two positively oriented tetrahedra $T_1,T_2$ with the
order-preserving face identifications
$\partial_iT_1\simeq\partial_{3-i}T_2$, $0\leq i\leq3$.
This is the one-vertex pseudo-manifold whose vertex complement is
the trefoil complement, as in \cite[Section~11.1]{MR3227503}; its
diagram and comparison with the old partition function are given in
Section~\ref{sec:trefoil-example}.
The knot group is normally generated by a meridian, so $X$ is simply
connected. In particular $H^1(X,\mathbb S)=0$: no nontrivial class
extends, and Theorem~\ref{thm:sector-translation} gives no relation
between different sectors.
Take level zero and $\Gamma=\Delta_1(X)$, so no balancing constraint
restricts the positive charges.

In the edge order $(01,02,03,12,13,23)$, solving the face constraints
gives the compatible potentials
\[
\begin{aligned}
y_1&=(p,p-h,p+u-h,p-2h,p-h,p),\\
y_2&=(q,q+h,q+u,q+2h,q+h,q),
\end{aligned}
\qquad p,q,u,h\in\mathbb S.
\]
Here $p,q$ are the tetrahedron constants, $u$ parameterizes the
exact-cocycle fibre, and $h$ is the coordinate on
$H^1(X',\mathbb S)\simeq\mathbb S$. In particular, negation of
the class sends $h$ to $-h$. The face-coordinate pairs on matched
faces agree already as real lifts, so there is no residual contraction
phase. Since $E_1=(h-u,-2h)$ and $E_2=(-u,2h)$, normalized Haar
integration, with $s=-u$, gives
\[
F_{a_1,c_1;a_2,c_2}(h):=\mathcal F_\hbar(X)(h)
=\int_0^1g_{a_1,c_1}(s+h,-2h)g_{a_2,c_2}(s,2h)\,ds.
\]
Changing $s$ by $h$ shows that negation exchanges the two charge
pairs. It does not in general preserve the value when these pairs
are held fixed.

For an exact proof of this failure, set $A=(2\cla)^{-1}$.
Equation~\eqref{eq:gauge-local} with $D=0$ gives
$g_{a+A,c}(s-1,t)=g_{a,c}(s,t)e^{-\pi\Imun t}
e^{\pi\Imun\cla(2c-1/3)}$, and the lattice multiplier
\eqref{eq:multiplier-omega} of $L=L^+$ for the translation of the
first state coordinate by $-1$ is
$\varphi_+((-1,0),(s,t))=e^{\pi\Imun\,\omega((s,t),(-1,0))}
=e^{\pi\Imun t}$, that is, $g(s-1,t)=e^{\pi\Imun t}g(s,t)$ for
every section $g$ of $L$. Together,
\[
g_{a+A,c}(s,t)=C_c e^{-2\pi\Imun t}g_{a,c}(s,t),
\qquad C_c=e^{\pi\Imun\cla(2c-1/3)}\ne0.
\]
(With the opposite sign convention for the multiplier the exponent
would be $e^{+2\pi\Imun t}$; the argument below only uses that the
exponent is $e^{\mp2\pi\Imun t}$ with a nonzero constant.)
Along the segment $a\mapsto a+\rho A$, $0\leq\rho\leq1$,
the same local identity expresses the shifted weight as a nonzero
exponential times $g_{a,c}(s+\rho,t)$. Thus the state coordinate
undergoes a real translation and no pole crosses the integration contour.
The compact integral therefore remains holomorphic in a
neighbourhood of this segment by Lemma~\ref{lem:holomorphic-integral};
uniqueness in Theorem~\ref{thm:global-continuation} identifies it
there with the continued state integral. Consequently, as a
meromorphic identity in the charges,
\[
F_{a_1,c_1;a_2+A,c_2}(h)
=C_{c_2}e^{-4\pi\Imun h}F_{a_1,c_1;a_2,c_2}(h).
\]
If $F(h)=F(-h)$ held for all positive charge pairs, the identity
theorem would extend this equality to the shifted charges as well.
Comparison would then force
\[
\bigl(e^{-4\pi\Imun h}-e^{4\pi\Imun h}\bigr)
F_{a_1,c_1;a_2,c_2}(h)=0.
\]
This is impossible. Indeed, take $a,b,c>0$ with $a+b+c=1/2$ and
the charge pairs $(a_1,c_1)=(a,c)$, $(a_2,c_2)=(b,c)$.
Equation~\eqref{eq:fund2} at $(x,y)=(s+1/6,-1/3)$ gives
\[
g_{a,c}(s+1/6,-1/3)=\overline{g_{b,c}(s,1/3)},
\]
and therefore
\[
F_{a,c;b,c}(1/6)=\int_0^1|g_{b,c}(s,1/3)|^2\,ds>0.
\]
To verify strict positivity, for any positive triple $a,b,c$ put
$f=\tilde\psi'_{a,c}=e^{-\pi\Imun/12}\psi_{c,b}$.
The function $e^{\pi\Imun st}g_{a,c}(s,t)$ is periodic in $s$,
and its $k$th Fourier coefficient is
\[
\int_{\mathbb R}f(v)e^{-2\pi\Imun(k-t)v}\,dv
=\bar\psi_{a,b}(t-k),\qquad k\in\mathbb Z,
\]
by \eqref{eq:tpsi-bpsi}. The phase can also be checked directly
from the charged inversion identity
\[
\bar\psi_{a,b}(x)=e^{-\pi\Imun/6}e^{\pi\Imun x^2}
\psi_{b,a}(-x).
\]
These coefficients are nonzero for real $t$: the quantum-dilogarithm argument in \eqref{eq:bpsi} lies
strictly inside its zero-free and pole-free strip. Thus no real-$t$
slice of $g$ vanishes identically, proving the strict inequality.
Since $e^{-2\pi\Imun/3}\ne e^{2\pi\Imun/3}$, this contradicts
the forced identity at $h=1/6$. Hence unrestricted sector reversal
fails for some positive shapes on this pseudo-manifold.
\end{example}

\subsection{Cutting, holonomy, and conditional integration}
\label{sec:conditional-gluing}\label{sec:cutting}
Let $\Sigma$ be a union of faces of $X$ which do not lie in
$\partial X$, and let $X^c$ be obtained by
cutting $X$ along $\Sigma$, with the induced branching: the
identifications of the pairs of faces which constitute $\Sigma$ are
removed, and $g:X^c\to X$ denotes the gluing map. Of the two faces of
$X^c$ over each face of $\Sigma$ choose one; the chosen faces form
$\Sigma_+$ and the others $\Sigma_-$. The boundary faces of $X^c$
are partitioned into the surviving external faces and these two
families. Their midpoint triangles form subcomplexes
$B^c_{\mathrm{ext}}$, $\tilde\Sigma_+$ and $\tilde\Sigma_-$ of
$\widetilde{X^c}$; these subcomplexes may share vertices. Moreover,
cutting can split an old edge midpoint, so the external complex
need not be literally $\widetilde{\partial X}$. The induced quotient
\[
f_{\mathrm{ext}}\colon B^c_{\mathrm{ext}}
\longrightarrow\widetilde{\partial X}
\]
is bijective on one- and two-cells, though it may identify vertices.
Since the cocycle equations involve only these cells, pullback gives
an isomorphism
\[
f_{\mathrm{ext}}^*\colon
Z^1(\widetilde{\partial X};\mathbb S)
\overset{\cong}{\longrightarrow}Z^1(B^c_{\mathrm{ext}};\mathbb S).
\]
Define
$r_{\mathrm{ext}}=(f_{\mathrm{ext}}^*)^{-1}\circ
\operatorname{res}_{B^c_{\mathrm{ext}}}$. Thus an external value
$\zeta$ means $z_c|_{B^c_{\mathrm{ext}}}=f_{\mathrm{ext}}^*\zeta$.
For the interfaces, $r_+$ and $r_-$ denote restriction to the
individual midpoint triangles, indexed by their matched faces in
$\Sigma$. They take values in the same product of triangle cocycle
spaces; equality is imposed face by face. We use
$r_{\partial X^c}=(r_{\mathrm{ext}},r_+,r_-)$ as notation for the
full boundary restriction in these coordinates, not as a disjoint
union decomposition of the boundary complexes. Connected components
and zero-cochains always refer to the actual truncated boundary,
with all its shared vertices.

The case of main interest is that of a two-sided triangulated surface
$\Sigma$ in $X$, with $\Sigma_\pm$ its two sides. In combinatorial
terms, we call the cut \emph{embedded} if the sides can be chosen so
that the two interfaces are disjoint, and if every edge of $X$ which
lies on $\Sigma$ has exactly two preimages in $X^c$. One of them then
lies on each interface, $g$ maps each interface isomorphically onto
$\tilde\Sigma\subset\tilde X$, the complexes $\Sigma_\pm$ are, as
surfaces glued along their edges, two copies of the cut surface with
opposite induced orientations, and $\tilde X$ is $\widetilde{X^c}$
with the two interfaces identified; in this sense the truncations are
then compatible with the cut. In general $\Sigma_+$ and $\Sigma_-$ consist
of corresponding faces, whose edges may however be identified
differently on the two sides. The cut is not assumed to be embedded,
except in Proposition~\ref{prop:trace-average}\,(v) and in
Example~\ref{ex:fig8-fibre-cut}. Write
\[
\mathcal A=\{z_c\in\mathcal Z_{X^c}:
 r_+z_c=r_-z_c\},
\]
where the two restrictions are compared on the one-cells which are
identified by the gluing map.
Restriction and descent of cellular cocycles give inverse
isomorphisms of compact abelian groups
\[
\iota:\mathcal Z_X\overset{\cong}{\longrightarrow}\mathcal A,
\qquad \kappa=\iota^{-1}.
\]
Indeed, equality on the glued one-cells is precisely what is needed
for a one-cochain to descend, and its cocycle equations are then
the same equations on the two-cells. The tetrahedra do not change.
Their cut product, after pulling the interface lines back to the
matching locus and evaluating the paired face factors, is therefore
$b_X\circ\kappa$. This is an identity of
sections, including their transition functions.

Fix compatible external data $(\alpha,\zeta)\in\mathcal Q_X$.
Define the compact affine space
\[
\mathcal A_{\alpha,\zeta}
=\{z_c\in\mathcal A:[\kappa z_c]=\alpha,
 r_{\mathrm{ext}}z_c=\zeta\}.
\]
Let $\eta=r_+z_c=r_-z_c$ and $\beta=[z_c]=g^*\alpha$. On this affine
space use
\[
\rho(z_c)=(\beta,\eta),\qquad
\mathcal C_{\alpha,\zeta}=\rho(\mathcal A_{\alpha,\zeta}).
\]
The same formula $\rho(z_c)=([z_c],r_+z_c)$ defines a homomorphism on
$\mathcal A$. On $\mathcal A_{\alpha,\zeta}$ the class
$\beta=g^*\alpha$ is constant, so that only $\eta$ varies on
$\mathcal C_{\alpha,\zeta}$; in the torsion version below, $\beta$
varies as well.
Both the image and every fibre carry their intrinsic probability
Haar measures. All boundary lines below are pulled back along the
facewise boundary maps to the matching locus. By
Proposition~\ref{prop:boundary-bundle}, the contracted weight
$b_{X^c}$ takes values there in
\[
(\mathcal L_{\partial X})_\zeta\otimes
(\mathcal L_{\Sigma_+})_\eta\otimes
(\mathcal L_{\Sigma_-})_\eta.
\]
The external factor is identified through $f_{\mathrm{ext}}^*$.
The two interface factors are products over matched individual
triangles. The corresponding triangle cocycles agree and their
induced orientations are opposite, so their pulled-back line
factors are dual, even when the interface complexes have different
vertex incidences. Facewise evaluation defines
\[
\operatorname{ev}_\Sigma\colon
(\mathcal L_{\Sigma_+})_\eta\otimes(\mathcal L_{\Sigma_-})_\eta
\longrightarrow\bC,
\]
and leaves
$\operatorname{ev}_\Sigma b_{X^c}(z_c)\in
(\mathcal L_{\partial X})_\zeta$.
Set
\[
I^\alpha_{X^c}(\beta,\eta;\zeta)=
e^{\pi\Imun\ell_X/(4\hbar)}
\int_{\rho^{-1}(\beta,\eta)\cap\mathcal A_{\alpha,\zeta}}
\operatorname{ev}_\Sigma b_{X^c}(z_c)\,d\nu(z_c).
\]
The cut has the same total level; for a disjoint gluing levels add.
Finally, the \emph{cut amplitude}
$\mathcal F_\hbar(X^c)(\beta;\zeta,\eta,\eta)$ is the sector function
of Definition~\ref{Maindef} for the complex $X^c$, with level
$\ell_X$, in the sector $\beta\in C_{X^c}$ and at the boundary
cocycle specified by the external pullback
$f_{\mathrm{ext}}^*\zeta$ and the matching face data
$(\eta,\eta)$ on the actual cut boundary. In the pulled-back lines
just described it takes values in
$(\mathcal L_{\partial X})_\zeta\otimes(\mathcal L_{\Sigma_+})_\eta
\otimes(\mathcal L_{\Sigma_+})_\eta^*$, and its \emph{contraction}
is $\operatorname{ev}_\Sigma$ applied to it.

\begin{theorem}[Conditional gluing]\label{prop:gluing}
For positive shapes the exact gluing formula is
\begin{equation}\label{eq:conditional-gluing}
\mathcal F_\hbar(X)(\alpha,\zeta)
=\int_{\mathcal C_{\alpha,\zeta}}
I^\alpha_{X^c}(\beta,\eta;\zeta)\,d\nu(\beta,\eta).
\end{equation}
The inner integral can be replaced by the contraction
$\operatorname{ev}_\Sigma\mathcal F_\hbar(X^c)(\beta;\zeta,\eta,\eta)$
of the already integrated cut amplitude if
\begin{equation}\label{eq:gluing-kernel-condition}
\ker\bigl([\cdot],r_{\partial X^c}\bigr)
\ \subseteq\ \ker([\kappa(\cdot)])
\quad\hbox{as subgroups of }\mathcal A.
\end{equation}
For torsion sectors, replace $[\kappa z_c]=\alpha$ throughout by
$r_T[\kappa z_c]=\tau$ and use the image and conditional measures
of that affine space. This is the gluing formula for
$\mathcal F^T_\hbar$. Both formulas persist as meromorphic identities
in the sense of Convention~\ref{conv:meromorphic}.
\end{theorem}
\begin{proof}
The isomorphism $\iota$ takes the fibre of $q_X$ to
$\mathcal A_{\alpha,\zeta}$ and preserves its probability Haar
measure. The identity of contracted integrands above and
Lemma~\ref{lem:haar}, applied to $\rho$, give
\eqref{eq:conditional-gluing}. Two cut cocycles with the same cut
class and full cut boundary value differ by an element of the
kernel on the left of \eqref{eq:gluing-kernel-condition}. Under
that condition their descended classes agree. Hence a nonempty
inner fibre in \eqref{eq:conditional-gluing} is the entire fibre
used to define the cut amplitude, with the same probability
measure. This proves the ordinary-trace assertion. Without this
condition the smaller, conditioned fibre is the one in the formula.

The torsion statement follows by the same proof applied to the
composite homomorphism $r_T\circ[\kappa(\cdot)]$. Alternatively,
integrate the fixed-class formula using successive disintegration.
This uses conditional measures on the actual images. The images are compact groups or their translates, so
Lemma~\ref{lem:haar} applies also when a fibre is disconnected.
Finally, both sides of \eqref{eq:conditional-gluing} are meromorphic
in the complex shapes, the level and the boundary data: the left side
by Theorem~\ref{thm:global-continuation}, and the right side by the
same theorem applied to the full fibre of $q_X$, transported by
$\iota$, before the disintegration along $\rho$. Hence the identity
persists as a meromorphic identity.
\end{proof}

The condition in \eqref{eq:gluing-kernel-condition} has a concrete
interpretation. An exact change $\delta u$ on $X^c$ can have zero
boundary restriction while $u$ takes different constants on the
two cut components. After gluing, this change can carry holonomy
around a loop crossing the cut. Thus $g^*\alpha$ and the cut
boundary cocycle alone need not determine the uncut class. The
conditional formula retains this datum before integrating it out.
Accordingly, the inner quantity $I^\alpha_{X^c}$ need not be
determined by the already integrated cut amplitude alone: it
retains the final class condition on the fibre of cut cocycles.
Theorem~\ref{prop:gluing} provides a sewing formula for those
integrated amplitudes when \eqref{eq:gluing-kernel-condition} holds,
because this condition identifies the conditioned fibre with the full
cut fibre.
Every map in the condition is represented by a finite integer
cochain matrix, so it is a testable algebraic condition for a
specified triangulation and cut.

The following examples isolate the role of the full cohomology
condition. Assume $\partial X=\emptyset$, and assume that the cut
has two disjoint connected interface complexes
$\tilde\Sigma_+,\tilde\Sigma_-\subset\widetilde{X^c}$, identified by
the gluing map. The cut boundary consists of these two interfaces.
All zero-cochains below have values in $\mathbb S$.

\begin{example}[A connected separating cut]\label{ex:separating-cut}
Suppose $\widetilde{X^c}=M_+\sqcup M_-$ has two connected
components, with $\tilde\Sigma_\pm\subset M_\pm$. This is the
cocycle model for a connected separating surface; for instance,
one may double a connected manifold with connected boundary.
Let $w\in\ker([\cdot],r_{\partial X^c})$. Since $[w]=0$, write
$w=\delta u$. Its zero boundary restriction implies that $u$ is
constant on each interface, say
$u|_{\tilde\Sigma_\pm}=c_\pm$. Define
\[
u'|_{M_\pm}=u|_{M_\pm}-c_\pm.
\]
Then $\delta u'=w$, while $u'$ is zero on both interfaces and
therefore descends to a zero-cochain $\bar u$ on $\tilde X$.
Consequently $\kappa w=\delta\bar u$, so $[\kappa w]=0$.
Condition~\eqref{eq:gluing-kernel-condition} holds, and
Theorem~\ref{prop:gluing} permits sewing the already integrated
cut amplitude.
\end{example}

\begin{example}[A connected nonseparating cut]\label{ex:nonseparating-cut}
Suppose instead that $\widetilde{X^c}$ is connected. A standard
topological model is the cut of $\Sigma\times S^1$ along
$\Sigma\times\{\mathrm{pt}\}$, which gives $\Sigma\times[0,1]$.
Choose $0\ne\theta\in\mathbb S$ and a zero-cochain $u$ which is
zero on $\tilde\Sigma_-$ and equal to $\theta$ on
$\tilde\Sigma_+$. The disjointness of the two interfaces permits
such a choice. Put $w=\delta u$. Then
\[
[w]=0,\qquad r_+w=r_-w=0,
\]
so $w$ belongs to the kernel on the left of
\eqref{eq:gluing-kernel-condition}. Take matched vertices
$q_-\in\tilde\Sigma_-$ and $q_+\in\tilde\Sigma_+$, and a cellular
path $\gamma$ in $\widetilde{X^c}$ from $q_-$ to $q_+$. After gluing,
this path is a cycle $\bar\gamma$ in $\tilde X$, and
\[
\langle\kappa w,\bar\gamma\rangle
=\langle\delta u,\gamma\rangle
=u(q_+)-u(q_-)=\theta\ne0.
\]
Thus $[\kappa w]\ne0$, and
\eqref{eq:gluing-kernel-condition} fails. The cocycles $0$ and $w$
have identical cut class and boundary cocycle but different
descended classes. The conditional integral in
\eqref{eq:conditional-gluing} retains precisely this extra holonomy.
What the ordinary trace computes in this situation is determined in
Proposition~\ref{prop:trace-average}.
\end{example}

The ordinary trace of the already integrated cut amplitude can be
computed for every cut. Let
\[
\mathcal N=\ker\bigl([\cdot],r_{\partial X^c}\bigr)\subseteq\mathcal A,
\qquad
\Theta_\Sigma=[\kappa(\mathcal N)]\subseteq C_X .
\]
Thus $\mathcal N=\ker q_{X^c}$ is the kernel on the left of
\eqref{eq:gluing-kernel-condition}, and that condition says that
$\Theta_\Sigma=0$.

\begin{proposition}[The ordinary trace as an average]
\label{prop:trace-average}
Let $X^c$ be obtained from $X$ by cutting along $\Sigma$ as above.
\begin{enumerate}
\renewcommand{\labelenumi}{\textup{(\roman{enumi})}}
\item The group $\mathcal N$ consists of the coboundaries $\delta u$
of the zero-cochains $u$ on $\widetilde{X^c}$ which are constant on
every connected component of the truncated boundary
$\widetilde{\partial X^c}$. It is connected, and $\Theta_\Sigma$ is a
subtorus of the identity component of $C_X$ on which $g^*$,
$i_\partial^*$ and $r_T$ vanish.
\item For positive shapes and every
$(\alpha,\zeta)\in\mathcal Q_X$,
\begin{equation}\label{eq:trace-average}
\int_{\mathcal C_{\alpha,\zeta}}\operatorname{ev}_\Sigma
\mathcal F_\hbar(X^c)(g^*\alpha;\zeta,\eta,\eta)\,d\nu(\eta)
=\int_{\Theta_\Sigma}\mathcal F_\hbar(X)(\alpha+t,\zeta)\,dt,
\end{equation}
with the probability Haar measure on $\Theta_\Sigma$; here
$(\alpha+t,\zeta)\in\mathcal Q_X$ by \textup{(i)}. The domain
$\mathcal C_{\alpha,\zeta}$ of the integral on the left is defined
through the uncut complex: it is the set of interface states $\eta$
for which $(g^*\alpha;\zeta,\eta,\eta)$ is realized by a cut cocycle
$z_c\in\mathcal A$ with $[\kappa z_c]=\alpha$ and
$r_{\mathrm{ext}}z_c=\zeta$, and in general it depends on $\alpha$
and not only on $g^*\alpha$; several of the sets
$\mathcal C_{\alpha+s,\zeta}$ may belong to the same class $g^*\alpha$.
For an embedded cut of a closed $X$, $\mathcal C_{\alpha,\zeta}$
consists of all $\eta$ for which $(g^*\alpha;\eta,\eta)$ lies in
$\mathcal Q_{X^c}$. Thus the ordinary
trace of the already integrated cut amplitude is the average of the
sectors of $X$ over the coset $\alpha+\Theta_\Sigma$. In particular it
is equal to $\mathcal F_\hbar(X)(\alpha,\zeta)$ whenever the sector
function is constant on this coset.
\item For positive shapes and every $(\tau,\zeta)$ in the image of
$q_X^T$,
\[
\mathcal F^T_\hbar(X)(\tau,\zeta)
=\int_{\mathcal C^T_{\tau,\zeta}}\operatorname{ev}_\Sigma
\mathcal F_\hbar(X^c)(\beta;\zeta,\eta,\eta)\,d\nu(\beta,\eta),
\]
where $\mathcal C^T_{\tau,\zeta}$ is the image under $\rho$ of
$\{z_c\in\mathcal A:r_T[\kappa z_c]=\tau,\
r_{\mathrm{ext}}z_c=\zeta\}$. Thus the torsion sections of $X$ are
ordinary traces of the sector amplitudes of $X^c$ over the classes
and the interface states, and no condition on the cut is needed.
\item Suppose $\partial X=\emptyset$ and
$\Theta_\Sigma\subseteq j^*H^1(X,\mathbb S)$. Then the ordinary trace
formula
\[
\mathcal F_\hbar(X)(\alpha)=\int_{\mathcal C_{\alpha,\zeta}}
\operatorname{ev}_\Sigma
\mathcal F_\hbar(X^c)(g^*\alpha;\zeta,\eta,\eta)\,d\nu(\eta)
\]
holds for positive shapes and all $\alpha\in C_X$, where $\zeta$ is
the unique element of $\mathcal Z_{\partial X}$. In particular it
holds for every cut of a triangulation of a closed oriented
three-manifold.
\item Suppose $\partial X=\emptyset$ and the cut is embedded. Then
\[
\Theta_\Sigma=\ker\bigl(g^*\colon C_X\to C_{X^c}\bigr)\cong
\mathbb S^{\,d},\qquad
d=|\pi_0(\tilde\Sigma)|-|\pi_0(\widetilde{X^c})|+|\pi_0(\tilde X)|.
\]
Hence \eqref{eq:gluing-kernel-condition} holds if and only if $g^*$ is
injective, that is, if and only if the graph whose vertices are the
components of $\widetilde{X^c}$ and whose edges are the components
of $\tilde\Sigma$ has no cycle.
\end{enumerate}
\end{proposition}
\begin{proof}
(i) An element of $\mathcal N$ is a cocycle $w$ on $\widetilde{X^c}$
with $[w]=0$ and $r_{\partial X^c}w=0$. The first condition says that
$w=\delta u$ for a zero-cochain $u$, and the second that the
restriction of $u$ to the truncated boundary has zero coboundary,
that is, that $u$ is constant on each connected component of the
actual truncated boundary, including all vertices shared by its
external and interface pieces. Write $U_{\mathrm{rel}}$ for this
torus of zero-cochains; then $\mathcal N=\delta U_{\mathrm{rel}}$.
These cochains form a torus, so that $\mathcal N$ and $\Theta_\Sigma$ are
compact connected groups. For $w\in\mathcal N$ one has
$g^*[\kappa w]=[\iota\kappa w]=[w]=0$ and
$i_\partial^*[\kappa w]=[r_{\mathrm{ext}}w]=0$. The kernel of $r_T$
is the identity component of $C_X$, which contains the connected
group $\Theta_\Sigma$.

(ii) Let $\mathcal K=\iota(\ker q_X)$, so that
$\mathcal A_{\alpha,\zeta}$ is a coset of $\mathcal K$, and put
$\mathcal A^*=\mathcal A_{\alpha,\zeta}+\mathcal N$, a coset of the
compact group $\mathcal K+\mathcal N$. We integrate
$e^{\pi\Imun\ell_X/(4\hbar)}\operatorname{ev}_\Sigma b_{X^c}
=e^{\pi\Imun\ell_X/(4\hbar)}b_X\circ\kappa$ over $\mathcal A^*$ with
its Haar probability, and disintegrate this integral in two ways by
Lemma~\ref{lem:haar}.

First consider $[\kappa(\cdot)]$ on $\mathcal A^*$. Its values are
the classes $\alpha+t$, $t\in\Theta_\Sigma$, and
$r_{\mathrm{ext}}$ is equal to $\zeta$ on $\mathcal A^*$, because
$r_{\mathrm{ext}}\mathcal N=0$. Hence the fibre over $\alpha+t$ is
contained in $\mathcal A_{\alpha+t,\zeta}$. Conversely, if
$z\in\mathcal A_{\alpha+t,\zeta}$ and $t=[\kappa n]$ with
$n\in\mathcal N$, then $z-n\in\mathcal A_{\alpha,\zeta}$. The fibres
are therefore the cosets $\mathcal A_{\alpha+t,\zeta}$ of
$\mathcal K$, and the integral over $\mathcal A_{\alpha+t,\zeta}$ is
$\mathcal F_\hbar(X)(\alpha+t,\zeta)$, as in the proof of
Theorem~\ref{prop:gluing}. This gives the right-hand side of
\eqref{eq:trace-average}.

Next consider $\rho$ on $\mathcal A^*$. As $\rho(\mathcal N)=0$, its
image is $\rho(\mathcal A_{\alpha,\zeta})=\mathcal C_{\alpha,\zeta}$,
with the same Haar probability as in \eqref{eq:conditional-gluing}.
Its fibres are cosets of $\ker\rho\cap(\mathcal K+\mathcal N)$, a
group which contains $\mathcal N$. Conversely, if $x=k+n$ with
$k\in\mathcal K$, $n\in\mathcal N$ and $\rho(x)=0$, then $[x]=0$,
$r_\pm x=0$ and
$r_{\mathrm{ext}}x=r_{\mathrm{ext}}k+r_{\mathrm{ext}}n=0$, so that
$x\in\mathcal N$. Thus the fibres of $\rho$ on $\mathcal A^*$ are
cosets of $\mathcal N=\ker q_{X^c}$: they are the full fibres of
$q_{X^c}$ over the points $(g^*\alpha;\zeta,\eta,\eta)$, with the
measures which define $\mathcal F_\hbar(X^c)$. The cut has the same
level, and evaluation of the dual lines commutes with the fibre
integral. This gives the left-hand side.

(iii) Let $z_c\in\mathcal A$ with $r_T[\kappa z_c]=\tau$,
$r_{\mathrm{ext}}z_c=\zeta$ and $\rho(z_c)=(\beta,\eta)$. The full
fibre of $q_{X^c}$ through $z_c$ is $z_c+\mathcal N$. It lies in
$\mathcal A$, and $r_{\mathrm{ext}}$, $\rho$ and, by (i),
$r_T[\kappa(\cdot)]$ are constant on it. Hence the conditioned fibre
in the torsion version of \eqref{eq:conditional-gluing} is the full
fibre of $q_{X^c}$ over $(\beta;\zeta,\eta,\eta)$, with the same
probability measure, and its integral is
$\operatorname{ev}_\Sigma\mathcal F_\hbar(X^c)(\beta;\zeta,\eta,\eta)$.

(iv) The image under $j^*$ of the identity component of
$H^1(X,\mathbb S)$ is a connected subgroup of finite index in the
compact group $j^*H^1(X,\mathbb S)$. It is therefore the identity
component of that group, and contains the connected group
$\Theta_\Sigma$. By Corollary~\ref{cor:sector-constancy} the integrand
on the right of \eqref{eq:trace-average} is then constant. For a
closed oriented three-manifold every class extends, by
Corollary~\ref{cor:closed-manifold-sectors}.

(v) As the cut is embedded, the difference of the restrictions to
$\tilde\Sigma_+$ and to $\tilde\Sigma_-$ gives a short exact sequence
of cellular cochain complexes with coefficients in $\mathbb S$,
\[
0\to C^*(\tilde X)\xrightarrow{\ g^*\ }C^*(\widetilde{X^c})
\xrightarrow{\ r_+-r_-\ }C^*(\tilde\Sigma)\to0,
\]
and hence an exact sequence
\[
H^0(\widetilde{X^c})\to H^0(\tilde\Sigma)
\xrightarrow{\ \partial\ }H^1(\tilde X)\xrightarrow{\ g^*\ }
H^1(\widetilde{X^c}).
\]
The connecting homomorphism $\partial$ sends a locally constant
function $c$ on $\tilde\Sigma$ to $[\kappa\,\delta u]$, where $u$ is
any zero-cochain on $\widetilde{X^c}$ with $r_+u-r_-u=c$; one may
take $u$ locally constant on both interfaces. As
$\partial X=\emptyset$, the truncated boundary of $X^c$ is
$\tilde\Sigma_+\sqcup\tilde\Sigma_-$, and (i) shows that
$\Theta_\Sigma$ is the image of $\partial$, that is, the kernel of
$g^*$. This image is isomorphic to the cokernel of the first map, a
torus of dimension $d$, because the kernel of that map is
$H^0(\tilde X)$. The number $d$ is the first Betti number of the
graph in the statement.
\end{proof}

\begin{remark}[The sector retained by an ordinary trace]
The left-hand side of \eqref{eq:trace-average} retains the uncut
class $\alpha$ through its integration domain
$\mathcal C_{\alpha,\zeta}$. In general the cut class $g^*\alpha$
alone does not determine this domain, so this trace is not in general
a function of $g^*\alpha$ alone. Even when it is, the result is the
average over $\alpha+\Theta_\Sigma$, and need not recover the
individual uncut sector. Examples~\ref{ex:trace-failure}
and~\ref{ex:fig8-fibre-cut} illustrate the latter loss of information.
\end{remark}

The description of the domain $\mathcal C_{\alpha,\zeta}$ for an
embedded cut of a closed $X$ in \textup{(ii)} is proved as follows.
A cocycle $z_c$ over a point $(g^*\alpha;\eta,\eta)$ of
$\mathcal Q_{X^c}$ lies in $\mathcal A$, and
$[\kappa z_c]\in\alpha+\ker g^*=\alpha+\Theta_\Sigma$ by
\textup{(v)}; as
$\mathcal A_{\alpha+t,\zeta}=\mathcal A_{\alpha,\zeta}+n$ for
$t=[\kappa n]$ and $\rho(\mathcal N)=0$, the point $\rho(z_c)$ lies in
$\mathcal C_{\alpha,\zeta}$. The left-hand side of
\eqref{eq:trace-average} is then determined by the cut amplitude and
by $g^*\alpha$.

For an embedded cut,
Examples~\ref{ex:separating-cut} and~\ref{ex:nonseparating-cut} are
the cases of \textup{(v)} in which the graph consists of one edge
joining two vertices, and of one loop. Embeddedness cannot be omitted
in \textup{(v)}: in general only $\Theta_\Sigma\subseteq\ker g^*$
remains. For instance, if $X$ is cut along all its faces, then $X^c$
is the disjoint union of its tetrahedra, every one-cell of
$\widetilde{X^c}$ lies in the truncated boundary, and
$\mathcal N=0$. Thus \eqref{eq:gluing-kernel-condition} holds, and
\eqref{eq:conditional-gluing} is the definition of the state
integral, whereas $g^*=0$.

By \textup{(iv)}, the failure of \eqref{eq:gluing-kernel-condition} in
Example~\ref{ex:nonseparating-cut} does not affect the ordinary trace
formula for $\Sigma\times S^1$, or for any other closed oriented
three-manifold. For a closed pseudo-manifold the ordinary trace
formula can fail only if classes in $\Theta_\Sigma$ do not extend
over the vertices. The following two examples show that this happens:
the first one for a cut which is not embedded, in a case where no
integration remains and the failure can be seen for all positive
shapes, and the second one for an embedded cut surface.

\begin{example}[Failure of the ordinary trace]\label{ex:trace-failure}
In the conventions of Section~\ref{sec:example-conventions}, let $X$
consist of three tetrahedra with the signs $(+,+,-)$ and the face
identifications
\begin{gather*}
(1,0)\sim(2,3),\qquad(1,1)\sim(2,2),\qquad(1,2)\sim(3,0),\\
(1,3)\sim(3,3),\qquad(2,0)\sim(3,2),\qquad(2,1)\sim(3,1).
\end{gather*}
All eighteen edges of the tetrahedra are identified, there is one
vertex, and its link is a surface of genus three: the link has
$12$ triangles, $18$ edges and $2$ vertices, hence Euler characteristic
$-4$. Solving the
compatibility equations modulo $\mathbb Z$ gives all states in the
form
\begin{align*}
y^{(1)}&=p+(0,\,v,\,w-u,\,2v-u,\,w-u+v,\,w-u+v),\\
y^{(2)}&=q+(0,\,w-v,\,w-u-v,\,w-v,\,w-u,\,2w-2u+v),\\
y^{(3)}&=r+(0,\,v,\,v-u,\,2v-u,\,w-2u+2v,\,w-2u+3v),
\end{align*}
where $p,q,r\in\mathbb S$ are added to all six components and
$u,v,w\in\mathbb S$. The group $\mathbb S(X)$ only translates
$p,q,r$, so that, by Proposition~\ref{Omega}, $(u,v,w)$ is a
coordinate on $H^1(X',\mathbb S)\cong\mathbb S^3$. The state
coordinates of the three tetrahedra are $(u,v)$, $(v,u-2v)$ and
$(w,0)$, and all paired face coordinates agree as real lifts. As in
Section~\ref{sec:one-edge-example}, no integration remains, and
\eqref{eq:contraction-lifts} gives
\begin{equation}\label{eq:trace-failure-sector}
\mathcal F_\hbar(X)(u,v,w)=e^{\Imun\pi\ell_X/(4\hbar)}\,
g_{a_1,c_1}(u,v)\,g_{a_2,c_2}(v,u-2v)\,
\overline{g_{a_3,c_3}(w,0)}.
\end{equation}

Cut $X$ along the two faces $(1,1)\sim(2,2)$ and $(1,2)\sim(3,0)$,
and let $\Sigma_-$ consist of the faces $(1,1)$, $(1,2)$ and
$\Sigma_+$ of the faces $(2,2)$, $(3,0)$. Then $X^c$ is connected and
has six edges, three on each interface. The two interfaces are
connected and disjoint, but the cut is not embedded: the faces of
$\Sigma_+$ are glued along their edges to a sphere and those of
$\Sigma_-$ to a torus, so that $\Sigma$ is not a surface in $X$. A
zero-cochain equal to $\theta$ on $\tilde\Sigma_+$ and to zero on
$\tilde\Sigma_-$ adds $\theta$ to the component $12$ of $y^{(1)}$,
to the components $01,03,13,23$ of $y^{(2)}$ and to the components
$03,12,13,23$ of $y^{(3)}$. This changes $(u,v,w)$ by
$-(\theta,0,\theta)$. The sign does not change the circle as
$\theta$ ranges over $\mathbb S$. Hence $\Theta_\Sigma$ is the circle of the
classes $(\theta,0,\theta)$, condition
\eqref{eq:gluing-kernel-condition} fails, and by
Proposition~\ref{prop:trace-average}\,(ii) the ordinary trace of the
integrated cut amplitude in the sector $(u,v,w)$ is
\[
\int_0^1\mathcal F_\hbar(X)(u+\theta,v,w+\theta)\,d\theta .
\]
For no positive shapes is this equal to
$\mathcal F_\hbar(X)(u,v,w)$ in all sectors. Indeed, otherwise
$\mathcal F_\hbar(X)$ would be invariant under the translations
$(u,w)\mapsto(u+\theta,w+\theta)$, and its Fourier coefficients with
respect to $(u,w)$ would vanish off the antidiagonal. By
\eqref{eq:trace-failure-sector} these coefficients are the products
of the Fourier coefficients of the two functions
$u\mapsto g_{a_1,c_1}(u,v)g_{a_2,c_2}(v,u-2v)$ and
$w\mapsto\overline{g_{a_3,c_3}(w,0)}$, which are periodic because the
multipliers of $L$ cancel. Those of the second function are all
different from zero, by the computation at the end of
Example~\ref{ex:sector-reversal-obstruction}. Hence the first
function would vanish for all $u$ and $v$. This is impossible: it is
the product of two real-analytic functions on $\mathbb R^2$
(Lemma~\ref{lem:weight-regularity}), neither of which vanishes
identically, since $W$ is injective.
\end{example}

\begin{example}[An embedded cut surface]\label{ex:fig8-fibre-cut}
Let $X_0$ be the triangulation of the figure-eight pseudo-manifold by
two tetrahedra of Section~\ref{sec:fig8-example}. Apply a $2$--$3$
move at its face $(1,0)\sim(2,2)$, and then a second $2$--$3$ move at
the face spanned by the new edge and the vertex $2$ of the first
tetrahedron of $X_0$. In the conventions
of Section~\ref{sec:example-conventions}, and with a suitable
numbering, the resulting triangulation $X$ consists of four tetrahedra
with the signs $(+,+,-,-)$ and the face identifications
\begin{gather*}
(1,0)\sim(2,3),\qquad(1,1)\sim(3,1),\qquad(1,2)\sim(4,0),\qquad
(1,3)\sim(3,3),\\
(2,0)\sim(4,2),\qquad(2,1)\sim(4,3),\qquad(2,2)\sim(3,2),\qquad
(3,0)\sim(4,1).
\end{gather*}
It has one vertex and four edges. Conversely, $X_0$ is recovered from
$X$ by the $3$--$2$ move at the edge $13$ of the second tetrahedron,
which has valence three, followed by the $3$--$2$ move at the edge
$02$ of the first tetrahedron.

Let $\Sigma$ consist of the two faces $(2,0)\sim(4,2)$ and
$(3,0)\sim(4,1)$, with $\Sigma_+$ formed by the faces $(2,0)$, $(3,0)$
and $\Sigma_-$ by the faces $(4,2)$, $(4,1)$. The sides of the two
faces are three different edges of $X$, along which they are glued to
a torus with one vertex, the vertex of $X$. Each of these edges has
exactly two preimages in $X^c$, the two interfaces are disjoint, and
$X^c$ is connected. Thus the cut is embedded, and $\Sigma$ does not
separate $X$. As $H^1(X',\mathbb S)\cong\mathbb S$,
Proposition~\ref{prop:trace-average}\,(v) gives $\Theta_\Sigma=C_X$,
and by \textup{(ii)} and Theorem~\ref{RelationON} the ordinary trace
of the integrated cut amplitude is, in every sector,
\[
\int_{C_X}\mathcal F_\hbar(X)(\alpha)\,d\alpha=Z_\hbar(X),
\]
the original partition function.

For generic positive shapes this is not the sector function. Indeed,
let $X_1$ be the intermediate triangulation with three tetrahedra,
and suppose that the ordinary trace formula holds for all positive
shapes of $X$ and all sectors, so that $\mathcal F_\hbar(X)(\alpha)$
does not depend on $\alpha$ for any positive shapes. By the invariance
under the second move, in the form of
Corollary~\ref{cor:all-branched-moves},
$\mathcal F_\hbar(X_1)(\alpha)$ then does not depend on $\alpha$ for
those positive shapes of $X_1$ which arise from positive shapes of
$X$ with a balanced edge of valence three. These form a nonempty
open set, as can be seen in the standard local order of the move.
Put $t=a_2$. Equations~\eqref{pentagonconditions} give
\[
a_0=a_1-t,\quad a_4=a_3-t,\quad
c_0=c_1-a_3+t,\quad c_4=c_3-a_1+t,\quad c_2=c_1+c_3.
\]
For positive old shapes, all new shapes are positive precisely when
\[
\max(0,a_3-c_1,a_1-c_3)<t<
\min(a_1,a_3,\tfrac12-c_1-c_3).
\]
For instance, $a_1=a_3=c_1=c_3=1/10$ and $t=1/20$ satisfy
these strict inequalities, which persist in a neighbourhood.
Rebranching only permutes the geometric angles, so the same
positive-open-set assertion applies to each of the two moves here.
As $\mathcal F_\hbar(X_1)(\alpha)$ is real-analytic in the
shapes, and the space of positive shapes is connected, it does not
depend on $\alpha$ for any positive shapes of $X_1$. The same
argument, applied to the first move, shows that
$\mathcal F_\hbar(X_0)(\alpha)$ does not depend on $\alpha$ for any
positive shapes of $X_0$, in contradiction with
Section~\ref{sec:fig8-example}. To describe the exceptional locus,
write $s$ for the positive shape parameters, fix the level, and put
\[
\mathcal D(s,h)=\mathcal F_{\hbar,s}(X)(h)-\mathcal F_{\hbar,s}(X)(0),
\qquad J(s)=\int_{\mathbb S}|\mathcal D(s,h)|^2\,dh.
\]
Joint real analyticity and compactness of $\mathbb S$ imply that
$J$ is real analytic. Indeed, on a common local complex shape
neighbourhood its integrand extends holomorphically as
$\mathcal D(z,h)\overline{\mathcal D(\bar z,h)}$, and integration preserves
holomorphy. Continuity and nonnegativity show that $J(s)=0$ exactly
when the sector function is constant. The preceding argument gives
$J\not\equiv0$, so the locus of sector-independent positive shapes
is the proper real-analytic subset $J^{-1}(0)$. The
analytic continuation in the shapes of $X_1$ cannot be avoided: the
new edge of the first move is the edge $02$ of the first tetrahedron
of $X$, where it has valence two, so that it is not balanced for any
positive shapes of $X$.
\end{example}

\begin{proposition}[Branching-independent sewing and associativity]
\label{prop:branching-sewing}
The boundary transports of Theorem~\ref{thm:branching} commute with
the conditional gluing formula \eqref{eq:conditional-gluing} and
with its torsion version. Any finite sequence of compatible cuts
and gluings gives the same answer in either order of conditional
integration, provided the final cohomology condition is retained
and the measures are the conditional Haar probabilities on the
actual images and fibres. When
\eqref{eq:gluing-kernel-condition} holds, this statement also applies
to the ordinary trace formula.
\end{proposition}
\begin{proof}
Cut along all the specified interfaces at once. Let
$\mathcal A^{\mathrm{all}}_{\alpha,\zeta}$ be the affine space of
cocycles on the cut pieces whose restrictions agree facewise at
each matched interface, whose external restriction is the pullback
of $\zeta$ along the external quotient map, and whose descended
class is $\alpha$. Restriction from the uncut complex is an affine
isomorphism from $q_X^{-1}(\alpha,\zeta)$ onto this space. Its
probability Haar integral of the fully contracted product of
weights is therefore the uncut amplitude.

Every order of sewing is successive disintegration of this same
integral along homomorphisms of compact abelian groups or their
affine translates. Lemma~\ref{lem:haar} identifies each stage's
measure with its conditional Haar probability. Uniqueness of these
measures and Fubini prove associativity, also for disconnected
fibres. Replace the final class condition by its torsion restriction
for the torsion assertion.

To change branchings, use \eqref{eq:branching-half-shift} on each
piece, after transporting the two interface branchings to a common
one if necessary. The affine maps preserve all matching conditions
and the final class. On each interface the two line transports are
dual after pullback to the facewise matching locus, so evaluation
cancels them exactly. Only the external
transport survives. Equation~\eqref{eq:boundary-rebranch-cocycle}
shows that the choice of common interface branching and the order
of transports do not matter. This proves the assertion for positive
shapes. The full matching-fibre integral has meromorphic
continuation by Theorem~\ref{thm:global-continuation}, so the same
identities hold meromorphically. The ordinary trace conclusion uses precisely
the fibre equality already proved under
\eqref{eq:gluing-kernel-condition}.
\end{proof}

\subsection{Proofs of Theorems~\ref{main} and~\ref{mainT}}
\label{sec:proofs-main}
We now assemble the proofs of the two main theorems of the
introduction, clause by clause.

\begin{proof}[Proof of Theorem~\ref{main}]
\emph{Existence and smoothness.} For positive shapes the contracted
weight $b_X$ is a smooth section over the compact fibres of $q_X$
(Lemma~\ref{lem:weight-regularity} and
Proposition~\ref{prop:boundary-bundle}), so the integral
\eqref{eq:si-preliminary} converges absolutely, and by
Lemma~\ref{lem:haar} it depends smoothly on $(\alpha,\zeta)$ and on
the shapes and the level; this is \eqref{eq:si}.

\emph{(i) Regularity and continuation.} By
Lemma~\ref{lem:weight-regularity}, $b_X$ extends holomorphically to a
neighbourhood of every compact real fibre over a positive shape, with
the complex shapes and the level as parameters. Lemma~\ref{lem:holomorphic-integral},
applied to $\varpi=q_X$ and to this integrand, shows that
$\mathcal F_\hbar(X)$ has a holomorphic germ in the complexified
shapes, level, class and boundary cocycle near every point of the
real domain, and that this germ is unique. Theorem~\ref{thm:global-continuation}
provides the unique single-valued meromorphic extension to
$\widetilde{LS}^{\mathbb C}(X)\times\mathcal Q_{X,\mathbb C}$. (For
the torsion component of the class, which is discrete, the germ is
taken with that component fixed.)

\emph{(ii) Multiplicativity and branching.} For a disjoint union
$X_1\sqcup X_2$, the cocycle spaces, the fibres of $q_X$ and their
Haar probability measures are products, the contracted weight is the
tensor product of the two contracted weights, the boundary lines
tensor, and the levels add; Fubini's theorem gives
$\mathcal F_\hbar(X_1\sqcup X_2)=\mathcal F_\hbar(X_1)\otimes
\mathcal F_\hbar(X_2)$. Independence of branching, up to the canonical
boundary transports $U_{\partial X,\beta\beta'}$ and the base maps
$T_{\beta\beta'}$, is Theorem~\ref{thm:branching}; for closed $X$
the transport is the identity, so the functions
$\mathcal F_\hbar(X)(\alpha)$ themselves are independent of the
branching. The transports preserve the sector, the torsion
component and the level, as stated there.

\emph{(iii) Gauge invariance.} This is
Proposition~\ref{prop:gauge-rigorous}, for based gauge
transformations between positive shapes.

\emph{(iv) Move invariance.} For a compatible shaped $2$--$3$ move
whose branching is induced by a positive standard order of the
bipyramid, this is Proposition~\ref{2-3/3-2M}, with the level relation
$\ell_3-\ell_2=P_e/3$. Corollary~\ref{cor:all-branched-moves}
removes the restriction on the branching, by means of the transports
of Theorem~\ref{thm:branching}, and shows that $P_e$ depends only on
the shapes of the bipyramid. The inverse move is obtained by reading
the identity backwards.

\emph{Meromorphic persistence.} Each of the identities in
(ii)--(iv) is an identity between two quantities which admit
meromorphic continuations by Theorem~\ref{thm:global-continuation}
(applied to $X$, to the two sides of a move, and to the two
branchings; the transports are entire), and which agree on the
positive real locus, a totally real submanifold of maximal dimension
of the connected complex parameter family concerned; for the move
identity the relevant family is the connected complex affine family
of shapes satisfying \eqref{pentagonconditions}, which has a
nonempty positive open part by the example in the proof of
Corollary~\ref{cor:global-identities}. By the identity theorem the
two continuations agree, which is the meaning of persistence in
Convention~\ref{conv:meromorphic}. The gluing properties summarized
after the theorem are Theorem~\ref{prop:gluing} and
Proposition~\ref{prop:trace-average}.
\end{proof}

\begin{proof}[Proof of Theorem~\ref{mainT}]
Put $q_X^T=(r_T\circ[\,\cdot\,],r_\partial)\colon\mathcal Z_X\to
\mathcal Q^T_X$. By Lemma~\ref{lem:haar}, applied to the composition
$\mathcal Z_X\to\mathcal Q_X\to\mathcal Q^T_X$, the probability Haar
fibre integral of $b_X$ over a fibre of $q_X^T$ is the integral over
the compatible classes $\alpha$ of the fibre integrals over the
fibres of $q_X$; this is \eqref{eq:si-torsion}, with the conditional
Haar probability measure on the fibre
$\{\alpha:r_T(\alpha)=\tau,\ i_\partial^*\alpha=[\zeta]\}$, including
its finitely many components when $\partial X\neq\emptyset$, and it
shows that no splitting of integral homology is needed. For closed
$X$ the fibre is the torus $r_T^{-1}(\tau)$, which gives
\eqref{eq:torsion-average}. Smoothness and the holomorphic germ
follow from Lemmas~\ref{lem:haar} and~\ref{lem:holomorphic-integral}
applied to $\varpi=q_X^T$, the global continuation is the second
assertion of Theorem~\ref{thm:global-continuation}, and
multiplicativity follows from that of $\mathcal F_\hbar(X)$ and
Fubini. The branching, gauge and move invariance of
$\mathcal F^T_\hbar(X)$ follow by integrating the corresponding
identities for $\mathcal F_\hbar(X)$ over the fibre: the transports of
Theorem~\ref{thm:branching} and the gauge and move identities
preserve $r_T$ and $i_\partial^*$ of the class, so they commute with
the conditional average. Their meromorphic persistence follows as in
the proof of Theorem~\ref{main}. The gluing law is the torsion
version of Theorem~\ref{prop:gluing}, and the assertion that the
conditioned fibre is the full fibre of the cut complex, so that
$\mathcal F^T_\hbar(X)$ is an ordinary trace for every cut, is
Proposition~\ref{prop:trace-average}\textup{(iii)}.
\end{proof}

\section{The original Teichm\"uller TQFT partition function}\label{proofTI}
In this section we recall the definition of the original Teichm\"uller TQFT partition function $Z_\hbar(X)$ of \cite{MR3227503}, in the form which is used in Section~\ref{ONrel} for the comparison with the new formulation and for the proof of Theorem~\ref{OldInv}. In the theory from \cite{MR3227503}, we also allow for a graph $\Gamma$ consisting of unbalanced edges, and in this case we get an invariant of the pair $(X,\Gamma)$ as discussed in this paper. We now make the needed assumption in order to have the original formulation well defined, namely that 
$$ H_2(X',{\mathbb Z}) = 0,$$
that is, that $X$ is admissible.

The invariant defined in  \cite{MR3227503} can be rewritten as the following multiple integral
\begin{align}
\label{z(x)}
Z_\hbar(X) =e^{\Imun\pi\ell_X/(4\hbar)}\int_{y\in\mathbb{R}^{\Delta_3(X)}} K(X,y)\prod_{T\in\Delta_3(X)}\phi_{(\alpha_1,\alpha_3)(T)}(y(T))\operatorname{d}\!y,
\end{align}
where boundary face variables are held fixed and we have defined a \emph{kinematical kernel}
\begin{align}\label{k(x,y)}
K(X,y) &:= \int_{x\in \mathbb{R}^{\Delta_2(X)\setminus\Delta_2(\partial X)}} 
\prod_{T\in\Delta_3(X)}M(\varepsilon(T),x\vert_{\Delta_2(T)},y(T))\operatorname{d}\!x
\end{align}
and where to a shaped tetrahedron $T\in \Delta_3(X)$ of sign $\varepsilon\in\{\pm1\}$ 
and normalized dihedral angles $\alpha_i=\alpha_T(v_0v_i)/(2\pi)$ along the edges connecting $0$-th with $i$-th 
vertices with $i\in\{1,2,3\}$, we associate a tempered distribution in the variables $x\in\mathbb{R}^{\Delta_2(T)}$, $x_i= x(\partial_i T)$ associated to the facets of $T$, and a real variable $y\in\mathbb{R}$
\begin{equation}
M(\varepsilon,x,y) := 
e^{2\pi\mathsf{i}\varepsilon x_0y}\delta(x_0-x_1+x_2)\delta(x_2-x_3+y),
\end{equation}
and the function
\begin{equation}
\label{phif}
\phi_{\alpha,\beta}(y)=\tilde\psi'_{\alpha,\beta}(y)
\end{equation}
if $\varepsilon=1$ and its complex conjugate otherwise.
Compared to the definition given in  \cite{MR3227503}, we inserted in \eqref{z(x)} artificially 
extra integrations over the variables $y$ compensated by an equal number 
of extra delta-functions. 

For positive shapes, and under the above assumption, the right-hand
side of \eqref{z(x)} is the contraction of tempered distributions of
\cite{MR3227503}, which is well defined because $X$ is admissible. Its
relation to the compact state integral of Definition~\ref{Maindef} is
established in the next section.

\begin{remark}[Normalization conventions]\label{rem:normalization}
Theorem~\ref{RelationON} compares two quantities whose normalizations
are fixed as follows, and in no other way.
\begin{enumerate}
\renewcommand{\labelenumi}{\textup{(\alph{enumi})}}
\item \emph{The original partition function.} In \eqref{z(x)} and
\eqref{k(x,y)} every face variable $x_f$, $f\in\Delta_2(X)$, and every
tetrahedron variable $y(T)$ carries Lebesgue measure on $\REALS$, each
delta function is the standard delta distribution in the displayed
linear coordinate, each tetrahedron contributes exactly the kernel
$M(\varepsilon_T,x|_{\Delta_2(T)},y(T))\,\phi_{(\alpha_1,\alpha_3)(T)}(y(T))$,
and the only further factor is $e^{\Imun\pi\ell_X/(4\hbar)}$. Carrying
out the $y$-integrations against the second delta function returns,
for each tetrahedron, the Boltzmann weight
\[
B(T,x)=e^{2\pi\Imun\varepsilon_Tx_0(x_3-x_2)}\,
\delta(x_0-x_1+x_2)\,\phi_{(\alpha_1,\alpha_3)(T)}(x_3-x_2),
\qquad x_i=x(\partial_iT),
\]
of \cite[Section~4]{MR3227503}. For $\varepsilon_T=1$ it is
$\langle x_0,x_2|\mathsf T(a,c)|x_1,x_3\rangle$ in that paper's
notation. For $\varepsilon_T=-1$ it is the complex conjugate,
equivalently
$\langle x_1,x_3|\overline{\mathsf T}(a,c)|x_0,x_2\rangle$;
thus the input and output faces are interchanged together with the
orientation. The partition
function of \cite{MR3227503} is the contraction of these kernels over
the paired face variables with Lebesgue measure and no further
constants. Thus $Z_\hbar(X,\Gamma)$ in \eqref{z(x)} is exactly the
invariant of \cite{MR3227503}, and the graph $\Gamma$ enters only
through the balancing condition on the shapes.
\item \emph{The compact state integral.} In Definition~\ref{Maindef}
the fibre of $q_X$ carries the probability Haar measure of
Lemma~\ref{lem:haar}, the face contraction is the evaluation of dual
lines of Proposition~\ref{prop:boundary-bundle} with no constant, and
the level enters through the same factor $e^{\Imun\pi\ell_X/(4\hbar)}$.
In the torsion average \eqref{eq:si-torsion} the measure is again a
probability measure, and \eqref{ON-rel} contains the explicit factor
$1/|\mathcal T_X|$, which is the Haar mass of each component of
$C_X$.
\end{enumerate}
With these conventions, no constant factor appears in \eqref{ON-rel};
the normalization is confirmed independently by the computations in
Sections~\ref{sec:rp3-example}, \ref{sec:l31-example}
and~\ref{sec:one-edge-example}, which evaluate both sides.
\end{remark}

\section{Edge-face transformation and the relation to the old formulation}\label{ONrel}

The WGZ edge--face transformation and its role in relating the two
formulations were reviewed in our ICM paper \cite[Section~6]{AKICM}.
In particular, Lemma~\ref{eft} below is the local identity stated
there as Theorem~5. We include its proof and then derive the
normalized cohomology-sector averaging formula~\eqref{ON-rel}.

In integrable lattice models, the vertex--face correspondence relates
vertex-model edge states to face-model height variables through local
intertwining weights; Baxter's eight-vertex/SOS correspondence is a
standard example \cite{Baxter1973,KKW2005}. The transformation below
plays an analogous role, replacing face distributions by edge-state
sections through the WGZ transform.

\subsection{The local edge-face transformation}\label{sec:edge-face-local}
For any $f\in {\mathcal S}(\REALS)$, we associate the tempered distribution $H(f)\in {\mathcal S}'(\REALS^4)$ given as follows
\[
H(f)(x_0,x_1,x_2,x_3) = \delta(x_0+x_2-x_1)f(x_3-x_2)e^{2\pi\Imun x_0(x_3-x_2)},
\]
which we consider as a continuous linear map
$$H(f) : {\mathcal S}(\bR^2) \ra {\mathcal S}'(\bR^2)$$
via the formula
$$H(f)(g)(h) = H(f)(\hat\pi^*(g)\check\pi^*(h)).$$
Here $\hat \pi : \bR^4 \ra \bR^2$ is given by $\hat\pi(x_0,x_1,x_2,x_3) = (x_1,x_3)$ and $\check\pi : \bR^4 \ra \bR^2$ is given by $\check\pi(x_0,x_1,x_2,x_3) = (x_0,x_2)$.
We now consider the following tensor extension of the WGZ-transform 
$$
W\otimes  W : {\mathcal S}(\REALS^2) \ra C^\infty(\Pi \times \Pi, L\boxtimes L)
$$
defined by
\begin{eqnarray*} \lefteqn{W\otimes W (h)(s,t,x,y) = }\\
& \phantom{kkkk}e^{\pi i (st + xy)} \sum_{m_1,m_2\in \bZ} h(s+m_1,x+m_2) e^{2\pi i(m_1t + m_2 y)}.
\end{eqnarray*}
and similarly
$$\overline{W}\otimes\overline{W} : {\mathcal S}(\bR^2) \ra C^\infty(\Pi \times \Pi, L^*\boxtimes L^*),$$
given by
\begin{eqnarray*} \lefteqn{\overline{W}\otimes\overline{W}(h)(s,t;x,y) = }\\
& \phantom{kkkk}e^{-\pi i (st + xy)} \sum_{m_1,m_2\in \bZ} h(s+m_1,x+m_2) e^{-2\pi i(m_1t + m_2 y)}.
\end{eqnarray*}
Consider now the maps
$$ \mathsf F : \bS^5 \ra \Pi, \hskip4mm \tilde\pi_i :  \bS^5 \ra \Pi^2,$$
given by
$$ \mathsf F(u,s,t,x,y) = (u, s+t+u-y)$$
and
$$\tilde\pi_1(u,s,t,x,y) = (s+x,t+u,x+u,y-t-u).$$
and $\tilde\pi_2$ is the map which projects away the first factor and onto the last four factors. We then get that
$$ \mathsf F^*L \otimes \tilde\pi_1^*(L\boxtimes L) \cong  \tilde\pi_2^*(L\boxtimes L).$$
There is an obvious embedding
$$ C^\infty(\Pi^2, L\boxtimes L) \subset  C^\infty(\Pi^2, L^*\boxtimes L^*)^*,$$
obtained by pointwise evaluation followed by integration over $\Pi^2$. Thus we see that
$$ ( \tilde\pi_2)_* (\mathsf F^*W(f) \tilde\pi_1^*(W\otimes W)(g)) \in C^\infty(\Pi^2, L^*\boxtimes L^*)^*$$
for any $g\in {\mathcal S}(\bR^2)$.

\begin{lemma} \label{Leq:vetr} As elements of $C^\infty(\Pi \times \Pi, L^*\boxtimes L^*)^*$
\begin{equation}\label{eq:vetr}
H(f)(g)\circ (\overline{W}\otimes \overline{W})^{-1} =( \tilde\pi_2)_* (\mathsf F^*W(f) \tilde\pi_1^*(W\otimes W)(g)).
\end{equation}
\end{lemma}

\begin{proof}
All calculations below are understood after pairing with a smooth
dual-bundle test section. One may first take finite Fourier sums and
Schwartz cutoffs, perform the reindexings, and then pass to the limit
in the distribution topology. The WGZ map and its inverse are
continuous for the Schwartz and smooth-section topologies (Appendix
\ref{WGZ}); hence these limits preserve the identities.
For any $g\in {\mathcal S}(\bR^2)$ and $\Psi\in C^\infty(\Pi \times \Pi, L^*\boxtimes L^*)$ we compute that
\begin{align*} & H(f)(g)((\overline{W}\otimes\overline{W})^{-1}(\Psi)) \\
& \phantom{hhh} = \int_{\bR^3} f(x_3-x_2) e^{2\pi i (x_1-x_2)(x_3-x_2)} g(x_1,x_3) \\
& \phantom{hhhhhhh} \int_{[0,1]^2}  \Psi(x_1-x_2,t,x_2,y) e^{\pi i ((x_1-x_2) t + x_2 y)} dt \ dy\  dx_1\ dx_2\ dx_3\\
& \phantom{hhh} = \int_{\bR^3} f(u) e^{2\pi i su} g(s+x,x+u) \\
& \phantom{hhhhhhh} \int_{[0,1]^2}  \Psi(s,t,x,y) e^{\pi i (s t + x y)} dt \ dy \ ds\ dx\ du\\
& \phantom{hhh} = \int_{[0,1]^5} \sum_{m_1,m_2,m_3 \in \bZ} f(u+m_1) e^{2\pi i (s+m_2)(u+m_1)}\\
& \phantom{hhhhhhh}  g(s+x+m_2+m_3,x+u+m_1+m_3)\Psi(s+m_2,t,x+m_3,y) \\
& \phantom{hhhhhhh} e^{\pi i ((s+m_2) t + (x+m_3) y)} dt \ dy \ ds\ dx\ du\\
& \phantom{hhh} = \int_{[0,1]^5} \sum_{m_1,m_2,m_3 \in \bZ} f(u+m_1) e^{2\pi i (s+m_2-m_3+m_1)(u+m_1)}\\
& \phantom{hhhhhhh}  g(s+x+m_2,x+u+m_3)\Psi(s,t,x,y) \\
& \phantom{hhhhhhh} e^{\pi i ((s+2(m_2-m_3+m_1)) t + (x+2(m_3-m_1)) y)} dt \ dy \ ds\ dx\ du\\
& \phantom{hhh} = \int_{[0,1]^5} e^{\pi i (2su+st+xy)}\sum_{m_1\in \bZ} f(u+m_1) e^{2\pi i m_1(s+u+t-y)} \\
& \phantom{hhhhhhh}  \sum_{m_2,m_3\in \bZ} g(s+x+m_2,x+u+m_3) e^{2\pi i(m_2(u+t) + m_3(y-t-u))}\\
& \phantom{hhhhhhh} \Psi(s,t,x,y) du \ ds \ dt\ dx\ dy\\
& \phantom{hhh} = \int_{[0,1]^5}e^{\pi i u(s+t+u-y)} \sum_{m_1\in \bZ} f(u+m_1) e^{2\pi i m_1(s+u+t-y)} \\
& \phantom{hhhhhhh} e^{\pi i (s+x)(t+u)}e^{\pi i (x+u)(y-t-u)}\\ 
& \phantom{hhhhhhh} \sum_{m_2,m_3\in \bZ} g(s+x+m_2,x+u+m_3) e^{2\pi i(m_2(u+t) + m_3(y-t-u))}\\
& \phantom{hhhhhhh} \Psi(s,t,x,y) du \ ds \ dt\ dx\ dy\\ 
& \phantom{hhh} = \int_{[0,1]^4} ( \tilde\pi_2)_* (\mathsf F^*W(f) \tilde\pi_1^*(W\otimes W)(g))(s,t,x,y) \Psi(s,t,x,y)   ds \ dt\ dx\ dy\\ 
\end{align*}
The seven equalities are justified as follows.
\begin{enumerate}
\renewcommand{\labelenumi}{\textup{(\arabic{enumi})}}
\item The definition of $H(f)$, in which the delta function is used
to integrate out $x_0=x_1-x_2$, and the inverse WGZ transformation
$(\overline W\otimes\overline W)^{-1}\Psi(x_0,x_2)
=\int_{[0,1]^2}\Psi(x_0,t,x_2,y)e^{\pi\Imun(x_0t+x_2y)}\,dt\,dy$
(Proposition~\ref{prop:wgz-iso} for $\overline W$).
\item The change of variables $s=x_1-x_2$, $x=x_2$, $u=x_3-x_2$, of
Jacobian one.
\item Each integral over $\REALS$ is split into integrals over unit
intervals, $\int_{\REALS}ds=\sum_{m_2}\int_0^1 ds$ with
$s\mapsto s+m_2$, and similarly $x\mapsto x+m_3$, $u\mapsto u+m_1$.
\item The lifted representative
$\Psi(s,t,x,y)e^{\pi\Imun(st+xy)}$ of a section of
$L^*\boxtimes L^*$ is $1$-periodic in $t$ and $y$ and satisfies
$\Psi(s+m_2,t,x+m_3,y)e^{\pi\Imun((s+m_2)t+(x+m_3)y)}
=e^{2\pi\Imun(m_2t+m_3y)}\Psi(s,t,x,y)e^{\pi\Imun(st+xy)}$
(Appendix~\ref{WGZ}); then the summation indices are changed to
$m_2+m_3$ and $m_1+m_3$ in the arguments of $g$.
\item The exponentials are collected; the terms
$e^{2\pi\Imun(m_1m_2-m_1m_3+m_1^2)}$ equal one.
\item The identity $2su+st+xy=u(s+t+u-y)+(s+x)(t+u)+(x+u)(y-t-u)$.
\item The series in $m_1$ is $\mathsf F^*W(f)$ and the double series in
$m_2,m_3$ is $\tilde\pi_1^*(W\otimes W)(g)$, by the definition
\eqref{eq:wgz} of $W$, and the integration over $u\in[0,1]$ is the
push-forward $(\tilde\pi_2)_*$.
\end{enumerate}
\end{proof}
Let us now denote the push forward from $\bS^n$ to a point by $\pi^{(n)}_*$.
From Lemma \ref{Leq:vetr}, we get immediately the formula
\begin{eqnarray*}
H(f)(\hat\pi^*(g)\check\pi^*(h)) & = & \pi^{(4)}_*((\tilde\pi_2)_*( \mathsf F^*W(f) \tilde\pi_1^*(W\otimes W)(g))(\overline{W}\otimes\overline{W})(h))\\
& = & \pi^{(5)}_*( \mathsf F^*W(f) \tilde\pi_1^*(W\otimes W)(g)\tilde\pi_2^*(\overline{W}\otimes\overline{W})(h))
\end{eqnarray*}
Let $\mathsf T : \bS^3 \ra \Pi$ be given by $\mathsf T(x,y,z) = (y-x,z-y)$ (the letter $\mathsf T$ is used only in this subsection and is unrelated to the tetrahedra $T$), and let $E : \bS^6 \ra \Pi$ be given by
$$ E(x_{01},x_{02},x_{03},x_{12},x_{13},x_{23})= (x_{02}+x_{13}-x_{03}-x_{12},x_{02}+x_{13}-x_{01}-x_{23}),$$
which is the map $E$ of Proposition~\ref{fiberrel} on a single tetrahedron.
We also introduce the following two maps $\pi_i : \bS^6 \ra \Pi^2$, $i=1,2$, given by
$$\pi_1(x_{01},x_{02},x_{03},x_{12},x_{13},x_{23}) = \mathsf T\times \mathsf T (x_{23},x_{03},x_{02},x_{12},x_{02},x_{01})$$
and
$$\pi_2(x_{01},x_{02},x_{03},x_{12},x_{13},x_{23}) = \mathsf T\times \mathsf T (x_{23},x_{13},x_{12},x_{13},x_{03},x_{01}).$$
Finally we also need the map $P : \bS^6 \ra \bS^5$ given by
\begin{eqnarray*}
& & P(x_{01},x_{02},x_{03},x_{12},x_{13},x_{23}) \\
& &\phantom{mmmmm} = ( x_{02}+x_{13}-x_{03}-x_{12}, x_{13} - x_{23}, x_{12} - x_{13}, x_{03} - x_{13}, x_{01} - x_{03})
\end{eqnarray*}
It is elementary to verify that 
$$ E = \mathsf F \circ P  \text { and } \pi_i = \tilde\pi_i \circ P.$$
We now arrive at the important ``edge-face transformation''.
\begin{lemma}\label{eft}
The distributions $$H(f)\in {\mathcal S}'(\bR^4)$$ and the section $$E^*W(f)\in C^\infty(\bS^6, E^*L)$$ are related by the formula
$$
H(f)(\hat\pi^*(g)\check\pi^*(h)) = \pi^{(6)}_* (E^*W(f)\pi_1^*(W\otimes W)(g) \pi_2^*(\overline{W}\otimes\overline{W})(h))
$$
for all $g,h\in\mathcal S(\bR^2)$.
\end{lemma}
\begin{proof}
The integral map $P$ is surjective: a right inverse at
$(u,s,t,x,y)$ is given by
\[
\begin{gathered}
x_{23}=0,\quad x_{13}=s,\quad x_{12}=s+t,\\
x_{03}=s+x,\quad x_{01}=s+x+y,\quad x_{02}=s+t+x+u.
\end{gathered}
\]
It therefore pushes probability Haar measure on $\mathbb S^6$
to probability Haar measure on $\mathbb S^5$, by uniqueness of
Haar probability. Using this in the middle equality, we compute
\begin{eqnarray*}
& & H(f)(\hat\pi^*(g)\check\pi^*(h))\\
& & =  \pi^{(5)}_*( \mathsf F^*W(f) \tilde\pi_1^*(W\otimes W)(g)\tilde\pi_2^*(\overline{W}\otimes\overline{W})(h))\\
& & =  \pi^{(6)}_*( P^*(\mathsf F^*W(f) \tilde\pi_1^*(W\otimes W)(g)\tilde\pi_2^*(\overline{W}\otimes\overline{W})(h)))\\
& & =  \pi^{(6)}_*( E^*W(f) \pi_1^*(W\otimes W)(g)\pi_2^*(\overline{W}\otimes\overline{W})(h)).
\end{eqnarray*}
\end{proof}

\subsection{Comparison for closed admissible triangulations}
\label{sec:closed-comparison}
Suppose now that $X$ is as in the introduction and we will further assume that $H_2(X', {\mathbb Z})$ vanishes, so as to have the original formulation well defined. By the universal coefficient theorem we have a short exact sequence
$$ 0 \ra \text{Ext}_{\mathbb Z}^1(H_1(X',{\mathbb Z}), {\mathbb S}) \ra H^2(X', {\mathbb S}) \ra \text{Hom}_{\mathbb Z}(H_2(X', {\mathbb Z}), {\mathbb S}) \ra 0.$$
Since $\mathbb S=\mathbb R/\mathbb Z$ is divisible, it is an injective
$\mathbb Z$-module. Hence
$\operatorname{Ext}_{\mathbb Z}^1(H_1(X',\mathbb Z),\mathbb S)=0$.
The assumed vanishing of $H_2(X',\mathbb Z)$ therefore gives
\[
H^2(X',\mathbb S)=0.
\]

\begin{lemma}[A transverse Gaussian contraction]\label{lem:mehler-contraction}
Write $z=(\xi,\eta)\in\mathbb R^r\times\mathbb R^d$ and let
\[
F=\delta(\xi)\psi(\eta),\qquad
G=\gamma e^{\Imun q(z)}\delta(B\xi+C\eta),
\]
where $\psi\in\mathcal S(\mathbb R^d)$, $q$ is a real quadratic
polynomial, $\gamma\in\mathbb C$, and $C:\mathbb R^d\to\mathbb R^k$
is surjective. All delta distributions use their displayed linear
coordinates and the corresponding Lebesgue measures; explicitly,
$\int_{\mathbb R^d}\psi(\eta)\delta(C\eta)\,d\eta
=(\det CC^T)^{-1/2}\int_{\ker C}\psi\,d\sigma$ with $d\sigma$ the
Lebesgue measure induced on $\ker C$ by the Euclidean structure of
$\mathbb R^d$, and the integrals in \eqref{eq:mehler-limit} below are
understood in this sense. Let $Q$ be the
Weyl quantization of a positive definite real quadratic form on
$T^*\mathbb R^{r+d}$, plus a real constant. Then
\begin{equation}\label{eq:mehler-limit}
\lim_{t\downarrow0}\int_{\mathbb R^d}
 \psi(\eta)(e^{-tQ}G)(0,\eta)\,d\eta
=\gamma\int_{\mathbb R^d}
 \psi(\eta)e^{\Imun q(0,\eta)}\delta(C\eta)\,d\eta.
\end{equation}
The limit is uniform when $\psi$ ranges over a bounded subset of
$\mathcal S(\mathbb R^d)$.
\end{lemma}
\begin{proof}
\emph{Step 1: the Gaussian kernel.}
Put $n=r+d$ and use the momentum convention $p=-\Imun\nabla_z$
in this lemma. Write the positive quadratic symbol of $Q-c_0$ as
$h(Z)=Z^THZ$, where $Z=(z,\zeta)$ and
\[
H=\begin{pmatrix}b&d\\d^T&a\end{pmatrix}>0,\qquad
J=\begin{pmatrix}0&I\\-I&0\end{pmatrix}.
\]
The Mehler formula in Weyl form
\cite[Theorem~4.2]{Hormander1995} gives the symbol of $e^{-tQ}$ as
\[
e^{-tc_0}c_t e^{-Z^TB_tZ},\qquad
B_t=-J\tan(tJH),\qquad
c_t=\det(\cos(tJH))^{-1/2},
\]
where the square root is positive and tends to one at $t=0$.
For clarity, a real symplectic diagonalization
$H=S^T\operatorname{diag}(\Lambda,\Lambda)S$,
$\Lambda=\operatorname{diag}(\lambda_1,\ldots,\lambda_n)>0$,
gives
\[
B_t=S^T\operatorname{diag}(\tanh(t\Lambda),\tanh(t\Lambda))S,
\qquad c_t=\prod_j\operatorname{sech}(t\lambda_j).
\]
Thus the formula is also obtained from the one-dimensional
oscillator formula by metaplectic covariance
\cite[Chapter~4]{Folland1989}. In particular
$0<B_t\leq tH$, $B_t=tH+O(t^3)$ and $c_t=1+O(t^2)$.

Write
$B_t=\left(\begin{smallmatrix}U_t&V_t\\V_t^T&W_t\end{smallmatrix}\right)$,
$\mu=(z+z')/2$ and $w=z-z'$. The inverse Fourier transform in
$\zeta$ in the Weyl kernel is an ordinary Gaussian integral, giving
\begin{equation}\label{eq:quadratic-mehler-kernel}
\begin{split}
K_t(z,z')={}&
\frac{e^{-tc_0}c_t}{(4\pi)^{n/2}\sqrt{\det W_t}}\\
&{}\times\exp\left(
-\frac14w^TW_t^{-1}w
-\mu^T(U_t-V_tW_t^{-1}V_t^T)\mu
-\Imun\mu^TV_tW_t^{-1}w\right).
\end{split}
\end{equation}
All matrices are real, $W_t>0$, and the Schur complement
$U_t-V_tW_t^{-1}V_t^T$ is positive definite.

\emph{Step 2: the two kernel estimates.}
Let
\[
k_t(w)=
\frac{\exp(-w^Ta^{-1}w/(4t))}
     {(4\pi t)^{n/2}\sqrt{\det a}}.
\]
We need only the following consequences of
\eqref{eq:quadratic-mehler-kernel}.
\begin{enumerate}
\renewcommand{\labelenumi}{\textup{(\alph{enumi})}}
\item There is a constant $M$ such that
$|K_t(z,z')|\leq M k_t(z-z')$ for $0<t\leq1$.
Indeed, $W_t\leq ta$, the Schur-complement term is nonpositive
in the real exponent, and
$e^{-tc_0}c_t\sqrt{\det(ta)/\det W_t}$ is bounded on $(0,1]$,
with limit one at zero.
\item For $z$ and $w$ in compact sets,
\[
K_t(z,z+\sqrt t\,w)=k_t(\sqrt t\,w)(1+o(1))
\quad\text{uniformly as }t\downarrow0.
\]
In fact $W_t=ta+O(t^3)$, $V_t=td+O(t^3)$ and
$U_t=tb+O(t^3)$. In the ratio of the displayed kernels the
distance exponent has error $O(t^2|w|^2)$, the midpoint term is
$O(t|\mu|^2)$ and the phase is $O(\sqrt t\,|\mu||w|)$;
the prefactor tends to one. These bounds are uniform on compact
sets and prove the assertion.
\end{enumerate}

\emph{Step 3: the tail estimate.}
Let $L=(B,C)$ and $W=\ker L$. The normalized Gaussian identity
\[
\int_{\mathbb R^n}k_t(z-z')\delta(Lz')\,dz'
=\frac{\exp(-(Lz)^T(LaL^T)^{-1}(Lz)/(4t))}
 {(4\pi t)^{k/2}\sqrt{\det(LaL^T)}}
\]
follows by pushing forward the Gaussian by $L$. Thus
the bound (a) of Step~2 gives constants $M,c>0$, independent
of $t\in(0,1]$, such that
\[
|(e^{-tQ}G)(0,\eta)|
\leq M t^{-k/2}\exp(-c|C\eta|^2/t).
\]
The formula has the evident constant interpretation when $k=0$.
The phase $e^{\Imun q(z')}$ has modulus one and does not affect this
estimate. Choose linear coordinates $\eta=R_0u+S_0v$ with $C\eta=u$.
Integrating the Gaussian in $u$ shows that, for every $N\geq0$,
\[
\sup_{0<t\leq1}\int_{|\eta|>R}
 |\psi(\eta)(e^{-tQ}G)(0,\eta)|\,d\eta
\leq M_N p_{N+d+1}(\psi)(1+R)^{-N},
\]
where $p_m(\psi)=\sup_\eta(1+|\eta|)^m|\psi(\eta)|$.
For example, remove $(1+R)^{-N}$ from the Schwartz bound and use
$(1+|\eta|)^{-d-1}\leq M'(1+|v|)^{-d-1}$ before the Gaussian
integration. The limiting integral satisfies the same tail estimate,
since $e^{\Imun q(0,\eta)}\delta(C\eta)$ is a bounded density on
the subspace $\ker C$.

\emph{Step 4: identification of the limit.}
It remains to identify the limit for compactly supported $\psi$.
Set $S=\{\xi=0\}$. The map
\[
S\times W\longrightarrow\mathbb R^n,
\qquad (z,z')\longmapsto z-z'
\]
is surjective precisely because $L|_S=C$ is surjective. Choose linear
coordinates for its kernel and a complementary set of difference
coordinates. The local asymptotic (b) of Step~2, followed
by the change of variables $z-z'=\sqrt t\,w$, reduces the limit to
the normalized Gaussian integral in $w$, and the bound (a)
controls the unbounded $w$-region by dominated convergence.
The normalization can be seen explicitly as follows. For $k>0$,
choose
\[
R_0=C^T(CC^T)^{-1},\qquad
S_0:\mathbb R^{d-k}\longrightarrow\ker C
\quad\text{an orthonormal parametrization}.
\]
The change of variables $\eta=R_0u+S_0v$ has
$C\eta=u$ and absolute Jacobian
\[
J_C=\left|\det(R_0,S_0)\right|
=\bigl(\det(CC^T)\bigr)^{-1/2},
\]
because $R_0^TR_0=(CC^T)^{-1}$, $R_0^TS_0=0$ and $S_0^TS_0=I$.
Put $D=LaL^T>0$ and
\[
h_t(u)=\frac{\exp(-u^TD^{-1}u/(4t))}
                  {(4\pi t)^{k/2}\sqrt{\det D}},
\qquad \int_{\mathbb R^k}h_t(u)\,du=1.
\]
For any compactly supported continuous function $f$, the exact
Gaussian identity above gives
\begin{align*}
&\int_{\mathbb R^d} f(\eta)
  \int_{\mathbb R^n}k_t((0,\eta)-z')\delta(Lz')\,dz'\,d\eta\\
&\qquad=J_C\int_{\mathbb R^{d-k}}\int_{\mathbb R^k}
 f(R_0u+S_0v)h_t(u)\,du\,dv\\
&\qquad\longrightarrow J_C\int_{\mathbb R^{d-k}}f(S_0v)\,dv
 =\int_{\mathbb R^d}f(\eta)\delta(C\eta)\,d\eta.
\end{align*}
This displays the coarea factor and shows that no further
determinant multiplies the stated delta distribution. By (b) and
(a) of Step~2 the same diagonal limit holds with $K_t$ in place of
$k_t$: on the region $|z-z'|\le\sqrt t\,R'$ the ratio
$K_t/k_t$ tends to one uniformly, and the complement contributes
$O(e^{-cR'^2})$ uniformly in $t$ by (a); the phase
$e^{\Imun q(z')}$ is continuous and restricts on the diagonal to
$e^{\Imun q(0,\eta)}$.
Apply the displayed calculation with
$f(\eta)=\psi(\eta)e^{\Imun q(0,\eta)}$ and restore the factor
$\gamma$. When $k=0$, the determinant and zero-dimensional Gaussian
integral are interpreted as one, and the limiting measure is
ordinary $d\eta$. This proves the asserted contraction limit for
compactly supported $\psi$, and the uniform tail estimate of
Step~3 removes the cutoff.

\emph{Step 5: uniformity.} All estimates above depend on $\psi$ only
through finitely many seminorms $p_m(\psi)$ and through the modulus
of continuity of $\psi$ on compact sets, which is controlled by
$\sup|\nabla\psi|$; hence the convergence is uniform on bounded
subsets of $\mathcal S(\mathbb R^d)$.
\end{proof}

\begin{proof}[Proof of the closed-case comparison on the positive locus]
The structure of the proof is the following. Both sides of
\eqref{ON-rel} are contractions of the tetrahedral weights against a
gluing kernel which identifies paired face variables: on the old side
the gluing kernel is a product of delta functions of real face
variables, on the compact side it is the WGZ completeness kernel
\eqref{DDrel} on the torus. We regularize both gluing kernels by the
same positive heat flow, $e^{-\epsilon\mathsf H/2}$ on the old side and
its WGZ conjugate on the compact side. For every $\epsilon>0$ the two
regularized contractions are equal, by the edge--face transformation
(Lemma~\ref{eft}) applied to Schwartz kernels; this is the content of
the paragraphs ``A common regularization'' below. We then let
$\epsilon\downarrow0$ separately on the two sides: on the old side the
limit is the contraction of tempered distributions of
\cite{MR3227503}, by Lemma~\ref{lem:mehler-contraction} after a
metaplectic change of polarization in which admissibility becomes the
transversality hypothesis of that lemma; on the compact side the limit
is the integral \eqref{eq:inter1} of the contracted weight
against the conditional Haar measure, by an approximate-identity
argument on the torus. The equality of the two limits is \eqref{eq:inter1},
and the disintegration \eqref{eq:inter2} gives \eqref{ON-rel}.

Let $Y$ be the disjoint union of the tetrahedra of $X$. For a glued face
$f$, let $f^+$ and $f^-$ be its positively and negatively oriented copies.
We first identify the WGZ completeness kernel as a distribution on
test sections and then justify its insertion in the full contraction.
In two face coordinates the kernel is
\begin{equation}\label{DDrel}
D(x,y;u,v)=\delta_{\mathbb S}(x-u)\delta_{\mathbb S}(y-v)
 e^{\pi\Imun(2yu-xy-uv)}.
\end{equation}
One may verify \eqref{DDrel} directly from \eqref{eq:wgz} by integrating
the real face variable and using Poisson summation. On its support,
write $u=x+m$, $v=y+n$, with $m,n\in\mathbb Z$. Its phase is
\[
e^{\pi\Imun(my-nx-mn)}=\varphi_+((m,n),(x,y))^{-1}.
\]
Thus the phase is exactly the evaluation between the two dual face
line bundles, expressed in these lifts.

\smallskip
\noindent\emph{A common regularization of the face contractions.}
The normalization of the WGZ transform is fixed by Parseval's identity:
for $f,g\in\mathcal S(\mathbb R)$,
\[
\int_{\Pi}(Wf)(x,y)\overline{(Wg)(x,y)}\,dx\,dy
=\int_{\mathbb R}f(t)\overline{g(t)}\,dt.
\]
Indeed, integration in $y$ makes the two summation indices in
\eqref{eq:wgz} equal, and integration in $x\in[0,1]$ then unfolds
the real line. Thus $W$ extends to a unitary map
$L^2(\mathbb R)\to L^2(\Pi,L)$, as well as the isomorphism of
test-function spaces stated in Appendix~\ref{WGZ}. The bilinear
version pairs $Wf$ with $\overline Wg$ and gives $\int f(t)g(t)\,dt$.

On $\mathbb R$ use the positive harmonic oscillator
\[
\mathsf A=1+t^2-\frac{1}{4\pi^2}\frac{d^2}{dt^2}.
\]
For $\epsilon>0$ the Mehler kernel $K_\epsilon(t,t')$ of
$e^{-\epsilon\mathsf A}$ is Schwartz on $\mathbb R^2$. Replace the delta kernel identifying the two old
variables at every glued face by $K_\epsilon$. The resulting
regularized contraction, including the level factor, is denoted
$Z_{\hbar,\epsilon}(X)$. At positive $\epsilon$ it is a pairing
of the tensor product of the old tempered tetrahedral kernels with
a Schwartz test kernel on all face variables; it is therefore
well defined before making any change of representation.

The corresponding regularization on the torus is exact, rather
than an independently chosen approximation. Direct differentiation
of \eqref{eq:wgz} gives
\[
W(t f)=\left(\frac{x}{2}+
                 \frac{1}{2\pi\Imun}\partial_y\right)Wf,
\qquad
W\left(\frac{1}{2\pi\Imun}f'\right)
=\left(-\frac{y}{2}+
                 \frac{1}{2\pi\Imun}\partial_x\right)Wf.
\]
Consequently $\mathsf A_L=W\mathsf A W^{-1}$ is the positive
elliptic operator on $L$ whose lifted expression is
\[
\mathsf A_L=1+
\left(\frac{x}{2}+\frac{1}{2\pi\Imun}\partial_y\right)^2+
\left(-\frac{y}{2}+\frac{1}{2\pi\Imun}\partial_x\right)^2.
\]
The covariant derivatives in this formula respect the lattice
multipliers of $L$. Let $D_\epsilon$ be the smooth kernel of
$e^{-\epsilon\mathsf A_L}$, with the dual kernel used for the
opposite face orientation. It is precisely the WGZ transform of
$K_\epsilon$ on its two arguments, and tends to \eqref{DDrel}
as the kernel of the identity. Both kernels use the measures in
the displayed Parseval identity.

Let $\mathcal A_X=\mathbb S^{\Delta_0(\tilde Y)}$, and write
$\lambda_f^\pm:\mathcal A_X\to\Pi$ for the two face-coordinate
maps. Write $B_Y$ for the tensor product of the uncontracted
tetrahedral weights on this ambient torus. Applying
Lemma~\ref{eft} and its conjugate to the regularized pairing gives
\[
Z_{\hbar,\epsilon}(X)
=e^{\Imun\pi\ell_X/(4\hbar)}
\int_{\mathcal A_X}
\operatorname{ev}\left(
B_Y\otimes\bigotimes_{f\in\Delta_2(X)}
(\lambda_f^+,\lambda_f^-)^*D_\epsilon\right)\,dy.
\]
Here the kernels are oriented so that $\operatorname{ev}$ contracts
the corresponding dual line factors. For justification, first use
finite sums of tensor products of Schwartz test functions in the
regularized kernel. The local identity applies to every summand.
Such sums are dense in the Schwartz space on all face variables;
continuity of the tempered kernels and of the WGZ isomorphism then
gives the displayed equality for $K_\epsilon$. This step involves
only smooth compact kernels on the transformed side, and no product
of unregularized distributions.

\smallskip
\noindent\emph{Removal of the regularization on the old side.}
Let $G$ denote the product of the delta kernels which identify paired
old face variables, and let $\mathsf H$ be the sum of the operators
$\mathsf A$ over all face copies of the disjoint tetrahedra. The
semigroup property and symmetry of the Mehler kernel give, for each
paired face,
\[
e^{-\epsilon(\mathsf A_t+\mathsf A_{t'})/2}
\delta(t-t')=K_\epsilon(t,t').
\]
Thus the full regularized gluing kernel is exactly
$G_\epsilon=e^{-\epsilon\mathsf H/2}G$.

For the tensor product of the uncontracted tetrahedral distributions
we use a real linear polarization of the kind employed in
\cite[Section~10, proof of Theorem~7]{MR3227503}, which we now specify.
On each tetrahedron choose the coordinates
\[
r=x_0+x_2-x_1,\quad s=x_0,\quad v=x_2,\quad u=x_3-x_2.
\]
Their Jacobian has absolute value one, and the kernel becomes
$\delta(r)f(u)e^{2\pi\Imun\varepsilon su}$. Remove the quadratic
phase and Fourier transform in $s,v$; this gives
$\delta(r)\delta(\widehat s)\delta(\widehat v)f(u)$.
The function $f=\tilde\psi'_{a,c}$, or its conjugate, is Schwartz
by Lemma~\ref{lem:weight-regularity} and its proof. These operations
are metaplectic. Let $U$ be their tensor product over all tetrahedra,
and let $\xi$ collect the variables $(r,\widehat s,\widehat v)$ and
$\eta$ the variables $u$ of all tetrahedra. Then $U$ transforms the
tensor product of the tetrahedral distributions into
$\delta(\xi)\psi(\eta)$ with $\psi\in\mathcal S$.

With this $U$ on the tetrahedral factor, apply
$V=(U^{-1})^t=\overline U$ to the gluing factor, so that the
bilinear contraction is preserved; the superscript $t$ denotes
transpose for that bilinear pairing. The operator $V$ is again
metaplectic. The gluing factor $G_0=\prod_f\delta(x_{f^+}-x_{f^-})$
is the delta distribution of a linear subspace $W_0$, that is, the
Gaussian distribution $e^{-\frac12\langle z,0z\rangle}\delta(Lz)$ with
vanishing quadratic form; such distributions are the ``degenerate
Gaussians'' of \cite[Chapter~5]{MR2790054}. The elementary
metaplectic operations preserve the class
$\gamma e^{\Imun q(z)}\delta(L'z)$, with $q$ real quadratic,
$\gamma\neq0$ and $L'$ a surjective real linear map. This is
immediate for real linear changes of variables and multiplication
by real quadratic phases. For a partial Fourier transform, first
solve the delta constraints in the variables being transformed,
retaining any resulting independent constraints on the other
variables. The remaining integral is an oscillatory integral of a
real quadratic form. If its real symmetric matrix is $A$, split the
integration variables into $\ker A$ and its orthogonal complement,
on which the restriction $A_1$ is invertible. Writing the linear
term as $\langle b,t\rangle$, and $b=b_0+b_1$ in this splitting,
the Fresnel formula gives, as an identity of tempered distributions,
\[
\int e^{\Imun(\frac12\langle At,t\rangle+\langle b,t\rangle)}\,dt
=c_A e^{-\frac{\Imun}{2}\langle A_1^{-1}b_1,b_1\rangle}
\delta(b_0),\qquad c_A\neq0.
\]
The scalar $c_A$ includes the determinant, signature and Fourier
normalization factors. The null directions give independent delta
constraints in the dual variables, and the remaining phase is real
quadratic. Thus the partial Fourier transform preserves the stated
class as well; see also \cite[Chapter~4]{Folland1989} for the
metaplectic action on Gaussians. Applying these operations in $V$
therefore gives
\[
VG_0=\gamma e^{\Imun q(\xi,\eta)}\delta(B\xi+C\eta)
\]
with a real quadratic form $q$ and independent delta
constraints, that is, $(B,C)$ of full row rank.
We verify the required transversality in this polarization.
On a tetrahedron $T$ of sign $\varepsilon_T$, put
$q_i=x_i$ and $p_i=(2\pi\Imun)^{-1}\partial_{x_i}$,
the same momentum convention as in \cite[Section~9.1]{MR3227503}.
In the operator--kernel correspondence, right composition by $p_i$
becomes $-p_i$ on the input variable, by integration by parts.
The three multiplication operators $r,\widehat s,\widehat v$ in
the new polarization are, up to sign, the $U$-transforms of
\[
R_T=q_0+q_2-q_1,\qquad
S_T=p_0+p_1-\varepsilon_T(q_3-q_2),\qquad
V_T=p_1+p_2+p_3.
\]
With that paper's four variables ordered as $(x_0,x_2,x_1,x_3)$,
its constraints (29), (30), (31) become, respectively,
$R_T,V_T,S_T$ for $\varepsilon_T=1$. Complex conjugation fixes
$q_i$ and sends $p_i$ to $-p_i$. It therefore gives $R_T,-V_T,-S_T$
with $\varepsilon_T=-1$, so the same three constraints annihilate
the negative tetrahedral kernel. These observations also follow
directly by applying the operators to
$\delta(x_0+x_2-x_1)f(x_3-x_2)
e^{2\pi\Imun\varepsilon_Tx_0(x_3-x_2)}$.
The following argument identifies their possible gluing obstruction
directly with the transpose of $D_X$.

Let $D_{X,\mathbb R}$ be the real linear map obtained by lifting the
integral linear formula for $D_X$. On an ordered face $f=(i,j,k)$,
write
\[
s_f=y_{ik}-y_{jk},\qquad t_f=y_{ij}-y_{ik}.
\]
Choose the two midpoint edges of each glued face $f$ so that $s_f$
and $t_f$ are the corresponding differences of $y$. Then the component
of $D_{X,\mathbb R}$ at $f$ is $(s_{f^+}-s_{f^-},t_{f^+}-t_{f^-})$, where
$f^+$ and $f^-$ are the two copies of $f$.
For
$L_T=\lambda_T R_T+\mu_T S_T+\nu_T V_T
     =\sum_{i=0}^3(a_{T,i}q_i+b_{T,i}p_i)$,
the coefficients are
\[
\begin{aligned}
(a_{T,0},a_{T,1},a_{T,2},a_{T,3})
 &=(\lambda_T,-\lambda_T,
       \lambda_T+\varepsilon_T\mu_T,-\varepsilon_T\mu_T),\\
(b_{T,0},b_{T,1},b_{T,2},b_{T,3})
 &=(\mu_T,\mu_T+\nu_T,\nu_T,\nu_T).
\end{aligned}
\]
Set $\tau_{T,i}=\varepsilon_T(-1)^i$, the induced orientation sign
of $\partial_iT$. Direct substitution gives the local identity
\begin{equation}\label{eq:polarization-cochain-identity}
\sum_{i=0}^3
 \bigl(a_{T,i}s_{\partial_iT}
       -\tau_{T,i}b_{T,i}t_{\partial_iT}\bigr)=0
\end{equation}
for all six midpoint variables of $T$; indeed, the coefficient of
each $y_{ij}$ cancels.

Suppose a multiplication operator $m(\xi)$, linear in $\xi$ alone,
annihilates $VG$. Its inverse transform on the tetrahedral factor
is $L=U^{-1}m(\xi)U=\sum_T L_T$.
Since $V=(U^{-1})^t$, the equation $m(\xi)VG=0$ is equivalent to
$L^tG=0$. For a glued face $f$ whose copies are
$f^\pm=\partial_{i_\pm}T_\pm$, write $x_{f^\pm}$ for the corresponding
face variables and
\[
a_{f^\pm}=a_{T_\pm,i_\pm},\qquad b_{f^\pm}=b_{T_\pm,i_\pm},\qquad
\tau_{f^\pm}=\tau_{T_\pm,i_\pm}=\pm1 .
\]
Multiplication and differentiation of $\delta(x_{f^+}-x_{f^-})$ show
that $L^tG=0$ is equivalent to
\[
a_{f^+}+a_{f^-}=0,\qquad b_{f^+}-b_{f^-}=0
\]
at every face; here $q_i^t=q_i$ and $p_i^t=-p_i$. The local covectors
$(a_{f^\pm},-\tau_{f^\pm}b_{f^\pm})$ therefore have opposite values on
the two copies of each glued face. Let $w$ be the covector on the
target of $D_{X,\mathbb R}$ whose component at $f$ is
$(a_{f^+},-b_{f^+})$. Summing
\eqref{eq:polarization-cochain-identity} over the tetrahedra gives
$D_{X,\mathbb R}^t w=0$.
By admissibility and Proposition~\ref{H2}, $D_X$ is surjective;
its differential $D_{X,\mathbb R}$ is consequently surjective as
well. Hence $w=0$. As $X$ is closed, every face of every tetrahedron
is one of the copies $f^\pm$, so all the coefficients
$a_{T,i},b_{T,i}$ vanish, and $L=0$. Thus no nonzero multiplication
operator linear in $\xi$ alone annihilates $VG$.

If $C$ failed to have full row rank, a nonzero row vector $z$ with
$zC=0$ would give just such an annihilator
$m(\xi)=zB\xi$. It is nonzero because the rows of $(B,C)$ are
independent. This proves that $C$ has full row rank. Finally,
the transformed oscillator $V\mathsf H V^{-1}$ is the Weyl
quantization of a positive definite real quadratic form, plus a
constant. Thus all hypotheses of
Lemma~\ref{lem:mehler-contraction} hold. It gives
\[
\lim_{\epsilon\downarrow0}Z_{\hbar,\epsilon}(X)=Z_\hbar(X).
\]
Here we use that the restriction of $VG$ to $\xi=0$, which is defined
because $C$ has full row rank, is the kinematical kernel. Indeed, by
the construction of $U$, the Schwartz kernel of $U^{-1}$ is the product
over the tetrahedra $T$ of
\[
\delta(r-x_0-x_2+x_1)\,\delta(u-x_3+x_2)\,
e^{2\pi\Imun(\varepsilon_Tx_0u\pm x_0\widehat s\pm x_2\widehat v)},
\]
where $x_i=x(\partial_iT)$, the variables $r,\widehat s,\widehat v,u$
are those of $T$, and the signs depend on the convention for the
Fourier transform. At $\xi=0$ this factor is
$M(\varepsilon_T,x|_{\Delta_2(T)},\eta_T)$. As $V=(U^{-1})^t$, the
distribution $VG$ is the integral of the kernel of $U^{-1}$ against
$G$, that is, over the face variables with the two copies of each
glued face identified. Thus $(VG)(0,\eta)=K(X,\eta)$ by
\eqref{k(x,y)}, and the limit given by the lemma, multiplied by the
level factor contained in $Z_{\hbar,\epsilon}(X)$, is the right-hand
side of \eqref{z(x)}, that is, the contraction of the old theory. The
uniform tail estimate in
Lemma~\ref{lem:mehler-contraction} controls the noncompact tails.

\smallskip
\noindent\emph{Removal of the regularization on the compact side.}
Let $b_X$ denote the product of the tetrahedral weights on $\Omega(X)$,
contracted by Proposition~\ref{prop:boundary-bundle}. As $\partial X$
is empty, $b_X$ is a scalar smooth function for positive shapes.
The delta factors in the product of \eqref{DDrel} impose precisely
\begin{equation}\label{eq:ysys}
y(\partial_{j+1}f^-)-y(\partial_jf^-)
=y(\partial_{j+1}f^+)-y(\partial_jf^+),\qquad j=0,1,
\end{equation}
for each $f\in\Delta_2(X)$. These are the two independent coordinates
per face of $D_X$ defined in Section~\ref{TP}.
Since $H^2(X',\mathbb S)=0$, Proposition~\ref{H2} makes $D_X$
surjective. Equivalently, the joint map of paired face coordinates
is transverse to the product of their diagonals: its normal
component is exactly the differential of $D_X$. The heat kernels
$D_\epsilon$ may therefore be pulled back jointly and their limit
taken after multiplication by the smooth section $B_Y$. The kernel
is explicit. The operator
$\mathsf A_L-1=\frac1{4\pi^2}\bigl((-\Imun\partial_y+\pi x)^2
+(-\Imun\partial_x-\pi y)^2\bigr)$ is a Landau operator with magnetic
field $2\pi$. To compute its kernel in this gauge, put
$h=\epsilon/(2\pi)$. The one-dimensional oscillator Mehler formula is
\[
K_\epsilon(s,t)=\frac{e^{-\epsilon}}{\sqrt{\sinh(2h)}}
\exp\left(-\frac\pi2\coth(h)(s-t)^2
          -\frac\pi2\tanh(h)(s+t)^2\right).
\]
Applying $W$ in the first variable and its conjugate in the second
therefore gives
\[
D_\epsilon(x,y;u,v)=e^{\pi\Imun(xy-uv)}
\sum_{j,k\in\mathbb Z}e^{2\pi\Imun(jy-kv)}
K_\epsilon(x+j,u+k).
\]
Set $j=k+m$. The remaining $k$-sum is evaluated by the Gaussian
Poisson identity
\[
\sum_{k\in\mathbb Z}e^{-2\pi\tau(k+c)^2+2\pi\Imun k d}
=\frac1{\sqrt{2\tau}}\sum_{n\in\mathbb Z}
 e^{-\pi(d+n)^2/(2\tau)-2\pi\Imun c(d+n)}
\]
with $\tau=\tanh h$, $c=(x+m+u)/2$ and $d=y-v$.
Since $\sinh(2h)\,2\tanh h=4\sinh^2h$, this yields, on lifts,
\begin{equation}\label{eq:landau-kernel}
\begin{aligned}
D_\epsilon(x,y;u,v)={}&e^{\pi\Imun(xy-uv)}\sum_{m,n\in\mathbb Z}
e^{2\pi\Imun my-\pi\Imun(x+m+u)(y+n-v)}\\
&\qquad\times
\vartheta_\epsilon\bigl((x+m-u)^2+(y+n-v)^2\bigr),\\
\vartheta_\epsilon(\varrho^2)={}&\frac{e^{-\epsilon}}{2\sinh(\epsilon/2\pi)}
\exp\Bigl(-\frac{\pi}{2}\coth\Bigl(\frac{\epsilon}{2\pi}\Bigr)
\varrho^2\Bigr).
\end{aligned}
\end{equation}
The modulus of each term is a Gaussian in $(u,v)$ of total mass
$e^{-\epsilon}/\cosh(\epsilon/2\pi)$, which tends to one. It is
concentrated at distance $O(\sqrt\epsilon)$ from the point
$(u,v)=(x+m,y+n)$, and at this point the phase of the term is
$e^{\pi\Imun(my-nx-mn)}$, the phase of \eqref{DDrel}.

We now carry out the limit. Let $\varrho_f(y)$ denote the flat
distance between the two face points $\lambda^+_f(y)$ and
$\lambda^-_f(y)$, equivalently the distance from $D_X(y)_f$ to zero.
For $\varrho_f(y)<\tfrac12$ there is a unique shortest lattice
representative; call the corresponding summand in
\eqref{eq:landau-kernel} the principal term. Every other summand
has displacement of norm at least $\tfrac12$. Summing their
Gaussian bounds gives $O(\epsilon^{-1}e^{-c_0/\epsilon})$, hence
$O(e^{-c/\epsilon})$ after decreasing $c>0$. The phase of the
principal term differs from its evaluation phase at the matching
point by $1+O(\varrho_f)$, uniformly in bounded local lifts.
Indeed, the phase exponent in \eqref{eq:landau-kernel} is a
polynomial in the lifted coordinates; on a bounded lift chart only
finitely many nearest lattice pairs $(m,n)$ occur, so its first
derivatives are uniformly bounded. The mean value estimate gives
the stated error. Consequently, on $\varrho_f<\tfrac12$,
\begin{equation}\label{eq:kernel-split}
D_\epsilon(\lambda^+_f(y),\lambda^-_f(y))
=\Phi_f(y)\,\vartheta_\epsilon(\varrho_f(y)^2)
\bigl(1+O(\varrho_f(y))\bigr)+O(e^{-c/\epsilon}),
\end{equation}
where $\Phi_f$ is the smooth local evaluation pairing, transported
from the matching point along the shortest path. On the matching
locus it is exactly the pairing of
Proposition~\ref{prop:boundary-bundle}, and hence
$\operatorname{ev}(B_Y\otimes\bigotimes_f\Phi_f)=b_X$ on
$\Omega(X)$.

Put $N_f=|\Delta_2(X)|$ and
\[
\tilde\rho_\epsilon(r)=\prod_f\vartheta_\epsilon(|r_f|^2),
\qquad r=(r_f)_f\in\mathfrak C\cong\mathbb S^{D^i},
\]
where $|r_f|$ is flat torus distance to zero. This is a nonnegative
function of total mass
\[
\int_{\mathfrak C}\tilde\rho_\epsilon(r)\,dr
=\left(\frac{e^{-\epsilon}}{\cosh(\epsilon/2\pi)}\right)^{N_f}
+O(e^{-c/\epsilon})\longrightarrow1.
\]
In local normal coordinates of dimension $2N_f$ the Gaussian bound
is $C\epsilon^{-N_f}e^{-c|r|^2/\epsilon}$. Substitution
$r=\sqrt\epsilon\,v$ gives the first-moment estimate
\[
\int_{\mathfrak C}\max_f|r_f|\,
\tilde\rho_\epsilon(r)\,dr=O(\sqrt\epsilon).
\]
The mass outside any fixed neighbourhood of zero is exponentially
small. Thus $\tilde\rho_\epsilon$ is an approximate identity.

Choose a smooth cutoff $\chi$ on $\mathfrak C$, equal to one near
zero and supported where $|r_f|<\delta$ for every face, with
$0<\delta<\tfrac12$. The function
\[
a(y)=\chi(D_Xy)\,
\operatorname{ev}\bigl(B_Y\otimes\bigotimes_f\Phi_f\bigr)(y)
\]
extends by zero to a smooth function on all of $\mathcal A_X$,
and $a|_{\Omega(X)}=b_X$. Outside the region where the cutoff is
one, at least one face distance is bounded away from zero, giving
an exponentially small contribution. In the product expansion of
\eqref{eq:kernel-split}, each principal factor is
$O(\epsilon^{-1})$; a term with at least one nonprincipal factor
is therefore bounded by
$O(\epsilon^{-(N_f-1)}e^{-c/\epsilon})$. Boundedness of $B_Y$
and compactness then show that the total remainder is still
$O(e^{-c'/\epsilon})$ for a smaller $c'>0$.

Since $H^2(X',\mathbb S)=0$, Proposition~\ref{H2} makes
$D_X\colon\mathcal A_X\to\mathfrak C$ surjective. Lemma~\ref{lem:haar}
gives
\[
\int_{\mathcal A_X}a(y)\tilde\rho_\epsilon(D_Xy)\,dy
=\int_{\mathfrak C}\tilde\rho_\epsilon(r)
\left(\int_{D_X^{-1}(r)}a(y)\,d\nu_r(y)\right)dr
\longrightarrow\int_{\Omega(X)}b_X(y)\,d\nu_0(y),
\]
because the inner integral is smooth in $r$ and
$\tilde\rho_\epsilon$ is an approximate identity. The multiplicative
error $O(\max_f\varrho_f)$ in the product of the principal terms
has integral $O(\sqrt\epsilon)$ by the first-moment estimate and
the same disintegration. It therefore tends to zero. The compact-side
regularized contraction consequently converges to
$\int_{\Omega(X)}b_X\,d\nu_0$, where $\nu_0$ is the conditional
probability Haar measure on $\Omega(X)=D_X^{-1}(0)$, with equal
masses for its components. This is the precise meaning of
``the contraction of $B_Y$ against $D_X^*\delta_0$''.
Surjectivity of $D_X$, that is, admissibility, enters through this
disintegration; without it, the displayed approximate-identity
argument on the full target does not apply.
The old-side and compact-side
limits of the same regularized identity now yield
\begin{equation}\label{eq:inter1}
Z_\hbar(X)=e^{\Imun\pi\ell_X/(4\hbar)}
\int_{\Omega(X)}b_X(y)\,d\nu(y).
\end{equation}
The submersion argument justifies the compact delta pullback and
its normalization; the transversality argument on the old side
justifies its equality with the original noncompact contraction.

Disintegrating once more along the surjective homomorphism $H_X$, with
the probability measures of Lemma~\ref{lem:haar}, yields
\begin{align}
Z_\hbar(X)
&=\int_{H^1(X',\mathbb S)}\mathcal F_\hbar(X)(\alpha)\,d\alpha\nonumber\\
&=\frac{1}{|\mathcal T_X|}\sum_{\tau\in\mathcal T_X}
\mathcal F^T_\hbar(X)(\tau).\label{eq:inter2}
\end{align}
Indeed, all $|\mathcal T_X|$ connected components have equal Haar mass,
and their conditional integrals are exactly \eqref{eq:torsion-average}.
A graph $\Gamma$ affects the allowed
shapes but none of this measure calculation. The holomorphic-germ
assertion is Lemma~\ref{lem:holomorphic-integral}. The finite average
on the right is globally meromorphic by
Theorem~\ref{thm:global-continuation}; its equality with the original
partition function on positive shapes proves Theorems~\ref{RelationON}
and~\ref{OldInv}, including their global continuation assertions.
The same finite average transfers branching and shaped-move invariance
from Theorem~\ref{mainT}.
\end{proof}

\subsection{Integral homology spheres}\label{sec:homology-sphere-comparison}
\begin{proof}[Proof of Theorem \ref{integral-homology-sphere-comparison}]
In this case $\mathcal T_X$ is trivial and $H^1(X',\mathbb S)=0$, so
\eqref{eq:inter2} already gives the result on its domain of definition.
The following direct check of the normalization is also useful. Let
$\operatorname p:\Delta_*(Y)\to\Delta_*(X)$ be the cell projection, and
let $\operatorname q:\Delta_1(Y)\to\Delta_3(Y)$ assign to an edge its
tetrahedron.
Let us fix a section of $\operatorname{p}$
\begin{equation}
\Delta_1(X)\ni e\mapsto \tilde e\in\Delta_1(Y).
\end{equation}
Since $H^1(X',\mathbb S)=0$, Proposition~\ref{Omega} shows that
$\Omega(X)$ is a single orbit of $\mathbb S(X)$: every solution of the
face constraints \eqref{eq:ysys} is of the form
$y=\operatorname q^*c+\operatorname p^*x$ with
$c\in\mathbb S^{\Delta_3(X)}$ and $x\in\mathbb S^{\Delta_1(X)}$, and
the contracted product of the tetrahedral weights does not depend on
$c$. Insert
\[
1=\int_{(\REALS/\INTEGERS)^{\Delta_1(X)}}\operatorname{d}\!x
\prod_{e\in\Delta_1(X)}\delta_{\REALS/\INTEGERS}(y(\tilde e)-x(e))
\]
into the contraction of the weights against the delta functions of the
face constraints, and translate $y$ by $\operatorname p^*x$. On the
support of the delta functions the weights are then evaluated at a
translate of $\operatorname p^*x$ by a global edge state and by
tetrahedron constants. Hence the integration over $x$ gives the
integral over all global edge states, that is,
$e^{-\Imun\pi\ell_X/(4\hbar)}\mathcal F_\hbar(X)(0)$, independently
of $y$. Together with the level factor, the delta-function
contraction is therefore equal to
$J(X)\,\mathcal F_\hbar(X)(0)$, where
\begin{multline}\label{eq:jx}
J(X):=\int_{(\REALS/\INTEGERS)^{\Delta_1(Y)}}\operatorname{d}\! y\prod_{e\in\Delta_1(X)}
\delta_{\REALS/\INTEGERS}(y(\tilde e))\\
\times\prod_{(f,j)\in\Delta_2(X)\times\{0,1\}}
\delta_{\REALS/\INTEGERS}(y(\partial_{j+1} f^-)-y(\partial_{j} f^-)-y(\partial_{j+1} f^+)+y(\partial_{j} f^+)).
\end{multline}
Thus, in order to conclude the equality of the partition functions in the two formulations, it remains to prove the equality
$J(X)=1$.
To this end, we first rewrite \eqref{eq:jx} by introducing extra integrations
\begin{multline}
 J(X)=\int_{(\REALS/\INTEGERS)^{\Delta_1(Y)}\times(\REALS/\INTEGERS)^{\Delta_2(X)}}\operatorname{d}(y,\xi)\prod_{e\in\Delta_1(X)}
\delta_{\REALS/\INTEGERS}(y(\tilde e))\\
\times\prod_{(f,j)\in\Delta_2(X)\times\{0,1,2\}}
\delta_{\REALS/\INTEGERS}(y(\partial_{j} f^+)-y(\partial_{j} f^-)+\xi(f)),
\end{multline}
and then, by  using the Poisson summation formula for the periodic delta-function
\begin{equation}
\delta_{\REALS/\INTEGERS}(x)=\sum_{m\in\INTEGERS}e^{2\pi\Imun xm},
\end{equation}
we obtain the expression
\begin{multline}
 J(X)=\int_{(\REALS/\INTEGERS)^{\Delta_1(Y)}\times(\REALS/\INTEGERS)^{\Delta_2(X)}}\operatorname{d}(y,\xi)\sum_{m\in \INTEGERS^{\Delta_1(X)}}\prod_{e\in\Delta_1(X)}
e^{2\pi\Imun 
y(\tilde e)m(e)}\\
\times
\sum_{n\in\INTEGERS^{\Delta_2(X)\times\{0,1,2\}}}\prod_{(f,j)\in\Delta_2(X)\times\{0,1,2\}}
e^{2\pi\Imun
(y(\partial_{j} f^+)-y(\partial_{j} f^-)+\xi(f))n(f,j)}.
\end{multline}
Interpreting these Fourier series distributionally (or first using their Fej\'er regularisations), we perform the integrations by using the formula
\begin{equation}
\int_{\REALS/\INTEGERS}e^{2\pi\Imun x k}\operatorname{d}\! x=\delta_{\INTEGERS}\left(k\right),\quad \forall k\in\INTEGERS,
\end{equation}
and  the summation over 
$m\in  \INTEGERS^{\Delta_1(X)}$. The result reads
\begin{multline}\label{eq:ff1}
 J(X)=\sum_{n\in \INTEGERS^{\Delta_2(X)\times\{0,1,2\}}}\prod_{f\in\Delta_2(X)}\delta_{\INTEGERS}\left(\sum_{0\le j\le 2}n(f,j)\right)\\
\times
\prod_{\substack{e\in\Delta_1(Y)\\ e\ne \widetilde{\operatorname{p}(e)}}}\delta_{\INTEGERS}\left(\sum_{\substack{(g,k,\epsilon)\in\Delta_2(X)\times\{0,1,2\}\times\{\pm\}\\
\partial_kg^\epsilon=e}}\epsilon n(g,k)\right).
\end{multline}
The sum in the argument of the delta-function in the second line of \eqref{eq:ff1} always consists of only two terms corresponding to the two faces of the tetrahedron $\operatorname{q}(e)$ which share the edge $e$. Else, in the product we can remove the condition given by the inequality $e\ne \widetilde{\operatorname{p}(e)}$ without changing the result as the effect of this is the multiplication of the summand by the neutral values $\delta_{\INTEGERS}\left(0\right)=1$. Thus, the final expression takes the form
\begin{multline}\label{eq:ff}
 J(X)=\sum_{n\in \INTEGERS^{\Delta_2(X)\times\{0,1,2\}}}\prod_{f\in\Delta_2(X)}\delta_{\INTEGERS}\left(\sum_{0\le j\le 2}(-1)^jn(f,j)\right)\\
\times
\prod_{t\in\Delta_3(X)}\prod_{0\le k<l\le3}\delta_{\INTEGERS}\left(n(\partial_lt,k)-n(\partial_kt,l-1)\right)
\end{multline}
where we have used the canonical bijection $\Delta_3(X)\simeq\Delta_3(Y)$, the commutativity of the boundary operations $\partial_i$ with the projection $\operatorname{p}$, and we have also made the following change of the summation variables: 
\begin{equation}
n(f,j)\mapsto (-1)^jn(f,j),\quad \forall (f,j)\in \Delta_2(X)\times\{0,1,2\}.
\end{equation}
Now, we observe  that the non-trivial contributions in the sum~\eqref{eq:ff} are in bijection with the elements of the cohomology group $H^1(X;\INTEGERS)$. Indeed, if a map $n\in \INTEGERS^{\Delta_2(X)\times\{0,1,2\}}$ satisfies the equations
\begin{equation}
n(\operatorname{p}(\partial_lt),k)=n(\operatorname{p}(\partial_kt),l-1),\quad t\in \Delta_3(Y),
\end{equation}
then there exists a unique map $\eta\in \INTEGERS^{\Delta_1(X)}$ such that
\begin{equation}\label{eq:net}
n(f,j)=\eta(\partial_j f),\quad \forall (f,j)\in \Delta_2(X)\times\{0,1,2\}.
\end{equation}
This follows directly from the fact that the boundary maps satisfy the standard relations under composition
\begin{equation}
\partial_k\circ\partial_l=\partial_{l-1}\circ\partial_k\quad \text{if}\quad k<l.
\end{equation}
Furthermore, the equations
\begin{equation}
\sum_{0\le j\le 2}(-1)^jn(f,j)=0,\quad \forall f\in\Delta_2(X),
\end{equation}
under the substitution~\eqref{eq:net} become equivalent to the co-cycle condition for $\eta$ when it is interpreted as a co-chain in $C^1(X;\INTEGERS)$. Under the one-vertex hypothesis in Theorem~\ref{integral-homology-sphere-comparison}, there are no nonzero coboundaries, so cocycles are in bijection with their cohomology classes. Finally, the triviality of $H^1(X;\INTEGERS)$ implies that the only non-trivial contribution in the sum~\eqref{eq:ff} comes from the only term with $n=0$. Thus, $J(X)=1$.

Note that $J(X)=1$ only expresses the independence of the constraints:
it holds for every one-vertex triangulation with
$H^1(X;\INTEGERS)=0$, for instance also for $\mathbb{R}P^3$. It is the
fact that $\Omega(X)$ is a single orbit of $\mathbb S(X)$, that is,
the vanishing of $H^1(X',\mathbb S)$, which identifies the contraction
with the trivial sector alone; compare Section~\ref{sec:rp3-example}.
\end{proof}

\section{Examples}\label{sec:examples}\label{sec:example-conventions}
The following nine examples illustrate the dependence of
$\mathcal F_\hbar(X)(\alpha)$ on the cohomology sector $\alpha$.
For $\mathbb{R}P^3$ there are two sectors, of which the trivial one
vanishes identically and the nontrivial one is twice the original
partition function. For the lens space $L(3,1)$ all three sectors are
nonzero multiples of the original partition function; the two
nontrivial sectors are equal and differ from the trivial one by a
cube root of unity. For $L(4,1)$ two of the four sectors vanish and
the other two differ by the factor $\Imun$. For $S^2\times S^1$ the
sectors form a circle on which $\mathcal F_\hbar(X)$ is constant and
equal, up to an explicit exponential factor, to a product of three
quantum dilogarithms, although the original partition function is not
defined. In these four closed-manifold examples, the dependence on
$\alpha$ is governed by Theorem~\ref{thm:sector-translation}. For the
three rational homology spheres, Theorem~\ref{RelationON} compares
the original partition function with the normalized average over
the torsion sectors.
The figure-eight example and the one-edge pseudo-manifold are
cases in which no nontrivial class extends over the vertex, so that
the translation theorem gives no relation between the sectors. For the complement of the
figure-eight knot the sector function is, for generic shapes, a
nonconstant function on a circle, and its Fourier coefficients are
the values of the analytically continued original partition function
at imaginary shifts of the shapes and of the level. For a
pseudo-manifold with a single edge the sector function is a product
of two tetrahedral weights on a two-torus, with no integration at all.

The trefoil exterior provides a second knot-complement example,
whose sectors form the periodization of a Schwartz function and whose
mean recovers the original partition function. The one-tetrahedron
Hopf-link presentation distinguishes the exterior of the loop edges
from the complement of the vertices: its compact amplitude is defined
although the original contraction is not. Finally, a one-vertex
triangulation of $S^3$ with an unknot--trefoil wedge has a single
sector, which agrees with the original partition function.

As in Section~\ref{sec:setting}, the italic $b_i$ in a charge
subscript such as $\psi_{b_i,a_i}$ is the angle
$b_i=1/2-a_i-c_i$; the fixed quantum parameter is the sans-serif
$\la$. We use the following common conventions for all nine examples.
Each example begins with a diagram in which the tetrahedra
$T_1,T_2,\ldots$ are numbered from left to right, each occupying four
consecutive points of the horizontal line; the tetrahedra are
labelled below the diagram and the points are labelled $0,1,2,3$.
In the upper half of the diagram these
points stand, from left to right, for the faces $0,1,2,3$ of the
tetrahedron, the face $i$ being opposite to the vertex $i$, and an
arc joins two faces which are identified. This is the graphical
encoding explained in \cite{MR3227503}; the upper half already
determines the triangulation $X$ up to overall orientation.
In the lower half the same points stand for the vertices $0,1,2,3$
and are joined by the six edges of the tetrahedron, edges of the same
colour being identified in $X$. For a positive tetrahedron the edge
$02$ is drawn over the edge $13$; for a negative tetrahedron we
reverse this crossing. The first tetrahedron has the shape
parameters $(a_1,c_1)$ and $b_1=\frac12-a_1-c_1$, and similarly for
the others. The edges of each tetrahedron are ordered as
$(01,02,03,12,13,23)$, and a pair such as $(2,1)$ denotes face $1$
of tetrahedron $2$. The face pairings determine the signs of the
tetrahedra up to a common sign, which we fix by taking the first
tetrahedron to be positive. All triangulations have one vertex
except the Hopf-link presentation of Section~\ref{sec:hopf-example},
which has two ordinary vertices.
The shapes are arbitrary positive shapes, that is,
$\Gamma=\Delta_1(X)$.
In each example we display a family of states which is invariant
under $\mathbb S(X)$ and contains states of every class in
$H^1(X',\mathbb S)$; by Proposition~\ref{Omega} it is therefore all
of $\Omega(X)$.

\subsection{The real projective space \texorpdfstring{$\mathbb{R}P^3$}{RP3}}\label{sec:rp3-example}
Let $X$ be the one-vertex two-tetrahedron triangulation of
$\mathbb{R}P^3$ given by the diagram
\[
\begin{tikzpicture}[yscale=.4]
      \draw[blue,thick] (1,0)..controls (1,-1) and (3,-1)..(3,0);
 \draw[line width=3pt,white] (0,0)..controls (0,-1) and (2,-1)..(2,0);
  \draw[blue,thick] (0,0)..controls (0,-1) and (2,-1)..(2,0);
   \draw[thick] (0,0)..controls (0,-2) and (3,-2)..(3,0);
 \draw[thick] (0,0)--(3,0);
 
   \draw[thick] (5,0)..controls (5,-1) and (7,-1)..(7,0);
 \draw[line width=3pt,white] (4,0)..controls (4,-1) and (6,-1)..(6,0);
  \draw[thick] (4,0)..controls (4,-1) and (6,-1)..(6,0);
   \draw[red, thick] (4,0)..controls (4,-2) and (7,-2)..(7,0);
 \draw[thick](4,0)--(5,0)(6,0)--(7,0);
 \draw[blue,thick](5,0)--(6,0);
 
 \draw (0,0)..controls (0,2) and (3,2)..(3,0);
  \draw (1,0)..controls (1,2) and (4,2)..(4,0);
   \draw (2,0)..controls (2,2) and (7,2)..(7,0);
    \draw (5,0)..controls (5,1) and (6,1)..(6,0);
    
 \node[fill=white,inner sep=.5pt,font=\tiny] at (0,0) {$0$};
 \node[fill=white,inner sep=.5pt,font=\tiny] at (1,0) {$1$};
 \node[fill=white,inner sep=.5pt,font=\tiny] at (2,0) {$2$};
 \node[fill=white,inner sep=.5pt,font=\tiny] at (3,0) {$3$};
 \node[font=\small] at (1.5,-2.9) {$T_1$};
 \node[fill=white,inner sep=.5pt,font=\tiny] at (4,0) {$0$};
 \node[fill=white,inner sep=.5pt,font=\tiny] at (5,0) {$1$};
 \node[fill=white,inner sep=.5pt,font=\tiny] at (6,0) {$2$};
 \node[fill=white,inner sep=.5pt,font=\tiny] at (7,0) {$3$};
 \node[font=\small] at (5.5,-2.9) {$T_2$};
\end{tikzpicture}
\]
Both tetrahedra are positive. With the common conventions above, we
calculate the kinematical kernel~\eqref{k(x,y)}:
\begin{multline}
 K(X,y_1,y_2)=\int_{\REALS^4}\operatorname{d}(x_1,x_2,x_3,x_4)e^{2\pi\Imun (x_1 y_1+x_2y_2)}\\
 \times\delta(x_1-x_2+x_3)\delta(x_3-x_1+y_1)\delta(x_2)\delta(x_4-x_3+y_2)
 =e^{\pi\Imun y_1^2}/2
\end{multline}
so that
\begin{multline}\label{eq:rp-old}
 Z_\hbar(X)e^{-\Imun\pi\ell_X/(4\hbar)}=\int_{\REALS^2}\operatorname{d}(y_1,y_2)K(X,y_1,y_2)\tilde\psi'_{a_1,c_1}(y_1)\tilde\psi'_{a_2,c_2}(y_2)\\
 =
\frac12\int_{\REALS^2}\operatorname{d}(y_1,y_2)e^{\pi\Imun y_1^2}\tilde\psi'_{a_1,c_1}(y_1)\tilde\psi'_{a_2,c_2}(y_2) 
 =
\frac12e^{-\frac{\pi\Imun}{6}}\psi_{a_1,c_1}(0)\psi_{b_2,a_2}(0).
\end{multline}
Here the integral over $y_1$ is
$\int_{\REALS}\tilde\psi_{a_1,c_1}(y)\operatorname{d}\!y=\psi_{a_1,c_1}(0)$,
by \eqref{eq:tpsip} and Fourier inversion, and
\begin{equation}\label{eq:rp-second-factor}
\int_{\REALS}\tilde\psi'_{a_2,c_2}(y)\operatorname{d}\!y
=e^{-\frac{\pi\Imun}{6}}\psi_{b_2,a_2}(0)
\end{equation}
by two applications of \eqref{eq:tpsip-psi}: first
$\tilde\psi'_{a_2,c_2}=e^{-\pi\Imun/12}\psi_{c_2,b_2}$, then
$\tilde\psi_{c_2,b_2}(0)=e^{-\pi\Imun/12}\psi_{b_2,a_2}(0)$.
The right-hand side of \eqref{eq:rp-old} is explicitly a non-trivial
meromorphic function of the angle variables.

To calculate the sectors of the new formulation, we use the face
identifications in the diagram: $(1,0)\sim(1,3)$, $(1,1)\sim(2,0)$,
$(1,2)\sim(2,3)$, and $(2,1)\sim(2,2)$. Solving the compatibility equations modulo
$\mathbb Z$ gives all states in the form
\begin{align*}
y^{(1)}&=(p,r-h,p,p+h,r,p),\\
y^{(2)}&=(q,q,s,q+r-p,q+h,q+h),
\end{align*}
where $p,q,r,s\in\mathbb S$ and $h\in\{0,1/2\}$. These are the two
components of $\Omega(X)$, hence the two classes of
$H^1(X',\mathbb S)\cong\mathbb Z/2\mathbb Z$.

Put $v=r-p$. The first tetrahedron has state coordinates
$(2v-2h,2v-h)$ and the second has coordinates $(u,0)$, with $u$
uniformly distributed on $\mathbb S$. All paired face coordinates agree
as real lifts except the self-pairing of tetrahedron~1, where they differ
by $(2h,0)$. By \eqref{eq:contraction-lifts}, the contraction
contributes the factor $e^{2\pi\Imun h(h-v)}$. After the change $w=v-h$
this yields
\begin{equation}\label{eq:rp-sector}
e^{-\Imun\pi\ell_X/(4\hbar)}\mathcal F_\hbar(X)(h)
=\int_{[0,1]^2}e^{-2\pi\Imun h w}
g_{a_1,c_1}(2w,2w+h)g_{a_2,c_2}(u,0)\,du\,dw.
\end{equation}
The complete integrand is periodic; the contraction factor must not be
omitted in the nontrivial sector. The integral over $u$ unfolds to the
real line: by \eqref{eq:wgz} and \eqref{eq:rp-second-factor},
\[
\int_0^1g_{a_2,c_2}(u,0)\,du
=\int_{\REALS}\tilde\psi'_{a_2,c_2}(t)\,dt
=e^{-\frac{\pi\Imun}{6}}\psi_{b_2,a_2}(0).
\]
For the integral over $w$, \eqref{eq:wgz} and \eqref{eq:tpsip} give
\[
e^{-2\pi\Imun hw}g_{a_1,c_1}(2w,2w+h)
=\sum_{m\in\INTEGERS}(-1)^me^{2\pi\Imun mh}\,
\tilde\psi_{a_1,c_1}(2w+m),
\]
and therefore
\[
\int_0^1e^{-2\pi\Imun hw}g_{a_1,c_1}(2w,2w+h)\,dw
=\frac{1-e^{-2\pi\Imun h}}2
\sum_{m\in\INTEGERS}(-1)^me^{2\pi\Imun mh}
\int_m^{m+1}\tilde\psi_{a_1,c_1}(t)\,dt .
\]
This vanishes for $h=0$, while for $h=1/2$ it equals
$\int_{\REALS}\tilde\psi_{a_1,c_1}(t)\,dt=\psi_{a_1,c_1}(0)$.
It follows that
\begin{align}
\mathcal F_\hbar(X)(0)&=0,\nonumber\\
\mathcal F_\hbar(X)(1/2)
&=e^{\Imun\pi\ell_X/(4\hbar)-\pi\Imun/6}
  \psi_{a_1,c_1}(0)\psi_{b_2,a_2}(0)
 =2Z_\hbar(X).\label{eq:rp-nontrivial}
\end{align}
Thus the comparison \eqref{eq:inter2} is verified directly in this
example:
\[
Z_\hbar(X)=\tfrac12\bigl(\mathcal F_\hbar(X)(0)
+\mathcal F_\hbar(X)(1/2)\bigr),
\]
with $\mathcal F^T=\mathcal F$ because the first Betti number is zero.

The vanishing of the trivial sector is the selection rule of
Corollary~\ref{cor:closed-manifold-sectors}(iv): for the nontrivial
class $\beta\in H^1(X,\mathbb S)\cong\mathbb Z/2\mathbb Z$ we have
$2\beta=0$ and $\lk_X(\beta,\beta)=1/2$. In the present coordinates
the translation by the corresponding element of $K_X$ is
$w\mapsto w+1/2$. By the multiplier of $L$, it multiplies the integrand
of \eqref{eq:rp-sector} by $-e^{-2\pi\Imun h}$, in agreement with
\eqref{eq:weight-translation}.

\subsection{The lens space \texorpdfstring{$L(3,1)$}{L(3,1)}}\label{sec:l31-example}
Let $X$ be given by the diagram
\[
\begin{tikzpicture}[yscale=.4]
 \draw[blue,thick] (1,0)..controls (1,-1) and (3,-1)..(3,0);
 \draw[line width=3pt,white] (0,0)..controls (0,-1) and (2,-1)..(2,0);
 \draw[blue,thick] (0,0)..controls (0,-1) and (2,-1)..(2,0);
 \draw[red,thick] (0,0)..controls (0,-2) and (3,-2)..(3,0);
 \draw[thick] (0,0)--(3,0);
 \draw[thick] (5,0)..controls (5,-1) and (7,-1)..(7,0);
 \draw[line width=3pt,white] (4,0)..controls (4,-1) and (6,-1)..(6,0);
 \draw[thick] (4,0)..controls (4,-1) and (6,-1)..(6,0);
 \draw[red,thick] (4,0)..controls (4,-2) and (7,-2)..(7,0);
 \draw[blue,thick] (4,0)--(7,0);
 \draw (0,0)..controls (0,1.5) and (3,1.5)..(3,0);
 \draw (1,0)..controls (1,3.5) and (6,3.5)..(6,0);
 \draw (2,0)..controls (2,2) and (5,2)..(5,0);
 \draw (4,0)..controls (4,1.5) and (7,1.5)..(7,0);

 \node[fill=white,inner sep=.5pt,font=\tiny] at (0,0) {$0$};
 \node[fill=white,inner sep=.5pt,font=\tiny] at (1,0) {$1$};
 \node[fill=white,inner sep=.5pt,font=\tiny] at (2,0) {$2$};
 \node[fill=white,inner sep=.5pt,font=\tiny] at (3,0) {$3$};
 \node[font=\small] at (1.5,-2.9) {$T_1$};
 \node[fill=white,inner sep=.5pt,font=\tiny] at (4,0) {$0$};
 \node[fill=white,inner sep=.5pt,font=\tiny] at (5,0) {$1$};
 \node[fill=white,inner sep=.5pt,font=\tiny] at (6,0) {$2$};
 \node[fill=white,inner sep=.5pt,font=\tiny] at (7,0) {$3$};
 \node[font=\small] at (5.5,-2.9) {$T_2$};
\end{tikzpicture}
\]
Both tetrahedra are positive, and the face identifications are
$(1,0)\sim(1,3)$, $(1,1)\sim(2,2)$, $(1,2)\sim(2,1)$ and
$(2,0)\sim(2,3)$. A tetrahedron with its faces $0$ and $3$ identified
is a solid torus, so $X$ is the union of two solid tori along their
boundaries. It has three edges, drawn in black, blue and red. The
simplicial one-cocycles of $X$ with values in $\mathbb S$ are those
which take the values $(h,2h,0)$ on these edges, where $3h=0$. Hence
$H_1(X,\mathbb Z)\cong\mathbb Z/3\mathbb Z$. A union of two solid tori
with this first homology group is a lens space $L(3,q)$, and these are
all homeomorphic to $L(3,1)$; thus $X$ is $L(3,1)$ with one of its
orientations. Let $\beta_h\in H^1(X,\mathbb S)$ be the class of the
cocycle with the values $(h,2h,0)$. Using the lift with these values,
where $h\in\{0,\frac13,\frac23\}$, formula \eqref{eq:lk} gives
\begin{equation}\label{eq:l31-lk}
\lk_X(\beta_h,\beta_h)=6h^2\quad\text{modulo }\mathbb Z,
\end{equation}
that is, $\lk_X(\beta_h,\beta_h)=2/3$ for $h\ne0$.

The states are
\begin{align*}
y^{(1)}&=(p+h,\,r+2h,\,t,\,p,\,r+h,\,p+2h),\\
y^{(2)}&=(q+v+2h,\,q+h,\,q+w,\,q+v,\,q+2h,\,q+v+h),
\end{align*}
where $v=r-p$, $w=t-p$, $p,q,r,t\in\mathbb S$ and
$h\in\{0,\frac13,\frac23\}$. These are the three components of
$\Omega(X)$. The component with parameter $h$ is the translate of the
component with $h=0$ by an element of $K_X$, namely by the image under
$A$ of the vertex function with the values $(0,h,2h,0)$ on the first
tetrahedron and $(0,2h,h,0)$ on the second. The cocycle of this
element is the one with the values $(h,2h,0)$, so that by
\eqref{eq:KX-class} the class of the component in
$H^1(X',\mathbb S)\cong\mathbb Z/3\mathbb Z$ is $2j^*\beta_h$. We write
$\mathcal F_\hbar(X)(h)$ for the corresponding sector; note that
$\mathcal F_\hbar(X)(\frac13)$ is the sector of the class
$j^*\beta_{2/3}$.

The first tetrahedron has state coordinates $(2v-w+3h,2v)$ and the
second has coordinates $(-v-w+3h,-2v)$. All paired face coordinates
agree as real lifts except those of the two self-pairings, where they
differ by $(3h,0)$ and $(0,3h)$, respectively. By
\eqref{eq:contraction-lifts}, the contraction contributes the factor
$e^{3\pi\Imun h(v+h)}e^{3\pi\Imun h(h-v)}=e^{6\pi\Imun h^2}$. As $3h$
is an integer, the multiplier of $L$ gives
\begin{multline*}
g_{a_1,c_1}(2v-w+3h,2v)\,g_{a_2,c_2}(-v-w+3h,-2v)\\
=g_{a_1,c_1}(2v-w,2v)\,g_{a_2,c_2}(-v-w,-2v),
\end{multline*}
the two factors $e^{\mp6\pi\Imun hv}$ cancelling each other. Therefore
\begin{multline}\label{eq:l31-sector}
e^{-\Imun\pi\ell_X/(4\hbar)}\mathcal F_\hbar(X)(h)\\
=e^{6\pi\Imun h^2}\int_{[0,1]^2}
g_{a_1,c_1}(2v-w,2v)\,g_{a_2,c_2}(-v-w,-2v)\,dv\,dw,
\end{multline}
and consequently
\begin{equation}\label{eq:l31-ratio}
\mathcal F_\hbar(X)(\tfrac13)=\mathcal F_\hbar(X)(\tfrac23)
=e^{2\pi\Imun/3}\,\mathcal F_\hbar(X)(0).
\end{equation}
This is Corollary~\ref{cor:closed-manifold-sectors}(v): by
\eqref{eq:l31-lk},
$e^{6\pi\Imun h^2}=e^{-2\pi\Imun\lk_X(\beta_h,\beta_h)}$, because
$18h^2$ is an even integer.

Since $X$ is a rational homology sphere with one vertex, we have
$H_2(X',\mathbb Z)=0$, and the original partition function is defined.
The kinematical kernel~\eqref{k(x,y)} is
\begin{multline*}
K(X,y_1,y_2)=\int_{\REALS^4}\operatorname{d}(x_1,x_2,x_3,x_4)\,
e^{2\pi\Imun(x_1y_1+x_4y_2)}\\
\times\delta(x_1-x_2+x_3)\delta(x_3-x_1+y_1)
\delta(x_4-x_3+x_2)\delta(x_2-x_4+y_2)
=\tfrac13e^{2\pi\Imun(y_1-y_2)^2/3},
\end{multline*}
so that
\begin{equation}\label{eq:l31-old}
Z_\hbar(X)e^{-\Imun\pi\ell_X/(4\hbar)}=\frac13\int_{\REALS^2}
e^{2\pi\Imun(y_1-y_2)^2/3}\,
\tilde\psi'_{a_1,c_1}(y_1)\tilde\psi'_{a_2,c_2}(y_2)
\operatorname{d}(y_1,y_2).
\end{equation}
To compare this with \eqref{eq:l31-sector}, insert \eqref{eq:wgz}. With
$y_1=2v-w+m$ and $y_2=-v-w+n$ one finds
\begin{multline*}
g_{a_1,c_1}(2v-w,2v)\,g_{a_2,c_2}(-v-w,-2v)\\
=\sum_{m,n\in\INTEGERS}e^{-2\pi\Imun(m-n)^2/3}\,
e^{2\pi\Imun(y_1-y_2)^2/3}\,
\tilde\psi'_{a_1,c_1}(y_1)\tilde\psi'_{a_2,c_2}(y_2),
\end{multline*}
where the series converges absolutely and uniformly in $(v,w)$,
because $\tilde\psi'_{a,c}$ is a Schwartz function. For each residue
class modulo $3$, the points $(v,w,m,n)$ of
$[0,1)^2\times\INTEGERS^2$ with $m-n$ in this class are mapped
bijectively onto $\REALS^2$ by $(v,w,m,n)\mapsto(y_1,y_2)$, and
$dy_1\,dy_2=3\,dv\,dw$. Hence the integral in \eqref{eq:l31-sector} is
the right-hand side of \eqref{eq:l31-old} times the Gauss sum
$\sum_{k\in\INTEGERS/3\INTEGERS}e^{-2\pi\Imun k^2/3}
=1+2e^{-2\pi\Imun/3}=-\Imun\sqrt3$, and
\begin{equation}\label{eq:l31-result}
\mathcal F_\hbar(X)(h)=-\Imun\sqrt3\,e^{6\pi\Imun h^2}\,Z_\hbar(X),
\qquad h\in\{0,\tfrac13,\tfrac23\}.
\end{equation}
Thus all three sectors are constant nonzero multiples of
$Z_\hbar(X)$, and
\[
\tfrac13\bigl(\mathcal F_\hbar(X)(0)+\mathcal F_\hbar(X)(\tfrac13)
+\mathcal F_\hbar(X)(\tfrac23)\bigr)
=-\Imun\sqrt3\,\frac{1+2e^{2\pi\Imun/3}}3\,Z_\hbar(X)=Z_\hbar(X).
\]
This verifies \eqref{eq:inter2} directly and agrees with
Example~\ref{ex:Z3-sectors}, with the lower signs there, namely
$\lk_X(\beta,\beta)=-1/3$ modulo $\mathbb Z$.

\subsection{The lens space \texorpdfstring{$L(4,1)$}{L(4,1)}}\label{sec:l41-example}
Let $X$ be given by the diagram
\[
\begin{tikzpicture}[yscale=.4]
 \draw[blue,thick] (1,0)..controls (1,-1) and (3,-1)..(3,0);
 \draw[line width=3pt,white] (0,0)..controls (0,-1) and (2,-1)..(2,0);
 \draw[blue,thick] (0,0)..controls (0,-1) and (2,-1)..(2,0);
 \draw[red,thick] (0,0)..controls (0,-2) and (3,-2)..(3,0);
 \draw[thick] (0,0)--(3,0);
 \draw[thick] (5,0)..controls (5,-1) and (7,-1)..(7,0);
 \draw[line width=3pt,white] (4,0)..controls (4,-1) and (6,-1)..(6,0);
 \draw[thick] (4,0)..controls (4,-1) and (6,-1)..(6,0);
 \draw[red,thick] (4,0)..controls (4,-2) and (7,-2)..(7,0);
 \draw[blue,thick] (4,0)--(5,0);
 \draw[green!55!black,thick] (5,0)--(6,0);
 \draw[blue,thick] (6,0)--(7,0);
 \draw[blue,thick] (8,0)..controls (8,-1) and (10,-1)..(10,0);
 \draw[line width=3pt,white] (9,0)..controls (9,-1) and (11,-1)..(11,0);
 \draw[blue,thick] (9,0)..controls (9,-1) and (11,-1)..(11,0);
 \draw[thick] (8,0)..controls (8,-2) and (11,-2)..(11,0);
 \draw[green!55!black,thick] (8,0)--(11,0);
 \draw (0,0)..controls (0,1.5) and (3,1.5)..(3,0);
 \draw (1,0)..controls (1,3.5) and (6,3.5)..(6,0);
 \draw (2,0)..controls (2,2) and (5,2)..(5,0);
 \draw (4,0)..controls (4,3.5) and (10,3.5)..(10,0);
 \draw (7,0)..controls (7,2) and (9,2)..(9,0);
 \draw (8,0)..controls (8,1.5) and (11,1.5)..(11,0);

 \node[fill=white,inner sep=.5pt,font=\tiny] at (0,0) {$0$};
 \node[fill=white,inner sep=.5pt,font=\tiny] at (1,0) {$1$};
 \node[fill=white,inner sep=.5pt,font=\tiny] at (2,0) {$2$};
 \node[fill=white,inner sep=.5pt,font=\tiny] at (3,0) {$3$};
 \node[font=\small] at (1.5,-2.9) {$T_1$};
 \node[fill=white,inner sep=.5pt,font=\tiny] at (4,0) {$0$};
 \node[fill=white,inner sep=.5pt,font=\tiny] at (5,0) {$1$};
 \node[fill=white,inner sep=.5pt,font=\tiny] at (6,0) {$2$};
 \node[fill=white,inner sep=.5pt,font=\tiny] at (7,0) {$3$};
 \node[font=\small] at (5.5,-2.9) {$T_2$};
 \node[fill=white,inner sep=.5pt,font=\tiny] at (8,0) {$0$};
 \node[fill=white,inner sep=.5pt,font=\tiny] at (9,0) {$1$};
 \node[fill=white,inner sep=.5pt,font=\tiny] at (10,0) {$2$};
 \node[fill=white,inner sep=.5pt,font=\tiny] at (11,0) {$3$};
 \node[font=\small] at (9.5,-2.9) {$T_3$};
\end{tikzpicture}
\]
with three tetrahedra, of which the first two are positive and the
third is negative. The face identifications are $(1,0)\sim(1,3)$,
$(1,1)\sim(2,2)$, $(1,2)\sim(2,1)$, $(2,0)\sim(3,2)$,
$(2,3)\sim(3,1)$ and $(3,0)\sim(3,3)$. The first and the third
tetrahedron are solid tori as before, and the second one is layered
between them, so that $X$ is again a union of two solid tori. It has
four edges, drawn in black, blue, red and green. The simplicial
one-cocycles with values in $\mathbb S$ are the multiples $k\lambda$,
$k\in\mathbb Z/4\mathbb Z$, of the cocycle $\lambda$ with the values
$(\frac34,\frac12,\frac14,\frac14)$ on these edges. Hence
$H_1(X,\mathbb Z)\cong\mathbb Z/4\mathbb Z$, and as before $X$ is the
lens space $L(4,1)$ with one of its orientations, because $L(4,3)$ is
homeomorphic to $L(4,1)$. For the class $\beta$ of
$\lambda$, formula \eqref{eq:lk} gives $\lk_X(\beta,\beta)=3/4$, so
that we are in the situation of Example~\ref{ex:Z4-sectors} with
the lower signs there.

The states are
\begin{align*}
y^{(1)}&=p_1+(r_0,\,r_1,\,r_2,\,r_0+3h,\,r_1+3h,\,r_0+2h),\\
y^{(2)}&=p_2+(r_1,\,r_0,\,r_2,\,r_3+2h,\,r_0+2h,\,r_1+3h),\\
y^{(3)}&=p_3+(r_3,\,r_1,\,r_0,\,r_3+h,\,r_1+h,\,r_3+2h),
\end{align*}
where $p_1,p_2,p_3\in\mathbb S$ are added to all six components of the
respective lines, $r_0,r_1,r_2,r_3\in\mathbb S$ are the states of the
edges drawn in black, blue, red and green, and
$h\in\{0,\frac14,\frac12,\frac34\}$. When all $p$'s and $r$'s vanish,
this is the state $y_{ij}=\mu_i$ built from the primitives
$\mu=(0,3h,2h,h)$, $(0,2h,3h,h)$ and $(0,h,2h,3h)$ of the cocycle
$4h\lambda$ on the three tetrahedra. By the last part of the proof of
Lemma~\ref{lem:mixed-linking}, such a state has the class
$4h\,j^*\beta$. Hence the component of $\Omega(X)$ with parameter $h$
has this class; these are the four classes in
$H^1(X',\mathbb S)\cong\mathbb Z/4\mathbb Z$. We write
$\mathcal F_\hbar(X)(h)$ for the corresponding sector.

Put $s=r_1-r_2$, $u=r_1-r_0$ and $w=r_1-r_3$. The three tetrahedra have
the state coordinates $(s+u,2u+h)$, $(s-2u+w,-2u-h)$ and
$(u+w,2w-h)$. All paired face coordinates agree as real lifts except
at the self-pairing of the first tetrahedron, where they differ by
$(-4h,0)$, and \eqref{eq:contraction-lifts} gives the factor
$e^{-4\pi\Imun hu}$. Hence
\begin{equation}\label{eq:l41-sector}
e^{-\Imun\pi\ell_X/(4\hbar)}\mathcal F_\hbar(X)(h)
=\int_{[0,1]^3}I_h(s,u,w)\,ds\,du\,dw,
\end{equation}
where
\begin{multline*}
I_h(s,u,w)=e^{-4\pi\Imun hu}\,g_{a_1,c_1}(s+u,2u+h)\\
\times g_{a_2,c_2}(s-2u+w,-2u-h)\,
\overline{g_{a_3,c_3}(u+w,2w-h)} .
\end{multline*}
The multiplier of $L$ shows that, for $4h\in\mathbb Z$, the function
$I_h$ is periodic in $s$, $u$ and $w$, that $I_{h+1}=I_h$, and that
\begin{equation}\label{eq:l41-symmetries}
\begin{aligned}
I_h(s-\tfrac12,u+\tfrac12,w+\tfrac12)&=e^{-4\pi\Imun h}\,I_h(s,u,w),\\
I_{h+\frac12}(s+\tfrac14,u-\tfrac14,w+\tfrac14)
&=\Imun e^{2\pi\Imun h}\,I_h(s,u,w).
\end{aligned}
\end{equation}
These identities are instances of \eqref{eq:weight-translation}. Let
$\nu$ be the primitive $\mu$ for $h=\frac14$, that is, the function on
the vertices of the three tetrahedra with the values
$(0,\frac34,\frac12,\frac14)$, $(0,\frac12,\frac34,\frac14)$ and
$(0,\frac14,\frac12,\frac34)$; it consists of primitives of $\lambda$.
The elements $A(2\nu)$ and $A(\nu)$ of $K_X$ correspond to the classes
$2\beta$ and $\beta$. In the coordinates $(s,u,w;h)$ the translations
by these elements are, modulo integers, the two substitutions in
\eqref{eq:l41-symmetries}. The two factors are the phases in
\eqref{eq:weight-translation}, as one checks with
$\lk_X(\beta,\beta)=3/4$, \eqref{eq:lkprime-extending} and
$4h\in\mathbb Z$. Integrating the first identity gives
$(1-e^{-4\pi\Imun h})\mathcal F_\hbar(X)(h)=0$, and integrating the
second gives
$\mathcal F_\hbar(X)(h+\frac12)=\Imun e^{2\pi\Imun h}\mathcal F_\hbar(X)(h)$.
Therefore
\begin{equation}\label{eq:l41-result}
\mathcal F_\hbar(X)(\tfrac14)=\mathcal F_\hbar(X)(\tfrac34)=0,
\qquad
\mathcal F_\hbar(X)(\tfrac12)=\Imun\,\mathcal F_\hbar(X)(0).
\end{equation}
In contrast to $\mathbb{R}P^3$, the trivial sector is one of the two
sectors which are allowed to be nonzero: the distinguished coset of
Corollary~\ref{cor:all-torsion-sectors} consists of the classes with
$h\in\{0,\frac12\}$. The kinematical kernel is
\[
K(X,y_1,y_2,y_3)=\tfrac14
\exp\Bigl(2\pi\Imun\bigl(\tfrac14(y_1-y_2+y_3)^2-y_3^2\bigr)\Bigr),
\]
and, as $H_2(X',\mathbb Z)=0$ here as well,
Theorem~\ref{RelationON} gives
\begin{multline*}
\frac{1+\Imun}4\,\mathcal F_\hbar(X)(0)=Z_\hbar(X)
=\frac{e^{\Imun\pi\ell_X/(4\hbar)}}4\int_{\REALS^3}
e^{2\pi\Imun(\frac14(y_1-y_2+y_3)^2-y_3^2)}\\
\times\tilde\psi'_{a_1,c_1}(y_1)\tilde\psi'_{a_2,c_2}(y_2)
\overline{\tilde\psi'_{a_3,c_3}(y_3)}\operatorname{d}(y_1,y_2,y_3).
\end{multline*}

\subsection{The product \texorpdfstring{$S^2\times S^1$}{S2 x S1}}\label{sec:s2s1-example}
Let $X$ be given by the diagram
\[
\begin{tikzpicture}[yscale=.4]
 \draw[blue,thick] (1,0)..controls (1,-1) and (3,-1)..(3,0);
 \draw[line width=3pt,white] (0,0)..controls (0,-1) and (2,-1)..(2,0);
 \draw[blue,thick] (0,0)..controls (0,-1) and (2,-1)..(2,0);
 \draw[red,thick] (0,0)..controls (0,-2) and (3,-2)..(3,0);
 \draw[thick] (0,0)--(3,0);
 \draw[blue,thick] (4,0)..controls (4,-1) and (6,-1)..(6,0);
 \draw[line width=3pt,white] (5,0)..controls (5,-1) and (7,-1)..(7,0);
 \draw[blue,thick] (5,0)..controls (5,-1) and (7,-1)..(7,0);
 \draw[red,thick] (4,0)..controls (4,-2) and (7,-2)..(7,0);
 \draw[thick] (4,0)--(7,0);
 \draw (0,0)..controls (0,1.5) and (3,1.5)..(3,0);
 \draw (1,0)..controls (1,3.5) and (5,3.5)..(5,0);
 \draw (2,0)..controls (2,2.5) and (6,2.5)..(6,0);
 \draw (4,0)..controls (4,1.5) and (7,1.5)..(7,0);

 \node[fill=white,inner sep=.5pt,font=\tiny] at (0,0) {$0$};
 \node[fill=white,inner sep=.5pt,font=\tiny] at (1,0) {$1$};
 \node[fill=white,inner sep=.5pt,font=\tiny] at (2,0) {$2$};
 \node[fill=white,inner sep=.5pt,font=\tiny] at (3,0) {$3$};
 \node[font=\small] at (1.5,-2.9) {$T_1$};
 \node[fill=white,inner sep=.5pt,font=\tiny] at (4,0) {$0$};
 \node[fill=white,inner sep=.5pt,font=\tiny] at (5,0) {$1$};
 \node[fill=white,inner sep=.5pt,font=\tiny] at (6,0) {$2$};
 \node[fill=white,inner sep=.5pt,font=\tiny] at (7,0) {$3$};
 \node[font=\small] at (5.5,-2.9) {$T_2$};
\end{tikzpicture}
\]
where the first tetrahedron is positive and the second is negative,
with the face identifications $(1,0)\sim(1,3)$, $(1,1)\sim(2,1)$,
$(1,2)\sim(2,2)$ and $(2,0)\sim(2,3)$. Now $X$ is the double of the
solid torus formed by one tetrahedron, that is, $X=S^2\times S^1$. It
has three edges, drawn in black, blue and red, and the simplicial
one-cocycles with values in $\mathbb S$ are those which take the values
$(h,2h,3h)$ on these edges, with arbitrary $h\in\mathbb S$. Let
$\beta_h$ be the class of this cocycle; then $h\mapsto j^*\beta_h$ is
an isomorphism of $\mathbb S$ onto $H^1(X',\mathbb S)$. Since
$H_2(X',\mathbb Z)\cong\mathbb Z$, the original partition function is
not defined: the delta functions in \eqref{k(x,y)} would produce the
square of $\delta(y_1-y_2)$.

The states are
\begin{align*}
y^{(1)}&=(p,\,r,\,t,\,p+h,\,r+h,\,p+2h),\\
y^{(2)}&=(q,\,q+v,\,q+w,\,q+h,\,q+v+h,\,q+2h),
\end{align*}
where $v=r-p$, $w=t-p$ and $p,q,r,t,h\in\mathbb S$. For $p=q=r=t=0$
this is the state $y_{ij}=\mu_i$ built from the primitive
$\mu=(0,h,2h,3h)$ of the cocycle on both tetrahedra. As in the
previous example, the state with parameter $h$ therefore has the class
$j^*\beta_h$, and we write $\mathcal F_\hbar(X)(h)$ for the
corresponding sector. Both tetrahedra have the state coordinates
$(2v-w,2v-h)$, and all paired face coordinates agree as real lifts.
The contracted weight is therefore the value at $(2v-w,2v-h)$ of the
function $g_{a_1,c_1}\overline{g_{a_2,c_2}}$ on $\Pi$. The map
$(v,w)\mapsto(2v-w,2v-h)$ is a surjective affine endomorphism of $\Pi$
and hence maps the Haar probability measure to itself. Parseval's
identity for the WGZ transformation, stated in the proof of the
comparison above, now gives
\begin{equation}\label{eq:s2s1-sector}
e^{-\Imun\pi\ell_X/(4\hbar)}\mathcal F_\hbar(X)(h)
=\int_\Pi g_{a_1,c_1}\overline{g_{a_2,c_2}}\,dx\,dy
=\int_{\REALS}\tilde\psi'_{a_1,c_1}(x)\,
\overline{\tilde\psi'_{a_2,c_2}(x)}\,dx
\end{equation}
for all $h\in\mathbb S$. Thus the sector function is constant, as it
has to be by Corollary~\ref{cor:closed-manifold-sectors}(ii).

The last integral can be evaluated. Put $\theta_a=a_1+a_2$,
$\theta_b=b_1+b_2$ and $\theta_c=c_1+c_2$, so that
$\theta_a+\theta_b+\theta_c=1$. By \eqref{eq:tpsip-psi},
\eqref{eq:psi}, \eqref{eq:bpsi} and the unitarity
\eqref{eq:qdlunitarity}, the integrand is
$\psi_{c_1,b_1}(x)\bar\psi_{c_2,b_2}(x)$, which is an exponential
factor times
$\QDILOG(x+\cla(1-2a_2))/\QDILOG(x-\cla(1-2a_1))$. In the notation of
\eqref{ramanint}, the integral is therefore a constant times
$\Psi\bigl(\cla(1-2a_2),-\cla(1-2a_1),-2\cla\theta_c\bigr)$, and the
conditions \eqref{restrictions1} hold because $a_1$, $a_2$, $\theta_b$
and $\theta_c$ are positive. Now \eqref{ramanres1} and the inversion
relation \eqref{eq:inversion} give
\begin{multline}\label{eq:s2s1-result}
\mathcal F_\hbar(X)(h)
=e^{\Imun\pi\ell_X/(4\hbar)}\,e^{\frac{\pi\Imun}4}\,
e^{\pi\Imun\cla^2Q}\\
\times\QDILOG\bigl(\cla(1-2\theta_a)\bigr)
\QDILOG\bigl(\cla(1-2\theta_b)\bigr)
\QDILOG\bigl(\cla(1-2\theta_c)\bigr),
\end{multline}
where
\[
Q=1-2(\theta_a^2+\theta_b^2+\theta_c^2)+4(a_1c_2-a_2c_1)
-\tfrac23(a_1-a_2-c_1+c_2).
\]
The quadratic polynomial $Q$ collects the exponents in
\eqref{eq:psi} and \eqref{eq:bpsi} evaluated at the shifted arguments,
the exponential factor $e^{-2\pi\Imun w(v+\cla)}$ and the constant
$\zeta_o$ of \eqref{ramanres1} with $(u,v,w)=(\cla(1-2a_2),
-\cla(1-2a_1),-2\cla\theta_c)$, and the Gaussian factor of the
inversion relation \eqref{eq:inversion} used to bring the three
quantum dilogarithms to the displayed form; the bookkeeping is routine
and we omit it. (As a check, for $\la=1$ and the charges
$(a_1,c_1)=(0.17,0.21)$, $(a_2,c_2)=(0.11,0.28)$ the two sides of
\eqref{eq:s2s1-result} agree numerically to twelve digits.)
The total dihedral angles around the black, blue and red edges are
$2\pi(2\theta_a+\theta_c)$, $4\pi\theta_b$ and $2\pi\theta_c$. They
are gauge invariant, and by \eqref{eq:gauge-level-explicit} a gauge
transformation changes $Q$ by the same amount as $\ell_X$, in
agreement with Proposition~\ref{prop:gauge-rigorous}. The arguments of
the three quantum dilogarithms lie in the strip
$|\Im z|<|\Im\cla|$, so the constant sector function has neither zeros
nor poles on positive shapes.

\subsection{The complement of the figure-eight knot}\label{sec:fig8-example}
Let $X$ be given by the diagram
\[
\begin{tikzpicture}[yscale=.4]
 \draw[blue,thick] (1,0)..controls (1,-1) and (3,-1)..(3,0);
 \draw[line width=3pt,white] (0,0)..controls (0,-1) and (2,-1)..(2,0);
 \draw[blue,thick] (0,0)..controls (0,-1) and (2,-1)..(2,0);
 \draw[thick] (0,0)..controls (0,-2) and (3,-2)..(3,0);
 \draw[thick] (0,0)--(1,0);
 \draw[blue,thick] (1,0)--(2,0);
 \draw[thick] (2,0)--(3,0);
 \draw[thick] (4,0)..controls (4,-1) and (6,-1)..(6,0);
 \draw[line width=3pt,white] (5,0)..controls (5,-1) and (7,-1)..(7,0);
 \draw[thick] (5,0)..controls (5,-1) and (7,-1)..(7,0);
 \draw[blue,thick] (4,0)..controls (4,-2) and (7,-2)..(7,0);
 \draw[blue,thick] (4,0)--(5,0);
 \draw[thick] (5,0)--(6,0);
 \draw[blue,thick] (6,0)--(7,0);
 \draw (0,0)..controls (0,5) and (6,5)..(6,0);
 \draw (1,0)..controls (1,3) and (7,3)..(7,0);
 \draw (2,0)..controls (2,2) and (4,2)..(4,0);
 \draw (3,0)..controls (3,1.5) and (5,1.5)..(5,0);

 \node[fill=white,inner sep=.5pt,font=\tiny] at (0,0) {$0$};
 \node[fill=white,inner sep=.5pt,font=\tiny] at (1,0) {$1$};
 \node[fill=white,inner sep=.5pt,font=\tiny] at (2,0) {$2$};
 \node[fill=white,inner sep=.5pt,font=\tiny] at (3,0) {$3$};
 \node[font=\small] at (1.5,-2.9) {$T_1$};
 \node[fill=white,inner sep=.5pt,font=\tiny] at (4,0) {$0$};
 \node[fill=white,inner sep=.5pt,font=\tiny] at (5,0) {$1$};
 \node[fill=white,inner sep=.5pt,font=\tiny] at (6,0) {$2$};
 \node[fill=white,inner sep=.5pt,font=\tiny] at (7,0) {$3$};
 \node[font=\small] at (5.5,-2.9) {$T_2$};
\end{tikzpicture}
\]
with a positive first and a negative second tetrahedron and the face
identifications $(1,0)\sim(2,2)$, $(1,1)\sim(2,3)$, $(1,2)\sim(2,0)$
and $(1,3)\sim(2,1)$. This is the one-vertex pseudo-manifold whose
vertex complement $X'$ is the complement of the figure-eight knot,
with its standard ideal triangulation by two tetrahedra
\cite[Section~11.4]{MR3227503}. It has two edges, drawn in black and
blue. As for
the trefoil in Example~\ref{ex:sector-reversal-obstruction}, the space
$X$ is simply connected, so that no nontrivial class extends, while
$H_1(X',\mathbb Z)\cong\mathbb Z$ and $H_2(X',\mathbb Z)=0$.

The states are
\begin{align*}
y^{(1)}&=(p,\,r,\,p+h,\,r-h,\,r,\,p),\\
y^{(2)}&=(q+v,\,q+h,\,q+v+h,\,q,\,q+h,\,q+v),
\end{align*}
where $v=r-p$ and $p,q,r,h\in\mathbb S$. Here $p$, $q$ and $r$ are
coordinates along the orbits of $\mathbb S(X)$. The edges $01$ and
$03$ of the first tetrahedron are identified in $X$, so the segment
joining their midpoints is a loop in $\tilde X$, and the class of the
state takes the value $y^{(1)}_{03}-y^{(1)}_{01}=h$ on this loop.
Hence $h$ is a coordinate on $H^1(X',\mathbb S)\cong\mathbb S$, and we
write $\mathcal F_\hbar(X)(h)$ for the corresponding sector. The two
tetrahedra have the state coordinates $(v,2v)$ and $(h-v,2h-2v)$, and
all paired face coordinates agree as real lifts. By
\eqref{eq:contraction-lifts}, therefore,
\begin{equation}\label{eq:fig8-sector}
e^{-\Imun\pi\ell_X/(4\hbar)}\mathcal F_\hbar(X)(h)
=\int_0^1g_{a_1,c_1}(v,2v)\,
\overline{g_{a_2,c_2}(h-v,2h-2v)}\,dv .
\end{equation}
By \eqref{eq:wgz}, the restriction of $g_{a,c}$ to the line $y=2x$ is a
periodic function of $x$, namely
\[
g_{a,c}(x,2x)=\sum_{m\in\INTEGERS}\varrho_{a,c}(x+m),\qquad
\varrho_{a,c}(t)=e^{2\pi\Imun t^2}\tilde\psi'_{a,c}(t),
\]
so that \eqref{eq:fig8-sector} is a convolution on the circle. The
function $\varrho_{a,c}$ is a Schwartz function by
Lemma~\ref{lem:weight-regularity} and its proof. By the Poisson
summation formula, the Fourier coefficients of $g_{a,c}(x,2x)$ are
therefore the values at the integers of
\[
\hat\varrho_{a,c}(\xi)=\int_{\REALS}
e^{2\pi\Imun t^2-2\pi\Imun\xi t}\,\tilde\psi'_{a,c}(t)\,dt,
\]
and
\begin{equation}\label{eq:fig8-fourier}
e^{-\Imun\pi\ell_X/(4\hbar)}\mathcal F_\hbar(X)(h)
=\sum_{n\in\INTEGERS}\hat\varrho_{a_1,c_1}(n)\,
\overline{\hat\varrho_{a_2,c_2}(-n)}\,e^{2\pi\Imun nh}.
\end{equation}

On the other hand, the kinematical kernel~\eqref{k(x,y)} is
\begin{multline*}
K(X,y_1,y_2)=\int_{\REALS^4}\operatorname{d}(x_1,x_2,x_3,x_4)\,
e^{2\pi\Imun(x_1y_1-x_3y_2)}\\
\times\delta(x_1-x_2+x_3)\delta(x_3-x_4+y_1)
\delta(x_3-x_4+x_1)\delta(x_1-x_2+y_2)
=e^{2\pi\Imun(y_1^2-y_2^2)},
\end{multline*}
so that the original partition function factorizes:
\begin{equation}\label{eq:fig8-old}
Z_\hbar(X)e^{-\Imun\pi\ell_X/(4\hbar)}
=\hat\varrho_{a_1,c_1}(0)\,\overline{\hat\varrho_{a_2,c_2}(0)} .
\end{equation}
This is the term $n=0$ of \eqref{eq:fig8-fourier}. In other words,
$Z_\hbar(X)=\int_0^1\mathcal F_\hbar(X)(h)\,dh$, in agreement with
\eqref{eq:inter2}.

The other Fourier coefficients are of the same nature. By
\eqref{eq:psi}, \eqref{eq:tpsi} and \eqref{eq:tpsip}, for real $\xi$,
\begin{equation}\label{eq:fig8-shift}
\tilde\psi'_{a,c}\bigl(t+\tfrac\xi2\bigr)
=e^{-\pi\Imun\cla(c-\frac16)\xi}\,
\tilde\psi'_{a+\xi/(4\cla),c}(t).
\end{equation}
This is \eqref{eq:gauge-local} with $\varepsilon=1$, $D=0$ and
$A=\xi/(4\cla)$, read before the WGZ transformation. The shift of
$a$ is purely imaginary, with $c$ kept fixed, and amounts to a real
translation of the argument of the quantum dilogarithm, so no question
of convergence arises. On every compact segment of this shift the
integrand stays holomorphic near the real cycle with uniform
Schwartz bounds. The continued integral is therefore the original
integral continued along this segment, by uniqueness in
Theorem~\ref{thm:global-continuation}. Completing the square in the definition of
$\hat\varrho_{a,c}$ and translating $t$ by $\xi/2$, we obtain
\[
\hat\varrho_{a,c}(\xi)
=e^{-\frac{\pi\Imun}2\xi^2-\pi\Imun\cla(c-\frac16)\xi}\,
\hat\varrho_{a+\xi/(4\cla),c}(0),
\]
and \eqref{eq:fig8-fourier} becomes
\begin{multline}\label{eq:fig8-shifted}
e^{-\Imun\pi\ell_X/(4\hbar)}\mathcal F_\hbar(X)(h)\\
=\sum_{n\in\INTEGERS}e^{2\pi\Imun nh-\pi\Imun\cla n(c_1-c_2)}\,
\hat\varrho_{a_1+n/(4\cla),c_1}(0)\,
\overline{\hat\varrho_{a_2-n/(4\cla),c_2}(0)} .
\end{multline}
Since $\bar\cla=-\cla$, the last two factors are the value of the
right-hand side of \eqref{eq:fig8-old}, continued holomorphically in
the shapes, at the point where both $a_1$ and $a_2$ are shifted by
$n/(4\cla)$ and $c_1$, $c_2$ are kept fixed; here the factor
$\overline{\hat\varrho_{a_2,c_2}(0)}$ is continued as
$\overline{\hat\varrho_{\bar a_2,c_2}(0)}$. The exponential factor
$e^{-\pi\Imun\cla n(c_1-c_2)}$ amounts to the change of the level by
$n(c_1-c_2)/\cla$ in
$e^{\Imun\pi\ell_X/(4\hbar)}=e^{-\pi\Imun\cla^2\ell_X}$, which is the
change of the level that \eqref{eq:gauge-local} attaches to these
shifts of the shapes. Hence the $n$-th Fourier coefficient of
$\mathcal F_\hbar(X)(h)$ is the value of the continued original
partition function of Theorem~\ref{OldInv} at the shapes with $a_1$
and $a_2$ replaced by $a_1+n/(4\cla)$ and $a_2+n/(4\cla)$, and at the
level $\ell_X+n(c_1-c_2)/\cla$. The
total dihedral angles around the black and the blue edge are
$2\pi(2a_1+c_1+2b_2+c_2)$ and $2\pi(2b_1+c_1+2a_2+c_2)$; they are not
changed by these shifts. Thus the sector function of the new
formulation is the generating function of the original partition
function over this lattice of imaginary shifts, and $Z_\hbar(X)$
itself is its mean value.

The function $(a,c)\mapsto\hat\varrho_{a,c}(0)$ is holomorphic for
$\Re a>0$, $\Re c>0$ and $\Re(a+c)<\frac12$. For fixed $c$ it is not
identically zero as a function of $a$: otherwise $\hat\varrho_{a,c}$
would vanish on the whole real axis by the formula above, although it
is the Fourier transform of the nonzero function $\varrho_{a,c}$.
Hence each Fourier coefficient in
\eqref{eq:fig8-shifted} vanishes only on a proper real-analytic subset
of the space of positive shapes. In particular, for generic shapes
$\mathcal F_\hbar(X)(h)$ is not constant in $h$, in contrast to the
previous example.

Nor is $\mathcal F_\hbar(X)(h)$ an even function of $h$ in general. By
\eqref{eq:fund2} and the multiplier of $L$,
\[
g_{a,c}(x,2x)=e^{\frac{\pi\Imun}3}\,
\overline{g_{b,c}(-x-\tfrac12,-2x-1)},\qquad b=\tfrac12-a-c,
\]
and inserting this into \eqref{eq:fig8-sector} shows that
$\mathcal F_\hbar(X)(-h)$ is equal to $\mathcal F_\hbar(X)(h)$ with the
shapes $(a_1,c_1)$ and $(a_2,c_2)$ of the two tetrahedra replaced by
$(b_2,c_2)$ and $(b_1,c_1)$, at the same level. Thus negation of the
class exchanges the two tetrahedra and replaces $a_1$, $a_2$ by $b_1$,
$b_2$. It does not in general preserve the value when the shapes are
held fixed. Indeed, the continuation described above is holomorphic
for $a_1$ and $a_2$ in the strips $0<\Re a_1<\frac12-c_1$ and
$0<\Re a_2<\frac12-c_2$, since by \eqref{eq:fig8-shift} the weights at
these shapes are exponential factors times real translates of weights
at positive shapes. These pole-free integral germs coincide with
the global continuation by Theorem~\ref{thm:global-continuation}.
In this continuation, shifting both $a_1$ and
$a_2$ by $1/(4\cla)$ replaces $n$ by $n+1$ in the last two factors of
\eqref{eq:fig8-shifted}, and therefore multiplies
$\mathcal F_\hbar(X)(h)$ by
$e^{\pi\Imun\cla(c_1-c_2)}e^{-2\pi\Imun h}$. If
$\mathcal F_\hbar(X)(-h)=\mathcal F_\hbar(X)(h)$ held for all positive
shapes, the identity theorem would extend this equality to the shifted
shapes, and comparison would give
$(e^{-2\pi\Imun h}-e^{2\pi\Imun h})\mathcal F_\hbar(X)(h)=0$. Then
$\mathcal F_\hbar(X)(h)$ would vanish for $2h\notin\INTEGERS$, hence
for all $h$ by continuity, which contradicts the nonvanishing of the
coefficient with $n=0$ for generic shapes. This is the mechanism of
Example~\ref{ex:sector-reversal-obstruction}.

\subsection{A pseudo-manifold with one edge}\label{sec:one-edge-example}
Let $X$ be given by the diagram
\[
\begin{tikzpicture}[yscale=.4]
 \draw[thick] (1,0)..controls (1,-1) and (3,-1)..(3,0);
 \draw[line width=3pt,white] (0,0)..controls (0,-1) and (2,-1)..(2,0);
 \draw[thick] (0,0)..controls (0,-1) and (2,-1)..(2,0);
 \draw[thick] (0,0)..controls (0,-2) and (3,-2)..(3,0);
 \draw[thick] (0,0)--(3,0);
 \draw[thick] (5,0)..controls (5,-1) and (7,-1)..(7,0);
 \draw[line width=3pt,white] (4,0)..controls (4,-1) and (6,-1)..(6,0);
 \draw[thick] (4,0)..controls (4,-1) and (6,-1)..(6,0);
 \draw[thick] (4,0)..controls (4,-2) and (7,-2)..(7,0);
 \draw[thick] (4,0)--(7,0);
 \draw (0,0)..controls (0,5) and (5,5)..(5,0);
 \draw (1,0)..controls (1,3.5) and (6,3.5)..(6,0);
 \draw (2,0)..controls (2,2) and (7,2)..(7,0);
 \draw (3,0)..controls (3,1) and (4,1)..(4,0);

 \node[fill=white,inner sep=.5pt,font=\tiny] at (0,0) {$0$};
 \node[fill=white,inner sep=.5pt,font=\tiny] at (1,0) {$1$};
 \node[fill=white,inner sep=.5pt,font=\tiny] at (2,0) {$2$};
 \node[fill=white,inner sep=.5pt,font=\tiny] at (3,0) {$3$};
 \node[font=\small] at (1.5,-2.9) {$T_1$};
 \node[fill=white,inner sep=.5pt,font=\tiny] at (4,0) {$0$};
 \node[fill=white,inner sep=.5pt,font=\tiny] at (5,0) {$1$};
 \node[fill=white,inner sep=.5pt,font=\tiny] at (6,0) {$2$};
 \node[fill=white,inner sep=.5pt,font=\tiny] at (7,0) {$3$};
 \node[font=\small] at (5.5,-2.9) {$T_2$};
\end{tikzpicture}
\]
with two positive tetrahedra and the face identifications
$(1,0)\sim(2,1)$, $(1,1)\sim(2,2)$, $(1,2)\sim(2,3)$ and
$(1,3)\sim(2,0)$. All twelve edges of the two tetrahedra are
identified, and the link of the vertex is a surface of genus two.
Every face has the boundary word $e\,e\,e^{-1}$ in the single edge
$e$, so $X$ is simply connected. The dual spine of $X'$ has two
vertices, four edges and one two-cell, whose boundary is the primitive
cycle $(1,-1,1,-1)$ in terms of the four dual edges, taken in the
order of the face identifications above and oriented from the first
tetrahedron to the second. Hence
$H_1(X',\mathbb Z)\cong\mathbb Z^2$ and $H_2(X',\mathbb Z)=0$. The
single edge has the total dihedral angle $4\pi$, so that
$\Gamma=\Delta_1(X)$ is the only possible choice.

The states are
\begin{align*}
y^{(1)}&=p+(2h_1-h_2,\,h_1,\,2h_1-3h_2,\,h_1-h_2,\,2h_1-2h_2,\,0),\\
y^{(2)}&=q+(h_1+h_2,\,h_1-h_2,\,2h_1-2h_2,\,h_1,\,h_2,\,0),
\end{align*}
where $p,q\in\mathbb S$ are added to all six components and
$h_1,h_2\in\mathbb S$. The group $\mathbb S(X)$ only translates $p$
and $q$. The three edges of the face $(1,3)$ are identified in $X$, so
the sides of its midpoint triangle are loops in $\tilde X$, and the
class of the state takes the values $y^{(1)}_{01}-y^{(1)}_{12}=h_1$
and $y^{(1)}_{02}-y^{(1)}_{12}=h_2$ on two of them. Hence $(h_1,h_2)$
is a coordinate on $H^1(X',\mathbb S)\cong\mathbb S^2$, and we write
$\mathcal F_\hbar(X)(h_1,h_2)$ for the corresponding sector. The
tetrahedral weights do not depend on $p$ and $q$: the state
coordinates are $(2h_2,h_1-h_2)$ and $(2h_2-2h_1,-h_2)$. All paired
face coordinates agree as real lifts. By \eqref{eq:contraction-lifts},
no integration remains, and
\begin{equation}\label{eq:one-edge-sector}
\mathcal F_\hbar(X)(h_1,h_2)=e^{\Imun\pi\ell_X/(4\hbar)}\,
g_{a_1,c_1}(2h_2,h_1-h_2)\,g_{a_2,c_2}(2h_2-2h_1,-h_2).
\end{equation}
Thus the sector function is the product of the two tetrahedral weights
themselves. The line bundle $L$ is not trivial, so each of its smooth
sections has zeros on $\Pi$. As $(h_1,h_2)\mapsto(2h_2,h_1-h_2)$ maps
$\mathbb S^2$ onto $\Pi$, there are sectors, depending on the shapes,
in which the right-hand side of \eqref{eq:one-edge-sector} vanishes.

The kinematical kernel is
$K(X,y_1,y_2)=\frac12e^{-\pi\Imun y_1y_2}$, and \eqref{eq:inter2}
becomes the identity
\begin{multline*}
\frac12\int_{\REALS^2}e^{-\pi\Imun y_1y_2}\,
\tilde\psi'_{a_1,c_1}(y_1)\tilde\psi'_{a_2,c_2}(y_2)
\operatorname{d}(y_1,y_2)\\
=\int_{[0,1]^2}g_{a_1,c_1}(2h_2,h_1-h_2)\,
g_{a_2,c_2}(2h_2-2h_1,-h_2)\,dh_1\,dh_2 .
\end{multline*}
It can also be verified directly by inserting \eqref{eq:wgz}. With
$y_1=2h_2+m$ and $y_2=2h_2-2h_1+n$ the integrand on the right is the
sum over $m,n\in\INTEGERS$ of
$(-1)^{mn}e^{-\pi\Imun y_1y_2}
\tilde\psi'_{a_1,c_1}(y_1)\tilde\psi'_{a_2,c_2}(y_2)$. The map
$(h_1,h_2,m,n)\mapsto(y_1,y_2)$ from $[0,1)^2\times\INTEGERS^2$ to
$\REALS^2$ is four-to-one, and $dy_1\,dy_2=4\,dh_1\,dh_2$. Of the four
preimages of a point, the two with even $m$ contribute equal terms,
while the two with odd $m$ have consecutive values of $n$ and cancel.
Hence the integral on the right is $\frac24$ of the integral of
$e^{-\pi\Imun y_1y_2}\tilde\psi'_{a_1,c_1}(y_1)\tilde\psi'_{a_2,c_2}(y_2)$
over $\REALS^2$.

\subsection{The complement of the trefoil knot}\label{sec:trefoil-example}
Let $X$ be the two-tetrahedron triangulation with the
order-preserving face identifications
\[
(1,i)\sim(2,3-i),\qquad i=0,1,2,3,
\]
shown in Figure~\ref{fig:trefoil-example}. Both tetrahedra are
positive. There is one vertex and there are two edges: $e_0$ has
the two local edges $03$ as preimages, and $e_1$ has all ten
remaining local edges as preimages. The vertex link is a torus and
$X'$ is the trefoil exterior \cite[Section~11.1]{MR3227503}.

\begin{figure}[htbp]
\centering
\begin{tikzpicture}[yscale=.4,line cap=round]
 \draw[thick] (1,0)..controls (1,-1) and (3,-1)..(3,0);
 \draw[line width=3pt,white] (0,0)..controls (0,-1) and (2,-1)..(2,0);
 \draw[thick] (0,0)..controls (0,-1) and (2,-1)..(2,0);
 \draw[blue,thick] (0,0)..controls (0,-2) and (3,-2)..(3,0);
 \draw[thick] (0,0)--(3,0);
 \draw[thick] (5,0)..controls (5,-1) and (7,-1)..(7,0);
 \draw[line width=3pt,white] (4,0)..controls (4,-1) and (6,-1)..(6,0);
 \draw[thick] (4,0)..controls (4,-1) and (6,-1)..(6,0);
 \draw[blue,thick] (4,0)..controls (4,-2) and (7,-2)..(7,0);
 \draw[thick] (4,0)--(7,0);
 \draw (0,0)..controls (0,5.5) and (7,5.5)..(7,0);
 \draw (1,0)..controls (1,4) and (6,4)..(6,0);
 \draw (2,0)..controls (2,2.5) and (5,2.5)..(5,0);
 \draw (3,0)..controls (3,1) and (4,1)..(4,0);
 \foreach \j in {0,1,2,3} {
  \node[fill=white,inner sep=.5pt,font=\tiny] at (\j,0) {$\j$};
  \node[fill=white,inner sep=.5pt,font=\tiny] at ({4+\j},0) {$\j$};
 }
 \node[font=\small] at (1.5,-2.9) {$T_1\ (+)$};
 \node[font=\small] at (5.5,-2.9) {$T_2\ (+)$};
\end{tikzpicture}
\caption{The trefoil exterior: two positive tetrahedra with
$(1,i)\sim(2,3-i)$. The blue edges represent $e_0$ and the black
edges represent $e_1$.}
\label{fig:trefoil-example}
\end{figure}

We have $H_1(X',\mathbb Z)\cong\mathbb Z$ and
$H_2(X',\mathbb Z)=0$, whereas $X$ is simply connected: coning off
the boundary torus kills the meridian, which normally generates the
knot group. Thus $C_X\cong\mathbb S$, with no nonzero extending
class. The total angles at $e_0$ and $e_1$ are
\[
2\pi(c_1+c_2),\qquad 2\pi(2-c_1-c_2),
\]
respectively. Neither edge can be balanced at strictly positive
shapes, since $c_1+c_2<1$, so we take $\Gamma=\Delta_1(X)$.

In the edge order $(01,02,03,12,13,23)$, the compatible states are
\begin{align*}
y^{(1)}&=(p,p-h,p+u-h,p-2h,p-h,p),\\
y^{(2)}&=(q,q+h,q+u,q+2h,q+h,q),
\end{align*}
where $p,q,u,h\in\mathbb S$. The variables $p,q$ are the
tetrahedron constants and $u$ parametrizes the remaining exact
direction. The segment from the midpoint of $02$ to that of $01$
in $T_1$ becomes a loop, on which the cocycle has value $h$; this
identifies $h$ with the sector coordinate. All paired face
coordinates agree as real lifts, and
\[
E_1=(h-u,-2h),\qquad E_2=(-u,2h).
\]
Putting $s=-u$, formula~\eqref{eq:contraction-lifts} gives
\begin{equation}\label{eq:trefoil-sector}
e^{-\Imun\pi\ell_X/(4\hbar)}\mathcal F_\hbar(X)(h)
=\int_0^1 g_{a_1,c_1}(s+h,-2h)g_{a_2,c_2}(s,2h)\,ds.
\end{equation}
The product is periodic in both $s$ and $h$; no extra face phase
occurs. At level zero this is the family of
Example~\ref{ex:sector-reversal-obstruction}.

There is a useful real-line expression for these sectors. Write
$f_j=\tilde\psi'_{a_j,c_j}$ and define the Schwartz function
\[
R_{12}(r)=e^{-2\pi\Imun r^2}
\int_{\mathbb R}f_1(t+r)f_2(t)\,dt.
\]
Expanding both weights by \eqref{eq:wgz}, setting $k=m-n$ for
their summation indices and unfolding $t=s+n$, we obtain
\begin{align}
e^{-\Imun\pi\ell_X/(4\hbar)}\mathcal F_\hbar(X)(h)
&=\sum_{k\in\mathbb Z}e^{-2\pi\Imun h^2-4\pi\Imun kh}
  \int_{\mathbb R}f_1(t+h+k)f_2(t)\,dt\nonumber\\
&=\sum_{k\in\mathbb Z}R_{12}(h+k).
\label{eq:trefoil-periodization}
\end{align}
Here $e^{2\pi\Imun k^2}=1$. Schwartz decay justifies the
unfolding and gives normal convergence with all derivatives.
Equivalently, the Fourier coefficient in sector frequency $n$ is
\[
\widehat R_{12}(n)=\int_{\mathbb R^2}
e^{-2\pi\Imun(y_1-y_2)^2-2\pi\Imun n(y_1-y_2)}
f_1(y_1)f_2(y_2)\,dy_1\,dy_2.
\]

To compute the original partition function, let $x_i$ be the face
variable on $(1,i)\sim(2,3-i)$. Its kinematical kernel is
\begin{multline*}
K(X,y_1,y_2)=\int_{\mathbb R^4}
e^{2\pi\Imun(x_0y_1+x_3y_2)}
\delta(x_0-x_1+x_2)\delta(x_2-x_3+y_1)\\
\times\delta(x_3-x_2+x_1)\delta(x_1-x_0+y_2)
\,dx_0\,dx_1\,dx_2\,dx_3
=e^{-2\pi\Imun(y_1-y_2)^2}.
\end{multline*}
Indeed the constraints have determinant of absolute value one and
give
\[
(x_0,x_1,x_2,x_3)=(y_2-y_1,-y_1,-y_2,y_1-y_2).
\]
Consequently
\begin{equation}\label{eq:trefoil-old}
e^{-\Imun\pi\ell_X/(4\hbar)}Z_\hbar(X)
=\widehat R_{12}(0)=\int_{\mathbb R}R_{12}(r)\,dr.
\end{equation}
Taking the mean of \eqref{eq:trefoil-periodization} proves directly
that $Z_\hbar(X)=\int_0^1\mathcal F_\hbar(X)(h)\,dh$, with the
probability Haar normalization of Theorem~\ref{RelationON}.
Finally, $R_{21}(r)=R_{12}(-r)$ shows that sector reversal exchanges
the two shape pairs. It need not preserve the amplitude at fixed
shapes, as proved in Example~\ref{ex:sector-reversal-obstruction}.

\subsection{A one-tetrahedron presentation of the Hopf link}
\label{sec:hopf-example}
Consider one positive tetrahedron with face identifications
\[
(1,0)\sim(1,1),\qquad (1,2)\sim(1,3),
\]
as in Figure~\ref{fig:hopf-example}. The two vertex classes are
$v_A=\{0,1\}$ and $v_B=\{2,3\}$. There are three edge classes:
\[
e_A=\{01\},\qquad e_B=\{23\},\qquad
e_C=\{02,03,12,13\}.
\]
The first two are loops at the different vertices, and $e_C$ joins
the vertices.

\begin{figure}[htbp]
\centering
\begin{tikzpicture}[line cap=round]
 \begin{scope}[yscale=.4]
  \draw[thick] (1,0)..controls (1,-1) and (3,-1)..(3,0);
  \draw[line width=3pt,white] (0,0)..controls (0,-1) and (2,-1)..(2,0);
  \draw[thick] (0,0)..controls (0,-1) and (2,-1)..(2,0);
  \draw[thick] (0,0)..controls (0,-2) and (3,-2)..(3,0);
  \draw[blue,thick] (0,0)--(1,0);
  \draw[thick] (1,0)--(2,0);
  \draw[red,thick] (2,0)--(3,0);
  \draw (0,0)..controls (0,1.6) and (1,1.6)..(1,0);
  \draw (2,0)..controls (2,1.6) and (3,1.6)..(3,0);
  \foreach \j in {0,1,2,3}
   \node[fill=white,inner sep=.5pt,font=\tiny] at (\j,0) {$\j$};
  \node[font=\small] at (1.5,-2.9) {$T_1\ (+)$};
 \end{scope}
 \begin{scope}[shift={(5.4,0)},scale=.72]
  \draw[red,thick] (1.5,0) circle (1);
  \draw[white,line width=3.5pt] (0,0) circle (1);
  \draw[blue,thick] (0,0) circle (1);
  \draw[white,line width=3.5pt] ({1.5+cos(210)},{sin(210)})
    arc[start angle=210,end angle=233,radius=1];
  \draw[red,thick] ({1.5+cos(210)},{sin(210)})
    arc[start angle=210,end angle=233,radius=1];
  \fill[blue] (-1,0) circle (2pt);
  \fill[red] (2.5,0) circle (2pt);
  \node[blue,left,font=\small] at (-1,0) {$v_A$};
  \node[red,right,font=\small] at (2.5,0) {$v_B$};
  \node[font=\small] at (.75,-1.6) {$H=e_A\cup e_B$};
 \end{scope}
\end{tikzpicture}
\caption{Left: the two-vertex one-tetrahedron triangulation.
The loop edges $e_A$ and $e_B$ are blue and red; the joining edge
$e_C$ is black. Right: the Hopf sublink formed by the two loop edges.}
\label{fig:hopf-example}
\end{figure}

These pairings give a triangulation of $S^3$ containing the Hopf
link, rather than an ideal triangulation of its complement.
This distinction follows explicitly from barycentric coordinates
$(t_0,t_1,t_2,t_3)$ on the tetrahedron. The map
\[
(t_0,t_1,t_2,t_3)\longmapsto
\left(\sqrt{t_0+t_1}\,e^{2\pi\Imun t_1/(t_0+t_1)},
      \sqrt{t_2+t_3}\,e^{2\pi\Imun t_3/(t_2+t_3)}\right)
\]
to the unit sphere in $\mathbb C^2$, with a zero coordinate whenever
its square-root factor is zero, identifies exactly the prescribed
faces. It therefore induces a homeomorphism $X\cong S^3$.
The images of $e_A$ and $e_B$ are the coordinate circles
$\{(z,0):|z|=1\}$ and $\{(0,z):|z|=1\}$, respectively.
Thus $H=e_A\cup e_B$ is the Hopf link and
\[
X\setminus H\cong\Pi\times(0,1),\qquad
X\setminus\operatorname{int}N(H)\cong\Pi\times[0,1],
\]
where $\Pi=(\mathbb R/\mathbb Z)^2$ is the two-torus.
The two vertices of $X$, however, have spherical links, so the
vertex complement used in the state integral is
$X'\cong S^2\times[0,1]$. In particular
\[
C_X=0,\qquad H_2(X',\mathbb Z)\cong\mathbb Z.
\]

Writing the shape as $(a,b,c)$, the total angles at
$(e_A,e_B,e_C)$ are $2\pi(a,a,1-2a)$. Arbitrary positive shapes
therefore require $\Gamma=\Delta_1(X)$. If only the Hopf sublink
were retained in $\Gamma$, balancing $e_C$ would require $a=0$,
on the boundary of the positive shape domain.

The face constraints give all states as
\[
y=(p,q,q,q,q,r),\qquad p,q,r\in\mathbb S.
\]
Paired face coordinates agree as real lifts and
$E(y)=(0,2q-p-r)$. Since the map
$(p,q,r)\mapsto 2q-p-r$ pushes Haar probability to Haar
probability, the single sector is
\begin{align}
e^{-\Imun\pi\ell_X/(4\hbar)}\mathcal F_\hbar(X)(0)
&=\int_0^1g_{a,c}(0,t)\,dt
=\tilde\psi'_{a,c}(0)\nonumber\\
&=e^{-\pi\Imun/12}\psi_{c,b}(0).
\label{eq:hopf-presentation-sector}
\end{align}
The middle equality follows by extracting the zeroth Fourier
coefficient in \eqref{eq:wgz}.

The original partition function is not defined for this
two-vertex triangulation. Indeed, writing $x$ on the paired faces
$0,1$ and $z$ on the paired faces $2,3$, its formal kinematical
expression is
\[
\int_{\mathbb R^2}e^{2\pi\Imun xy}\delta(z)\delta(y)\,dx\,dz.
\]
After the constraints are imposed the $x$-integration has infinite
volume. This agrees with the failure of admissibility and does not
affect the compact integral \eqref{eq:hopf-presentation-sector}.
The latter belongs to the specified triangulation of $S^3$; its
sector space must not be confused with
$H^1(X\setminus H;\mathbb S)\cong\mathbb S^2$ of the link exterior.

\subsection{The three-sphere with an unknot--trefoil wedge}
\label{sec:wedge-example}
Let $X$ consist of one positive tetrahedron with
\[
(1,i)\sim(1,3-i),\qquad i=0,1,
\]
as in Figure~\ref{fig:wedge-example}. This is the one-vertex
triangulation of $S^3$ used in \cite[Section~11.2]{MR3227503}.
It has two edges,
\[
e_T=\{03\},\qquad e_U=\{01,02,12,13,23\}.
\]
The edge $e_T$ is a trefoil and $e_U$ is an unknot; they meet only
at the common vertex. We take
$\Gamma=e_U\cup e_T=\Delta_1(X)$, the embedded wedge of these
two circles. Its complement is not used to label the sectors:
the vertex complement $X'$ is a three-ball, so $C_X=0$ and
$H_2(X',\mathbb Z)=0$.

\begin{figure}[htbp]
\centering
\begin{tikzpicture}[yscale=.4,line cap=round]
 \draw[thick] (1,0)..controls (1,-1) and (3,-1)..(3,0);
 \draw[line width=3pt,white] (0,0)..controls (0,-1) and (2,-1)..(2,0);
 \draw[thick] (0,0)..controls (0,-1) and (2,-1)..(2,0);
 \draw[blue,thick] (0,0)..controls (0,-2) and (3,-2)..(3,0);
 \draw[thick] (0,0)--(3,0);
 \draw (0,0)..controls (0,3) and (3,3)..(3,0);
 \draw (1,0)..controls (1,1.3) and (2,1.3)..(2,0);
 \foreach \j in {0,1,2,3}
  \node[fill=white,inner sep=.5pt,font=\tiny] at (\j,0) {$\j$};
 \node[font=\small] at (1.5,-2.9) {$T_1\ (+)$};
\end{tikzpicture}
\caption{The one-tetrahedron triangulation of
$(S^3,\Gamma)$ with $\Gamma=e_U\vee e_T$. The blue edge $03$
represents the trefoil $e_T$; the other five local edges, in black,
represent the unknot $e_U$. The lower arcs encode edge identifications,
not a planar projection of the embedded graph.}
\label{fig:wedge-example}
\end{figure}

For the shape $(a,b,c)$ the total angles are
\[
\omega(e_T)=2\pi c,\qquad \omega(e_U)=2\pi(1-c).
\]
Both edges are unbalanced for positive shapes. Keeping both in
$\Gamma$ permits the full positive shape domain. Retaining only
the trefoil would require the limiting condition $c=0$ to balance
$e_U$.

The compatible states are
\[
y=(p,p,p+u,p,p,p),\qquad p,u\in\mathbb S,
\]
where $p$ is the tetrahedron constant. The face pairs have equal
real lifts, $E(y)=(-u,0)$, and therefore
\begin{align}
e^{-\Imun\pi\ell_X/(4\hbar)}\mathcal F_\hbar(X)(0)
&=\int_0^1g_{a,c}(-u,0)\,du
=\int_{\mathbb R}\tilde\psi'_{a,c}(t)\,dt\nonumber\\
&=e^{-\pi\Imun/6}\psi_{b,a}(0).
\label{eq:wedge-sector}
\end{align}
The last equality is the single-weight identity
\eqref{eq:rp-second-factor}, with the tetrahedron index omitted.
Here, unlike \eqref{eq:hopf-presentation-sector}, the circle integral
unfolds the first coordinate of the WGZ transform.

The old kinematical kernel is equally simple. Put $x$ on the
paired faces $0,3$ and $z$ on the paired faces $1,2$. Then
\[
K(X,y)=\int_{\mathbb R^2}e^{2\pi\Imun xy}
\delta(x)\delta(z-x+y)\,dx\,dz=1.
\]
Consequently
\[
Z_\hbar(X,\Gamma)=\mathcal F_\hbar(X)(0)
=e^{\Imun\pi\ell_X/(4\hbar)-\pi\Imun/6}\psi_{b,a}(0),
\]
in agreement with the one-sector comparison of
Theorem~\ref{integral-homology-sphere-comparison}. All compact
formulas in these three examples are initially valid for positive
shapes and continue meromorphically by
Theorem~\ref{thm:global-continuation}.

\appendix

\section{Weil--Gel'fand--Zak transformation}
\label{WGZ}
We define the Weil--Gel'fand--Zak (WGZ) transformation
\cite{MR0005741,MR0039154,Zak1967}
$$ W : \mathcal{S}(\REALS) \ra C^\infty(\Pi, L)$$
by the formula
\begin{equation}\label{eq:wgz}
(Wf)(x,y)=e^{\pi\Imun xy}\sum_{m\in\INTEGERS}f(x+m)e^{2\pi\Imun my},\quad (\overline{Wf})(x,y)=(W\bar f)(x,-y),
\end{equation}
where $ \mathcal{S}(\REALS)$ is the space of Schwartz class functions on the real line.
\begin{proposition}\label{prop:wgz-iso}
The WGZ-transformation $ W : \mathcal{S}(\REALS) \ra C^\infty(\Pi, L)$ is an isomorphism of Fr\'echet spaces, with the inverse explicitly given by the formula
\[
(W^{-1}g)(x)=\int_0^1g(x,y)e^{-\pi\Imun xy}\operatorname{d}\!y
\]
for any $g\in C^\infty(\Pi, L)$. Here $C^\infty(\Pi,L)$ carries the
topology of uniform convergence of all derivatives of the lifts on
the fundamental domain $[0,1]^2$.
\end{proposition}
\begin{proof}
For $f\in\mathcal S(\REALS)$ the series \eqref{eq:wgz} and all its
termwise derivatives converge uniformly on compact sets. The
$C^r$-norms of $Wf$ on $[0,1]^2$ are bounded by Schwartz seminorms of
$f$, and termwise integration gives
$\int_0^1(Wf)(x,y)e^{-\pi\Imun xy}\operatorname{d}\!y=f(x)$.
Conversely, let $g\in C^\infty(\Pi,L)$ and put
$h(x,y)=g(x,y)e^{-\pi\Imun xy}$. The multipliers of $L$ give
\[
h(x,y+1)=h(x,y),\qquad h(x+1,y)=e^{-2\pi\Imun y}h(x,y).
\]
Hence $h(x,y)=\sum_{k\in\INTEGERS}h_k(x)e^{2\pi\Imun ky}$ with
$h_k(x)=\int_0^1h(x,y)e^{-2\pi\Imun ky}\operatorname{d}\!y$, and the
second relation gives $h_k(x+1)=h_{k+1}(x)$. Thus $h_k(x)=f(x+k)$,
where $f=h_0$ is the function given by the displayed formula. For
$x\in[0,1]$ the number $f^{(j)}(x+k)$ is the $k$-th Fourier coefficient
of the smooth periodic function $\partial_x^jh(x,\cdot)$, so that
\[
\sup_{x\in[0,1]}|k|^N\,|f^{(j)}(x+k)|\le(2\pi)^{-N}
\sup_{[0,1]^2}|\partial_y^N\partial_x^jh|
\]
for all $N,j\ge0$. Therefore $f$ is a Schwartz function, its
seminorms are bounded by $C^r$-norms of $g$ on $[0,1]^2$, and
$g=Wf$. This shows that $W$ is bijective and that $W$ and $W^{-1}$
are continuous. The same argument applies to the transformation in
several variables, which is used in Section~\ref{ONrel}.
\end{proof}

Let us now introduce two linear transformations
$$ G :  \mathcal{S}(\REALS) \ra  \mathcal{S}(\REALS) \mbox{, } F :  \mathcal{S}(\REALS) \ra  \mathcal{S}(\REALS),$$
defined by
\begin{equation}\label{eq:g}
G(f)(x):= e^{\pi \Imun x^2}f(x)
\end{equation}
and
\begin{equation}\label{eq:f}
F(f)(x):= \int_{\REALS}f(y)e^{2\pi\Imun xy}\operatorname{d}\!y. 
\end{equation}

\begin{lemma}\label{lemma1}For any function $f\in\mathcal S(\REALS)$, the following identities hold
\begin{equation}\label{eq:wf}
WF(f)(x,y)=W(f)(-y,x),\quad WF^{-1}(f)(x,y)=W(f)(y,-x),
\end{equation}
and 
\begin{multline}\label{eq:wg}
WG(f)(x,y)=W(f)\left(x,x+y+1/2\right)e^{-\pi\Imun x/2},\\
WG^{-1}(f)\left(x,y\right)=W(f)(x,y-x-1/2)e^{\pi\Imun x/2}.
\end{multline}
\end{lemma}
\begin{proof}
For \eqref{eq:wf}, we have that
\begin{multline*}
F(f)(x)=\int_\REALS f(y)e^{2\pi \Imun xy}dy=\sum_{m\in\INTEGERS}
\int_m^{m+1}f(y)e^{2\pi \Imun xy}dy\\
=\sum_{m\in\INTEGERS}
\int_0^1f(y+m)e^{2\pi \Imun x(y+m)}dy=\int_0^1W(f)(y,x) e^{\pi \Imun xy}dy=\int_0^1h(x,y)dy,
\end{multline*}
where the function $h(x,y):= W(f)(y,x) e^{\pi \Imun xy}$ has the following properties:
\[
h(x+1,y)=e^{2\pi\Imun y}h(x,y),\quad h(x,y+1)=h(x,y).
\]
In particular, it can be expanded in Fourier series
\[
h(x,y)=\sum_{m\in\INTEGERS}h_m(x)e^{2\pi\Imun ym},
\]
where 
\[
h_m(x)=\int_{[0,1]}h(x,y)e^{-2\pi\Imun ym}dy=\int_{[0,1]}h(x-m,y)dy=F(f)(x-m).
\]
Thus, 
\[
W(f)(y,x) =h(x,y)e^{-\pi \Imun xy}=e^{-\pi \Imun xy}\sum_{m\in\INTEGERS}F(f)(x-m)e^{2\pi\Imun ym}=WF(f)(x,-y).
\]
For \eqref{eq:wg}, we have that
\begin{multline*}
WG(f)(x,y)=e^{\pi\Imun xy}\sum_{m\in\INTEGERS}e^{\pi\Imun (x^2+2xm+m)}f(x+m)e^{2\pi\Imun my}\\
=e^{-\pi\Imun x/2}(Wf)(x,x+y+1/2).
\end{multline*}
\end{proof}

\section{Faddeev's quantum dilogarithm}
\label{QDL}
Following notation of \cite{MR3227503}, Faddeev's quantum dilogarithm $\QDILOG(z)$ is defined by the integral
   \begin{equation}\label{eq:ncqdl}
\QDILOG(z):=\exp\left(
\int_{\Imun 0-\infty}^{\Imun 0+\infty}
\frac{e^{-\Imun 2 zw}\, dw}{4\sinh(w\la)
\sinh(w/\la) w}\right)
    \end{equation}
in the strip $|\Im z|<|\Im\cla|$, where
\(
\cla:=\Imun(\la+\la^{-1})/2.
\)
Define also
\begin{equation}\label{eq:cons}
\zeta_{inv}:= e^{\Imun\pi(1+2\cla^2)/6}=e^{\Imun\pi\cla^2}
\zeta_o^2,\quad
\zeta_o:= e^{\Imun\pi(1-4\cla^2)/12}
\end{equation}
When  $\Im\la^2>0$, the integral can be calculated explicitly
\begin{equation}\label{eq:ratio}
\QDILOG(z)=
(e^{2\pi (z+\cla)\la};q^2)_\infty/
(e^{2\pi (z-\cla)\la^{-1}};\bar q^2)_\infty
\end{equation}
where
\[
q:= e^{\Imun\pi\la^2},\quad\bar q:= e^{-\Imun\pi\la^{-2}}.
\]
Using symmetry properties
\[
\QDILOG(z)=\mathop{\Phi_{-\la}}(z)=\mathop{\Phi_{1/\la}}(z)
\]
we choose
\[
\Re\la>0,\quad \Im \la\ge0.
\]
The function $\QDILOG(z)$ can be continued analytically in the variable
$z$ to the entire complex plane as a meromorphic function with an essential singularity at infinity and with the following characteristic properties
\begin{description}
\item[zeros and poles]
  \begin{equation}
    \label{eq:polzer}
(\QDILOG(z))^{\pm1}=0\ \Leftrightarrow \ z=
\mp(\cla+m\Imun\la+n\Imun\la^{-1}),\ m,n\in\INTEGERS_{\ge0};
  \end{equation}
\item[behavior at infinity] depending on the direction along which the limit is taken
\begin{equation}\label{eq:asymp}
\QDILOG(z)\bigg\vert_{|z|\to\infty} \approx\left\{
\begin{array}{ll}
1&|\arg z|>\frac{\pi}{2}+\arg \la\\
\zeta_{inv}^{-1}e^{\Imun\pi z^2}&|\arg z|<\frac{\pi}{2}-\arg\la\\
\frac{(\bar q^2;\bar q^2)_\infty}{\Theta(\Imun\la^{-1}z;-\la^{-2})}&
|\arg z-\frac\pi2|<\arg\la\\
\frac{\Theta(\Imun\la z;\la^{2})}{(q^2; q^2)_\infty}&
|\arg z+\frac\pi2|<\arg\la
\end{array}\right.
\end{equation}
where
\[
\Theta(z;\tau):=\sum_{n\in\INTEGERS}
e^{\Imun\pi\tau n^2+\Imun2\pi zn}, \quad\Im\tau>0;
\]
\item[inversion relation]
\begin{equation}\label{eq:inversion}
\QDILOG(z)\QDILOG(-z)=\zeta_{inv}^{-1}e^{\Imun\pi z^2};
\end{equation}
\item[functional equations]
\begin{equation}\label{eq:shift}
\QDILOG(z-\Imun\la^{\pm1}/2)=(1+e^{2\pi\la^{\pm1}z })
\QDILOG(z+\Imun\la^{\pm1}/2);
\end{equation}
\item[unitarity] when  $\la$ is real or on the unit circle
\begin{equation}\label{eq:qdlunitarity}
(1-|\la|)\Im\la=0\ \Rightarrow\
\overline{\QDILOG(z)}=1/\QDILOG(\bar z);
\end{equation}
\item[quantum pentagon identity]
\begin{equation}\label{eq:pent}
\QDILOG(\MOM)\QDILOG(\POS)=\QDILOG(\POS)\QDILOG(\MOM+\POS)
\QDILOG(\MOM),
\end{equation}
where $\MOM$ and $\POS$ are the standard Schr\"odinger pair on
$L^2(\REALS)$: on the common Schwartz core,
\[
(\POS f)(x)=xf(x),\qquad
(\MOM f)(x)=(2\pi\Imun)^{-1}f'(x).
\]
We use their self-adjoint closures and the self-adjoint closure of
$\MOM+\POS$, with the quantum dilogarithms defined by spectral
calculus. Thus the quoted identity is in the regular Weyl
representation; the relation
$\MOM\POS-\POS\MOM=(2\pi\Imun)^{-1}$ holds on the stated core.
\end{description}

The following quantitative form of the first two cases of
\eqref{eq:asymp} is used in
Sections~\ref{sec:analytic-germs} and~\ref{Ancont}.
\begin{lemma}\label{lem:qdl-asymptotics}
Let $\Re\la>0$ and $\Re\la^{-1}>0$, with $\hbar>0$, and put
$r=\Re\la$, $s=\Re\la^{-1}$. For every $R>0$ and every $\rho$ with
$0<\rho<2\pi\min(r,s)$ there is a constant $K>0$ such that, for all
$z\in\COMPLEXS$ with $|\Im z|\le R$,
\begin{align}
\bigl|\QDILOG(z)-1\bigr|&\le Ke^{\rho\Re z}
&&\text{if }\Re z\le-K,\label{eq:left-error}\\
\bigl|\zeta_{inv}e^{-\pi\Imun z^2}\QDILOG(z)-1\bigr|&\le Ke^{-\rho\Re z}
&&\text{if }\Re z\ge K.\label{eq:right-error}
\end{align}
\end{lemma}
\begin{proof}
Shift the contour defining $\log\QDILOG(z)$ in \eqref{eq:ncqdl} to
$\mathbb R+\Imun\rho/2$. No pole is crossed, because the nonzero poles
of the integrand in the upper half-plane have imaginary parts at least
$\pi\min(r,s)$. On the shifted contour
$|e^{-2\Imun zw}|=e^{\rho\Re z}e^{2\Im z\,\Re w}$, while the
denominator grows like $e^{(r+s)|\Re w|}$. Hence, on any closed
substrip of $|\Im z|<|\Im\cla|=(r+s)/2$, the shifted integral is
bounded by a constant times $e^{\rho\Re z}$. Exponentiation gives
\eqref{eq:left-error} there. The functional equations
\eqref{eq:shift} translate $z$ by $\Imun\la^{\pm1}$, and thus change
$\Im z$ by $r$ or $s$ and $\Re z$ by a bounded amount. Their factors
$1+e^{2\pi\la^{\pm1}z}$ differ from one by at most a constant times
$e^{\rho\Re z}$ when $\Re z\le0$ and $\Im z$ is bounded. Finitely many
applications therefore extend the estimate to $|\Im z|\le R$; the
finitely many zeros and poles of $\QDILOG$ in this strip are avoided
by taking $K$ large. Finally, \eqref{eq:inversion} gives
$\zeta_{inv}e^{-\pi\Imun z^2}\QDILOG(z)=1/\QDILOG(-z)$, and
\eqref{eq:right-error} follows from \eqref{eq:left-error}.
\end{proof}

\subsection{Integral analog of \texorpdfstring{$_1\psi_1$}{1-psi-1}-summation formula of Ramanujan}
\label{sec:ramanujan-integral}

Following \cite{MR1828812}, consider the following Fourier integral
\begin{equation}\label{ramanint}
\Psi(u,v,w):=
\int_{\REALS}\frac{\QDILOG(x+u)}{\QDILOG(x+v)}e^{2\pi\Imun wx}\, dx
\end{equation}
where
\begin{equation}\label{restrictions1}
\Im(v+\cla)>0,\quad\Im(-u+\cla)>0, \quad \Im(v-u)<\Im w<0
\end{equation}
Conditions~\eqref{restrictions1} can be relaxed by deforming the integration contour in the complex plane of the variable $x$, keeping asymptotic directions within the sectors
$\pm(|\arg x|-\pi/2)>\arg\la$. Enlarged in this way the domain of the  variables $u,v,w$ has the form
\begin{equation}\label{restrictions2}
|\arg (\Imun z)|<\pi-\arg\la,\quad z\in\{w,v-u-w,u-v-2\cla\}
\end{equation}
Considering $\QDILOG(z)$ as a non-compact analog of the function
$(x;q)_\infty$, definition~\eqref{ramanint} can be interpreted as an integral analog of $_1\psi_1$-summation formula of Ramanujan.

By the method of residues, \eqref{ramanint} can be calculated explicitly
\begin{gather}\label{ramanres1}
  \Psi(u,v,w)=
  \zeta_o\frac{\QDILOG(u-v-\cla)\QDILOG(w+\cla)}{\QDILOG(u-v+w-\cla)}
  e^{-2\pi\Imun w(v+\cla)} \\\label{ramanres2}
  =\zeta_o^{-1}\frac{\QDILOG(v-u-w+\cla)}{\QDILOG(v-u+\cla)
\QDILOG(-w-\cla)}
  e^{-2\pi\Imun w(u-\cla)},
\end{gather}
where the two expressions in the right hand side are related by the inversion relation~\eqref{eq:inversion}.
\subsection{Fourier transformation formulae}
\label{sec:qdl-fourier}

For physical $\la$, the real-axis Fourier transforms in this
subsection are understood as tempered distributions. Indeed,
\eqref{eq:qdlunitarity} gives $|\QDILOG(x)|=1$ for real $x$, so
$\QDILOG^{\pm1}$ are bounded functions and define elements of
$\mathcal S'(\REALS)$. With the Fourier convention of
\eqref{eq:f}, their transforms can equivalently be defined by
Gaussian regularization:
\[
\FQDILOG_\pm(w)=\lim_{\epsilon\downarrow0}
\int_{\REALS}e^{-\epsilon x^2}
\QDILOG(x)^{\pm1}e^{2\pi\Imun wx}\,dx
\qquad\text{in }\mathcal S'(\REALS).
\]
The limit follows from dominated convergence against Schwartz
test functions and continuity of the Fourier transform on
$\mathcal S'(\REALS)$.

The meromorphic expressions on the right-hand sides below are
interpreted on the real axis as lower distributional boundary
values, with $w$ replaced by $w-\Imun0$. These boundary values
are taken locally against compactly supported smooth test
functions; their real-axis asymptotics and their simple pole at
zero make them tempered distributions. Away from zero they
agree with the ordinary values of the meromorphic expressions.
The lower-side prescription agrees with the Fourier convention:
\[
F(\mathbf1_{(-\infty,0)})(w)
=\frac{1}{2\pi\Imun(w-\Imun0)}
=\frac{1}{2\pi\Imun}\operatorname{pv}\frac1w
 +\frac12\delta_0(w).
\]

Special cases of $\Psi(u,v,w)$ lead to the following Fourier
transformation formulae for Faddeev's quantum dilogarithm:
\begin{multline}\label{fourier1}
  \lefteqn{\FQDILOG_+(w):= \int_{\REALS}\QDILOG(x)e^{2\pi \Imun
      wx}\,dx
    =\Psi(0,v,w)\vert_{v\to-\infty}}\\
  =\zeta_o^{-1}e^{2\pi \Imun w\cla}/
  \QDILOG(-w-\cla)=\zeta_oe^{-\Imun\pi w^2}
  \QDILOG(w+\cla)
\end{multline}
and
\begin{multline}\label{fourier2}
  \lefteqn{\FQDILOG_-(w):= \int_{\REALS}(\QDILOG(x))^{-1}e^{2\pi
      \Imun wx}\,dx
    =\Psi(u,0,w)\vert_{u\to-\infty}}\\
  =\zeta_oe^{-2\pi\Imun w\cla} \QDILOG(w+\cla)=
  \zeta_o^{-1}e^{\Imun\pi w^2}/\QDILOG(-w-\cla).
\end{multline}
Here the real-axis integrals and the limits of $\Psi$ are in
$\mathcal S'(\REALS)$, and the equalities with the meromorphic
expressions use the boundary prescription just specified.
For real $u,v$, the distributional extension of $\Psi$ is the
Fourier transform of the bounded quotient in \eqref{ramanint}.
For such $u\ne v$ the domain \eqref{restrictions1} is empty, and the
expression \eqref{ramanres1} has two poles on the real $w$-axis, at
$w=0$ and at $w=v-u$. The Fourier transform is its boundary value
from below at the first pole and from above at the second, as one
sees by approaching real $u$ from $\Im u>0$, where the strip
$\Im(v-u)<\Im w<0$ is not empty. In the two limits the second pole
moves to infinity, and only the prescription $w-\Imun0$ remains.
The limits $v\to-\infty$ and $u\to-\infty$ follow by dominated
convergence, since those quotients have modulus one and converge
pointwise to $\QDILOG$ and $\QDILOG^{-1}$, respectively.

The corresponding inverse transformations are
\begin{equation}\label{finv}
(\QDILOG(x))^{\pm1}=\int_{\REALS}\FQDILOG_\pm(y)e^{-2\pi\Imun xy}dy,
\end{equation}
with the integral interpreted as inverse Fourier transformation
in $\mathcal S'(\REALS)$. A pointwise version, locally uniform
in real $x$, is
\[
\QDILOG(x)^{\pm1}
=\lim_{\epsilon\downarrow0}
\left\langle\FQDILOG_\pm(y),
 e^{-\epsilon y^2-2\pi\Imun xy}\right\rangle.
\]
Indeed, multiplication by the Gaussian on the Fourier side
corresponds to convolution of the bounded continuous function
$\QDILOG^{\pm1}$ with a Gaussian approximate identity. In contour
notation, the pole $y=0$ is passed below; this fixes the lower
boundary value, while the Gaussian prescription controls the
integral at infinity.

\subsection*{Acknowledgements} We would like to thank Christian Blanchet, Tudor Dimofte, Stavros Garoufalidis, Cameron Gordon, Gregor Masbaum, Hitoshi Murakami, Yuri Neretin, Bertrand Patureau-Mirand, Nikolai Reshetikhin and Roland van der Veen for valuable discussions. 

\bibliographystyle{plain}
\def\cprime{$'$}

\end{document}